\documentclass[11pt,reqno]{amsart}
\usepackage{amsmath,amsthm,amsfonts,amssymb,mathrsfs,bm,graphicx,stmaryrd}
\usepackage{mathtools}
\usepackage{dsfont}
\usepackage{multicol}
\usepackage[colorlinks=true,linkcolor=blue]{hyperref}
\usepackage[letterpaper,hmargin=1.0in,vmargin=1.0in]{geometry}
\parskip 	\smallskipamount

\newtheorem{theorem}{Theorem}[section]
\newtheorem{lemma}[theorem]{Lemma}
\newtheorem{corollary}[theorem]{Corollary}
\newtheorem{proposition}[theorem]{Proposition}

\newtheorem{maintheorem}{Theorem}

\def\P{\mathbb{P}}
\def\Z{\mathbb{Z}}
\def\R{\mathbb{R}}

\def\E{\mathbb{E}}
\newcommand{\cA}{\mathcal{A}}
\newcommand{\cB}{\mathcal{B}}
\newcommand{\cC}{\mathcal{C}}
\newcommand{\cD}{\mathcal{D}}

\newcommand{\cF}{\mathcal{F}}
\newcommand{\cG}{\mathcal{G}}
\newcommand{\cH}{\mathcal{H}}
\newcommand{\cI}{\mathcal{I}}
\newcommand{\cJ}{\mathcal{J}}

\newcommand{\cL}{\mathcal{L}}

\newcommand{\cO}{\mathcal{O}}
\newcommand{\cP}{\mathcal{P}}
\newcommand{\cQ}{\mathcal{Q}}
\newcommand{\cR}{\mathcal{R}}
\newcommand{\cS}{\mathcal{S}}
\newcommand{\cT}{\mathcal{T}}
\newcommand{\cU}{\mathcal{U}}
\newcommand{\cV}{\mathcal{V}}
\newcommand{\cW}{\mathcal{W}}

\newcommand{\cZ}{\mathcal{Z}}

\newcommand{\origin}{\mathbf{0}}
\newcommand{\ua}{\underline{a}}
\newcommand{\uk}{\underline{k}}
\newcommand{\bX}{\bar{X}}
\newcommand{\trans}{\mathfrak{W}}
\newcommand{\hQ}{{\widehat{Q}}}

\newcommand{\Var}{\mathrm{Var}}
\newcommand{\SD}{{\mathrm{SD}}}

\newcommand{\fH}{{\mathfrak{H}}}

\begin{document}
\title[Rotationally invariant FPP]{Rotationally invariant first passage percolation: Concentration and scaling relations}

\author{Riddhipratim Basu}
\address{Riddhipratim Basu, International Centre for Theoretical Sciences, Tata Institute of Fundamental Research, Bangalore, India} 
\email{rbasu@icts.res.in}
\author{Vladas Sidoravicius}
\address{Vladas Sidoravicius, Courant Institute of Mathematical Sciences, New York and NYU-ECNU Institute of Mathematical Sciences at NYU Shanghai}
\email{vs1138@nyu.edu}
\author{Allan Sly}
\address{Allan Sly, Department of Mathematics, Princeton University, Princeton, NJ, USA}
\email{allansly@princeton.edu}

\date{}
\maketitle

\begin{abstract}
For rotationally invariant first passage percolation (FPP) on the plane, we use a multi-scale argument to prove stretched exponential concentration of the first passage times at the scale of the standard deviation. Our results are proved under hypotheses which can be verified for many standard rotationally invariant models of first passage percolation, e.g.\ Riemannian FPP, Voronoi FPP and the Howard-Newman model. This is the first such tight concentration result known for any model that is not exactly solvable. As a consequence, we prove a version of the so called KPZ relation between the passage time fluctuations and the transversal fluctuations of geodesics as well as up to constant upper and lower bounds for the non-random fluctuations in these models. Similar results have previously been known conditionally under unproven hypotheses, but our results are the first ones that apply to some specific FPP models. Our arguments are expected to be useful in proving a number of other estimates which were hitherto only known conditionally or for exactly solvable models.  
\end{abstract}
\tableofcontents

\section{Introduction}
First passage percolation (FPP), a model of distance through a disordered landscape, introduced first in \cite{HW65} has been one of the central models in spatial probability over the last half century. Canonically, the first passage percolation model is studied on the $d$-dimensional Euclidean lattice $\Z^d$ with a random independent and identically distributed cost at each edge.  The first passage time between two vertices $u$ and $v$ is the minimum total cost over all paths joining $u$ and $v$, denoted as $X_{uv}$, and can be interpreted as a random distortion of the graph distance on $\Z^d$.  Many important results have been established such as the existence of a limit shape~\cite{CD81} and the sub-linear variance of passage times in many models~\cite{BKS04,BR08,DHS13} (see the excellent monograph~\cite{ADH15} for a comprehensive description of the state of the field).  However, understanding finer properties such as establishing the scaling exponents for the fluctuation of passage times or the fluctuations of the geodesics (optimal paths) remain major mathematical challenges. 

A key obstacle in progress for the study of the FPP is that the properties of the limit shape remain poorly understood. The connection between the properties of the limit shape (such as strong convexity and uniform curvature of the boundary), the fluctuation of the passage times and the properties of the geodesics are long since anticipated and to some extent understood. It is expected that, under mild conditions on the passage times, the limit shape is smooth with its boundary having bounded and positive curvature.  However, this is not known for any distribution of the canonical lattice FPP model. A program started in the nineties by Newman and coauthors explored understanding the properties of the lattice FPP models under unproven assumptions on the limit shape such as the ones described above. A parallel program developed several variants of FPP in continuum that are rotationally invariant and the limit shape is by symmetry a Euclidean ball in those cases. Even though many stronger results (e.g.\ improved variance lower bound or existence of semi-infinite geodesics) could be proved for such models, sharp results on fluctuations and concentration had remained out of reach so far.

In this paper we develop a general multi-scale scheme to control the passage time fluctuations of rotationally invariant FPP models on $\R^2$ and obtain several consequences of the same. The planar case is particularly interesting since the model in two dimension is believed to belong to the Kardar-Parisi-Zhang (KPZ) universality class \cite{KPZ86} and precise predictions for the large scale behaviour are available. Our scheme applies to a broad class of well-known models of rotationally invariant FPP including the Howard-Newman Model, (weighted) Voronoi FPP, Riemannian FPP and graph distances in random geometric graphs. To apply as generally as possible, our proofs are given for any FPP model that satisfies a set of conditions listed in Section~\ref{s:assumption}. We now describe our main results. 

\noindent
\textbf{Concentration of passage times:}
We let $X_n$ denote the passage time from the origin to $(n,0)$.  Our main result shows that when centred (by $\E X_n$) and normalized by its standard deviation (denoted $\SD(X_{n})$), the passage time is tight and has stretched exponential tails.
\begin{maintheorem}~\label{t:tightness}
For an FPP model on the plane satisfying the assumptions of Section~\ref{s:assumption}, there exist $\theta, C>0$ such that for all $n\geq 1$ and $x>0$
\[
\P\left[\frac{|X_n - \E X_n|}{\SD(X_n)} > x\right] \leq 2\exp(-C x^\theta).
\]
\end{maintheorem}
Note that the constants $C,\theta$ above will depend on the model parameters as defined in the assumptions. Based on the result for planar exactly solvable models (where the scaling limit is a scalar multiple of the GUE Tracy-Widom distribution) one also expects that the supports of the laws of $\frac{X_n - \E X_n}{\SD(X_n)}$ will not be bounded in $n$. We show this as well; see Proposition \ref{p:lowertail}. 

\medskip

\noindent
\textbf{Transversal fluctuations and the KPZ relation:} Another property of interest is the transversal fluctuation of a geodesic. This is the  maximal deviation of an optimal path in the FPP metric (called a geodesic) between $u$ and $v$  away from the straight line Euclidean path (defined formally in Section~\ref{s:outline}) joining $u$ and $v$, and is denoted $\trans_{uv}$. We will write $\trans_{n}$ for the transversal fluctuation of the geodesic from the origin to $(n,0)$. It is widely believed that there exist scaling exponents $\chi$ and $\xi$ such that the standard deviation and transversal fluctuations satisfy $\hbox{SD}(X_n)\asymp n^\chi$ and $\trans_n \asymp n^\xi$ in the sense that $n^{-\xi}\trans_{n}$ is uniformly tight with good tails and in particular $\E \trans_n \asymp n^{\xi}$. In dimension 2, it is predicted from the KPZ theory that $\chi=\frac13, \xi=\frac23$.  Even the existence of scaling exponents is unknown in any model of FPP; however, they are known rigorously in last passage percolation (LPP) models in $\Z^2$ with exponential or geometrically distributed passage times or Poissonian last passage model in $\R^2$ as these models are exactly solvable~\cite{J00,BSS14,BGZ21}.

In any dimension $d\geq 2$ the exponents are expected to satisfy the scaling relation
\begin{equation}\label{eq:scalingRelation1}
\chi = 2\xi-1
\end{equation}
or equivalently
\begin{equation}\label{eq:scalingRelation2}
\hbox{SD}(X_n) \asymp \left(\E\trans_n\right)^2 n^{-1}.
\end{equation}
Under strong assumptions (not known for any FPP model) Chatterjee proved~\cite{Cha11} the conditional result that if the scaling exponents exist in a certain sense then they satisfy~\eqref{eq:scalingRelation1} (see also \cite{Ale20}). While we do not establish the existence of the scaling exponents we prove~\eqref{eq:scalingRelation2} holds for our rotationally invariant models of FPP.
\begin{maintheorem}
\label{t:tf}
For an FPP model on the plane satisfying the assumptions of Section~\ref{s:assumption},
\begin{enumerate}
    \item[(i)] There exists $C,\theta_1>0$ such that for all $z>0$ and $n\ge 1$,
    $$\P(\trans_{n}\ge z\sqrt{n\SD(X_{n})})\le \exp(1-Cz^{\theta_1}).$$
    \item[(ii)] Given $\varepsilon>0$, there exists $\delta>0$ such that 
    $$\P(\trans_{n}\le \delta \sqrt{n\SD(X_{n})})\le \varepsilon$$
    for all $n\ge n_0(\delta)$. 
\end{enumerate}

In particular, there exist $C,C'>0$ such that
\[
C\hbox{SD}(X_n) \leq  \left(\E\trans_n\right)^2 n^{-1} \leq C' \hbox{SD}(X_n)
\]
for all $n$ {sufficiently large}. The same statement remains true if $\E\trans_n$ is replaced by the median of $\trans_n$.
\end{maintheorem}
\medskip

\noindent 
\textbf{Non-random fluctuations:}
For all the models we consider it is known that there exists $\mu\in (0,\infty)$ such that $n^{-1}\E X_{n}\to \mu$. By sub-additivity, it is also known that $\E X_{n}\ge \mu n$ for all $n$. The behaviour of the non-negative sequence $\{\E X_{n}-\mu n\}$, often referred to as non-random fluctuations, has also drawn a lot of interest over the years. It is believed that the non-random fluctuations $\approx n^{\gamma}$ for some $\gamma>0$, and it is predicted that $\gamma=\chi$, the fluctuation exponent, under mild conditions. Certain conditional results concerning inequalities among these exponents were previously known~\cite{Ale11,ADH13}. Comparing with the results known for exactly solvable models (see e.g.\cite{LR10,BGHH22}), one might also predict the stronger statement $\E X_{n}-n\mu=\Theta(\SD(X_{n}))$ under mild assumptions but results along these lines have only been known under strong unproved assumptions (see e.g.\ \cite{Ale21}). We establish this for the rotationally invariant FPP models we consider. 

\begin{maintheorem}
\label{t:nr}
For an FPP model on the plane satisfying the assumptions of Section~\ref{s:assumption}, there exist $\widehat{C}, \widetilde{C}>0$ such that for all $n\geq 1$,
\[
\widehat{C}\SD(X_n) \leq  \E X_{n}-n\mu  \leq \widetilde{C} \SD(X_n).
\]
\end{maintheorem}

As mentioned before, certain conditional results along these lines were known before, but the best known unconditional results were $\E X_{n}-n\mu=O((n\log n)^{1/2})$ \cite{Ale93,Ale97}, and $\E X_{n}-n\mu=O(\psi(n)\log^{(k)}n)$ ($\log^{(k)}$ is the $k$-fold iteration of $\log$ and $\psi(n)$ is a scale such that $X_{n}$ has stretched exponential concentration at this scale, in particular, $\SD(X_{n})\le \psi(n)\le \sqrt{n}$) for certain rotationally invariant models \cite{DW16}. Theorem \ref{t:nr} improves on these results for the models we consider. A more detailed discussion regarding the relationship between our main theorems and the existing literature is presented in Section \ref{s:literature}.

\medskip

Theorem \ref{t:tightness} and the upper bounds in Theorems \ref{t:tf} and \ref{t:nr} will be proved together using a multi-scale argument. This is the heart of the paper. The lower bounds in Theorems \ref{t:tf} and \ref{t:nr} are established separately using the above results.

\subsection{Rotationally Invariant First Passage Percolation}

There are several simple ways to construct first passage percolation in $\R^2$ in such a way that it is invariant under $ISO(2)$, the group of Euclidean rigid body motions generated by translations, rotations and reflections.  Our approach will apply to a range of models that satisfy a list of properties set out in Section~\ref{s:assumption}.  In each case the first passage percolation metric will be constructed from $\omega$, a random field on $\R^2$ taking values in some probability space $\Omega$, whose distribution is invariant under $ISO(2)$.  We will write $\omega_A$ to denote the restriction of the field $\omega$ to $A\subset \R^2$.  We will assume that $\omega$ has the spatial independence property that for pairwise disjoint measurable sets $\Lambda_i\subset \R^2$ the restricted subfields $\omega_{\Lambda_i}$ are independent.  The two natural examples of this are $\omega$ distributed as a Poisson Point Process or as Gaussian white noise, and these are the only two we shall use. 

For different models we shall also define a class of admissible paths $\Upsilon$. Elements of $\Upsilon$ will be continuous functions $\gamma:[0,1]\to \R^2$ (with further restrictions depending on the model). We let $\Upsilon_{uv}$ be the be the collection of all paths $\gamma$ in $\Upsilon$ such that $\gamma(0)=u, \gamma(1)=v$ (i.e., paths from $u$ to $v$). Then the passage time (of an admissible path) will be a measurable function
\[
X:\Upsilon\times\Omega \to [0,\infty]
\]
which we will usually denote by $X_\gamma^\omega$ or by $X_\gamma$ suppressing the dependence on $\omega$ when it is unambiguous. We require the following properties:
\begin{itemize}
  \item {\bf Invariance:} For any rigid body transformation $\tau\in ISO(2)$
  \[
  X_{\tau\gamma}^{\tau \omega} = X_\gamma^\omega.
  \]
  \item {\bf Independence of parametrisation:} $X$ does not depend on the parametrisation, that is if $\gamma$ and $\tilde{\gamma}$ are {different parametrisations} of the same path, then 
  $$X_{\gamma}=X_{\tilde{\gamma}}.$$ 
  \item {\bf Time reversal:} If $\tilde{\gamma}(t)=\tilde{\gamma}(1-t)$ then $X_\gamma=X_{\gamma'}$.
\end{itemize}
For $u,v\in\R^2$ point-to-point distances (the first passage time) are then defined as the infimum of the passage time over all admissible paths
\[
X^\omega_{uv}=\inf_{\gamma\in \Upsilon_{uv}} X_\gamma^\omega.
\]
By the invariance property, the law of $X^\omega_{uv}$ depends only on the distance $|u-v|$. We shall write
\[
X_n:=X_{\origin,(n,0)}
\]
denoting the origin as $\origin$.\footnote{In general, for $r\in \R$, $\mathbf{r}$ shall denote the point $(r,0)$.}

We will at times want to resample parts of the field $\omega$.  For $\Lambda\subset\R^2$ we let $\omega^\Lambda$ denote the field $\omega$ with $\Lambda^c$ resampled.  In particular this means that
\begin{itemize}
  \item The fields $\omega$ and $\omega^\Lambda$ are equal in distribution.
  \item The fields are equal, $\omega^\Lambda_\Lambda=\omega_\Lambda$, on the set $\Lambda$ .
  \item On $\Lambda^c$ the resampled field $\omega_{\Lambda^c}^\Lambda$ is independent of $\omega$.
\end{itemize}
If $\Lambda_1,\ldots,\Lambda_m\subset \R^2$ are a collection of subsets to be resampled  we assume that each resampling is done independently, that is $\omega, \omega_{\Lambda_1^c}^{\Lambda_1},\ldots, \omega_{\Lambda_m^c}^{\Lambda_m}$ are all independent.  For compactness of notation, we will write $X^\Lambda$ to denote $X^{\omega^\Lambda}$.  Note that if $\Lambda_1,\ldots,\Lambda_m$ are pairwise disjoint then the $X^{\Lambda_i}$ are independent.

For a path $\gamma_1\in \Upsilon_{uv}$ and a path $\gamma_2\in \Upsilon_{vw}$ we shall also need to consider a concatenated path $\gamma\in \Upsilon_{uv}$. The concatenation operation will be defined differently in different models. 

\subsubsection{Models}\label{s:model}
We now describe 4 classes of FPP models already considered in the literature that fit into the above framework. There have been two different approaches of defining rotationally invariant FPP models: first, a discrete model based on a graph whose vertices are sets of a Poisson point process (the first three models we shall consider will belong to this class), and second, a model based on a continuous random field (the final model we consider will belong to this class). 

For the first three models of FPP, the underlying field $\omega$ is a (Marked) Poisson Point Process (homogeneous with rate $1$, say) with point set $\Pi$  and marks $\{M_x\}_{x\in\Pi}$. 
We define $T:\R^2 \to \Pi$ so that $T(u)$ is the element of $\Pi$ closest to $u$. In case of a tie we shall break it by using some additional randomness (e.g. we can have another set of marks $\{\tilde{M}_{x}\}_{x\in \Pi}$ of i.i.d. nonnegative continuous random variables, and for a point $u$ if $A\subset \Pi$ is the set of its closest points then we define $T(u)=x_{i}$ where $\tilde{M}_{x_i}=\min_{x\in A} \tilde{M}_{x}$. It will be clear that this tie breaking rule would not affect the rotational invariance of the model). The Voronoi cell of a point $x\in \Pi$ is the subset of the plane which is closer to $x$ than any other vertex in $\Pi$. 

For these class of models, the set of all allowed paths between $u,v\in \R^2$ will be a subclass (depending on the specific model) of piece-wise linear paths joining $x_0=u, x_1, x_2,\ldots, x_{n-1},x_{n}=v$ where each $x_{i}\in \Pi$ for $i=1,2,\ldots, n-1$. We shall denote such a path by $\gamma=\{x_0,x_1,\ldots, x_{n}\}$ when there is no scope for confusion. For $\gamma_1=\{x_0,x_1,\ldots,x_{n_1}\}\in \Upsilon_{uv}$ and $\gamma_2=\{y_0,y_1,\ldots, y_{n_2}\}\in \Upsilon_{vw}$ we define the concatenation of $\gamma_1$ and $\gamma_2$, denoted by 
$\gamma_1\gamma_2=\{x_0,x_1,\ldots, x_{n_1}, y_1,\ldots, y_{n_2}\}$ if $x_{n_1}=v=y_0\in \Pi$ and $\gamma_1\gamma_2=\{x_0,x_1,\ldots, x_{n_1-1}, y_1,\ldots, y_{n_2}$ otherwise. Also, we shall now allow paths from $u$ to $u$, and define $X_{uu}=0$ for all $u\in \R^2$. We now move on to the specific model descriptions.

{\bf (Weighted) Voronoi FPP:} A first model of FPP (this type of model was studied in \cite{VW90,VW92}) is one in which we count the number of Voronoi cells the path $\gamma$ enters. The admissible class of paths $\Upsilon_{uv}$ consists of paths $\gamma=\{u=x_0,x_1, \ldots, x_{n-1},x_{n}=v\}$ where $x_1=T(u)$, $x_{n-1}=T(v)$, $x_{i}\in \Pi$ for $i=1,2,\ldots, n-1$, and the Voronoi cells of $x_{i}$ and $x_{i+1}$ are neighbouring for all $i=1,2,\ldots, n-2$. Notice that if $u$ or $v\in \Pi$ then there can be repetitions in the above list.

For an admissible $\gamma$ as above, we define the passage time in the unweighted case by
\[
X_\gamma:=n-1
\]
whereas in the weighted version we define
\[
X_\gamma:=\sum_{i=1}^{n-1} M_{x_{i}}.
\]

Notice that if a path comes back to a Voronoi cell after leaving it once it is counted with multiplicities in the definition of the passage time.

{\bf Graph distance in Random Geometric graphs:} A random geometric graph with threshold $L$ is constructed by taking the vertex set $\Pi$, and an edge between pairs $x,x'\in \Pi$ if $|x-x'|\leq L$. It is a well known fact (see e.g.\ \cite{MR96,Pen03}) that almost surely there exists a critical value $L_c\in (0,\infty)$ such that for $L>L_c$ the model is supercritical and there is a unique infinite component which we will denote by $\Pi^{(\infty)}$.  Let $T^{(\infty)}(u)$ denote the closest point to $u$ in $\Pi^{(\infty)}$. 

The admissible class of paths $\Upsilon_{uv}$ in this case consists of paths $\gamma=\{u=x_0,x_1, \ldots, x_{n-1},x_{n}=v\}$ where $x_1=T^{(\infty)}(u)$, $x_{n-1}=T^{(\infty)}(v)$ and $x_{i}$ and $x_{i+1}$ are neighbouring vertices in $\Pi^{(\infty)}$ for all $i=1,2,\ldots, n-2$. For such a $\gamma$,
we define 
\[
X_{\gamma}:=n-2,
\]
that is the number of edges in $\gamma$ between $T^{(\infty)}(u)$ and $T^{(\infty)}(v)$.

With this definition, $X_{uv}$ is simply the graph distance between $T^{(\infty)}(u)$ and $T^{(\infty)}(v)$ and by uniqueness of the infinite component is finite almost surely. A slightly different first passage percolation model on a random geometric graph was studied in \cite{HNGS15}.

{\bf Howard-Newman Model:}  The third model with parameter $\beta>1$ is defined as follows. The admissible class of paths $\Upsilon_{uv}$ in this case consists of all paths $\gamma=\{u=x_0,x_1, \ldots, x_{n-1},x_{n}=v\}$ where $x_1=T(u)$, $x_{n-1}=T(v)$ and $x_{i}\in \Pi$ for $i=1,2,\ldots, n-1$.
The passage time of the path $\gamma$ is defined by
\[
X_\gamma=\sum_{j=1}^{n-2} |x_j-x_{j+1}|^\beta.
\]
This model was introduced by Howard and Newman \cite{HN97}, and was dubbed Euclidean FPP by them.

{\bf Riemannian FPP:} Now we define the second type of model which is based on an isotropic random field defined in the continuum. This type of model was  first considered in \cite{LW10}. Fix a radially symmetric (i.e., $K(x)$ is depends on $x$ only through $|x|$), nonnegative, smooth $(C^{\infty})$, bounded kernel $K:\R^2\to\R$ that vanishes outside the unit ball.  We will take $\omega$ to be either a Gaussian white noise or a rate $\lambda$ Poisson Point Process on $\R^2$.  Then we write
\[
\Phi(x):=\int K(x-y) \omega(d y)
\]
which means the stochastic integral
\[
\Phi(x)=\int K(x-y) dB(y)
\]
where $B(y)$ is two dimensional Brownian motion in the case that $\omega$ is white noise and means
\[
\Phi(x)=\sum_{y\in \Pi} K(x-y)
\]
where $\Pi$ is the set of points in the Poisson point Process in the case $\omega$ is a Poisson point process. To ensure that we have a bounded, positive field we fix a monotone increasing smooth function  $\psi:\R\to (d_1,d_2)$ such that $\psi(\R)$ is supported on $[d_1,d_2]$ for some $0<d_1<d_2<\infty$. Then we set $\Psi(x)=\psi(\Phi(x))$ to be the underlying environment which the paths traverse. 

The allowed class of paths $\Upsilon_{uv}$ here consists of all continuous paths $\gamma:[0,1]\to\R^2$ with $\gamma(0)=u$, $\gamma(1)=v$ that are of bounded variation. The concatenation of two paths in this case is simply defined as the concatenation of two continuous curves. 

Using the random field $\Psi$, we define passage time of a path $\gamma$, the natural continuous analogue of the standard FPP case on the lattice, by 
\[
X_{\gamma}:= \int_{\gamma} \Psi(x) dx =  \int_{0}^{1} \Psi(\gamma(t))|\dot{\gamma}(t)| dt
\]
when $\gamma$ is piece-wise $C^1$. This definition can be extended to all bounded variation paths by taking suitable limits. Finally, we define the first passage time  
$$X_{uv}:=\inf_{\gamma\in \Upsilon_{uv}} X_{\gamma}$$ which defines a random Riemannian metric on $\R^2$. 

Next we develop a general axiomatic framework containing the above four models.

\subsection{Model Assumptions}~\label{s:assumption}
In this subsection we set out a list of assumed properties for the FPP models we consider.  For the remainder of the paper, all results will assume that these properties hold.  In what follows $D_0>0, \kappa>0$ and $n_0\geq 1$\footnote{Unlike the standard convention, for us $n,n_0$ etc.\ will denote positive real numbers and not necessarily integers.} are constants that depend only on the model.

\begin{enumerate}
  \item {\bf Invariance:} The passage times $\{X^\omega_{uv}\}_{u,v\in \R^2}$, defined as the infimum of passage times $X_{\gamma}$ as $\gamma$ varies over the admissible class of paths from $u$ to $v$, are $ISO(2)$ invariant and in particular invariant under rotations, translations, reflections. Further, $X_{\gamma}$ is invariant under time reversal and change of parametrisation. 
  
  \item {\bf Optimal Paths:} For every $u,v\in \R^2$ there exists a path $\gamma_{uv}$ such that $X_{uv}=X_\gamma$, that is there exists a path which attains the infimum. {In general, there may be multiple optimal paths but we assume that $\gamma_{uv}$ denotes a canonical measurable choice.}

  \item\label{as:speed} {\bf Speed:} The limiting speed of the process
  \[
  \mu:=\frac1{n}\lim_{n\to\infty} \E X_n
  \]
  exists and $0<\mu<\infty$. 
  
  \item~\label{as:concentration}{\bf Concentration:}
  There passage times are concentrated around their mean satisfying for all $n\geq 1$ and $x>0$,
  \[
  \P[|X_n - \E X_n| > x\sqrt{n}]\leq \exp(1-D_0x^\kappa).
  \]
  
  \item~\label{as:nr} {\bf Non-random fluctuations:}
  With $A_n := \E X_n - \mu n$ we have there exists $n_0\ge 1$ such that for all $n\geq n_0$,
  \[
  0\leq A_n \leq n^{\frac12}\log^2 n.
  \]
  \item\label{as:localDist} {\bf Local Passage Times:} Distances within a local neighbourhood are not too large
  \[
    \P[\max_{u,v\in B_3(\origin)}X_{uv} > x]\leq \exp(1-D_0x^\kappa).
  \]
  \item \label{as:tri} {\bf Triangle Inequality:} The passage times satisfy the triangle inequality
  \[
    X_{uw}+X_{wv}\geq X_{uv}.
  \]
  More generally, if $\gamma$ is the concatenation of paths $\gamma_1$ and $\gamma_2$ we have 
   \[
    X_{\gamma_1}+X_{\gamma_2}\geq X_{\gamma}.
  \]
  \item\label{as:intermediate} {\bf Intermediate Points\footnote{  Note that while for some models we may have that $X_{uw}+X_{wv} = X_{uv}$ whenever $w\in\gamma_{uv}$ this will not hold for all the models we are considering, mostly due to the fact that we are imposing continuous paths on what are really discrete models.}:} For $0=t_0<t_1<t_2<\ldots<t_M<t_{M+1}=1$ and for any $n\geq n_0$,
  \[
  \P[\max_{u,v\in B_n(\origin)} \max_{M\geq 1} \max_{0<t_1<\ldots<t_M<1} \sum_{i=1}^{M+1} X_{\gamma_{uv}(t_{i-1})\gamma_{uv}(t_i)} - X_{uv} > xM \log^{1/\kappa}n] \leq  \exp(1-D_0x^\kappa),
  \]
  which says that the triangle inequality is close to tight for intermediate points. Also
  \[
  \P[\max_{u\in B_n(\origin)} \max_{M\geq 1} \max_{0<t_1<\ldots<t_M<1} \sum_{i=1}^{M+1} X_{\gamma_{uv}(t_{i-1})\gamma_{uv}(t_i)} - X_{uv} > xM \log^{1/\kappa}(n+\max_{0<i\leq M}|\gamma_{uv}(t_i)| )] \leq  \exp(1-D_0x^\kappa),
  \]
  \item\label{as:varlb} {\bf Variance Lower Bound:} The passage times have variance growing at least polynomially, and for $n\geq n_0$,
  \[
  \hbox{Var}(X_n) \geq n^{1/6}
  \]
 and $\inf_{n\geq 1} \hbox{Var}(X_n) >0$.
  \item\label{as:cars} {\bf Local Transversal Fluctuations:}  For any $n\geq 1, M\geq \frac32$ and $|y_1|,|y_2|\leq Mn$ let $v_1=(0,y_1), v_2=(Mn,y_2)$ and let $(n,H)$ denote the first intersection of $\gamma_{\origin,v}$ with the line $x=n$.  Then for $x\geq 1$,
      \[
        \P[H \geq xn^{4/5} + y_1 + (y_2-y_1)/M] \leq \exp(1-D_0(xn^{1/10})^\kappa).
      \]
  \item {\bf Resampling 1:}\label{as:resamp1} Let $\Lambda=\{(x,y)\in\R^2:0\leq x \leq n\}$.  Then,
  \[
    \P[\max_{0\leq y_1,y_2\leq n}|X_{(0,y_1),(n,y_2)}-X_{(0,y_1),(n,y_2)}^\Lambda|>x\log^{1/\kappa} n]\leq \exp(1-D_0x^\kappa).
  \]
  \item {\bf Resampling 2:}\label{as:resamp2} Let $\Lambda\subset \R^2$ be an $n\times W$ rectangle with $W\leq n$ and let
  \[
  \Lambda^- = \{w\in \Lambda:d(w,\Lambda^c)\geq \log^2 n\},\quad \Lambda^+ = \{w\in \R^2:d(w,\Lambda)\leq \log^2 n\}.
  \]
  Whenever $n\geq n_0$, then resampled passage times $X^\Lambda$ satisfy
      \[
        \P[\bigcap_{\gamma\in\Upsilon,\gamma\subset\Lambda^-} \{X_\gamma=X_\gamma^{\Lambda}\}]\geq 1-n^{-10},\qquad \P[\bigcap_{\gamma\in\Upsilon,\gamma\subset(\Lambda^+)^c} \{X_\gamma=X_\gamma^{\Lambda^c}\}]\geq 1-n^{-10}
      \]
      that is with very large probability all paths at least $\log^2 n$ away from the resampled region are unaffected by the resampling.
\end{enumerate}

We note that the assumptions listed above are not optimal. Values of several exponents can be relaxed to some degree; we have not tried to work out an optimal set of hypotheses for which our arguments will go through. One other point to note is that some of the assumptions can be deduced from the others (see below for more details) and we have not made any attempt to reduce the list to an independent set of assumptions either. We note one specific case which might be of interest. The stretched exponential tails at scale $\sqrt{n}$ is not needed for our arguments; replacing $\sqrt{n}$ in Assumption \ref{as:concentration} by $n^{1/2+\delta}$ for some small $\delta>0$ works as well. Furthermore, in that case, one can also relax Assumption \ref{as:nr} to $A_{n}\le n^{1/2+\delta}\log^2 n$ (the scale at which one has stretched exponential concentration and the bound one can get for $A_n$ are closely related; see the discussion below). 

As already mentioned, these properties hold for each of the 4 models defined in Section~\ref{s:model}. We record this fact as the next theorem.   

\begin{theorem}
    \label{t:4models}
    All the four models of rotationally invariant FPP on the plane described in Section~\ref{s:model} satisfy the assumptions in Section~\ref{s:assumption} for some choices of the parameters $D_0, \kappa, n_0>0$. 
\end{theorem}

The proof of Theorem \ref{t:4models} is standard but long and tedious. Some of the statements exist in the literature for some of the models and the rest can be verified by adapting various existing arguments, which are not necessarily formulated for these models. As such, neither the statement nor the proof of the above theorem is interesting to the experts. For this reason, and to keep this (already rather long) paper at a reasonable length, we shall not provide a detailed proof of Theorem \ref{t:4models} here, but defer it to a companion paper \cite{BSS2}. We nevertheless briefly discuss below which arguments in the literature will be used to verify the different assumptions above. 

Invariance properties are clear from the definition in each case, while existence of geodesics follows from standard compactness arguments (existence of geodesics in the Riemannian FPP case follow from an application of the Hopf-Rinow theorem in geoemetry; see \cite{LW10}). 
A canonical choice for the geodesic $\gamma_{uv}$ can be made by using some additional randomness independent of the noise field.

Existence and finiteness of $\mu$ follows from rotational invariance and variants of the standard Cox-Durrett shape theorem e.g.\ \cite{CD81}. This was shown for the Howard-Newman model in \cite{HN97}, for Voronoi FPP in \cite{VW90, VW92} (see also \cite{Ser97}) and for the Riemannian FPP in \cite{LW10}. Stretched exponential concentration at the $\sqrt{n}$ scale can be proved by a variant of the martingale argument as in \cite {Kes93} or using a Talagrand inequality argument as in \cite{Tal95}. This was shown for the Howard-Newman model in \cite{HN01}, and a slightly weaker statement was shown for the Voronoi FPP in \cite{Pim11}. The upper bound for non-random fluctuations $A_{n}$ was established for the lattice model by Alexander \cite{Ale93} (see also \cite{DW16} for an improvement for Euclidean FPP) and a variant of the same argument will work for each of our models. 

The tail bounds for local passage times as well as the triangle inequality are clear from definitions in each of the cases. The intermediate points hypothesis is also easy to establish. The variance lower bound follows from a variant the Newman-Piza argument \cite{NP95}, which requires planarity and curvature of the limit shape and the consequent upper bound on transversal fluctuations, see \cite{H00} for the case of Howard-Newman model. The local transversal fluctuations can be bounded using an argument of Newman from \cite{New95} (a similar argument was used in \cite{BSS19} for an exactly solvable model). Finally, the resampling hypotheses are rather easy to check case by case using the fact that the underlying noise field has a short range of dependence.  

We point out that all the hypotheses except Assumption \ref{as:varlb} can be shown to hold at all higher dimensions, but we shall not make use of this fact in this paper. Planarity is crucial for Assumption \ref{as:varlb}; indeed it is not known whether the variance grows polynomially for any model in dimensions $d\ge 3$. We shall make crucial use of planarity in parts of our argument (both Assumption \ref{as:varlb}, and the argument used to prove it). 

\subsection{Related literature}
\label{s:literature}
Existence of the limit shape for the first passage percolation models was established under a very general hypothesis more than forty years ago \cite{CD81}. It is expected that under mild regularity conditions on the underlying noise the limit shape has a boundary that is uniformly curved, but this is not known for any lattice model. From the general theory of the Kardar-Parisi-Zhang universality \cite{KPZ86}, it is also expected that planar first passage percolation models under very general hypothesis exhibit the exponent pair $(\chi=1/3,\xi=2/3)$ for passage time fluctuations and transversal fluctuations of the geodesics; proving this without stronger additional hypothesis has remained a central challenge in the field. A Poincar\'e inequality argument by Kesten \cite{Kes93} showed a linear upper bound on the variance which was upgraded later to an $O(n\log^{-1}n)$ bound in \cite{BKS04} (see also \cite{BR08, DHS13}). This remains the best known upper bound for general FPP models to date. 

The intricate inter-relations between understanding of the limit shape regularity, passage times fluctuations, and geodesic geometry has been clear for quite some time; however, to get control on the limit shape has proved to be rather difficult. The effort to define FPP models with suitable symmetries that will give control on the limit shape goes back to late 80s and early 90s, when FPP models based on a Poisson point process were first introduced. Voronoi FPP was introduced and studied in \cite{VW90, VW92,VW93} where the limit shape is a Euclidean ball by rotational invariance. A slightly different approach was based on studying, in absence of unconditional results, conditional behaviour of the FPP models under certain hypothesis which are expected to be true under general conditions. This approach was initiated by Newman and co-authors \cite{New95,NP95,LNP96} in the 90s where they proved a number of conditional results under some hypothesis on the limit shape (typically uniform curvature or local curvature). These results include, a one sided form of the KPZ relation $2\xi-1\le \chi$ and the consequent upper bound $\xi\le 3/4$, and a lower bound $\chi\ge 1/8$ for the passage time fluctuation exponent. As an example of a models where the hypotheses can be verified, another Poisson process based rotationally invariant model, Euclidean FPP (often referred to as Howard-Newman model) was introduced in \cite{HN97} and studied further in \cite{HN98,H00, H01, HN01}. As another candidate for a rotationally invariant FPP model, Riemannian FPP models were defined later \cite{LW10} based on a continuum random field instead. 

However, even with the additional hypotheses on the limit shape (or even rotational invariance), all of the major questions (e.g.\ showing that the fluctuation exponent $\chi=1/3$ or even the weaker bound $\chi<1/2$, non-existence of bigeodesics, or the full KPZ relation $\chi=2\xi-1$, optimal bounds for non-random fluctuations) still remained out of reach. Therefore, even stronger hypotheses were considered. In \cite{Cha11}, Chatterjee obtained a proof for the KPZ relation under some strong unverified hypothesis (essentially that the fluctuation exponent is same in all directions and passage times around their means are exponentially concentrated at the standard deviation scale). In parallel, in the exactly solvable planar last passage percolation literature (where curvature of limit shapes and $n^{1/3}$ fluctuations for passage times are rigorously known), many of the parallel questions (e.g.\ non-existence of bigeodesics, coalescence of geodesics, optimal estimates for the midpoint problem etc.) were handled \cite{BSS19,BHS18,BG18,BGZ21,BB21,BBF22,HS18,Zha19}. These works primarily used two ingredients from the exactly solvable literature as black-box estimates: curvature of limit shape, and (stretched) exponential concentration for passage times at the fluctuation scale. 

Following this, several conditional results for FPP models were proved under the assumption of curvature of limit shape (or rotational invariance) together with exponential concentration of passage times at the fluctuation scale in \cite{Ale20,Ale21}. In particular, KPZ relation, non-existence of bigeodesics, tight upper bound for the non-random fluctuations and tail estimates for distance to coalescence was proved under variants of the above hypotheses (some of these results were also obtained for all dimensions, while the rest were specific to the planar setting).

A comprehensive discussion of all related literature is beyond our current scope but we refer the reader to the excellent monograph \cite{ADH15} for a detailed account of the progress until 2015. We should also mention that some impressive progress without unverifiable hypotheses has recently been made, e.g.\ \cite{AH16} where the midpoint problem was solved (without any assumptions on the limit shape) or \cite{DEP21} where quantitative estimates 
for the midpoint problem was given (with some assumptions on the number of sides of the limit shape boundary which can be verified for some choices of passage time distributions). 

Nevertheless, as far as we are aware, prior to the current work, there was no known examples of FPP or LPP models (except the exactly solvable ones) for which a version of Theorem \ref{t:tightness}, i.e., concentration of passage times at the standard deviation scale, was known. The previous best known concentration result was an exponential concentration for $X_{n}$ at the scale $\sqrt{n/\log n}$ \cite{DHS14}. Although several conditional variants of Theorem \ref{t:tf} and Theorem \ref{t:nr} have been proved, as explained above, our results are the first ones were these theorems have been proved in some concrete (non exactly solvable) examples.  

\subsection{Companion papers and potential future applications}
\label{s:future}

This paper is the first in a series investigating properties of planar rotationally invariant FPP. As already mentioned, in a companion paper \cite{BSS2}, we shall show that the four examples of rotationally invariant FPP described above (Riemannian FPP, Voronoi FPP, Howard-Newman model and distances in random geometric graphs) do indeed satisfy the assumptions in Section \ref{s:assumption}. In another companion paper \cite{BSS3}, under the assumptions of Section \ref{s:assumption} together with the FKG inequality (which is satisfied for two of our four examples, namely Riemannian FPP and distances in
random geometric graphs), we will give a multi-scale proof establishing improved upper bounds on the passage time fluctuations, proving that $\hbox{Var}(X_n) \leq n^{1-\epsilon}$ for some $\epsilon>0$. This will be achieved by using the results of this paper together with various percolation arguments to show that the geodesic is chaotic at a large number of scales and then using the ``chaos implies superconcentration" principle (see e.g.\ \cite{Cha08}) in a multi-scale scheme. This will improve upon the best known variance upper bound $\hbox{Var}(X_n)=O(n/\log n)$ previously mentioned which was proved using hypercontractivity \cite{BKS04}, and later generalised using the log-Sobolev inequality \cite{BR08,DHS14}. This will demonstrate the first examples of non-exactly solvable models for which the strict inequality $\chi<1/2$ is known for the fluctuation exponent. The results of this paper can also allow one to establish non-existence of bigeodesics in weighted Voronoi FPP (where geodesics are almost surely unique). As mentioned above, non-existence of bigeodesics have been proved conditionally (essentially for lattice models under the assumptions of curvature of limit shape and a variant of Theorem \ref{t:tightness}), but it has not been established for any specific example so far.

Further, as mentioned above, many recent results on geometry of geodesics and passage time landscapes for exactly solvable LPP models were established using primarily two ingredients: uniform curvature of limit shapes, and stretched exponential concentration of passage times at the fluctuation scale; apart from the ones already mentioned above see also \cite{MSZ21,SSZ21,BGHH22,GZ22,GH23,BBB23}. The results of this paper open the door for proving similar results (without any unproven assumptions) for rotationally invariant planar FPP models. We believe that under the hypotheses of the current paper (and perhaps some additional assumptions such as the FKG inequality, uniqueness of geodesics etc. which are verifiable in specific cases), one can prove variants of some of the results established for exactly solvable models. Note that the exactly solvable results also sometimes use the fact that the Tracy-Widom distribution (scaling limit for the passage times in the exactly solvable LPP cases) has negative mean; in absence of weak convergence results in the non-exactly solvable cases, this input can be replaced by (the lower bound in) Theorem \ref{t:nr}. 

It is also worth investigating the robustness of our arguments. It is natural to ask if our arguments can be carried out under weaker hypothesis such as curvature of limit shapes or in higher dimensions. The current argument we have uses both rotational invariance and planarity crucially. In fact, it is not even sufficient for our arguments to assume that the limit shape is a Euclidean ball; at the very least our current argument also needs the fact that the fluctuations of passage times in different directions can not grow at different scales. One might try to check if our argument goes through under the hypothesis of curvature of limit shape together with an assumption like  $\sup_{n\ge 1}\sup_{|u|, |v|\in (n/2,2n)} \frac{\mbox{Var}(X_{\mathbf{0},u})}{\mbox{Var}(X_{\mathbf{0},v})}<\infty$. We have not tried to verify this, it might be taken up elsewhere. As for planarity, one can check that parts of our argument (Sections \ref{s:rectpara}, \ref{s:percgen}, \ref{s:trans}) works in all higher dimensions. Parts of the arguments in Sections \ref{s:impconc}, \ref{s:record}, \ref{s:ltf} and \ref{s:left} use that the variance grows at least polynomially, which is not known in higher dimensions even for rotiationally invariant models. The final piece of the argument in Section~\ref{s:varlb} uses planarity in a fundamental way; we implement a complicated block version of the Newman-Piza argument \cite{NP95} which shows polynomial lower bound of the variance of planar FPP models under curvature assumption (in higher dimensions the argument gives a lower bound that goes to $0$ as $n\to \infty$). Therefore our current argument does not appear to extend to higher dimensions even under an additional assumption of polynomial growth of the variance.

We now move towards the proofs of our main results. The next section shall describe the basic set-up of our multi-scale scheme and provide a high level overview of our argument. A more detailed description of the organisation of the remainder of the paper will be given at the end of the next section.

\subsection*{Acknowledgements}
Riddhipratim Basu is supported by a MATRICS grant (MTR/2021/000093) from SERB, Govt.~of India, DAE project no.~RTI4001 via ICTS, and the Infosys Foundation via the Infosys-Chandrasekharan Virtual Centre for Random Geometry of TIFR. Allan Sly is supported by a Simons Investigator grant and a MacArthur Fellowship.


\section{Basic set-up and proof outline}
\label{s:outline}
From now on we shall work with a rotationally invariant planar FPP model which satisfies the assumptions in Section \ref{s:assumption} for fixed parameters $\kappa, D_0>0$ and $n_0\ge 1$. We will make use of two exponents $\alpha,\theta>0$ as well as an exponent $\epsilon>0$ chosen to be small satisfying
\[
\alpha \in [\min\{\frac{1}{30},\frac12 \kappa\}, \min\{\frac{1}{15},\kappa\}] ,\quad \theta = \frac{1\wedge\kappa^2}{10000}, \quad {\epsilon= \frac{\theta}{10^6}\leq 10^{-10}}.
\]
Generally $u,v,w$ will denote points in $\R^2$.  As we will need to take many union bounds throughout the proof, we will want to discretise $\R^2$ and so for $u\in \R^2$ we let $u^\Z$ denote the closest point in $\Z^2$ to $u$ with some arbitrary rule to break ties.

While the ultimate goal of the proof is to renormalize the centred passage times by the variance, it turns out that this is not the most convenient quantity for our multiscale analysis. Instead we define a quantity $\hQ_n$ by
\[
\hQ_n := \inf\left\{Q:\sup_{x>0}\frac{\P[|X_n-\E X_n|>xQ]}{\exp(1-x^\theta)}\leq 1\right\}.
\]
By Assumption~\ref{as:concentration} since $\theta<\kappa$ we have that $\hQ_n$ is finite and
\[
\hQ_n \leq D_0^{1/\kappa}\sqrt{n}
\]
for all $n\geq 1$.  We define
\[
Q_n := \sup_{1\leq m \leq n} \left( \frac{n}{m}\right)^\alpha \hQ_m,
\]
and so we have that \footnote{recall that as mentioned before, $m$ varies over all reals in the above supremum and not just integers.}
\begin{equation}\label{eq:Q}
\P[|X_n-\E X_n|>xQ_{n}]\leq \exp(1-x^\theta).
\end{equation}
In our analysis we will inductively normalize the centred passage times by $Q_n$. Observe that since $\alpha>0$, by definition $Q_{n}$ is strictly increasing in $n$; We use $Q_n$ rather than $\hQ_n$ in order to avoid the possibility that there is a large range of $n$ for which $\hQ_n$ does not increase (or possibly decreases).

By construction we have that
\begin{equation}\label{eq:QnRootNBound}
\hQ_n \leq Q_n \leq D_0^{1/\kappa}\sqrt{n},
\end{equation}
since $\alpha<\frac12$.  Note that
\begin{align}\label{eq:varUpperQn}
\Var(X_n) &= \int_0^\infty 2x \P[|X_n-\E X_n| > x] dx \nonumber\\
&\leq \int_0^\infty 2x \exp(1-(x/\hQ_n)^\theta) dx \nonumber\\
&= \hQ_n^2 \int_0^\infty 2x \exp(1-x^\theta) dx \leq C \hQ_n^2,
\end{align}
and hence $\SD(X_n) \leq C\hQ_n \leq C Q_n$. 
Much of our proof goes to establishing the other direction of these inequalities and showing that
\begin{equation}\label{eq:QequivSD}
Q_n\asymp \hQ_n \asymp \SD(X_n)
\end{equation}
which immediately implies Theorem~\ref{t:tightness}.

\subsection{Outline of the argument}
Now we give a sketch of our argument showing \eqref{eq:QequivSD}. Observe first that since $\SD(X_n)$ grows at least as fast as $n^{1/12}$, which is faster than $n^\alpha$, there must be a sequence of $n_i$ tending to infinity such $Q_{n_i} = \hQ_{n_i}$. We shall call such $n_i$ the $\alpha$-record points, and $n\ge 1$ such that $Q_{n}=\sup_{1\le m\le n} (n/m)^{\alpha}\hQ_{m}\le C\hQ_{n}$ will be called $(C,\alpha)$-quasi record points. We shall establish \eqref{eq:QequivSD} by showing that (i) for all quasi record points $n$, we have $\Var X_{n}\ge CQ^2_{n}$, and (ii) every $n\ge 1$ is a quasi-record point. By way of establishing (i) and (ii) via a multi-scale argument, we shall establish Theorem~\ref{t:tightness} together with the upper bounds in Theorems~\ref{t:tf} and~\ref{t:nr}. Lower bounds in Theorems \ref{t:tf} and \ref{t:nr} are proved separately later; we shall give a brief outline of that argument at the end of this section. 

\noindent
\textbf{Transversal Fluctuations and Canonical rectangles.}
For points $u,v\in \R^2$ we denote by $\trans_{uv}$ the transversal fluctuation of $\gamma_{uv}$ defined as the maximal distance from the line passing through $u$ and $v$.  If $\underline{e}$ is a unit vector perpendicular to $v-u$ then this is equivalent to

\begin{equation}
 \label{e:tfdefn}
\trans_{uv} := |\max_{w\in\gamma_{uv}} \underline{e}\cdot(w-u)|.
\end{equation}

We set $\trans_n:=\trans_{\origin,(n,0)}$.  We define 
\begin{equation}
    \label{e:wndef}
    W_n=\sqrt{n Q_n}.
\end{equation}
The scaling relation of \eqref{eq:scalingRelation2} suggests we should expect transversal fluctuations of a path of length $n$ should be of the order $W_{n}$ (this is the content of Theorem \ref{t:tf}). It is expected (and known in the case of exactly solvable LPP models) that the collection of (centred and scaled) passage times between points in a rectangle of size $n\times W_{n}$ are uniformly tight. Therefore we consider such rectangles to be canonical on-scale rectangles. To move inductively from one scale to the next it is useful to control passage times from one side of a rectangle to another; it will be convenient to do this for rectangles and parallelograms that are more general than the canonical rectangles. In most cases the shorter sides will be parallel to the $y$-axis. We will generally refer to a rectangle as $\cR$ and let $L_\cR$ and $R_\cR$ denote its left and right sides (the shorter sides). Define the maximal and minimal side to side distances
\begin{align*}
Y^+_{\cR}:=\sup_{u\in L_\cR,v\in R_\cR} X_{uv},\\
Y^-_{\cR}:=\inf_{u\in L_\cR,v\in R_\cR} X_{uv}.
\end{align*}
Let $\cR_{n,W}$ denote the $n\times W$ rectangle with corners $(0,0),(n,0),(0,W)$ and $(n,W)$ ($W$ will typically vary around $W_{n}$ but can differ from it by only a factor that is a small polynomial of $n$). In the case of the canonical rectangles (i.e., when $W=W_{n}$) we shall denote it by $\cR_n$ and the corresponding passage times will be denoted $Y_n^\pm:=Y^\pm_{\cR_n}$.

\noindent
\textbf{Concatenation Bounds.}
When bounding passage times at larger scales in terms of passage times at smaller scales, we shall need to compare passage times of concatenated paths to the sum of passage times of their pieces (recall that we do not necessarily have passage times of a concatenated path to be equal to the sum of the passage times of its pieces for all our models). In order to give bounds where we concatenate paths and take union bounds over integer points in the plane we introduce the variable $\Gamma_n$ defined as
\begin{align*}
\Gamma_n:&= \max_{\substack{u,v\in B_{2n}(\origin)\\|u-v|\leq 1}} X_{uv} + \max_{u,v\in B_{2n}(\origin)} \max_{M\geq 1} \frac1{M} \max_{0<t_1<\ldots<t_M<1} \sum_{i=1}^{M+1} X_{\gamma_{uv}(t_{i-1})\gamma_{uv}(t_i)} - X_{uv}\\
&\qquad + \max_{u\in B_{2n}(\origin)}\max_{v\in\R^2} \max_{M\geq 1} \frac1{M} \max_{\substack{0<t_1<\ldots<t_M<1\\ |\gamma(t_i)| \leq 2n}} \sum_{i=1}^{M+1} X_{\gamma_{uv}(t_{i-1})\gamma_{uv}(t_i)} - X_{uv}.
\end{align*}
In the above equation and for its later usages, the notations $B_{r}(v)$ will denote the ball of radius $r$ around the point $v\in \R^2$. The following easy lemma, which follows immediately by Assumptions~\ref{as:localDist} and~\ref{as:intermediate} and a union bound, will be very useful to us. 
\begin{lemma}\label{l:Gamma}
There exists $C>0$ such that for any $n\geq 1$ and $x>0$
\[
\P[\Gamma_n > Cx\log^{C} n] \leq \exp(1-x^{\kappa}).
\]
\end{lemma}

\noindent
\textbf{Percolation argument and concentration.}
The first step is to prove that the definition of $Q_{n}$ implies that $Q_n^{-1}(Y^{-}_{n}-n\mu)$ has a stretched exponentially decaying left tail (in fact we prove something more general, see Section \ref{s:rectpara}). Using this, together with a percolation estimate, we first show that the geodesic from $\mathbf{0}$ to $(Mn,0)$ behaves sufficiently regularly (i.e., it cannot fluctuate too much), and then using it show that (see Lemma~\ref{l:AnBound})
\begin{equation}
    \label{e:an}
    A_{n}\le D_1 Q_{n}. 
\end{equation}
This will eventually prove the upper bound in Theorem \ref{t:nr} once we have \eqref{eq:QequivSD}. The next step is to show that $Q_n$ cannot grow too quickly (Lemma \ref{l:Qgrowth}). Using these, we finally show that $Q_{n}^{-1}|Y^{\pm}_{n}-n\mu|$ have stretched exponential concentration with tails decaying as $\exp(-Cx^{\theta})$ (see Lemma \ref{l:YMinusBound}, Lemma~\ref{l:YPlusBound}). 

\noindent
\textbf{Transversal fluctuations.}
The next step is to prove an upper bound of the transversal fluctuation $\trans_{n}$ at scale $W_{n}=\sqrt{nQ_{n}}$ (Theorem \ref{t:trans.main}). This done by first bounding the transversal fluctuation at the line $x=n/2$ and then using a chaining argument at dyadically decreasing scales. Once \eqref{eq:QequivSD} is established this immediately implies the upper bound in Theorem \ref{t:tf}.

Up until this point we have proved several estimates at scales $Q_{n}$ and $W_{n}$, which, once \eqref{eq:QequivSD} is established will give our main results. We note that analogues of some results in Sections \ref{s:percgen} and \ref{s:trans} are known in exactly solvable set-up \cite{BSS14, BGZ21} or conditionally in FPP \cite{Ale20,Ale21} under hypotheses similar to \eqref{eq:Q} and rotational invariance/curvature of limit shape. It is also worth pointing out that the arguments so far do not use Assumption \ref{as:varlb} that the variance grows at least polynomially. Note also that this is the only one among our assumptions which required planarity so all the results up until this point could be adapted to higher dimensions as well. Now we move towards establishing \eqref{eq:QequivSD} which will require crucial use of planarity (and Assumption \ref{as:varlb}). To this end we shall use the notion of (quasi-) record points defined above.

\noindent
\textbf{Record points.}
Recall our two step strategy to establish \eqref{eq:QequivSD} described at the start of this subsection. We first improve the results of Section \ref{s:percgen} to show that the exponent $\theta$ in Lemma \ref{l:YMinusBound} can be improved to $4\theta/3$ sufficiently far in the tails, and a similar improved concentration holds for $X_{n}$ (see Lemma \ref{l:ImprovedInduction}) . 

Using the improved concentration above, one can show that for record points (in fact for quasi-record points) $n$, one has $\Var(X_{n})\ge CQ^2_{n}$ (Lemma \ref{l:varbd}), which concludes the step (i) of showing~\eqref{eq:QequivSD}.

Next, we show that for each record point there exist numerous nearby record points at a range of nearly geometric scales (Lemma \ref{l:grmain}). We also improve Lemma \ref{l:Qgrowth} to show that $Q_{n}$ cannot grow too fast, in particular that it grows locally sublinearly (Lemma \ref{l:growth34}).

\noindent
\textbf{Local transversal fluctuations.}
We also need a local version of the transversal fluctuation estimate. In Section \ref{s:ltf} We show that for a geodesic from $(0,0)$ to $(n',0)$ for $n'\gg n$, the transversal fluctuations at the line $x=n$ has stretched exponential tails at scale $W_{n}$ (see Corollary \ref{c:local.trans} for a precise statement). Unlike the global transversal fluctuation results of Section \ref{s:trans}, this requires more control on the growth of $W_{n}$ given by Lemma \ref{l:growth34}. This result was also previously known in  exactly solvable set-up \cite{BSS19}.

\noindent
\textbf{Passage time tails carry positive mass arbitrarily far away to the left.} The final ingredient needed to show that all points are quasi-record points is the following: we show that given $L>0$, the probability that $Q_{n}^{-1}(X_{n}-n\mu)\le -L$ is bounded away from $0$ (uniformly in $n$)
for all sufficiently large quasi-record points $n$ (Proposition \ref{p:left1}). This shows that there is some positive probability of having a very good path (better than typical by an arbitrarily large multiple of  $Q_n$) at the scale of quasi-record points $n$, this will be used to construct favourable events in the following percolation argument. 

\noindent
\textbf{Lower bounding the variance.}
The final piece of the argument showing \eqref{eq:QequivSD} is to show that for sufficiently large and fixed $M$, we have 
\begin{equation}
    \label{e:varlb}
    \Var (X_{Mn})\ge CM^{1/10}\Var (X_{n})
\end{equation}
for all sufficiently large record points $n$. Using the growth estimate on $Q_{n}$ (Lemma \ref{l:growth34}) and \eqref{e:varlb} we next show that there are record points $n_0, n_{1}, \ldots$ such that $\frac{n_{i+1}}{n_{i}}\in [2,M]$ for all $i$. It follows that for every $n'\ge n_0$, there exists a record point $n\in [\frac{n'}{M},n']$. By a stronger variant of Proposition~\ref{p:left1} (see Proposition \ref{p:left}) this implies that there exists $\delta>0$ such that 
$$\P(X_{n'}\le n'\mu-Q_{n'})\ge \delta$$
which immediately implies $\Var (X_{n'})\ge \delta Q^2_{n'}$ showing that $n'$ is a quasi-record point. 

Establishing \eqref{e:varlb} is the most technically challenging part of this paper. Although the precise details of the implementation of the argument is slightly different from what is described below; essentially we do the following. The idea is to implement a block version of the Newman-Piza argument from \cite{NP95}. We divide the plane into a grid of $n\times W_{n}$ sized blocks. First we show by a percolation argument that with large probability most of the blocks the geodesic from $\mathbf{0}$ to $(Mn,0)$ passes through satisfy certain regularity properties. We decompose the variance into sum of contributions from resampling the blocks one by one. 
We show that on the event that the geodesic passes through a block $B$, resampling it contributes at least $CQ_{n}^2$ with positive probability. This shows (essentially)
$$\Var (X_{Mn})\ge CQ_{n}^2\sum_{B} p^2_{B}$$
where the sum is over ``all" blocks and $p_{B}$ denotes the probability of the geodesic passing through the block $B$. By the transversal fluctuation estimates (Theorem \ref{t:trans.main}) and the fact that $W_{Mn}=O(M^{7/8}W_{n})$ (which follows from the fact that $Q_{n}$ grows locally sub-linearly, Lemma \ref{l:growth34}) it follows that one essentially only needs to sum over $M^{15/8+\phi}$ many blocks for some small $\phi>0$, and the proof is completed by the Cauchy-Schwarz inequality, observing $\sum p_{B}=\Theta(M)$ and using the fact that $n$ is a record point. We point out that implementing this is technically much more difficult than in \cite{NP95} where a single edge was resampled at each step and requires carefully designing favourable events at multiple scales.

\noindent
\textbf{Lower bounds in Theorems \ref{t:tf} and \ref{t:nr}.}
For the lower bound in Theorem \ref{t:tf} we first prove a weaker statement which shows that there is a positive probability to have transversal fluctuation at scale $W_{n}$ that is bounded away from 0 (Lemma \ref{l:tflowerw}); in fact we show that for some $\delta>0$, the geodesic is better than all paths with transversal fluctuation less than $\delta W_{n}$ by at least an amount $\delta Q_{n}$ with probability at least $\delta$. The idea is to show that with positive probability there exists a good path nearby, and by a surgery near the end points one can use this good path to do better than all paths with small transversal fluctuation (see Figure \ref{f:trans.lower}). 

For the proof of the lower bound of Theorem \ref{t:nr}, we show using the above argument that with probability at least $\delta>0$, $X_{2n}$ does better than $X_{n}+X_{\mathbf{n},2\mathbf{n}}$ by at least $\delta Q_{2n}$. This together with the sub-additive structure of the sequence $X_{n}$ shows the lower bound in Theorem \ref{t:nr} (Lemma \ref{l:mean}). 

To complete the proof of the lower bound in Theorem \ref{t:tf} we  show that the passage time of the best path from $\mathbf{0}$ and $\mathbf{n}$ that has transversal fluctuation at most $\delta W_{n}$ for some small $\delta$ can be approximated by sums of many (approximately) independent passage times of length $n'$ where $n'\le n$ is such that $W_{n'}\approx \delta W_{n}$. Using Lemma \ref{l:meanpara} (which is a strengthened version of Lemma \ref{l:mean}) we show that each of these pieces lead to a penalty in mean of the order $Q_{n'}$ it follows that the sum will experience a penalty $\frac{n}{n'}Q_{n'}\gg Q_{n}$ (the last inequality can be shown by our estimates on the growth of $Q_{n}$). This together with Theorem \ref{t:tightness} shows that it is unlikely for the geodesic to have transversal fluctuation smaller than $\delta W_{n}$, as required (see the proof of Proposition \ref{p:tflower}).  

\subsection{Auxiliary results of independent interest}
For easy reference, here we would like to record a couple of results which we establish en route our arguments proving the main theorems. We believe these results to be of importance in future works. 

The first result deals with the concentration of first passage times across the canonical rectangle at the standard deviation scale. 

\begin{proposition}
    \label{p:conc}
    There exists $C,c>0$ such that for all $n\ge 1$ and $x>0$ we have 
    \begin{enumerate}
        \item[(i)] $\P(|Y^{+}_{n}-n\mu|\ge x \mathrm{SD}(X_{n}))\le Ce^{-cx^{\theta}}$;
        \item[(ii)] $\P(|Y^{-}_{n}-n\mu|\ge x \mathrm{SD}(X_{n}))\le Ce^{-cx^{\theta}}.$
    \end{enumerate}
\end{proposition}

These results are shown in Lemma \ref{l:YMinusBound} and Lemma \ref{l:YPlusBound} with $\mbox{SD}(X_{n})$ replaced by $Q_{n}$; and Proposition \ref{p:conc} follows immediately from those lemmas once \eqref{eq:QequivSD} is established for all $n\ge 1$. Similar results have been established in exactly solvable models \cite{BSS14, BGZ21} and turned out to be extremely useful in different problems in that setting; this will also be important in future studies of rotationally invariant FPP models. In particular, this estimate will be useful for our upcoming work on polynomial improvement on the upper bound of variance. 

The second result deals with existence of good paths. 

\begin{proposition}
    \label{p:lowertail}
    For $L>0$, there exists $\delta=\delta(L)>0$ such that 
    $$\P(X_{n}<n\mu-L\mathrm{SD}(X_{n}))\ge \delta$$
    for all $n$ sufficiently large.
\end{proposition}

This is first shown for quasi-record points $n$ with $\mbox{SD}(X_{n})$ replaced by $Q_{n}$ in Proposition~\ref{p:left1}; Proposition \ref{p:lowertail} immediately follows once \eqref{eq:QequivSD} is proved for all $n\ge 1$. This result is often useful in constructing favourable events (as in Section \ref{s:varlb}). Similar estimates have been proved to be very useful in the exactly solvable set-up and we expect that this will also play an important role in future works.

\subsection*{Notational Comment} We shall use $C,c,D$ etc. to denote generic constants that can change from equation to equation and also from line to line within the same equation. Specific constants which will be used multiple times will be marked with numbered subscripts, i.e., $C_1,D_2$ etc.

\subsection*{Organization of the paper}
The rest of the paper is organised as follows. Our multi-scale argument will require control of passage times not only across the sides of a canonical rectangle, but also of parallelograms. Such estimates are obtained in Section \ref{s:rectpara} by comparing passage times across parallelograms and rectangles. Section \ref{s:percgen} employs a general percolation argument to upper bound the non-random fluctuations, and obtains concentration of $Y^{\pm}_{n}$ at scale $Q_{n}$. Section \ref{s:trans} upper bounds the (global) transversal fluctuations of the geodesics at scale $W_n$. The results up until this point do not require the variance to grow polynomially and would be valid for rotationally invariant models all dimensions. The next sections utilize the planarity assumption. Section \ref{s:impconc} obtains improved concentration estimates which are used in Section \ref{s:record} to lower bound the variance for record points. Section \ref{s:record} also establishes the existence of many record points near a record point, and shows that $Q_{n}$ grows locally sublinearly. Using the control on the growth of $Q_{n}$, Section \ref{s:ltf} bounds local transversal fluctuations of a geodesic (of length typically $\gg n$) at a distance $n$ from its starting point at scale $W_{n}$. Section \ref{s:left} proves the left tail estimate Proposition \ref{p:lowertail} for the case of quasi-record points. Section \ref{s:varlb} completes the proof of \eqref{eq:QequivSD} by showing \eqref{e:varlb}. This completes the proof of Theorem \ref{t:tightness} and also the upper bounds in Theorems \ref{t:tf} and \ref{t:nr}. Finally, the lower bounds in Theorems \ref{t:tf} and \ref{t:nr} are proved in Section \ref{s:lower}. The argument in Section \ref{s:varlb} crucially uses a percolation estimate Proposition \ref{p:percevent} whose proof is provided in Section \ref{s:percevent}. The proof of another crucial percolation estimate,  Proposition \ref{p:perc1} is provided in Section \ref{s:appa}.

\section{Passage times across rectangles and parallelograms}
\label{s:rectpara}

Our first estimate gives of comparison of passage times between points close to the sides of a rectangle to passage times of their projection onto the side. It will be useful to consider rectangles with heights different from the canonical rectangle so in general we set let $W$ be the height and assume that
\begin{equation}\label{eq:variableQW}
n^{\frac12\alpha}\leq Q \leq  n^{\frac12+2\epsilon},\qquad W=\sqrt{n Q}
\end{equation}
which implies that
\[
n^{\frac12+\frac14\alpha}\leq W \leq n^{\frac34+\epsilon}.
\]
By Pythagoras' Theorem, and the fact that $\sqrt{x}$ has negative second derivative and positive third derivative
\begin{equation}~\label{eq:Pythag}
\max\{n,n+\frac{y^2}{2n} - \frac18\frac{y^4}{n^3}\} \leq |(n,y)|:=\sqrt{n^2+y^2} \leq n + \frac{y^2}{2n}.
\end{equation}
Furthermore since the derivative of $\sqrt{n^2+y^2}$ is increasing in $y$ positive, for $y\geq n$ we have that $\sqrt{n^2+y^2} \geq n+y/3$ and so
\begin{equation}~\label{eq:Pythag2}
|(n,y)|\geq n+(\frac{y^2}{3n}\wedge |y|/3).
\end{equation}

Define the sets
\begin{align*}
L^{(k)}_{\cR_{n,W}}&:=\{(x,y):\frac{(-k^2-2k)Q}{4}\leq x \leq \frac{(-3k^2+8k)Q}{16}, 0\leq y\leq W\},\\
R^{(k)}_{\cR_{n,W}}&:=\{(x,y): \frac{(3k^2-8k)Q}{16} \leq x-n\leq \frac{(k^2+2k)Q}{4}, 0\leq y\leq W\}
\end{align*}
For points $u_1=(x_1,y_1)$ and $u_2=(x_2,y_2)$ in $\R^2$ with $0\leq y_1,y_2\leq W$ let $u^\perp_1=(0,y_1)$ and $u^\perp_2=(n,y_2)$ be their projections onto $L_{\cR_{n,W}}$ and $R_{\cR_{n,W}}$ respectively (see Figure \ref{f:rect.ext1}).  We will compare $X_{u_1 u_2}$ and $X_{u^\perp_1 u^\perp_2}$. For convenience of taking union bounds, recall that we defined, for $u\in \R^2$, $u^{\Z}\in \Z^2$ which is the closest point to $u$ with some arbitrary rule for tie-breaking. 

\begin{lemma}\label{l:rectangleApprox}
There exist a constant $C>0$ such that for all $n$ sufficiently large and all $Q,W$ satisfying~\eqref{eq:variableQW} and all $k\in [1,\leq n/W]$ the following holds.  With $u_i,u^\perp_i$ defined as above,
\[
\P\Big[\max_{u_1\in L^{(k)}_{\cR_{n,W}}}\max_{u_2\in R^{(k)}_{\cR_{n,W}}} \left| X_{u_1 u_2}-X_{u^\perp_1 u^\perp_2} - (x_2-x_1-n)\mu \Big| >  (k^2 Q)^{9/10} \right] \leq \exp(1-C Q^{\kappa/10}).
\]
\end{lemma}

\begin{center}
\begin{figure}
\includegraphics[width=5in]{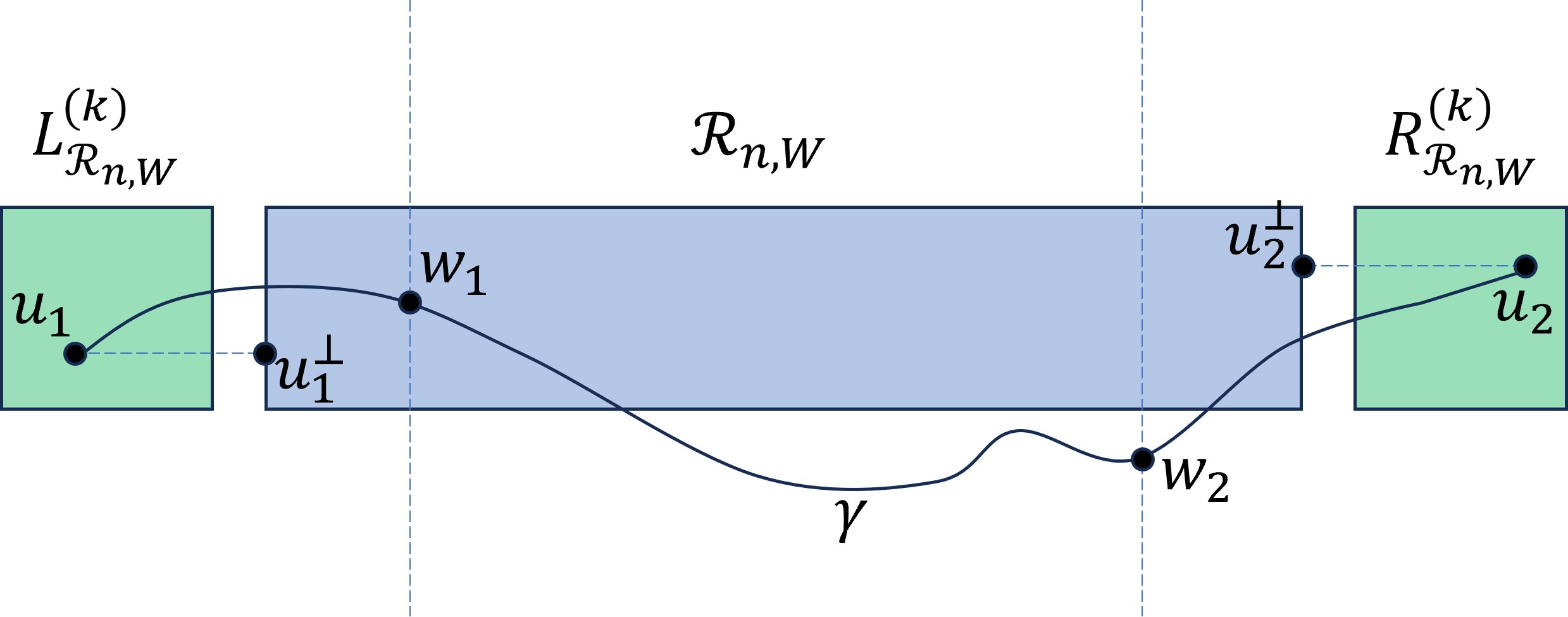}
\caption{For the rectangle $\cR_{n,W}$, Lemma \ref{l:rectangleApprox} compares the passage time $X_{u_1u_2}$
from a point $u_1$ slightly to its left to a point $u_2$ slightly to its right to the passage time $X_{u^\perp_1 u^\perp_2}$; it shows that this maximum of the difference (appropriately) as $u_1$ and $u_2$ vary over the green rectangles is unlikely to be large. This is achieved by considering the intermediate points $w_1$ and $w_2$ where the geodesic $\gamma_{u_1u_2}$ intersects intersects the line $x=2Q$ and $x=n-2Q$ respectively.}
\label{f:rect.ext1}
\end{figure}
\end{center}

\begin{proof}
Define
\[
J_{u_1 u_2}:=X_{u_1 u_2}-X_{u_1^\perp u_2^\perp} - (x_2-x_1-n)\mu.
\]
Let $\ell_1,\ell_2$ be the vertical lines $x=\lfloor 2 Q\rfloor$ and $x=\lfloor n- 2 Q\rfloor$ respectively.  Let $w_1=(\lfloor 2 Q \rfloor ,H_1)$ be the first intersection of $\gamma_{u_1^\Z u_2^\Z}$ with $\ell_1$ and let $w_2=(\lfloor n-2 Q\rfloor ,H_2)$ be its last intersection with $\ell_2$, see Figure~\ref{f:rect.ext1}.  Let $\cH_i$ be the event that for every choice of $u_1\in L^{(k)}_{\cR_{n,W}}$ and $u_2\in R^{(k)}_{\cR_{n,W}}$ we have that $|H_i-y_i^\Z| \leq (k^2Q)^{9/10}$.  By Assumption~\ref{as:cars} on local transversal fluctuations and taking a union bound over the $O(n^4)$ choices of $(u_1^\Z,u_2^\Z)$ we have that
\begin{equation}\label{eq:cHbound}
\P[\cH_i^c] \leq n^4\exp\Bigg(1-\bigg(\frac{(k^2Q)^{9/10}}{\big(2 Q+\frac{ ( (3k^2-8k)Q}{16}\big)^{4/5}}\bigg)^\kappa\Bigg) \leq \frac1{10}\exp(1-CQ^{\kappa/10}).
\end{equation}
Let $R_i=|u_i^\Z - w_i^\Z|$ and $S_i=|u_i^{\perp\Z} - w_i^\Z|$.
On the event $\cH_i$, we have that
\[
|R_i|,|S_i| \leq \frac{(k^2+2k)Q}{4} + 2Q + (k^2Q)^{9/10} \leq n \wedge 4k^2 Q,
\]
and $u_i,v_i,w_i,w_i^\Z\in B_{2n}(\origin)$.  Next let
\[
\Delta_1(u_1,u_2) := x_1\mu + X_{u_1^{\perp\Z} w_1^\Z} - X_{u_1^\Z w_1^\Z},\qquad \Delta_2(u_1,u_2) := (x_2-n)\mu + X_{u_2^{\perp\Z} w_2^\Z} - X_{u_2^\Z w_2^\Z}.
\]
By Assumption~\ref{as:intermediate},
\begin{align}\label{eq:JBound}
J_{u_1 u_2} &\geq X_{u_1^\Z u_2^\Z}-X_{u_1^{\perp\Z} u_2^{\perp\Z}} - (x_2-x_1-n)\mu - 4\Gamma_{2n}\nonumber\\
&\geq X_{u_1^\Z w_1}+X_{w_1 w_2} + X_{w_2 u_2^\Z} -X_{u_1^{\perp\Z} u_2^{\perp\Z}} - (x_2-x_1-n)\mu - 6\Gamma_{2n}\nonumber\\
&\geq X_{u_1^{\perp\Z} w_1^\Z}+X_{u_1^\Z w_1}-X_{u_1^\Z w_1^\Z}+X_{w_1 w_2} + X_{w_2 u_2^{\perp\Z} } +X_{w_2 u_2^\Z}-X_{w_2^\Z u_2^\Z} -X_{u_1^{\perp\Z} u_2^{\perp\Z}} - 6\Gamma_{2n}-\Delta_1-\Delta_2\nonumber\\
&\geq X_{u_1^{\perp\Z} w_1^\Z}+X_{w_1^\Z w_2^\Z} + X_{w_2^\Z u_2^{\perp\Z} }  -X_{u_1^{\perp\Z} u_2^{\perp\Z}} - 10\Gamma_{2n}-\Delta_1-\Delta_2\nonumber\\
&\geq - 10\Gamma_{2n}-\Delta_1-\Delta_2.
\end{align}
We bound $\Delta_1$ by
\begin{equation}\label{eq:DeltaBound}
|\Delta_1| \leq |X_{u_1^{\perp\Z} w_1^\Z}-\mu S_1 - A_{S_1}| + |X_{u_1^\Z w_1^\Z}- \mu R_1 - A_{S_1}| + |\mu R_1 - \mu S_1 + x_1 \mu| + |A_{S_1} +  A_{R_1}|,
\end{equation}
recalling that $A_n=\E X_n -n\mu$.  On the event $\cH_1$ we have that
\begin{equation}\label{eq:DeltaBound2}
|\mu R_1 - \mu S_1 + x_1 \mu| \leq 2\mu (k^2 Q)^{9/10}.
\end{equation}
By Assumption~\ref{as:speed},
\begin{align}\label{eq:DeltaBound3}
A_{S_1},A_{R_1} \leq \Big(\frac{(k^2+2k)Q}{4} + 2Q + (k^2Q)^{9/10}\Big)^{\frac12+\epsilon}\leq C(k^2Q)^{\frac12+\epsilon}
\end{align}
Noting that the points $u_1^\Z, u_1^{\perp\Z}, w_1^\Z$ are contained in $B_{2n}(\origin)$ on the event  $\cH_1$, by Assumption~\ref{as:concentration} on the concentration of passage times and taking a union bound we have that
\begin{align}\label{eq:concentrationUnion}
&\P[\cH_1,\max_{u_1,u_2} |X_{u_1^{\perp\Z} w_1^\Z}-\mu S_1 - A_{S_1}| + |X_{u_1^\Z w_1^\Z}- \mu R_1 - A_{R_1}| > 2(k^2 Q)^{7/8}]\nonumber\\
& \qquad \leq  \P[\max_{\substack{ u,u'\in B_{2n}(\origin)\cap\Z^2\\ |u-v|\leq 4k^2 Q}} |X_{uu'}-\E X_{uu'}|> (k^2 Q)^{7/8}]\nonumber\\
& \qquad \leq C n^4 \exp(1-C((k^2Q)^{1/3})^\kappa) \leq \exp(1-C(k^2 Q)^{\kappa/3}).
\end{align}
Altogether, combining equations~\eqref{eq:cHbound}, \eqref{eq:DeltaBound}, \eqref{eq:DeltaBound2}, \eqref{eq:DeltaBound3}, and~\eqref{eq:concentrationUnion} we have that
\[
\P[\max_{u_1,u_2}|\Delta_1(u_1,u_2)|\geq \frac1{5}(k^2 Q)^{9/10}] \leq \exp(1-C' Q^{\kappa/10}).
\]
Similarly
\[
\P[\max_{u_1,u_2}|\Delta_2(u_1,u_2)| \geq\frac1{5}(k^2 Q)^{9/10}] \leq \exp(1-C' Q^{\kappa/10}).
\]
By Lemma~\ref{l:Gamma}
\[
\P[\Gamma_{n} \geq \frac1{100}(k^2 Q)^{9/10}]\leq \exp(1-Q^{\kappa/2}).
\]
Combining the last three estimates with equation~\eqref{eq:JBound} we have that
\[
\P[\max_{u_1,u_2}J_{u_1 u_2}\geq  \frac1{2}(k^2 Q)^{9/10} ]\leq \exp(1-CQ^{\kappa/10}).
\]
By setting $w_i'$ as the (first for $i=1$, last for $i=2$) intersection of $\gamma_{u_1^{\perp\Z} u_2^{\perp\Z}}$ with $\ell_i$ and reversing the role of $u_i$ and $u^\perp_i$ we similarly have the reverse inequality that
\[
\P[\min_{u_1,u_2} -J_{u_1 u_2}\leq - \frac1{2}(k^2 Q)^{9/10} ]\leq \exp(1-CQ^{\kappa/10}),
\]
which completes the proof.
\end{proof}

\subsection{Parallelograms}

As the path can meander up outside of the bounds of a rectangle we also consider passage times across parallelograms which we will denote with $\cP$.  In general the parallelograms we construct will have left and right sides parallel to the $y$-axis.
We will let $\cP_{i,k,k',n,W}$ denote the parallelogram with left side
\[
L_{\cP_{i,k,k',n,W}} := \{((i-1)n,y):y\in [kW,(k+1)W]\}
\]
and right side
\[
R_{\cP_{i,k,k',n,W}} := \{(in,y):y\in [k'W,(k'+1)W]\}
\]
When the values of $n$ and $W$ are clear we will just write $\cP_{i,k,k'}$.  Observe that the parallelogram $\cP_{1,0,0,n,W}$ is simply the $n\times W$ rectangle $\cR_{n,W}$.  We now use Lemma~\ref{l:rectangleApprox} to relate crossing probabilities of parallelograms with rectangles.  First, analogously to $Y^\pm_n$ we denote
\[
Z^+_{i,k,k'}:=\sup_{u_1\in L_{\cP_{i,k,k'}}}\sup_{u_2\in R_{\cP_{i,k,k'}}} X_{u_1 u_2}, \qquad Z^-_{i,k,k'}:=\inf_{u_1\in L_{\cP_{i,k,k'}}}\inf_{u_2\in R_{\cP_{i,k,k'}}} X_{u_1 u_2}.
\]
By translation and reflection symmetry,
\[
Z^\pm_{i,k,k'}\stackrel{d}{=} Z^+_{1,0,|k-k'|}
\]
and therefore it only suffices to study $Z^{+}_{1,0,k}$ for $k>0$. Now given a parallelogram $\cP$ with dimensions $n$ and $W$ we will construct an $n\times W$ rectangle $\cR_{\cP}$, called the embedded rectangle, in the middle of $\cP$.  We define $\cR_{\cP}$ to be the rectangle that has the same center as $\cP$ and whose edges of length $n$ are parallel to the parallelograms top and bottom edges.  When $\cP=\cP_{1,0,k}$ we will write $\cR_{1,0,k}=\cR_{\cP_{1,0,k}}$ for the embedded rectangle. More formally $\cR_{1,0,k}$ is the rectangle with centre at $(n/2,(k+1)W/2)$ whose length $n$ edge is parallel to the line $y=\frac{kW}{n} x$.  Alternatively let $\tau\in ISO(2)$ be the composition of a translation of $(0,kW)$ followed by a rotation of $\tan^{-1} \frac{kW}{n}$ around $(n/2,(k+1)W/2)$.  Then $\cR_{1,0,k}=\tau \cR_{n,W}$; see Figure \ref{f:para.ext}.

Using basic geometry we will show that $L_{\cP_{1,0,k}}\subset\tau L^{(k)}_{\cR_{n,W}}$ and $R_{\cP_{1,0,k}}\subset\tau R^{(k)}_{\cR_{n,W}}$.
For $u_1\in L_{\cP_{1,0,k}}$ let $u^\perp$ be the orthogonal projection onto $L_{\cR_{1,0,k}}$.  Similarly for $u_2\in R_{\cP_{1,0,k}}$ let $u_2^\perp$ be the orthogonal projection onto $R_{\cR_{1,0,k}}$.
Let $\vec{v}=(n,kW)/|(n,kW)|$ be a unit  vector in the direction parallel to the long edge of $\cP_{1,0,k}$.
For $u_*=(0,W/2)$, the midpoint of $L_{\cP_{1,0,k}}$, its orthogonal projection is the midpoint of $L_{\cR_{1,0,k}}$ and so measuring the distance to the centre of the rectangle we have that by~\eqref{eq:Pythag}
\begin{align}~\label{eq:pythagPara}
(u_*^\perp-u_*)\cdot\vec{v} &=|u_* - (n/2,(k+1)W/2)| - \frac{n}{2}  = \sqrt{(\frac{n}{2})^2+(\frac{kW}{2})^2}-\frac{n}{2}\leq \frac{(kW)^2}{4n}=\frac{k^2 Q}{4}\nonumber\\
(u_*^\perp-u_*)\cdot\vec{v} &\geq \frac{k^2 Q}{4} - k^2 Q \frac{(kW)^2}{16n^2}
\end{align}
since $Q=\frac{W^2}{n}$.  
For $u_1=(0,\frac{W}{2}+y) \in L_{\cP_{1,0,k}}$, the difference between $|u-u^\perp|$ and $|u_*-u_*^\perp|$ is $|(0,y)\cdot \vec{v}|$, the distance between $u_*$ and the projection of $u_1$ onto the line joining $u_*$ to the centre of the rectangle; see Figure~\ref{f:para.ext}.  Thus by trigonometry,
\[
(u_1^\perp-u_1)\cdot\vec{v} - (u_*^\perp-u_*)\cdot\vec{v} = -(0,y)\cdot \vec{v}
\]
and
\[
|(0,y)\cdot \vec{v}|=|y \sin(\tan^{-1}(\frac{kW}{n}))|\leq \frac{W}{2} \cdot \frac{kW}{n}=\frac{Qk}{2}.
\]
Hence if $k\leq n/W$,
\begin{equation}~\label{eq:perpDistance}
\frac{(3k^2-8k)Q}{16} \leq \frac{k^2 Q}{4} - k^2 Q \frac{(kW)^2}{16n^2} - \frac{Qk}{2} \leq (u_1^\perp-u_1)\cdot\vec{v} \leq \frac{k^2 Q}{4} + \frac{Qk}{2}
\end{equation}
so $L_{\cP_{1,0,k}}\subset\tau L^{(k)}_{\cR_{n,W}}$ and similarly $R_{\cP_{1,0,k}}\subset\tau R^{(k)}_{\cR_{n,W}}$.  Then by Lemma~\ref{l:rectangleApprox} we have the following immediate corollary.
\begin{center}
\begin{figure}
\includegraphics[width=5in]{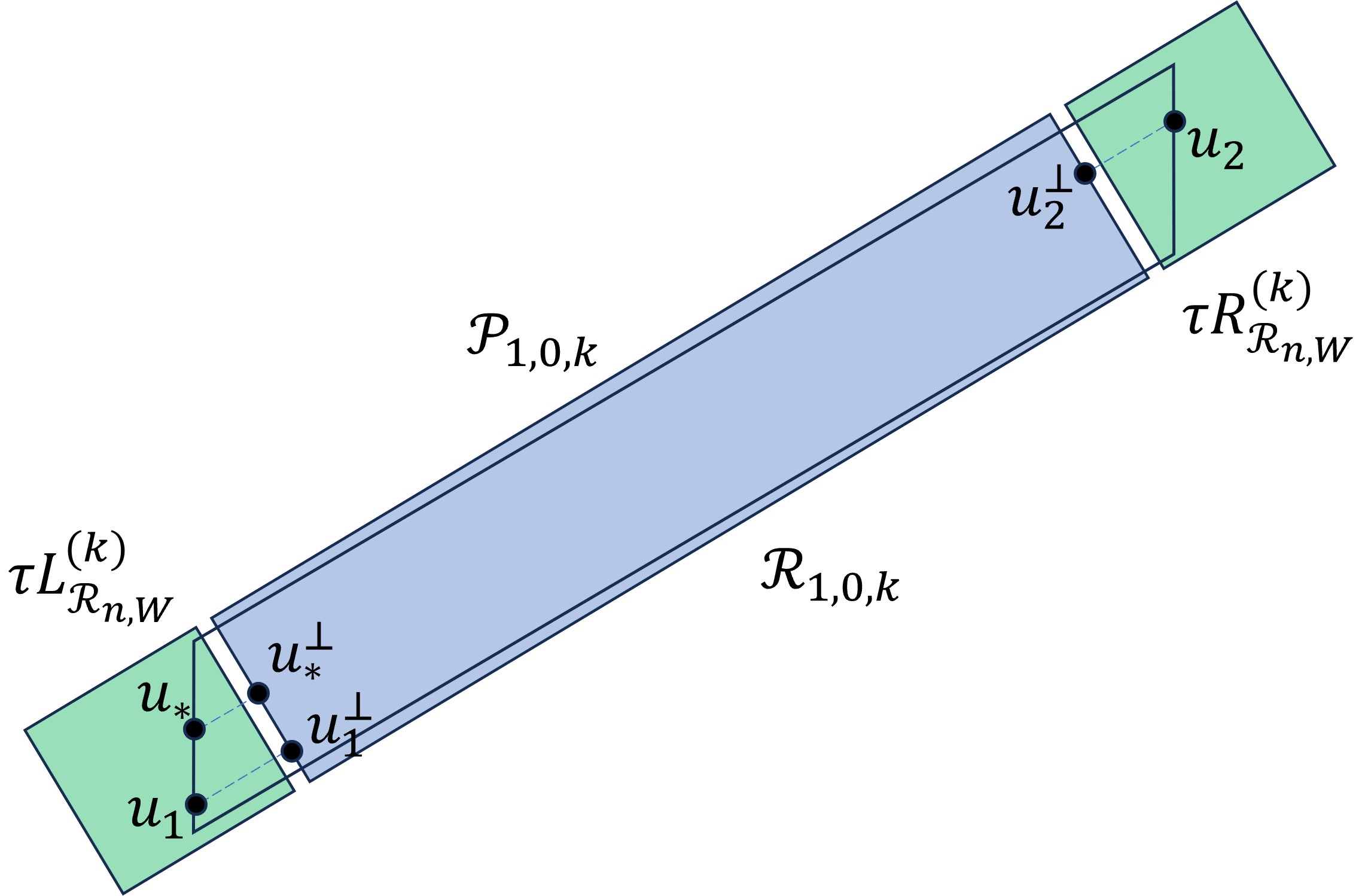}
\caption{Comparing passage times across parallelograms by projecting the endpoints on to the sides of an appropriate rectangle (the embedded rectangle $\cR_{1,0,k}$) and using Lemma \ref{l:rectangleApprox}: proof of Corollary \ref{c:paraRect}.}
\label{f:para.ext}
\end{figure}
\end{center}

\begin{corollary}\label{c:paraRect}
There exists $C>0$ such that for $Q,W$ satisfying~\eqref{eq:variableQW} and $0\leq k\leq n/W$ we have that for large enough $n$,
\begin{equation}\label{eq:paraRectA}
\P\Big[\max_{u_1\in L_{\cP_{1,0,k}}}\max_{u_2\in R_{\cP_{1,0,k}}} \left| X_{u_1 u_2}-X_{u^\perp_1 u^\perp_2} - ((u_2-u_1)\cdot \vec{v}-n)\mu \right| >  (k^2 Q)^{9/10} \Big] \leq \exp(1-C Q^{\kappa/10}).
\end{equation}
Moreover, for $z\in \R$
\begin{align}\label{eq:paraRectB}
\P[Z^-_{1,0,k}<n\mu - z Q] &\leq \P\Big[Y^-_{\cR_{n,W}} < n\mu -\frac{(k^2-8k)Q}{16} -zQ\Big] +\exp(1-CQ^{\kappa/10}).
\end{align}
\end{corollary}

\begin{proof}
Equation~\eqref{eq:paraRectA} follows from the fact that $L_{\cP_{1,0,k}}\subset\tau L^{(k)}_{\cR_{n,W}}$ and $R_{\cP_{1,0,k}}\subset\tau R^{(k)}_{\cR_{n,W}}$ and Lemma~\ref{l:rectangleApprox}.
By \eqref{eq:perpDistance},
\begin{equation}\label{eq:perpDistance2}
(u_2-u_1)\cdot \vec{v}-n\geq \frac{(3k^2-8k)Q}{16},
\end{equation}
and so for large enough $n$,
\begin{align*}
\P[Z^-_{1,0,k}<n\mu - z Q] &\leq \P\Big[Y^-_{\cR_{1,0,k}} < n\mu - \frac{(k^2-8k)Q}{16}  - zQ \Big]\\
&\qquad +\P\Big[\max_{u_1,u_2} (X_{u_1 u_2}-X_{u^\perp_1 u^\perp_2}) < \frac{(3k^2-8k)Q}{16}-\frac1{8} k^2 Q \Big]\\
&\leq \P\Big[Y^-_{\cR_{1,0,k}} < n\mu - \frac{(k^2-8k)Q}{16} - zQ \Big]\\
&\qquad +\P\Big[\max_{u_1,u_2}  X_{u_1 u_2}-X_{u^\perp_1 u^\perp_2} - ((u_2-u_1)\cdot \vec{v}-n)\mu < -\frac1{8} k^2 Q\Big]\\
&\leq \P\Big[Y^-_{\cR_{n,W}} < n\mu  - \frac{(k^2-8k)Q}{16} - zQ \Big] + \exp(1-CQ^{\kappa/10}),
\end{align*}
where the maximum is over $u_1\in L_{\cP_{1,0,k}},u_2\in R_{\cP_{1,0,k}}$, and the second inequality uses ~\eqref{eq:perpDistance2} and the final inequality uses \eqref{eq:paraRectA} and the fact that $Y^-_{\cR_{n,W}}\stackrel{d}{=}Y^-_{\cR_{1,0,k}}$.
\end{proof}

We now give another variant of this corollary which will be useful when we need to resample part of the field.

\begin{lemma}\label{l:paraChange}
There exists $C>0$ such that for $Q,W$ satisfying~\eqref{eq:variableQW} and $0\leq k\leq n/W$ we have that for large enough $n$ and $z>0$,
\begin{align}
\P\Big[\max_{u_1\in L_{\cP_{1,0,k}}}\max_{u_2\in R_{\cP_{1,0,k}}} \big| X_{u_1 u_2}- |u_2-u_1|\mu \big| \geq  (k^2 Q)^{9/10} +  (\mu + z) Q \Big] \nonumber\\
\leq \exp(1-C Q^{\kappa/10}) + \P[ |Y^+_{\cR_{n,W}} - n\mu|+ |Y^-_{\cR_{n,W}} -n\mu| \geq  z Q].
\end{align}
\end{lemma}

\begin{proof}
By Pythagoras' theorem (~\eqref{eq:Pythag}) and our choice of $Q,W$ we have that
\[
0\leq |u_1 - u_2| - (u_2-u_1)\cdot \vec{v}\leq Q,
\]
and so
\begin{align*}
&\P\Big[\max_{u_1\in L_{\cP_{1,0,k}}}\max_{u_2\in R_{\cP_{1,0,k}}} \big| X_{u_1 u_2}- |u_2-u_1|\mu \big| \geq  (k^2 Q)^{9/10} + (\mu + z) Q \Big] \nonumber\\
&\quad\leq\P\Big[\max_{u_1\in L_{\cP_{1,0,k}}}\max_{u_2\in R_{\cP_{1,0,k}}} \left| X_{u_1 u_2}-X_{u^\perp_1 u^\perp_2} - ((u_2-u_1)\cdot \vec{v}-n)\mu \right| >  (k^2 Q)^{9/10} \Big]\nonumber\\
&\qquad+\P\Big[\max_{v_1\in L_{\cR_{n,W}}}\max_{v_2\in R_{\cR_{n,W}}} \left|X_{v_1 v_2} - n\mu \right| >  zQ \Big]\nonumber\\
&\quad\leq \exp(1-C Q^{\kappa/10}) + \P[ |Y^+_{\cR_{n,W}} - n\mu|+ |Y^-_{\cR_{n,W}} -n\mu| \geq  z Q],
\end{align*}
completing the proof.
\end{proof}

Our next objective is to show that $Q^{-1}(Z^{-}_{1,0,k}-n\mu)-k^{2}$ is unlikely to be too small. For this we shall relate the minimal side to side distance $Y^-_{\cR_{1,0,k}}$ to a slightly longer point to point distance.  Choose $\delta$ small enough that $2\delta^\alpha<\frac12$. We set $Q=Q_{(1+\delta)n}$ and $W=W_{(1+\delta)n}$.

Define 
$$\Xi=A_{n+\delta n}-2\max_{1\leq m \leq \delta n} A_m$$
and set 

\begin{equation}
    \label{l:Deltadefn}
    \Delta=Q^{-1}[\Xi-(\frac{\mu}{\delta}+3)Q] = Q^{-1}[A_{n+\delta n}-2\max_{1\leq m \leq \frac{\delta}{2}n} A_m -(\frac{\mu}{\delta}+3)Q].
\end{equation}

We shall show later in Section \ref{s:percgen} that $\Delta$ is bounded above uniformly in $n$. Now we show the following result which will be used in the next section to show that $A_{n}\le CQ_{n}$. 

\begin{lemma}
    \label{l:parastatement}
    For $n$ sufficiently large and $z\ge 0$ we have 
    \begin{equation}\label{eq:paraBound2}
\P[Z^-_{1,0,k} -n\mu - \frac{k^2 Q}{32} -(\Delta-2) Q < - 3z Q] \leq 7\exp(1-z^\theta).
\end{equation}
\end{lemma}

\begin{center}
\begin{figure}
\includegraphics[width=4in]{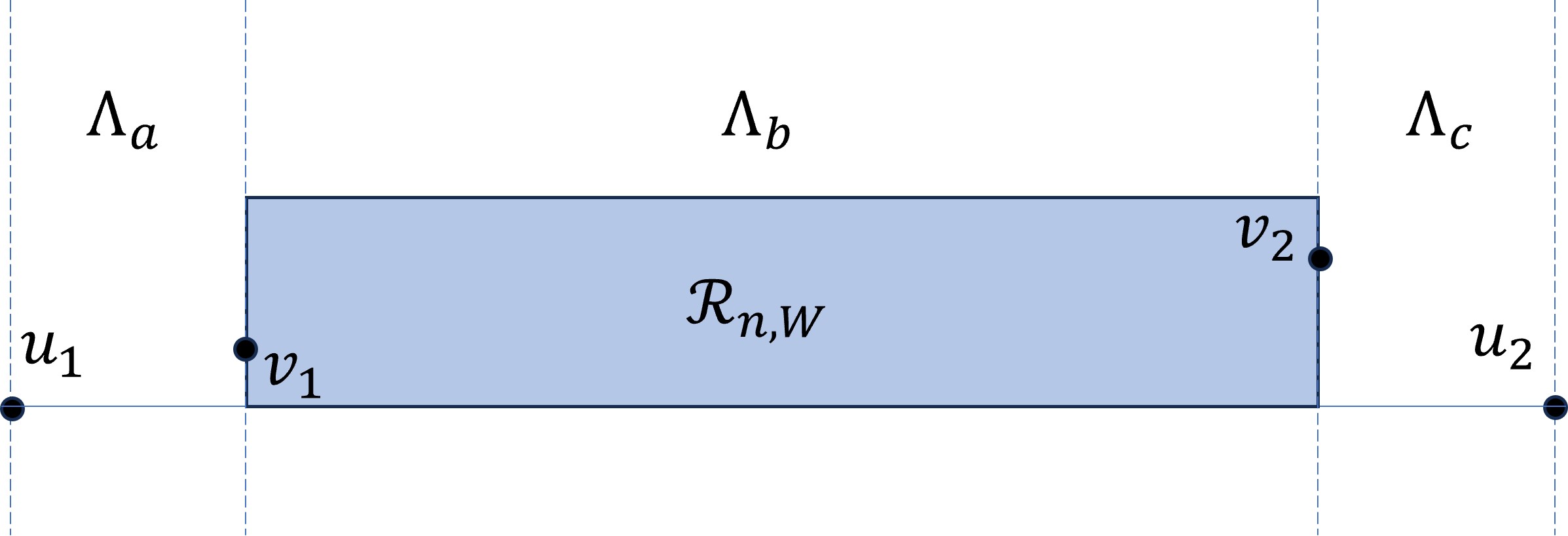}
\caption{Construction in the proof of Lemma \ref{l:parastatement}: for small $\delta$ (where $\delta n/2$ is the width of the columns $\Lambda_a$) we show the passage time $Y^{-}_{\cR_{n,W}}=X_{v_1v_2}$ can be approximated in an appropriate sense by $X_{u_1u_2}$.}
\label{f:XYcomp}
\end{figure}
\end{center}

\begin{proof}
Let $u_1=(-\frac{\delta}{2}n,0)$ and $u_2=((1+\frac{\delta}{2})n,0)$.
Let us define columns in $\R^2$ by
\begin{align*}
\Lambda_a &:= \{(x,y)\in\R^2:-\frac{\delta}{2}n< x \le 0\},\quad \Lambda_b:= \{(x,y)\in\R^2:0< x < n\},\\
  \Lambda_c &:= \{(x,y)\in\R^2:n\le  x < (1+\frac{\delta}{2})n\}.
\end{align*}
Let $v_1\in L_{\cR_{n,W}},v_2\in R_{\cR_{n,W}}$ be points minimizing the passage time across $\cR_{n,W}$ such that
\begin{equation}\label{eq:contMeanA}
X_{v_1,v_2}^{\Lambda_b} = Y_{\cR_{n,W}}^{-,\Lambda_b}.
\end{equation}

By Assumption~\ref{as:resamp1} (the first resampling hypothesis), we have that
\begin{align}\label{eq:contMeanB}
\P[|X_{u_1 v_1}^{\Lambda_a}-X_{u_1 v_1}|\geq Q] &\leq \exp(-Q^{\kappa/2})\nonumber\\
\P[|X_{v_1 v_2}^{\Lambda_b}-X_{v_1 v_2}|\geq Q] &\leq \exp(-Q^{\kappa/2})\nonumber\\
\P[|X_{u_2 v_2}^{\Lambda_c}-X_{u_2 v_2}|\geq Q] &\leq \exp(-Q^{\kappa/2})
\end{align}
Let $R_i=|u_i-v_i|$ and note that by~\eqref{eq:Pythag}
\begin{equation}\label{eq:RiLength}
\frac{\delta}{2} n\leq R_i\leq \frac{\delta}{2} n + \frac{W^2}{2\delta n}\leq \delta n.
\end{equation}
Observe that $v_i$ and $R_i$ are independent of $\omega^{\Lambda_2}$.
Now setting $\Xi=A_{n+\delta n}-2\max_{1\leq m \leq \delta n} A_m$ we have that
\begin{align*}
&\P[Y^-_{\cR_{n,W}}\leq n\mu +\Xi-\frac{\mu}{\delta}Q-3Q -zQ]\\
&=\P[Y^{-,\Lambda_b}_{\cR_{n,W}}\leq n\mu + \Xi-\frac{\mu}{\delta}Q-3Q -zQ]\\
&=\P[X_{v_1 v_2}^{\Lambda_b}\leq n\mu + \Xi-\frac{\mu}{\delta}Q-3Q -zQ]\\
&\leq\P[X_{v_1 v_2}\leq n\mu + \Xi-\frac{\mu}{\delta}Q-2Q -zQ]+\P[|X_{v_1 v_2}^{\Lambda_b}-X_{v_1 v_2}|\geq Q]\\
&\leq\P[X_{v_1 v_2}\leq n\mu + \Xi-\frac{\mu}{\delta}Q-2Q -zQ]+\exp(-Q^{\kappa/2}),
\end{align*}
where the first equality is by the exchangeability of $Y^-_{\cR_{n,W}}$ and $Y^{-,\Lambda_2}_{\cR_{n,W}}$, the second is by~\eqref{eq:contMeanA} and the second inequality is by~\eqref{eq:contMeanB}.  Now by the triangle inequality and equation~\eqref{eq:RiLength},
\begin{align*}
&\P[X_{v_1 v_2}\leq n\mu + \Xi-\frac{\mu}{\delta}Q-2Q -zQ]\\
&\leq\P[X_{u_1 u_2} - X_{u_1v_1} - X_{v_2 u_2} \leq n\mu + \Xi-\frac{\mu}{\delta}Q-2Q -zQ]\\
&\leq \P[X_{u_1 u_2} -(n+\delta n)\mu - A_{n+\delta n} \leq  -\frac{z}{3}Q]\\
&\qquad + \P[- X_{u_1v_1} + R_1\mu + A_{R_1} \leq  -\frac{z}{3}Q-Q]  + \P[- X_{u_2 v_2} + R_2\mu + A_{R_2} \leq  -\frac{z}{3}Q-Q]
\end{align*}
and by ~\eqref{eq:Q},and the choice $Q=Q_{(1+\delta)n}$, we have that
\begin{align*}
\P[X_{u_1 u_2} -(n+\delta n)\mu - A_{n+\delta n} \leq  -\frac{z}{3}Q]=\P[X_{u_1 u_2} -\E X_{u_1 u_2} \leq  -\frac{z}{3}Q]\leq \exp(1-(z/3)^\theta).
\end{align*}
Next we have that
\begin{align*}
\P[- X_{u_1v_1} + R_1\mu + A_{R_1} \leq  -\frac{z}{3}Q-Q]&\leq \P[- X^{\Lambda_a}_{u_1v_1} + R_1\mu + A_{R_1} \leq  -\frac{z}{3}Q] + \P[|X_{u_1 v_1}^{\Lambda_a}-X_{u_1 v_1}|\geq Q]\\
&=\P[- X^{\Lambda_a}_{u_1v_1} + R_1\mu + A_{R_1} \leq  -\frac{z}{3}Q] + \P[|X_{u_1 v_1}^{\Lambda_a}-X_{u_1 v_1}|\geq Q]\\
&\leq \P[- (X^{\Lambda_a}_{u_1v_1} - R_1\mu - A_{R_1}) \leq  -\frac{z}{3}Q]+\exp(-Q^{\kappa/2}),
\end{align*}
where we used~\eqref{eq:contMeanB}.  Furthermore, note that $v_1$ is independent of $\omega^{\Lambda_a}$ so we may treat $v_1$ and $R_1$ as a deterministic when considering $X^{\Lambda_a}_{u_1v_1}$.  Now since $Q_{R_1} \leq Q_{n+\delta n}=Q$, by applying ~\eqref{eq:Q} we have that
\[
\P[- (X^{\Lambda_a}_{u_1v_1} - R_1\mu - A_{R_1}) \leq  -\frac{z}{3}Q]\leq \exp(1-(z/3)^\theta).
\]
Altogether this gives that
\[
\P[- X_{u_1v_1} + R_1\mu + A_{R_1} \leq  -\frac{z}{3}Q-Q]\leq  \exp(1-(z/3)^\theta) + \exp(-Q^{\kappa/2}),
\]
and similarly
\[
\P[- X_{u_2v_2} + R_2\mu + A_{R_2} \leq  -\frac{z}{3}Q-Q]\leq  \exp(1-(z/3)^\theta) + \exp(-Q^{\kappa/2}).
\]
Combining the above estimates gives
\begin{align}\label{eq:contMeanC}
\P[Y^-_{\cR_{n,W}}\leq n\mu + A_{n+\delta n}-A_{R_1}-A_{R_2}-\frac{\mu}{\delta}Q-3Q -zQ]\leq 3\exp(1-(z/3)^\theta) + 3\exp(-Q^{\kappa/2}).
\end{align}
Since $n\mu + A_{n+\delta n}-A_{R_1}-A_{R_2}-\frac{\mu}{\delta}Q-3Q\leq 2n\mu$ the left hand side of~\eqref{eq:contMeanC} is 0 when $z\geq 2n\mu Q^{-1}$ so the inequality holds trivially.  When $z\leq 2n\mu Q^{-1}\leq 3n$ we have that
\[
(z/3)^\theta\leq n^{\theta} \leq n^{\alpha\kappa/2}\leq Q^{\kappa/2},
\]
and so
\[
\exp(-Q^{\kappa/2})\leq \exp(1-(z/3)^\theta).
\]
Now setting
\[
\Delta=Q^{-1}[\Xi-(\frac{\mu}{\delta}+3)Q] = Q^{-1}[A_{n+\delta n}-2\max_{1\leq m \leq \frac{\delta}{2} n} A_m -(\frac{\mu}{\delta}+3)Q]
\]
we may rewrite~\eqref{eq:contMeanC} as
\begin{equation}\label{eq:YwideMinus}
\P[Y^-_{\cR_{n,W}}-n\mu-\Delta Q \leq -3z Q]\leq 6\exp(1-z^\theta).
\end{equation}
Now combining this with Corollary~\ref{c:paraRect} we have that
\begin{align*}
\P[Z^-_{1,0,k} -n\mu - \frac{(k^2-8k)Q}{16} -\Delta Q < - 3z Q] &\leq \P\Big[Y^-_{\cR_{n,W}} -n\mu -\Delta Q<  -3zQ\Big] +\exp(1-CQ^{\kappa/10})\\
&\leq 6\exp(1-z^\theta) +\exp(1-CQ^{\kappa/10})
\end{align*}
and so since $k^2-8k \geq \frac12 k^2 -32$,
\begin{equation}\label{eq:paraBound}
\P[Z^-_{1,0,k} -n\mu - \frac{k^2 Q}{32} -(\Delta-2) Q < - z Q] \leq 6\exp(1-z^\theta) +\exp(1-CQ^{\kappa/10}).
\end{equation}
Now $n\mu + \frac{k^2 Q}{32} +(\Delta-2) Q \leq (2\mu+1) n$, so the left hand side of~\eqref{eq:paraBound} is 0 if $z\geq n\geq (2\mu+1) n Q^{-1}$.  If $z\leq n$ then for sufficiently large $n$.
\[
z^\theta\leq n^{\theta} \leq C n^{\alpha\kappa/10}\leq C Q^{\kappa/10}.
\]
Hence we may write
\begin{equation}
\P[Z^-_{1,0,k} -n\mu - \frac{k^2 Q}{32} -(\Delta-2) Q < - 3z Q] \leq 7\exp(1-z^\theta),
\end{equation}
completing the proof.
\end{proof}

\section{Estimates at scale $Q_n$}
\label{s:percgen}
In this section our objective is to prove Proposition \ref{p:conc} and the upper bound in Theorem \ref{t:nr} with $\SD(X_{n})$ replaced by $Q_n$ using a multi-scale argument. 

\subsection{A general percolation estimate}
A key piece of our multi-scale argument is to break a path crossing a rectangle of length $Mn$ into a sequence of paths of length approximately $n$.  We divide the plane into strips $\Lambda_i=\{(x,y):(i-1)n< x < in\}$ and view a path $\gamma$ from $\origin$ to $(Mn,0)$ as a series of crossings of parallelograms $\cP_{i,k,k'}$ as in Figure~\ref{f:grid}.

We will write $\gamma(t)=(\gamma_1(t),\gamma_2(t))$ to indicate its co-ordinates. For $1\leq i \leq M-1$ let
\begin{equation}\label{eq:tiDefn}
t_i:=\inf\{t\in[0,1]: \gamma_1(t)=in\},
\end{equation}
the first time $\gamma$ hits the line $x=in$ which we denote $\ell_i$.  We set $t_0=0$ and $t_M=1$.  As in the previous section we will rescale heights by length $W$ and we write
\[
k^\gamma_i = \lfloor \gamma_2(t_i)/W\rfloor
\]
to be the normalized height of the first hitting time of $\ell_i$.  With this notation, the path $\gamma$ makes a crossing of $\cP_{i,k^\gamma_{i-1},k^\gamma_i}$ for each $1\leq i \leq M$. Since we have $k^\gamma_0=k^\gamma_M=0$ we have that $\uk^\gamma\in\mathfrak{K}_M$ where
\[
\mathfrak{K}_M=\{(k_0,\ldots,k_M)\in\Z^{M+1}: k_0=0\}.
\]
It will also be useful for use to measure the fluctuations such a path and so we write
\[
\tau_1(\uk):=\sum_{i=1}^M |k_i-k_{i-1}|.
\]
Then
\begin{align}\label{eq:NBDecomposition}
X_{Mn} &\geq \sum_{i=1}^M X_{\gamma(t_{i-1})\gamma(t_{i})} - M\Gamma_{Mn}\nonumber\\
&\geq \sum_{i=1}^M X^{\Lambda_i}_{\gamma(t_{i-1})\gamma(t_{i})} - |X^{\Lambda_i}_{\gamma(t_{i-1})\gamma(t_{i})} - X_{\gamma(t_{i-1})\gamma(t_{i})}| - M\Gamma_{Mn}\nonumber\\
&\geq \sum_{i=1}^M Z^{-,\Lambda_i}_{i,k^\gamma_{i-1},k^\gamma_i} - |X^{\Lambda_i}_{\gamma(t_{i-1})\gamma(t_{i})} - X_{\gamma(t_{i-1})\gamma(t_{i})}| - M\Gamma_{Mn}\nonumber\\
&\geq \min_{\uk\in\mathfrak{K}_M}\sum_{i=1}^M Z^{-,\Lambda_i}_{i,k_{i-1},k_i} - \sum_{i=1}^M |X^{\Lambda_i}_{\gamma(t_{i-1})\gamma(t_{i})} - X_{\gamma(t_{i-1})\gamma(t_{i})}| - M\Gamma_{Mn}.
\end{align}
Note that in the sum $\sum_{i=1}^M Z^{-,\Lambda_i}_{i,k_{i-1},k_i}$ each of the summands are independent which motivates the following general proposition.

\begin{center}
\begin{figure}
\includegraphics[width=5in]{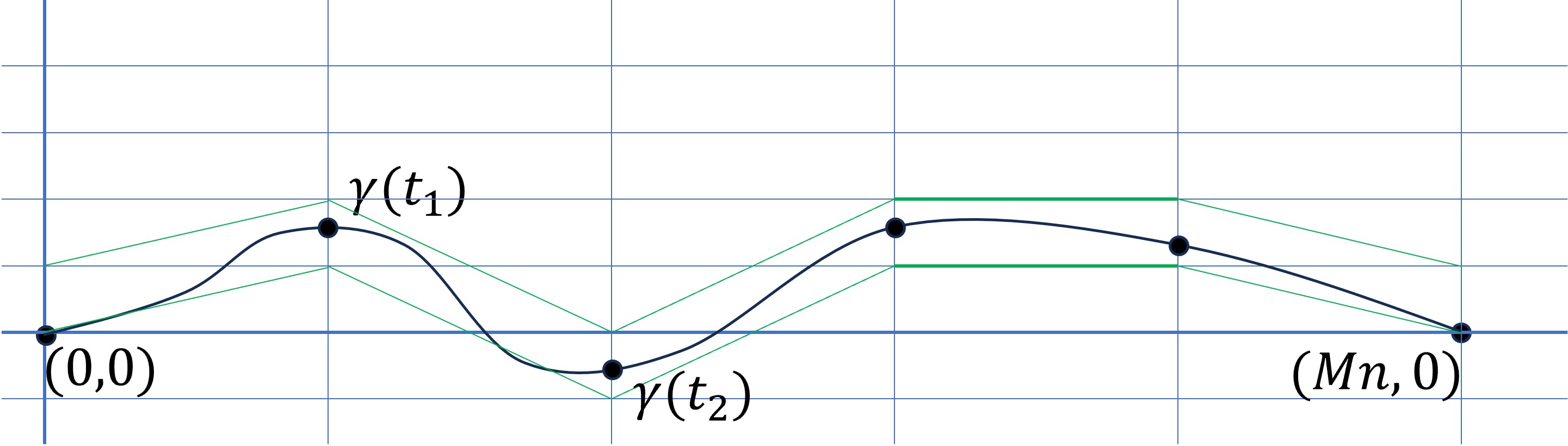}
\caption{The geodesic from $(0,0)$ to $(Mn,0)$ is decomposed into a series of crossings of parallelograms. Using this we can control $X_{Mn}$ using \eqref{eq:NBDecomposition} and Proposition \ref{p:perc1} with the latter result providing control on the minimum over (all possible choices of the crossings) of sums of the side-to-side passage times across the parallelograms.}
\label{f:grid}
\end{figure}
\end{center}

\begin{proposition}
\label{p:perc1}
For $i,k,k'\in \Z$, let $\cZ_{i,k,k'}$ be random variables such that for some $0<\beta<1$,
\[
\P[\cZ_{i,k,k'} \geq z]\leq C_1 \exp(-C_2 z^{\beta})
\]
and
\[
\P[\cZ_{i,k,k'} \leq z_{\max}]=1.
\]
Assume further that the collections of variables $\mathfrak{Z}_i=\{\cZ_{i,k',k'}\}_{k,k'\in \Z}$ are independent for different~$i$. Then there exist $C_3,C_4,C_5,C_6$ depending on $C_1,C_2$ and $\beta$ but not depending on $z_{\max}$ such that for all $R\geq 1$ and $z>0$ we have
\[
\P\left[\max_{\substack{\uk\in \mathfrak{K}_{M}\\ \tau_1(\uk)\leq R M}} \sum_{i=1}^M \cZ_{i,k_{i-1}, k_{i}} \geq (C_3+ \log^{C_4} (R))M + z \right] \leq C_5\exp\left(-C_6 z^{\beta}-C_6 (z/z_{\max})z_{\max}^{\beta}\right ).
\]
\end{proposition}

The proof of this abstract proposition, although quite important, is not directly related to the arguments that follow. This proof therefore is given later in {Section~\ref{s:appa}}. We now give the following corollary.
\begin{corollary}\label{c:percolation}
For any $\cZ_{i,k,k'}$ satisfying the hypothesis of Proposition~\ref{p:perc1} there exists $C_7, C_8, C_9$ such that for and any $\lambda>0$,
\[
\P\left[\max_{\uk\in \mathfrak{K}_{M}} \sum_{i=1}^M \cZ_{i,k_{i-1}, k_{i}} -  \lambda\tau_1(\uk)\geq C_7 M\log^{C_7}(3+\frac1{\lambda})  + z \right] \leq C_8\exp\left(-C_9 z^{\beta}-C_9 (z/z_{\max})z_{\max}^{\beta}\right ).
\]
\end{corollary}

\begin{proof}
If $\lambda\geq 1$ then for large enough $C_7$ we have that,
\[
C_7 + \frac14 \lambda x \geq C_7 + \frac14  x > C_3 +\log^{C_4} (x)
\]
for all $x\geq 1$. For $0<\lambda<1$, if $x\geq \lambda^{-2}>1$ then provided $C_7$ is large enough,
\[
C_7 + \frac14 \lambda x \geq C_7 + \frac14  x^{1/2} > C_3 +\log^{C_4} (x)
\]
while if $1\leq x\leq \lambda^{-2}$ then again provided $C_7$ is large enough,
\[
C_7 \log^{C_7}(3+\frac1{\lambda}) > C_3 + \log^{C_4} (\lambda^{-2}) \geq  C_3 +\log^{C_4} (x).
\]
Thus, by using these equations with $x=2^j$, we may pick $C_7>0$ large enough depending only on $C_3$ and $C_4$ such that for all $\lambda>0$ and $j\geq 0$,
\[
C_7 \log^{C_7}(3+\frac1{\lambda}) + \lambda 2^{j-2} > C_3+ \log^{C_4} (2^j).
\]
For $j\geq 1$ let $\cH_j$ be the event that there exists $\uk\in \mathfrak{K}_{M}$ with $2^{j-1} M \leq \tau_1(\uk)\leq 2^j M$ such that
\[
\sum_{i=1}^M \cZ_{i,k_{i-1}, k_{i}} > (C_3+ \log^{C_4} (2^j))M + \lambda 2^{j-2}M + z
\]
and let $\cH_0$ be the event that there exists $\uk\in \mathfrak{K}_{M}$ with $0\leq \tau_1(\uk)\leq M$ such that
\[
\sum_{i=1}^M \cZ_{i,k_{i-1}, k_{i}} > C_3 M + z.
\]
By our choice of $C_7$, if
\[
\sum_{i=1}^M \cZ_{i,k_{i-1}, k_{i}} - \lambda \tau_1(\uk)\geq C_7 M \log^{C_7}(3+\frac1{\lambda}) + z
\]
holds for some $\uk$ then $\cH_j$ holds for $j=\lceil\log_2 (\tau_1(\uk)/M)\rceil\vee 0$.  Hence by Proposition~\ref{p:perc1},
\begin{align*}
\P[\cup_{j=0}^\infty \cH_j] &\leq \sum_{j=0}^\infty  C_5\exp\left(-C_6 (\lambda 2^{j-2}M + z)^{\beta}-C_6 ((\lambda 2^{j-2}M + z)/z_{\max})z_{\max}^{\beta}\right )\\
&\leq C_8\exp\left(-C_9 z^{\beta}-C_9 (z/z_{\max})z_{\max}^{\beta}\right ),
\end{align*}
which completes the proof.
\end{proof}

\subsection{Bounding $A_n$}
Recall the set-up in Lemma \ref{l:parastatement}. We shall continue working with $Q=Q_{(1+\delta)n}$ and $W=W_{(1+\delta)n}$. We now show that \eqref{eq:paraBound2} is contradicted if $A_n$ is too large.
Let $\cJ$ be the event that none of the $k_i^\gamma$ from the path $\gamma=\gamma_{\origin,(Mn,0)}$ are too big
\[
\cJ=\{\max_{1\leq i \leq M-1} |k_i^\gamma| \leq \frac14 n/W\}.
\]
If $M\leq n^{1/10}$ then by Assumption~\ref{as:cars}
\begin{equation}\label{eq:cJBound}
\P[\cJ^c]\leq M\exp(1-D(\frac{n/4}{(Mn)^{8/10}}n^{1/10})^\kappa)\leq \exp(-Cn^{\kappa/5}).
\end{equation}

Recall $\Delta$ from \eqref{l:Deltadefn}. Now let
\[
\cZ_{i,k,k'}:=\left(-(Z^{-,\Lambda_i}_{i,k,k'} -n\mu)/Q + \frac{(k-k')^2}{32} +(\Delta-2) \right)I(|k|\vee |k'|\leq n/W).
\]
By ~\eqref{eq:paraBound2} and the independence of the $\omega^\Lambda_i$, the $\cZ_{i,k,k'}$ satisfy the hypothesis of Corollary~\ref{c:percolation} (withe $\beta=\theta$).
Suppose, for the sake of contradiction,  that
\begin{equation}\label{eq:DeltaContBound}
\Delta \geq 6+C_7 \log^{C_7}(35)
\end{equation}
where $C_7$ is the constant from Corollary~\ref{c:percolation} (where we take $\lambda=\frac1{32}$ and $\beta=\theta$).
Then by Corollary~\ref{c:percolation} we have that
\begin{align}\label{eq:percApplication1}
\P[\cJ,\sum_{i=1}^M Z^{-,\Lambda_i}_{i,k^\gamma_{i-1},k^\gamma_i} - nM\mu \leq  (4M -z) Q]&\leq \P[\max_{\uk\in\mathfrak{K}_M}\sum_{i=1}^M \cZ_{i,k_{i-1},k_i} - \frac{(k_{i-1}-k_i)^2}{32} \geq (\Delta-6)M  +z ]\nonumber\\
& \leq C_8\exp\left(-C_9 (M+z)^{\theta}\right ),
\end{align}
and hence by equation~\eqref{eq:NBDecomposition} and taking $M=n^{1/10}$, we have that
\begin{align}\label{eq:lowerTailContBound}
\P[X_{Mn}\leq Mn\mu + 2(Mn)^{1/20}]&\leq \P[\cJ^c] + \P[\cJ,\sum_{i=1}^M Z^{-,\Lambda_i}_{i,k^\gamma_{i-1},k^\gamma_i} - n\mu \leq  4M Q]\nonumber\\
&\quad + \P[ \sum_{i=1}^M |X^{\Lambda_i}_{\gamma(t_{i-1})\gamma(t_{i})} - X_{\gamma(t_{i-1})\gamma(t_{i})}| > MQ] +\P[M\Gamma_{Mn}> MQ]\nonumber\\
&\leq \exp(-Cn^{\kappa/5}) + C_8\exp\left(-C_9 M^{\theta}\right ) + 2M\exp\bigg(1-C\Big(\frac{Mn}{\log^C (Mn)}\Big)^{\kappa}\bigg)\nonumber\\
&\leq \exp(-2(Mn)^{2\epsilon})
\end{align}
where we applied equations~\eqref{eq:cJBound} and~\eqref{eq:percApplication1}, Lemma~\ref{l:Gamma} and Assumption~\ref{as:cars}.
With the following lemma we will derive a contradiction.
\begin{lemma}
For all large enough $n$, the passage times satisfy
\begin{equation}\label{eq:lowerTailContBound2}
\P[X_{n}\leq n\mu + 2n^{1/20}] > \exp(-2n^\epsilon).
\end{equation}
\end{lemma}
\begin{proof}
Suppose that~\eqref{eq:lowerTailContBound2} is false for some $n\geq n_0$.  Then by Lemma~\ref{l:Gamma} for all $n-2\leq m \leq n+2$,
\begin{equation}\label{eq:approxLength}
\P[X_{m}\leq n\mu + n^{1/20}] \leq \P[X_{n}\leq n\mu + 2n^{1/20}] +\P[2 \Gamma_n \geq n^{1/20}] \leq 2\exp(-2n^\epsilon).
\end{equation}
Let $\gamma=\gamma_{\origin,(n^2,0)}$ be the optimal path from the origin to $(n^2,0)$ and define $t_i$ such that $t_0=0$ and for $1\leq i\leq n$,
\[
t_{i} = \inf\{t>t_{i-1}:|\gamma(t)-\gamma(t_{i-1})|\geq n\}.
\]
Define $t_{n+1}=n+1$ and set $a_i=\gamma(t_{i})$.  Then we have that
\begin{align}\label{eq:pathDecomp}
X_{n^2} &\geq \sum_{i=1}^{n+1} X_{a_{i-1}a_i}-n\Gamma_{n^2}\geq \sum_{i=1}^{n} X_{a^\Z_{i-1}a^\Z_i}-3n\Gamma_{n^2}
\end{align}
Let $\cC$ be the event
\[
\cC=\{\forall u,v\in \Z^2\cap B_{2n^2}(\origin), n-2\leq |u-v| \leq n+2: X_{uv} \geq n\mu + n^{1/20}\}.
\]
Then by equation~\eqref{eq:approxLength},
\[
\P[\cC^c]\leq C n^4  \exp(-2n^\epsilon).
\]
By construction $n-2\leq|a^\Z_{i-1}-a^\Z_i| \leq n+2$ and $a^\Z_i\in B_{2n^2}(\origin)$ for $1\leq i \leq n$ and so
\[
\P[X_{n^2}< n(n \mu + \frac12 n^{1/20})] \leq \P[\cC^c] + \P[3n\Gamma_{n^2} \geq \frac12 n^{21/20}] \leq \exp(-n^\epsilon).
\]
Hence
\[
\E[X_{n^2}] \geq n(n \mu + \frac12 n^{1/20}) (1-\exp(-n^\epsilon)) \geq n^2 \mu + \frac14 n^{21/20} > n^2\mu + (n^2)^{\frac{51}{100}}
\]
which contradicts Assumption~\ref{as:speed} and so establishes~\eqref{eq:lowerTailContBound2}.
\end{proof}
Since~\eqref{eq:lowerTailContBound2} contradicts equation~\eqref{eq:lowerTailContBound} for $n$ large we have a contradiction to equation~\eqref{eq:DeltaContBound} and so $\Delta \leq  6+C_7 \log^{C_7}(35)$ and hence with $C'= 6+C_7 \log^{C_7}(35)$ for all large enough $n$,
\begin{equation}\label{eq:AnDeltaBound}
A_{n+\delta n}\leq 2\max_{1\leq m \leq \delta n} A_m + (\frac{\mu}{\delta}+3+C')Q_{n+\delta n}.
\end{equation}
\begin{lemma}\label{l:AnBound}
There exists $D_1$ such that for all $n\geq 1$ we have that
\[
A_n \leq D_1 Q_n.
\]
\end{lemma}
\begin{proof}
Let $n_*$ be large such that equation~\eqref{eq:AnDeltaBound} holds for all $n\geq n_*$.  Set
\[
D_1=4\left(\frac{\mu}{\delta} +C'+3\right)\vee \bigg(\max_{1\leq m \leq n_*(1+\delta)}2\frac{A_m}{Q_m}\bigg).
\]
Assume that the statement of the lemma fails for some $n$ in which case $n'\geq (1+\delta) n_*$ by construction of $D_1$.  We can choose $n'$ such that $A_{n'} \geq D_1 Q_{n'}$ and $A_m \leq D_1 Q_m$ for all $m\leq \delta n'$. Recall that we chose $\delta$ such that $2\delta^\alpha<\frac12$.  Then if $n'=(1+\delta)n$ then by equation~\eqref{eq:AnDeltaBound} 
\begin{align*}
D_1 Q_{n'}\le  A_{n'} &\leq 2\max_{1\leq m \leq \delta n'} A_m + (\frac{\mu}{\delta}+3 +C')Q_{n'}\\
&\leq 2 D_1 Q_{\delta n'} + (\frac{\mu}{\delta}+3+C')Q_{n'}\\
&\leq \left(2 D_1 \delta^\alpha +\frac14 D_1\right) Q_{n'} \leq \frac34 D_1 Q_{n'},
\end{align*}
which gives a contradiction.  Hence $A_n \leq D_1 Q_n$ for all $n\geq 1$.
\end{proof}

\subsection{Growth of $Q_n$}
Having controlled the size of $A_n$ in terms of $Q_n$ we will now show that $Q_n$ cannot grow too quickly.
\begin{lemma}\label{l:Qgrowth}
There exists $D_{6}>0$ such that $Q_{\frac32 m} \leq D_{6} Q_m$ for all $m\geq 1$.
\end{lemma}
\begin{proof}
Fix $n\geq 1$ and set $Q=Q_{(1+\delta)n}$.  We will bound the deviation of $X_{2n}$ from $\E X_{2n}$ starting with the positive deviation.
Since $\E X_{2n}\geq 2\mu n$, by the triangle inequality,
\begin{align*}
\P[X_{2n}\geq 2n\mu  + xQ]&\leq \P[X_{\origin,(n,0)}\geq n\mu + \frac12 xQ] + \P[X_{(n,0),(2n,0)}\geq n\mu + \frac12 xQ]\\
&\leq \P[X_{\origin,(n,0)}\geq \E X_n + (\frac12 x-D_1)Q] + \P[X_{(n,0),(2n,0)}\geq \E X_n + (\frac12 x-D_1)Q]\\
&\leq 2\exp(1-Cx^\theta)
\end{align*}
where the second inequality follows by Lemma~\ref{l:AnBound} and the final inequality holds for some small $C>0$.  Hence, since $\E[X_{2n}]\geq 2n\mu$, we have that
\begin{equation}\label{eq:xTwoNUpperBound}
\P[X_{2n}\geq \E[X_{2n}]  + xQ] \leq \exp(1-Cx^\theta).
\end{equation}
We can also use this to bound the mean of $\E[X_{2n}]$,
\begin{align}\label{eq:xTwoNMean}
\E [X_{2n}] &\leq 2n\mu + \E [(X_{2n}-2n\mu)^+]\nonumber\\
&=2n\mu + Q\int_0^\infty \P[X_{2n}> 2n\mu  + xQ] dx\nonumber\\
&\leq 2n\mu + Q \int_0^\infty 2\exp(1-Cx^\theta) dx \leq 2n\mu + \tilde{C}Q.
\end{align}
By equation~\eqref{eq:paraBound2} and Lemma~\ref{l:AnBound} and the definition of $\Delta$ there exists $C'>0$ such that for all $|k|< n/W$,
\begin{equation}\label{eq:paraBound3}
\P[Z^-_{1,0,k} \leq n\mu - ( C' + x) Q] \leq 7\exp(1-c(x+ \frac{k^2}{32})^\theta).
\end{equation}
Repeating the analysis of equation~\eqref{eq:lowerTailContBound}, for large enough $n$, and all $x>0$,
\begin{align}
&\P[X_{2n}< 2n\mu -2(1+C'+x)Q]\nonumber\\
&\quad\leq \P[\cJ^c] + \P[ \sum_{i=1}^2 |X^{\Lambda_i}_{\gamma(t_{i-1})\gamma(t_{i})} - X_{\gamma(t_{i-1})\gamma(t_{i})}| > Q] +\P[2\Gamma_{2n}> Q]\nonumber\\
&\qquad + \P[\cJ,Z^{-,\Lambda_1}_{1,0,k^\gamma_1} + Z^{-,\Lambda_2}_{2,k^\gamma_{1},0}\leq 2n\mu -2(C'+x)Q]\nonumber\\
&\quad\leq  \exp(-Cn^{\kappa/5}) + 4\exp\bigg(1-C\Big(\frac{n}{\log^C n}\Big)^{\kappa}\bigg)\nonumber\\
&\qquad +  \sum_{\ell=-n/(2W)}^{n/(2W)}\P[Z^{-,\Lambda_1}_{1,0,\ell}  \leq n\mu -(\tilde{C}+x)Q] + \P[Z^{-,\Lambda_2}_{2,\ell,0} \leq n\mu -(C'+x)Q]\nonumber\\
&\quad\leq 2\exp(-Cn^{\kappa/5}) + 14 \sum_{\ell=-n/(2W)}^{n/(2W)} \exp(1-c(x+ \frac{\ell^2}{32})^\theta). \nonumber\\
&\quad\leq 2\exp(-Cn^{\kappa/5}) + C\exp(1-cx^\theta).
\end{align}
Since the above equation holds trivially when $x>n\mu$ and since $n^{\kappa/5} > (n\mu)^\theta$ we have that
\begin{equation*}
  \P[X_{2n}< 2n\mu -2(1+C'+x)Q] \leq C\exp(1-cx^\theta)
\end{equation*}
and hence by equation~\eqref{eq:xTwoNMean}
\begin{equation}\label{eq:xTwoNLowerBound}
  \P[X_{2n}-\E X_{2n}\leq -x Q] \leq C\exp(1-Cx^\theta).
\end{equation}
Combining~\eqref{eq:xTwoNUpperBound} and~\eqref{eq:xTwoNLowerBound} we have that for some $\hat{C}>0$,
\[
\P[|X_{2n}-\E X_{2n}|\geq \hat{C} x Q]\leq \exp(1- x^\theta),
\]
and hence
\begin{equation}\label{eq:hQTwoN}
\hQ_{2n} \leq \hat{C} Q_{(1+\delta)n}
\end{equation}
for all $n\geq n'$ for some large constant $n'$.  Now set
\[
D_6=\max\{2 \hat{C} \left(\tfrac32\right)^\alpha,2 Q_{3 n'}/Q_1\},
\]
suppose that
\[
m'=\inf\{m\geq 1:Q_{\frac32 m} \geq D_6 Q_m\}<\infty.
\]
By our choice of $D_6$, we have that $m'>2n'$ and so,
\begin{align*}
Q_{\frac32 m'}&\leq \left(\tfrac32\right)^\alpha \max\{ Q_{m'}, \max_{m'\leq n \leq \frac32m'} \hQ_{n}\}\\
&\leq \left(\tfrac32\right)^\alpha \max\{ Q_{m'}, \max_{m'\leq n \leq \frac32m'} C_3  Q_{\frac12(1+\delta)n},\}\\
&\leq \hat{C} \left(\tfrac32\right)^\alpha Q_{m'} \leq \frac12 D_6 Q_m,
\end{align*}
where the second inequality follows by~\eqref{eq:hQTwoN}.  This contradicts our choice of $m'$ so for all $m\geq 1$,
$Q_{\frac32 m} \leq D_6 Q_m$.
\end{proof}

\subsection{Concentration of $Y_n^-$}

Having established that $Q_n$ does not grow too quickly we now bound the fluctuations of $Y_n^-$ in terms of $Q_n$.

\begin{lemma}\label{l:YMinusBound}
There exist $D_2,D_3$ such that for all $n\geq 1$ we have that
\begin{align*}
|\E Y_n^{-} - n\mu| &\leq D_2 Q_n,\\
\P[|Y_n^- - \E Y_n^-|> x D_3 Q_n]  &\leq \exp(1-x^\theta).
\end{align*}
\end{lemma}
\begin{proof}
Notice that by Assumption \ref{as:concentration} it suffices to prove this for $n$ sufficiently large. We will apply ~\eqref{eq:YwideMinus} which is in terms of  $Q=Q_{(1+\delta)n}$ and $W=W_{(1+\delta)n}$.
Note that since $\cR_{n,W_{(1+\delta)n}}$ is taller than $\cR_{n}=\cR_{n,W_{n}}$ we have that $Y_n^-\geq Y^-_{\cR_{n,W_{(1+\delta)n}}}$.  By Lemma~\ref{l:Qgrowth} we have that $Q_{(1+\delta) m} \leq D_6 Q_m$. Hence by equation~\eqref{eq:YwideMinus}, for some $C>0$ and for $x\ge 0$
\begin{equation}\label{eq:YminusLowerTail}
\P[Y^-_{n}-n\mu \leq -C x Q_n]\leq \exp(1-x^\theta).
\end{equation}
Furthermore,
\begin{align*}
\E [Y^-_{n}] &\geq n\mu - \E [(n\mu-Y^-_{n})^+]\\
&=2n\mu - C Q_n\int_0^\infty \P[Y^-_{n}-n\mu \leq -C x Q_n] dx\\
&\geq n\mu - C Q_n \int_0^\infty \exp(1-x^\theta) dx \geq n\mu - D_2Q_n
\end{align*}
and so since $\E Y^-_{n} \leq \E X_n$,
\begin{equation}\label{eq:YminusMean}
- D_2 Q_n \leq \E [Y^-_{n}]-n\mu \leq A_n \leq D_1 Q_n
\end{equation}
which establishes the first part of the lemma, by increasing the value of $D_2$, if necessary.  The second part follows from combining~\eqref{eq:YminusLowerTail} and~\eqref{eq:YminusMean} which gives that for some $D_3$
\begin{equation}\label{eq:YnLowerBoundA}
\P[Y^-_{n}-\E Y_n^- \leq -D_3 x Q_n]\leq \frac{1}{2}\exp(1-x^\theta).
\end{equation}
For the upper tail, since $X_n\geq Y^-_{n}$,
\[
\P[Y^-_{n}-\E Y_n^- \geq x Q_n]\leq \P[X_n-\E X_n \geq x Q_n +\E Y_n^- - \E X_n]
\]
and since $|Y_n^- - \E X_n|\leq (D_1 + D_2)Q_n$
by increasing the value of $D_3$ if necessary we get from \eqref{eq:Q} 
\[
\P[Y^-_{n}-\E Y_n^- \ge D_3 x Q_n]\leq \frac{1}{2}\exp(1-x^\theta).
\]
which together with~\eqref{eq:YnLowerBoundA} completes the second part of the lemma.
\end{proof}

\subsection{Bounds on $Y^+_n$.}
Next we prove the tail bounds on $Y_n^+$ via a chaining argument.

\begin{lemma}\label{l:YPlusFlexible}
There exists $D$ such  for all $n\geq 1$ we have that for all $x>0$
\begin{align*}
\P[Y_n^+ - n\mu > x D Q_n]  \leq \exp(1-x^{\theta}).
\end{align*}
\end{lemma}

\begin{center}
\begin{figure}
\includegraphics[width=5in]{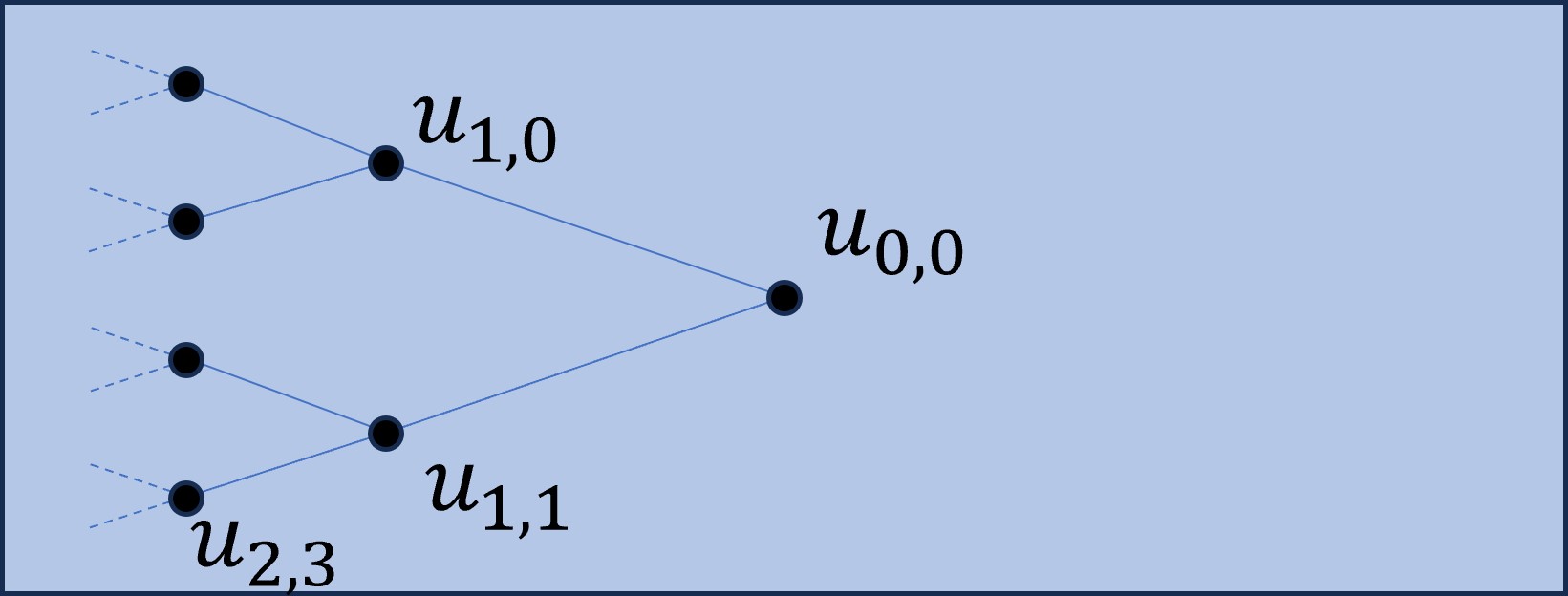}
\caption{The argument is Lemma \ref{l:YPlusFlexible}: the tree depicted above gives defines a collection of paths from the center of the tree to its side and provided that the passage time between the endpoints of each edge in the tree is not too large (and a similar event on the right side of the rectangle), $Y^{+}_{n}$ cannot be too large.}
\label{f:yplus.tree}
\end{figure}
\end{center}

\begin{proof}
Notice again that it suffices to prove this for all $n$ sufficiently large. We define a net of points forming a tree that joins $L_{\cR_n}$ to $u_{0,0}=(\tfrac{1}{2}n,\tfrac12 W_n)$, the centre of $\cR_n$.  Let $L=\lfloor \log_2 n\rfloor$ and for integers $1\leq \ell\leq L$ and $0\leq j \leq 2^\ell-1$ define
\[
u_{\ell,j}=\big(2^{-\ell-1}n, (j+\tfrac12)2^{-\ell}W_n\big).
\]
For $u=u_{\ell,j}$ we denote it's parent $u^+=(\ell-1,\lfloor j/2\rfloor)$ and let $K_\ell=|u-u^+|$ which does not depend on $j$.  This tree of points stretching from the center of the rectangle to its left side is illustrated in Figure~\ref{f:yplus.tree}. By~\eqref{eq:Pythag},
\[
K_\ell=|u-u^+|=\sqrt{(2^{-\ell-1}n)^2+(2^{-\ell-1}W_n)^2}\leq 2^{-\ell-1}(n+\frac{W^2_n}{2n})\leq 2^{-\ell-1}(n+Q_n/2).
\]
and so
\begin{equation}\label{eq:QKellBound}
Q_{K_\ell} \leq \left(\frac{2^{-\ell-1}(n+Q_n/2)}{n}\right)^{\alpha} Q_n \leq 2^{-\tfrac34 \ell\alpha} Q_n
\end{equation}
for all $1\leq \ell \leq L$ {provided $n$ is large enough.} Define the events
\[
\cD_{\ell,j,x}=\left\{X_{u_{\ell,j},u_{\ell,j}^+} - \E X_{u_{\ell,j},u_{\ell,j}^+} >x 2^{-\ell\alpha/2} Q_n \right\}
\]
and $\cD_x=\bigcup_{\ell=1}^L \bigcup_{j=0}^{2^\ell-1} \cD_{\ell,j,x}$.  Hence we have that
\begin{align*}
\P[\cD_{\ell,j,x}]&=\P[X_{K_\ell} - \E X_{K_\ell} >x 2^{-\ell\alpha/2} Q_n]\leq \P[X_{K_\ell} - \E X_{K_\ell} >x 2^{\ell\alpha/4} Q_{K_\ell}]\\
&\leq\exp(1-(x 2^{\ell\alpha/4} )^{\theta}).
\end{align*}
So for some $C>0$, by a union bound
\begin{equation}\label{eq:cDbound}
\P[\cD_x]\leq 1\wedge\sum_{\ell=1}^L 2^\ell\exp(1-(x 2^{\ell\alpha/4} )^{\theta})\leq \exp(1-C x^{\theta}).
\end{equation}

When $\cD_x^c$ holds, for any $u_{L,j}$, by adding up the passage times to its parents and ascendant up to $u_{0,0}$ we have that for large enough constants $C',C''>0$,
\begin{align}
|X_{u_{L,j}, u_{0,0}}| &\leq \sum_{\ell=1}^L |X_{u_{\ell,\lfloor 2^{\ell-L}j \rfloor }, u_{\ell-1,\lfloor 2^{\ell-1-L}j \rfloor}}|\nonumber\\
&\leq \sum_{\ell=1}^L \E X_{K_\ell} +  x 2^{-\ell\alpha/2} Q_n\nonumber\\
&\leq \sum_{\ell=1}^L \mu 2^{-\ell-1}(n+Q_n/2) + D_1 Q_{K_\ell} + x 2^{-\ell\alpha/2}  Q_n\nonumber\\
&\leq \frac12\mu n + (C'+2+C'' x) Q_n,
\end{align}
where the first inequality is by the triangle inequality, second used the event $\cD_x^c$, the third used Lemma~\ref{l:AnBound} and the final inequality used equation~\eqref{eq:QKellBound}.  Now every point in $L_{\cR_n}$ is within distance 2 of some $u_{L,j}$ so
\begin{align}\label{eq:leftToCentre}
\P[\max_{u\in L_{\cR_n}}X_{u, u_{0,0}} > \frac12\mu n + (C'+2+(C''+1) x) Q_n]&\leq \P[\Gamma_n> x Q_n] + \P[\cD_{x}]\nonumber\\
&\leq  2\exp(1-C x^{\theta}).
\end{align}
Similarly we have that
\begin{equation}\label{eq:rightToCentre}
\P[\max_{u\in R_{\cR_n}}X_{u, u_{0,0}} > \frac12\mu n + (C'+2+(C''+1) x) Q_n]\leq   2\exp(1-Cx^{\theta}).
\end{equation}
Combining~\eqref{eq:leftToCentre} and~\eqref{eq:rightToCentre} we have that
\begin{equation}
\P[Y_n^+ > \mu n + 2(C'+2+(C''+1) x) Q_n]\leq   4\exp(1-C x^{\theta}),
\end{equation}
which completes the proof, by taking $D$ sufficiently large.
\end{proof}

\begin{lemma}\label{l:YPlusBound}
There exists $D_2,D_3$ such that for all $n\geq 1$ we have that
\begin{align*}
|\E Y_n^{+} - n\mu| \leq D_2 Q_n,\\
\P[|Y_n^+ - \E Y_n^+|> x D_3 Q_n]  \leq \exp(1-x^\theta).
\end{align*}
\end{lemma}
Note that by increasing the values if necessary we can choose these constants to be the same as those in Lemma \ref{l:YMinusBound}. 
\begin{proof}
Integrating the probability of Lemma~\ref{l:YPlusFlexible}  similarly to~\eqref{eq:xTwoNMean} we have that for some $C>0$,
\begin{equation}\label{eq:YnPlusMean}
\E[Y_n^+] \leq \mu n + C Q_n.
\end{equation}
Since $n\mu\leq \E[X_n]\leq \E[Y_n^+]\leq \mu n + C Q_n$, this completes the first part of the lemma.  Substituting~\eqref{eq:YnPlusMean} into~Lemma~\ref{l:YPlusFlexible}  we have that for $D_3$ sufficiently large,
\begin{equation}\label{eq:YnPlusUpperA}
\P[Y_n^+ -\E Y_n^+ \geq x D_3 Q_n]\leq   \frac{1}{2}\exp(1- x^\theta).
\end{equation}
Since $X_n\leq Y_n^+$, for $x>2D_2$ we have that
\begin{align}\label{eq:YnPlusUpperB}
\P[Y_n^+ -\E Y_n^+ < - x Q_n] &\leq \P[X_n -\E Y_n^+ < - x Q_n]\nonumber\\
& \leq \P[X_n -\E X_n^+ < - (x-D_2) Q_n]\leq \exp(1- (x/2)^\theta).
\end{align}
Combining~\eqref{eq:YnPlusUpperA} and ~\eqref{eq:YnPlusUpperB} and increasing the value of $D_3$ if necessary completes the lemma.
\end{proof}

Finally, let us state the following useful result which bounds the passage times across parallelograms. This follows immediately from Lemmas \ref{l:YMinusBound}, \ref{l:YPlusBound} together with Lemma \ref{l:paraChange} and hence we shall omit the proof. 

\begin{lemma}
    \label{l:paraplusminus}
    There exists $C>0$ such that for $0\leq k\leq n/W_{n}$ and $z\ge 0$ we have for all $n$ sufficiently large
    \begin{align}
\P\Big[\max_{u_1\in L_{\cP_{1,0,k,n,W_{n}}}}\max_{u_2\in R_{\cP_{1,0,k,n,W_{n}}}} \big| X_{u_1 u_2}- |u_2-u_1|\mu \big| \geq  zQ_{n} \Big] \leq \exp(1-Cz^{\theta}).
\end{align}
\end{lemma}

\section{Transversal fluctuations}
\label{s:trans}
The aim of this section is to control transversal fluctuation of the geodesics between points at distance $n$ at the scale $W_{n}$. In particular, we shall prove Theorem \ref{t:tf}, (i) with $\SD(X_{n})$ replaced by $Q_{n}$ (i.e., $\sqrt{n\SD(X_{n})}$ replaced by $W_{n}$). Using a chaining argument, together with the concentration estimates from Section \ref{s:percgen} we construct a large probability event on which any path from $(0,0)$ to $(n,0)$ having large transversal fluctuation at scale $W_{n}$ will have length significant longer than $X_{n}$. The argument here uses ideas similar to the ones used in \cite[Theorem 11.1]{BSS14} or \cite[Proposition C.8]{BGZ21} where similar estimates were proved for exactly solvable LPP models. 

We write 
$\ell_{x,w,y}=\{(x,y'):w\leq y'\leq y\}$ and define 
\begin{align*}
\fH((x_1,w_1,y_1),(x_2,w_2,y_2)):=\sup_{ u_i\in \ell_{x_i,w_i,y_i}} \Big|X_{u_1u_2} -\mu|u_1-u_2|\Big|.
\end{align*}
and let
\[
\cH^{(n)}_{x,y,y',z}=\Big\{ \fH((xn,y W_n,(y+1)W_n),((x+1)n,y' W_n,(y'+1)W_n)) \leq (\frac{\theta^2}{1000}|y-y'| + z)Q_n \Big\}
\]
and
\[
{\cH^{(n)}_{x,y,z}=\bigcap_{y'=y-n^{1-\epsilon}/W_n}^{y+n^{1-\epsilon}/W_n} (\cH^{(n)}_{x,y,y',z} \cap \cH^{(n)}_{x,y',y,z})}
\]
Notice that the intersection above is over integers from $y-\lfloor n^{1-\epsilon}/W_n \rfloor$ to $y+\lfloor n^{1-\epsilon}/W_n \rfloor$; to keep notations simpler we shall drop the floor signs here and all similar intersections in the subsequent text. 
\begin{lemma}
\label{l:paraBounds}
There exists $D>0$ such that for all $x,y,y'$ with $W_n|y-y'|\leq n$ and $z>0$,
\[
\P[\cH^{(n)}_{x,y,y',z}]\geq 1-\exp(1-D(|z|+|y-y'|)^\theta), \qquad \P[\cH^{(n)}_{x,y,z}]\geq 1-\exp(1-D|z|^\theta)
\]
\end{lemma}
\begin{proof}
We being with $\cH^{(n)}_{x,y,y',z}$.
By translation invariance we can assume that $x=y=0$.  For the first case assume that $|y'|,z \leq n^\theta$.  Applying Lemma~\ref{l:paraChange} to the parallelogram with sides $\ell_{0,0,W_n}$ and $\ell_{n,y'W_n,(y'+1)W_n}$ we always have $k=|y|\leq n^{\theta}$ and so $(k^2 Q_n)^{9/10}\leq C_1 Q_n$.  Hence
\begin{align*}
\P[(\cH^{(n)}_{x,y,y',z})^c] &\leq \exp(1-C Q_n^{\kappa/10}) + \P[ |Y^+_{n} - n\mu|+ |Y^-_{n} -n\mu| \geq  (z+\frac{\theta^2}{1000}|y-y'|-\mu-1) Q_n]\\
&\leq \exp(1-D(|z|+|y-y'|)^\theta),
\end{align*}
for small enough $D$ where the first inequality is by  Lemma~\ref{l:paraChange} and the second is by Lemma~\ref{l:YMinusBound} and Lemma~\ref{l:YPlusBound}.

Next assume that $|y'|+z \geq n^\theta$.  Since $|y'|\leq n/W_n$, the diameter of the parallelogram is at most $2n$ and by Lemma~\ref{l:growth34}, $Q_{2n+|y'|W_n} \leq C Q_n$.  
Hence we have that,
\begin{align*}
\P[(\cH^{(n)}_{0,0,y',z})^c] 
&\leq \P\Big[\sup_{u \in \ell_{0,0,W_n}} \sup_{u' \in \ell_{n,y'W_n,(y'+1)W_n}} \big|X_{u u'} -\mu|u -u'|\big|  \geq (\frac{\theta^2}{1000}|y'| + z)Q_n\Big] \\
&\leq \P\Big[\sup_{u \in \ell_{0,0,W_n}} \sup_{u' \in \ell_{n,y'W_n,(y'+1)W_n}} \big|X_{u^\Z u^{'\Z}} -\mu|u^\Z -u^{'\Z}|\big| \geq \frac12 (\frac{\theta^2}{1000}|y'| + z)Q_n\Big]\\
&\qquad +\P[\Gamma_{n}\geq \frac12(\frac{\theta^2}{1000}|y'| + z)Q_n]\\
&\leq W_n^2  \exp\big(1-((\frac{\theta^2}{1000}|y'| + z)\frac{Q_n}{Q_{2n}})^\kappa\big) + \exp(1-(\frac12(\frac{\theta^2}{1000}|y'| + z)Q_n/\log^C n)^\kappa) \\
&\leq  \exp(1-D(|z|+|y'|)^\theta),
\end{align*}
which completes the proof of the first part of the lemma.  The second inequality follows by a union bound over $y'$.
\end{proof}

 We define the event $\cS^{(n)}_{z,u}$ to require that all passage times in $B_{3n}(u)$ are somewhat well behaved.
\begin{align*}
\cS^{(n)}_{z,u}&=
\begin{cases}
\Big\{\Gamma_{3n}\leq z^{2\theta/\kappa}n^\epsilon\Big\} \cap \Big\{\sup_{v,v'\in B_{3n}(u)} \big|X_{vv'}-\mu|v-v'|\big| \leq \frac1{10}\mu n^\epsilon Q_{3n} \Big\} &\hbox{ when } zW_n \leq n^{1-10\epsilon},\\
\Big\{\Gamma_{3zn}\leq \frac1{10} \mu {z^{1/10}Q_n} \Big\} \cap \Big\{\sup_{v,v'\in B_{3zn}(u)} \big|X_{vv'}-\mu|v-v'|\big| \leq \frac1{10} \mu {z^{1/10}Q_n}\Big\} &\hbox{ when } zW_n > n^{1-10\epsilon}.
\end{cases}
\end{align*}
For $u=(in,tW_{n})$, we also define the event $\cO^{(n)}_{z,u}$ which is designed to implement a chaining argument for transversal fluctuations.
\[
\cO^{(n)}_{z,u}=\bigcap_{j=0}^{{3\epsilon\log_2 n}} \bigcap_{x=0}^{2^j-1} \bigcap_{y=-z W_{n}/W_{n2^{-j}}}^{z W_{n}/W_{n2^{-j}}} {\cH^{(n2^{-j})}_{i2^{j}+x,tW_{n}/W_{n2^{-j}}+y,z2^{\theta j}\theta^2/1000}}.
\]

\begin{lemma}\label{l:trans.SOGam}
There exists $D>0$ such that for any $n\geq 1$ we have that for all $z\geq 10$,
\[
\P[\cO^{(n)}_{z,u}]\geq 1-\exp(1-Dz^{\theta}),\quad \P[\cS^{(n)}_{z,u}]\geq 1-\exp(1-D(zn)^{\frac12\epsilon\theta}).
\]
\end{lemma}
\begin{proof}
We start with $\cO^{(n)}_{z,u}$.  By Lemma~\ref{l:paraBounds}
\begin{align*}
\P[\cO^{(n)}_{z,u}] & \geq 1 - \sum_{j=0}^{3\epsilon\log_2 n} 2^j (2z  W_{n}/W_{n2^{-j}}) \exp\Big(1-D \big(z2^{\theta j}\theta^2/1000\big)^\theta\Big)\\
& \geq 1 - \sum_{j=0}^{3\epsilon\log_2 n} z C^{j+1} \exp\Big(1-D \big(z2^{\theta j}\theta^2/1000\big)^\theta\Big) \geq 1-\exp(1-Dz^{\theta})
\end{align*}

Next we consider $\cS^{(n)}_{z,u}$ when $zW_n \leq n^{1-10\epsilon}$ and assume $u=\origin$.
By Lemma~\ref{l:Gamma}
\[
\P[\Gamma_{3n}\leq z^{2\theta/\kappa}n^\epsilon]\geq 1-\frac13\exp(1-D(zn)^{\theta}).
\]
Then
\begin{align*}
&\P\Big[\sup_{v,v'\in B_{3n}(u)} \big|X_{vv'}-\mu|v-v'|\big| \geq \frac1{10}\mu n^\epsilon Q_{3n},\Gamma_{zn}\leq z^{2\theta/\kappa}n^\epsilon \Big]\\
&\qquad\leq \sum_{v,v'\in v,v'\in B_{3n}(u)\cap \Z^2}  \P\Big[\sup_{v,v'\in B_{3n}(u)} \big|X_{vv'}-\mu|v-v'|\big| \leq \frac1{10}\mu n^\epsilon Q_{3n}]\\
&\qquad\leq Cn^2 \exp(-C' n^{\epsilon\theta})\leq Cn^4 \exp(-C' (zn)^{\frac12\epsilon\theta})
\end{align*}
and so $\P[\cS^{(n)}_{z,u}]\geq 1-\exp(1-D(zn)^{\frac12\epsilon\theta})$. {The case of $zW_n > n^{1-10\epsilon}$ follows similarly.}
\end{proof}

We will rule out paths with large transversal fluctuations by showing that they have passage times significantly larger than the optimal path.  For $v_1\in\ell_{0,0,W_n},v_2\in\ell_{n,0,W_n}$
\[
\Upsilon_{n,v_1,v_2,z}=\{\gamma':\gamma'(0)=v_1,\gamma'(1)=v_2, \sup_t\inf_{x\in[0,n]} |\gamma'(t) - (x,0)|\geq zW_n\}
\]
denote the set of paths that travel at least distance $zW_n$ from the line segment joining $(0,0)$ and $(n,0)$; see Figure~\ref{f:upsilon}.  We let $\Upsilon_{n,v_1,v_2,z,u}$ denote the set of paths in $\Upsilon_{n,v_1-u,v_2-u,z}$ translated by $u\in \Z^2$.
For $v_1=\origin,v_2=(n,0)$ this contains the paths that have transversal fluctuations at least $zW_n$.
\begin{center}
\begin{figure}
\includegraphics[width=5in]{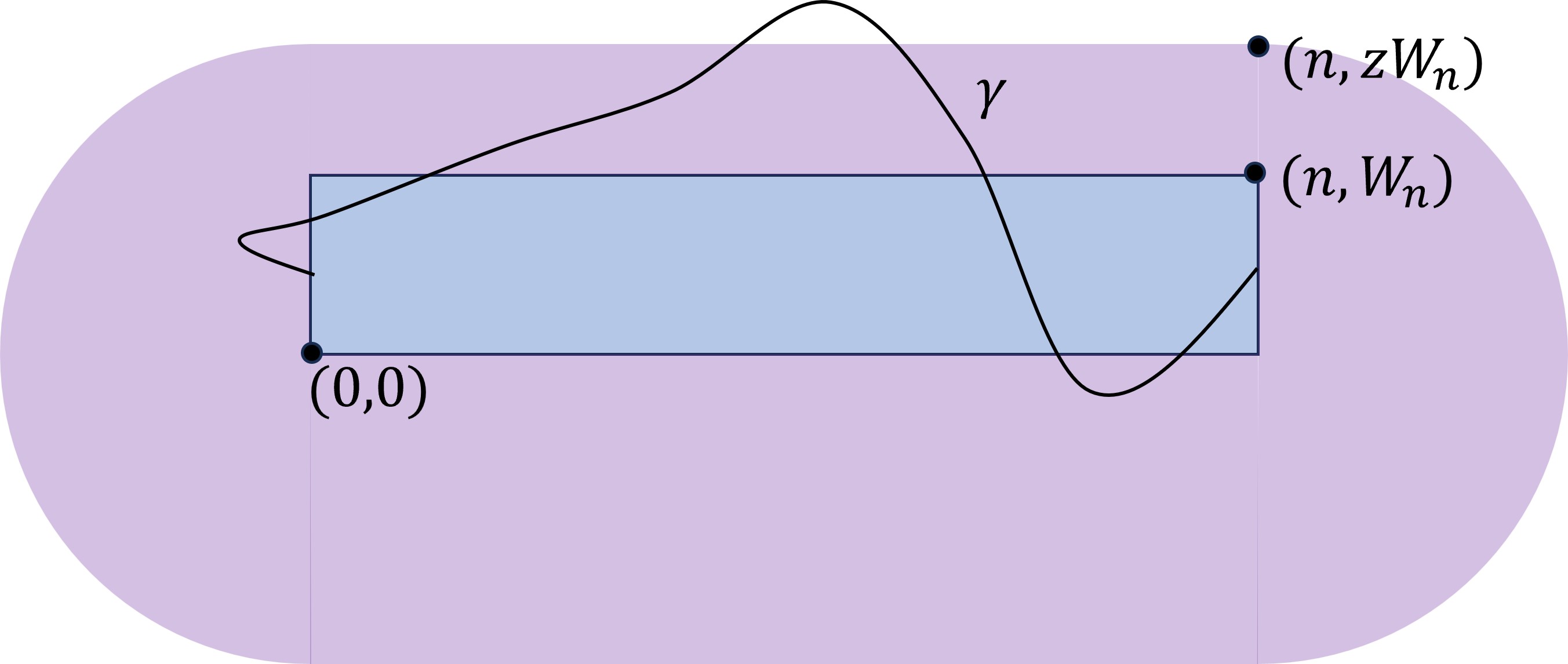}
\caption{The set of paths $\Upsilon_{n,v_1,v_2,z}$ are those that leave the region purple region, the set of points distance $zW_n$ from the line segment $(0,0)$ to $(n,0)$. Lemma \ref{l:trans.events} shows that for large $z$, on the large probability event $\cS^{(n)}_{z,\origin}\cap \cO^{(n)}_{z,\origin}$,  the lengths of all such paths exceed $X_{n}$ by a positive amount (at scale $Q_{n}$).}
\label{f:upsilon}
\end{figure}
\end{center}

\begin{lemma}\label{l:trans.events}
Suppose that the event $\cS^{(n)}_{z,\origin}\cap \cO^{(n)}_{z,\origin}$  holds for some $z\geq 4$. Then
\[
\inf_{v_1\in\ell_{0,0,W_n},v_2\in\ell_{n,0,W_n}} \inf \{X_{\gamma'} - X_{v_1,v_2}:\gamma'\in \Upsilon_{n,v_1,v_2,z}\} \geq \frac{z\theta^2 Q_{n}}{1000}.
\]
\end{lemma}

\begin{center}
\begin{figure}
\includegraphics[width=5in]{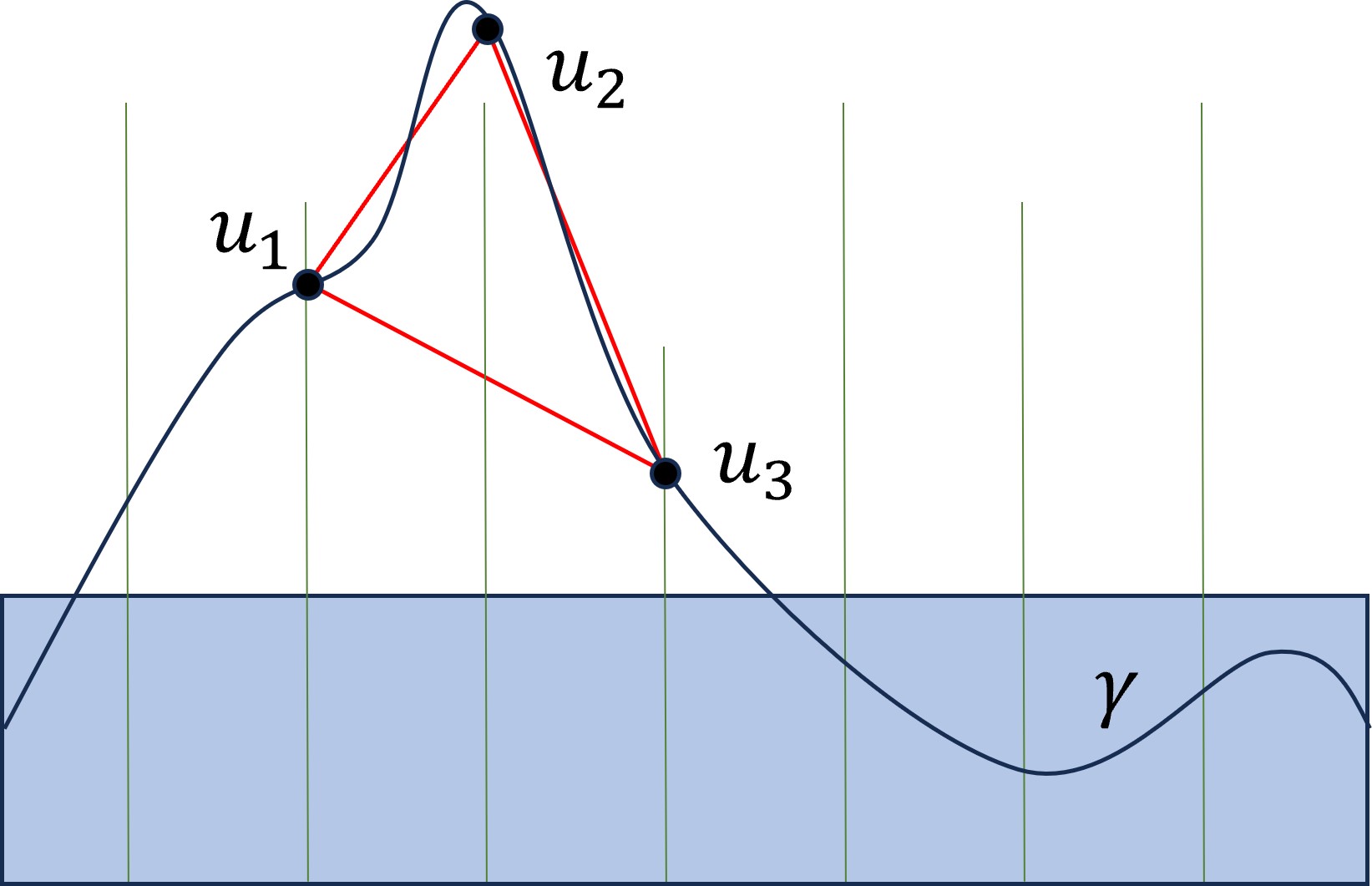}
\caption{The chaining argument in the proof of Lemma \ref{l:trans.events}: vertical lines represent intervals $\gamma$ must pass through to satisfy $A_j$ for different dyadic scales.  Here $A_j$ fails for $j=3$ at $i=3$.  The extra length required to traverse from $u_1$ to $u_3$ via $u_2$ rather than directly is used to rule out such paths under the event $\cS^{(n)}_{z,\origin}\cap \cO^{(n)}_{z,\origin}$.}
\label{f:chaining.trans}
\end{figure}
\end{center}

\begin{proof}
Suppose that we can find $\gamma'\in \Upsilon_{n,v_1,v_2,z}$ such that $X_{\gamma'} -X_{v_1 v_2}\leq \frac{z \theta^2 Q_{n}}{1000}$.  First we deal with the simple case when $zW_n \geq n^{1-10\epsilon}$.  Let $u$ be the first point in $\gamma'$ such that $\inf_{x\in[0,n]} |u - (x,0)|\geq zW_n$.  Then geometrically the point $u$ that will make the smallest difference in lengths is $(n/2,zW_n)$ so
\begin{align*}
|v_1-u|+|v_2-u|-|v_1-v_2|&\geq |v_1-(n/2,zW_n)|+|v_2-(n/2,zW_n)|-|v_1-v_2|\\
&\geq 2\sqrt{(n/2)^2 + ((z-1)W_n)^2} - n \\
&\geq n^{-10\epsilon}zW_n
\end{align*}
for large $n$.  Then by $\cS^{(n)}_{z,\origin}$
\begin{align*}
X_{\gamma'} -X_{v_1 v_2}&\geq \mu(|v_1-u|+|v_2-u|-|v_1-v_2|)\\
&\qquad -\Big(\big|X_{v_1 u} - \mu|v_1-u|\big| + \big|X_{v_2 u}  - \mu|v_2-u|\big| +  \big|X_{v_1 v_2} - \mu|v_1-v_2|\big|\Big) - 3 \Gamma_{zn}\\
&\geq \mu n^{-10\epsilon}z W_n - \frac6{10} \mu z^{1/10}Q_n \geq \frac4{10} \mu n^{-10\epsilon} zW_n > \frac{z \theta^2 Q_{n}}{1000},
\end{align*}
which contradicts our assumption.

So we continue to the case $zW_n \leq n^{1-10\epsilon}$. First suppose that there exists $0\leq s_1< s_2 < s_3\leq 1$ such that with $u_i=\gamma'(s_i)$ and
\begin{equation}\label{eq:big.fluctuation.pts}
|u_1-u_2| + |u_2-u_3| - |u_1-u_3| \geq n^\epsilon Q_{3n}.
\end{equation}
Then by $\cS^{(n)}_{z,\origin}$,
\begin{align*}
X_{\gamma'} -X_{v_1 v_2}&\geq \mu(|v_1-u_1|+ |u_1-u_2| + |u_2-u_3| +|v_2-u_3|-|v_1-v_2|)\\
&\qquad -\Big(\big|X_{v_1 u_1} - \mu|v_1-u_1|\big| + \big|X_{u_1 u_2} - \mu|u_1-u_2|\big| + \big|X_{u_2 u_3}  - \mu|u_2-u_3|\big| \\
&\qquad \quad + \big|X_{v_2 u_3}  - \mu|v_2-u_3|\big| +  \big|X_{v_1 v_2} - \mu|v_1-v_2|\big|\Big) + 5 \Gamma_{zn}\\
&\geq \mu n^\epsilon Q_{3n} - \frac5{10} \mu n^\epsilon Q_{3n} -5n^\theta > \frac{z \theta^2 Q_{n}}{1000},
\end{align*}
which contradicts our assumption.

So we may assume that~\eqref{eq:big.fluctuation.pts} fails for all triples of points along the path $\gamma'$.  This in particular implies that $\gamma'$ must be contained in $B_{3n}(\origin)$.
For $x'\in(0,n)$ define $f(x')$ such that $(x',f(x'))$ is the first point $\gamma'$ hits the line $x=x'$ and set $f(0)=v_1,f(n)=v_2$.  Since~\eqref{eq:big.fluctuation.pts} fails by geometry we can assume that
\[
\sup_{x\in[0,1]}|f(x)| \leq n^{9/10}.
\]
For $j\geq 0$, let $$r_j=\Big(1+\tfrac12(1-2^{-\theta(j+1)})z\Big)W_n\leq \frac12 zW_n$$ and let $A_j$ denote the event
\[
A_j=\Big\{\max_{i\in[0,2^j]\cap\Z} |f(i2^{-j}n)| \leq r_j)  \Big\}.
\]
Suppose that $A_j$ fails to holds for some $0\leq j \leq j_\star=3\epsilon\log_2 n$, illustrated in Figure~\ref{f:chaining.trans}.  Let $j$ be the smallest such value.  Since we set $f(0)=v_1,f(n)=v_2$ we immediately have $A_0$.  Since $A_j$ fails there is some odd $i\in[0,2^j]$ such that $f(i2^{-j}n) > r_j$.  So write $i=2i'+1$ and let $u_1 = (2i'2^{-j}n,f(2i'2^{-j}n))$, $u_2 = ((2i'+1)2^{-j}n,f((2i'+1)2^{-j}n))$ and $u_3 = ((2i'+2)2^{-j}n,f((2i'+2)2^{-j}n))$.  If $f(i2^{-j}n) > r_j$ then
\[
|f(i2^{-j}n)| > r_j > r_{j-1}\geq \max\{|f(2i'2^{-j}n))|,|f((2i'+2)2^{-j}n)|\}.
\]
To condense notation we will write $f$ in place of $f(i2^{-j}n)$ and so
\begin{align}\label{eq:trans.extra.length}
|u_1-u_2| + |u_2-u_3| - |u_1-u_3| &\geq |(2i'2^{-j}n,r_{j-1})-u_2|+|u_2-((2i'+2)2^{-j}n,r_{j-1})| - 2^{-(j-1)}n \nonumber\\
&\geq \frac13\frac{(|f| -r_{j-1})^2}{2^{-j}n}
\end{align}
where the first inequality holds by geometry and the second by ~\eqref{eq:Pythag2}.
Let us consider the case of $X_{u_1 u_2}$.  For some integers $y,y'$  we have that $u_1\in \ell_{(i-1)n2^{-j},yW_{n2^{-j}},(y+1)W_{n2^{-j}}}$ and $u_2\in \ell_{n2^{-j},y'W_{n2^{-j}},(y'+1)W_{n2^{-j}}}$.  Furthermore  $|y|\leq zW_n/W_{n2^{-j}}$ and 
\[
|y-y'| \leq \frac{|f|+zW_n}{W_{n2^{-j}}}
\]

For $x\geq r_j$, since  $W_{n2^{-j}}=\sqrt{n2^{-j} Q_{n2^{-j}}}$  and $\frac{r_{j}-r_{j-1}}{W_{n2^{-j}}}=\frac12 z(1-2^{-\theta})2^{-\theta j}\frac{W_n}{W_{n2^{-j}}}\geq  \frac{1}{5}z\theta 2^{-\theta j}$ and $\frac{W_n}{W_{n2^{-j}}}\geq 2^{j/2}$ we have that,
\begin{align*}
\frac{d}{dx} \frac1{20} \frac{(x-r_{j-1})^2}{2^{-j}n} - \frac{\theta^2}{1000}(\frac{x+zW_n}{W_{n2^{-j}}} Q_{n2^{-j}}+ z2^{\theta j}Q_{n}) 
&= \frac{x-r_{j-1}}{10\cdot 2^{-j}n} - \frac{\theta^2}{1000}\frac{Q_{n2^{-j}}}{W_{n2^{-j}}}\\
&\geq  \bigg(\frac{r_j-r_{j-1}}{10W_{n2^{-j}}} - \frac{\theta^2}{1000}\bigg)\frac{Q_{n2^{-j}}}{W_{n2^{-j}}} \geq 0.
\end{align*}
Also
\begin{align*}
\frac1{20} \frac{(r_j-r_{j-1})^2}{2^{-j}n} 
&\geq  \frac{z^2 2^{-\theta j}\theta^2 W_n^2}{500 \cdot 2^{-j}n}  \\
&\geq \frac{2^{-\theta j} \theta^2 W_n^2}{1000 \cdot 2^{-j}n}\bigg(2z 2^{\theta j-j/2}+ z 2^{2\theta j-j} \bigg)\\
&\geq \frac{\theta^2}{1000}\bigg(\frac{2zW_n W_{n 2^{-j}}}{n2^{-j}}+ \frac{z2^{\theta j} W_{n }^2}{n}\bigg)\\
&\geq \frac{\theta^2}{1000}\bigg(\frac{r_j+zW_n}{W_{n2^{-j}}}+ z2^{\theta j}\bigg )Q_{n2^{-j}} .
\end{align*}
Hence combining the last two equations we have that
\[
\frac1{20} \frac{(|f|-r_{j-1})^2}{2^{-j}n} \geq \frac{\theta^2}{1000} (\frac{|f|+zW_n}{W_{n2^{-j}}}Q_{n2^{-j}} + z2^{\theta j}Q_{2^{-j}n} ) \geq \frac{\theta^2}{1000}(|y-y'| + z2^{\theta j})Q_{n2^{-j}}
\]
and so
\begin{align*}
\left\{\big|X_{u_1 u_{2}} -\mu|u_1-u_{2}|\big| \leq \frac1{20} \frac{(|f|-r_{j-1})^2}{2^{-j}n}\right\}
&\supseteq\left\{\big|X_{u_1 u_{2}} -\mu|u_1-u_{2}|\big| \leq \frac1{1000\theta^2}(|y-y'| + z2^{\theta j})Q_{n2^{-j}}\right\}\\
&\supseteq \cH^{(n2^{-j})}_{(i-1),y,y',z 2^{\theta j}\theta^2/1000} \supseteq \cO^{(n)}_{z,\mathbf{0}}.
\end{align*}
We similarly have 
\begin{align*}
\left\{\big|X_{u_1 u_{3}} -\mu|u_1-u_{3}|\big| \leq \frac1{20}\frac{(|f| -r_{j-1})^2}{2^{-j}n}\right\} \supseteq \cO^{(n)}_{z,\mathbf{0}},\\
\left\{\big|X_{u_2 u_{3}} -\mu|u_2-u_{3}|\big| \leq \frac1{20}\frac{(|f| -r_{j-1})^2}{2^{-j}n}\right\} \supseteq \cO^{(n)}_{z,\mathbf{0}}.
\end{align*}
Thus on the event $\cO^{(n)}_{z,\mathbf{0}}$
\begin{align}
X_{u_1u_2} + X_{u_2 u_3} - X_{u_1 u_3} \geq \frac15 \frac{(|f| -r_{j-1})^2}{2^{-j}n}\geq \frac{z \theta^2 Q_{n}}{500}.
\end{align}
Combining with  $\cS^{(n)}_{z,\origin}$ we have that
\begin{align*}
X_{\gamma'} &\geq X_{v_1u_1} + X_{u_1u_2} + X_{u_2 u_3} + X_{u_3v_2}-4z^{2\theta/\kappa}n^\epsilon\\
&\geq X_{v_1u_1} + X_{u_1  u_3} + X_{u_3v_2}+\frac{z\theta^2 Q_{n}}{500}-4z^{2\theta/\kappa}n^\epsilon \\
&\geq X_{v_1v_2} +\frac{z \theta^2 Q_{n}}{1000}.
\end{align*}
Hence we have $A_j$ for all $j\leq j_{\star}$.  Now set $t$ as the first time such that $\inf_{x\in[0,n]} |\gamma'(t) - (x,0)|\geq zW_n$.  For some $i\in[2^{j_\star}]:=\{1,2,\ldots, 2^{j}\}$ we have that $t$ occurs after the first time $\gamma'$ hits the line $x=(i-1)2^{-j_\star}n$ and before it hits the line $x=i 2^{-j_\star}n$.  Let $u_1=((i-1)2^{-j_\star}n,f((i-1)2^{-j_\star}n))$ and $u_3 = (i 2^{-j_\star}n,f(i 2^{-j_\star}n))$ and $u_2=\gamma'(t)$.  On the event $A_{j_\star}$ we have that
\[
|u_1-u_2| + |u_2-u_3| - |u_1-u_3|\geq 2\sqrt{(n2^{-j_\star-1})^2+(zW_n/2)^2} - n2^{-j_\star} ,
\]
where in the first inequality we chose the points that geometrically minimize the left hand side given the definition of the $u_i$ and $A_{j_\star}$.  Since $n2^{-j_\star}=n^{1-3\epsilon} \gg n^{1-10\epsilon} \geq zW_n$ by Taylor Series we have that
\[
|u_1-u_2| + |u_2-u_3| - |u_1-u_3|\geq \frac{(zW_n)^2}{n2^{-j_\star-2}}\geq z^2 n^{3\epsilon}Q_n \geq  \mu n^\epsilon Q_{3n}> \frac{z \theta^2 Q_{n}}{1000},
\]
for large $n$ and so contradicts our assumption that~\eqref{eq:big.fluctuation.pts} fails which completes the proof.
\end{proof}

The following lemma now follows immediately combining Lemma~\ref{l:trans.SOGam} and~\ref{l:trans.events}.
\begin{theorem}\label{t:trans.main}
There exists $D>0$ such that for any $n\geq 1$ we have that for all $z\geq 0$,
\[
\P[\trans_n \geq z W_n]\leq \exp(1-Dz^{\frac12 \epsilon\theta}).
\]
\end{theorem}

\section{Improved concentration}
\label{s:impconc}

In this section we establish improved concentration bounds for passage times. We say that $\cL(n,R)$ holds if for all $x>0$
\begin{align}\label{eq:ImprovedBounds}
\P[|X_n-n\mu|>x R Q_n]&\leq \exp(1-x^{\tfrac43 \theta}),\nonumber\\
\P[|Y^-_n-n\mu|>x R Q_n]&\leq \exp(1-x^{\tfrac43\theta}).
\end{align}
We will establish inductively that for some sufficiently large $R$ that $\cL(n,R)$ always holds.  The following lemma establishes the base case of the induction.
\begin{lemma}\label{l:ImprovedBaseCase}
For sufficiently large $R$, we have that $\cL(n,R)$ for all $n\leq R$.
\end{lemma}
\begin{proof}
The bound for $X_n$ follows from Assumptions~\ref{as:speed} and~\ref{as:concentration} since $R \geq n$ and $\tfrac43\theta<\kappa$.  The side to side distance of $\cR_n$ ranges from $n$ to $\sqrt{n+W_n^2}$ by~\eqref{eq:Pythag} hence we have that
\begin{align}
|Y_n^- -n\mu|&\leq \sup_{u\in L_{\cR_n},v\in R_{\cR_n}} |X_{uv}-\mu|u-v|]+\frac12\mu Q_n\\
&\leq \sup_{u\in L_{\cR_n},v\in R_{\cR_n}} |X_{u^\Z v^\Z}-\mu|u^\Z-v^\Z|]+\frac12\mu Q_n + 2 + 2\Gamma_n.
\end{align}
Then for $x\geq 1$,
\begin{align*}
&\P[|Y^-_n-n\mu|\geq x R Q_n]\\
 &\qquad\leq \P[\sup_{u\in L_{\cR_n},v\in R_{\cR_n}} |X_{u^\Z v^\Z}-\mu|u^\Z-v^\Z|]\geq \frac14 x R Q_n] + \P[\Gamma_n \geq \frac14 x R Q_n]\\
&\qquad \leq n^2 \exp\left(1-\Big(\frac14 R^{1/2}x \Big)^\kappa\right) + \exp\bigg(1-\Big(\frac{R Q_{n}}{4\log^C n} x\Big)^\kappa\bigg)\\
&\qquad \leq \exp(1-x^{\tfrac43\theta})
\end{align*}
where the second inequality follows by a union bound over the at most $n^2$ choices of $u^\Z, v^\Z$, Assumption~\ref{as:concentration} and Lemma~\ref{l:Gamma}.
\end{proof}

The following two lemmas will complete the inductive step.

\begin{lemma}\label{l:ImprovedYmin}
For $M$ sufficiently large we have that if $n_\star\geq M^{100/\alpha}$ and $M^{10}\leq R$ then if $\cL(n',R)$ holds for all $n'\leq n_\star$ then
\begin{equation}\label{eq:ImprovedYmin}
\P[Y^-_{n''}-n''\mu \leq -x R Q_{n''}]\leq  \frac12\exp(1-x^{\tfrac43\theta}).
\end{equation}
holds for all $n_\star\leq n''\leq \frac{M}{2}n_\star$
\end{lemma}
\begin{proof}
For $n''\in[n_\star,\frac{M}{2}n_\star]$ we set $n=\frac1{M}n''$.  We first give a coarse bound for large $x$.  By Assumptions~\ref{as:speed} and~\ref{as:concentration} and a union bound over starting and ending points we have that for $x>n^{2/3}$
\begin{align}\label{eq:YminRbound1}
\P[Y_{Mn}^- - Mn\mu \leq -xRQ_{nM}] &\leq (2nM)^4 \P[ \min_{u,v\in B_{2nM}(\mathbf{0})}X_{u^\Z v^\Z} - \E X_{u^\Z v^\Z} \leq -xRQ_{nM}/2] \nonumber\\
&\qquad + 2\P[\Gamma_{2Mn} \geq xRQ_{nM}/2]\nonumber\\
& \leq \exp(1-x^{\kappa/10}) \leq \frac{1}{2}\exp(1-x^{\frac43 \theta}).
\end{align}

We will establish the improved tail bounds by following the percolation method from Section~\ref{s:percgen}.   We define $\ell_i$ to be the line $x=in$.  Let $\gamma$ be the optimal path for $Y^-_{nM}$ and let $t_i:=\inf\{t\in[0,1]: \gamma_1(t)=in\}$ and $k^\gamma_i = \lfloor \gamma_2(t_i)/W_n\rfloor$ where $\gamma_1$ and $\gamma_2$ denote the two coordinates of the path $\gamma$.  By Lemma~\ref{l:Qgrowth} we have that
\[
Q_{nM}\leq D_6^{\lceil \log_{3/2} M\rceil}Q_n
\]
and so $W_{Mn}/W_n \leq M^\zeta$ for some $\zeta>0$. It follows that $0 \leq k^\gamma_0 \leq M^\zeta$.
We denote
\[
\mathfrak{K}^+_M=\{(k_0,\ldots,k_M)\in\Z^{M+1}: 0\leq k_0 \leq M^\zeta\}
\]
so we have that $\uk^\gamma\in \mathfrak{K}^+_M$.  Next let $\cJ$ be the event
\[
\cJ=\{\max_{1\leq i \leq M-1} |k_i^\gamma| \leq \frac14 n/W_n\}
\]
and similarly to equation~\eqref{eq:cJBound} have that $\P[\cJ^c]\leq \exp(-Cn^{\kappa/5})$.
Setting
\[
\cZ_{i,k,k'}:=\left(-(Z^{-,\Lambda_i}_{i,k,k'} -n\mu)/Q_n + \frac{k^2}{32} \right)I(|k|\vee |k'|\leq n/W_n).
\]
we have by Corollary~\ref{c:paraRect} that,
\begin{align}\label{eq:ZtildeBound}
\P[\cZ_{i,k,k'} > z] &\leq \P\Big[Y^-_{n} < n\mu -\frac{(k^2-16k+C)Q}{32} -zQ\Big] +\exp(1-CQ^{\kappa/10})
\end{align}
and so
\[
\P[\frac1{R}\cZ_{i,k,k'}I(\cZ_{i,k,k'}\leq n) > z] \leq \exp(1-C'z^{\tfrac43\theta})
\]
and hence satisfies the hypothesis of Corollary~\ref{c:percolation}.
Let $\cG$ be the event 
\[
\cG=\{\max_{i,|k|,|k'|\leq \tfrac{n}{W_n}} \cZ_{i,k,k'}\leq n\}.
\]
By~\eqref{eq:ZtildeBound} we have that $\P[\cG^c]\leq n^5 \exp(1-CQ^{\kappa/10})$.
Then by Corollary~\ref{c:percolation}, similarly to ~\eqref{eq:percApplication1}, we have that
\begin{align}\label{eq:percApplication2}
&\P[\cJ,\cG,\sum_{i=1}^M Z^{-,\Lambda_i}_{i,k^\gamma_{i-1},k^\gamma_i} - nM\mu \leq  -(C_7 M \log^{C_7}(3+32R)  + x)R Q_n]\\
&\quad\leq \P[\max_{\uk\in\mathfrak{K}_M^+}\sum_{i=1}^M \frac1{R}\cZ_{i,k,k'}I(\cZ_{i,k,k'}\leq n) + \frac{(k_{i-1}-k_i)^2}{32R} \geq C_7 M \log^{C_7}(3+32R)  + x]\nonumber\\
&\quad \leq  C_8M^\zeta\exp\left(-C_9 x^{\tfrac43\theta}\right ).\nonumber
\end{align}
If $M^2 R\leq x\leq n^{2/3}$ and $M$ is sufficiently large then since $Q_{Mn}\geq M^\alpha Q_n$,
\begin{align}\label{eq:YminRbound2}
&\P[Y^-_{nM} - nM\mu \leq  -x R Q_{nM}]\\
&\quad\leq \P\left[\max_{\uk\in\mathfrak{K}^+_M}\sum_{i=1}^M \frac1{R}\cZ_{i,k,k'}I(\cZ_{i,k,k'}\leq n) - \frac{(k_{i-1}-k_i)^2}{32R} \geq (C_7 M \log^{C_7}(3+32R)  + \frac{x}{2})M^\alpha\right]\nonumber\\
&\qquad + \P[\cJ^c\cup\cG^c] + \P[\cJ,\cG,|X_{nM} - \sum_{i=1}^M X^{\Lambda_i}_{\gamma(t_{i-1}),t(\gamma_i)}| > Q_n]\nonumber\\
&\quad \leq  C_8M^\zeta\exp\left(-C_9 (M^\alpha x/2)^{\tfrac43\theta}\right ) + 2n^4\exp(1-CQ_n^{\kappa/10})\nonumber\\
&\quad \leq \exp\left(1-C' M^{\tfrac43\alpha\theta} x^{\tfrac43\theta}\right )\leq \frac12 \exp(1-x^{\tfrac43\theta})
\end{align}
for some $C'>0$ and large enough $M$.
Finally for $1\leq x \leq M^2 R$ since by Lemma \ref{l:YMinusBound} $\E |Y^-_{nM} - nM\mu| \leq D_2 Q_{nM} \leq \frac12 x R Q_{nM}$, we have by Lemma \ref{l:YMinusBound}, for some $D>0$,
\begin{align}\label{eq:YminRbound3}
\P[Y^-_{nM} - nM\mu \leq  -x R Q_{nM}] &\leq \P[Y^-_{nM} - \E Y^-_{nM} \leq  -\frac12 x R Q_{nM}]\nonumber\\
& \leq \exp(1-D(\frac12 x R)^{\theta})\nonumber\\
&\leq \exp(1-2x^{\tfrac43\theta})\leq  \frac12\exp(1-x^{\tfrac43\theta}).
\end{align}
since for large enough $M$ we have that $D(\frac12 R)^{\theta} \geq 2(M^2R)^{\tfrac13 \theta} \geq 2x^{\tfrac13 \theta}$.  Combining \eqref{eq:YminRbound1}, \eqref{eq:YminRbound2} and~\eqref{eq:YminRbound3} establishes equation~\eqref{eq:ImprovedYmin} when $M$ is large enough.
\end{proof}

\begin{lemma}\label{l:ImprovedX}
For $M$ sufficiently large we have that if $n_\star\geq M^{100/\alpha}$ and $M^{10}\leq R$ then if $\cL(n',R)$ holds for all $n'\leq n_\star$ then
\begin{equation}\label{eq:ImprovedX}
\P[X_{n''}-n''\mu \geq x R Q_{n''}]\leq \frac12\exp(1-x^{\tfrac43\theta})
\end{equation}
holds for all $n_\star\leq n''\leq \frac{M}{2}n_\star$.
\end{lemma}
\begin{proof}
For $n''\in[n_\star,\frac{M}{2}n_\star]$ we set $n=\frac1{M}n''$.  We first give a coarse bound for large $x$.  By Assumptions~\ref{as:speed} and~\ref{as:concentration} and a union bound over starting and ending points we have that for $x>n^{2/3}$
\begin{align}\label{eq:XRbound1}
\P[X_{Mn}^- - Mn\mu \geq xRQ_{nM}] &\leq \exp(1-x^{\kappa/10}) \leq \exp(1-x^{\tfrac43\theta}).
\end{align}

If $M^2 R\leq x\leq n^{2/3}$ and $M$ is sufficiently large then since $Q_{Mn}\geq M^\alpha Q_n$, by Proposition \ref{p:perc1} 
\begin{align}\label{eq:XRbound2}
\P[X_{nM} - nM\mu \geq  x R Q_{nM}]
&\leq \P[\sum_{i=1}^M X^{\Lambda_i}_{((i-1)n,0),(in,0)}-n\mu \geq (C_3 M R  + \frac{x}{2})M^\alpha Q_{n}]\nonumber\\
&\quad + \sum_{i=1}^M \P[|X_{((i-1)n,0),(in,0)} - X^{\Lambda^i}_{((i-1)n,0),(in,0)}| > Q_n]\nonumber\\
& \leq  C_4 \exp\left(-C_5 (M^\alpha x/2)^{\tfrac43\theta}\right ) + M\exp(1-CQ_n^{\kappa/10})\nonumber\\
& \leq  \frac12\exp(1-x^{\tfrac43\theta})
\end{align}
for large enough $M$.
Finally for $1\leq x \leq M^2 R$ since by Lemmas \ref{l:YMinusBound} and \ref{l:YPlusBound}, $\E |X_{nM} - nM\mu| \leq D_2 Q_{nM} \leq \frac12 x R Q_{nM}$, so we have by \eqref{eq:Q}
\begin{align}\label{eq:XRbound3}
\P[X_{nM} - nM\mu \geq  x R Q_{nM}] &\leq \P[X_{nM} - \E X_{nM} \leq  \frac12 x R Q_{nM}]\nonumber\\
& \leq \exp(1-(\frac12 x R)^{\theta})\leq  \exp(1-2x^{\tfrac43\theta})\leq  \frac12\exp(1-x^{\tfrac43\theta})
\end{align}
since for large enough $M$ we have that $(\frac12 R)^{\theta} \geq 2(M^2R)^{\tfrac13 \theta} \geq 2x^{\tfrac13 \theta}$.  Combining \eqref{eq:XRbound1}, \eqref{eq:XRbound2} and~\eqref{eq:XRbound3} establishes equation~\eqref{eq:ImprovedX} when $M$ is large enough.
\end{proof}

\begin{lemma}\label{l:ImprovedInduction}
For large enough $R$ we have that $\cL(n,R)$ holds for all $n\geq 1$.
\end{lemma}

\begin{proof}
Pick $M$ and $n_\star$ large enough so that Lemmas~\ref{l:ImprovedYmin} and~\ref{l:ImprovedX} hold.  By Lemma~\ref{l:ImprovedBaseCase} we can find $R$ large enough such that $\cL(n,R)$ holds for all $n\leq n_\star$ and $R\geq M^{10}$.  Then since $Y^-_n\leq X_n$  Lemmas~\ref{l:ImprovedYmin} and~\ref{l:ImprovedX} imply that $\cL(n,R)$ holds for all $n\leq \frac12 M  n_\star$.  Applying this inductively implies that $\cL(n,R)$ holds for all $n\geq 1$.
\end{proof}

\section{Record points and growth of $Q_n$}
\label{s:record}

Note that in our choice of parameters, a range of values of $\alpha$ are possible and that all our results thus far apply for any $\alpha$ in this range.  With this in mind let

\begin{equation}
    \label{e:alpharange}
\alpha<\alpha_+<\alpha_\star \in [\min\{\frac{1}{30},\frac12 \kappa\}, \min\{\frac{1}{15},\kappa\}]
\end{equation}
be three parameters in the range.  Recalling that $Q_n= \sup_{1\leq m \leq n} \left( \frac{n}{m}\right)^\alpha \hQ_m$ and we say that $n$ is an $\alpha$-record point if $Q_n= \hQ_n$.  Clearly if $\alpha'>\alpha$ then $\alpha'$-record points are also $\alpha$-record points.  We say that $n$ is a $(C,\alpha')$-quasi record point if 
\[
\sup_{1\leq m \leq n} \left( \frac{n}{m}\right)^{\alpha'} \hQ_m \leq C  \hQ_n
\]
Ultimately we will show that there exists a $C$ and $\alpha'>\alpha$ such that all $n$ are  $(C,\alpha')$-quasi record points.  

\begin{lemma}
\label{l:varbd}
For $C>0$, there exists $\hat{C},\tilde{C}>0$ such that if $n\geq 1$ is a $(C,\alpha)$ quasi-record point then
\[
\hat{C} Q_n^2\leq \Var(X_n)\leq \tilde{C} Q_n^2.
\]
\end{lemma}
\begin{proof}
The upper bound was established in ~\eqref{eq:varUpperQn}.  For the lower bound, observe that there exists $x_{\star}\ge 1$ such that 
\[
\P[|X_n-\E X_n|> x_\star \hQ_n]=\exp(1-x_\star^\theta).
\]
It is easy to see that such an $x_{\star}$ must be at least 1. To see that such a (finite, but possibly $n$-dependent) $x_{\star}$ exists one can either use Assumption \ref{as:concentration} and the fact that $\kappa>\theta$ or appeal to Lemma \ref{l:ImprovedInduction}. Since $n$ is a quasi-record point it follows that there exists an $x_\star\geq 1$ such that
\[
\P[|X_n-\E X_n|> x_\star Q_n]=\exp(1-cx_\star^\theta)
\]
for some $c>0$ ($c$ does not depend on $n$). 
By Lemma~\ref{l:ImprovedInduction} we have that
\[
\P[|X_n-n\mu|> x_\star Q_n] \leq \exp(1-(\frac1{R}x_\star)^{\tfrac43\theta}).
\]
Since $\E X_n - n\mu\leq D_1 Q_n$ (Lemma \ref{l:AnBound}) we have that
\[
\exp(1-cx_\star^\theta)=\P[|X_n-\E X_n|> x_\star Q_n]\leq \exp(1-(\frac1{R}(x_\star-D_1))^{\tfrac43\theta}).
\]
Then since  $Cy^\theta<(\frac1{R}(y-D_1))^{\tfrac43\theta}$ for all $y>c^{-\theta}(4D_1 R)^4$ so we must have that
\begin{equation}\label{eq:recordXstar}
1\leq x_\star\leq c^{-\theta}(4D_1 R)^4.
\end{equation}
Hence by Chebyshev's Inequality
\begin{align*}
\Var(X_n) \geq (x_{\star}Q_n)^2 \exp(1-x_{\star}^\theta) \geq \hat{C}Q_n^2.
\end{align*}
\end{proof}

Next we show how to establish the existence of numerous nearby record points on a range of geometric scales.

\begin{lemma}
\label{l:grmain}
Given $C\geq 1$ and $\alpha'>\alpha$ there exists $C_{10},C_{11}$ such that, if $n\geq 1$ is a  $(C,\alpha')$-quasi record point then for all $L\in [C_{10},n^{-1/(2C_{11})}]$ there exists $m\in [nL^{-C_9},nL^{-1}]$ with $m$ an $\alpha$-record point.
\end{lemma}

\begin{proof}
Let $x_0=\log n$ and let $f(x)=\log(Q_{e^x})$.
Then since $n$ is a  $(C,\alpha')$-quasi record point we have that for $0\leq x \leq x_0$,
\begin{equation}\label{eq:recordA}
f(x) \leq f(x_0)-\alpha'(x_0-x)+\log C.
\end{equation}
By Lemma~\ref{l:Qgrowth} we have that $Q_n \leq D_6^{\lceil\log_{3/2}(n/n')\rceil} Q_{n'}$ and so
\begin{equation}\label{eq:recordB}
f(x) \geq  f(x_0)-\frac{\log(D_6)}{\log(3/2)}(x_0-x)-\log(D_6).
\end{equation}
Now define
\[
x_\star=\inf\Big\{x:f(x)\geq f(x_0)-\frac{\log(D_6)}{\log(3/2)}\log (L)-\log(D_6) -\alpha (x_0-\log (L)-x)\Big\}.
\]
By equation~\eqref{eq:recordB} we have that
\begin{align*}
f(x_0-\log L)&\geq f(x_0)-\frac{\log(D_6)}{\log(3/2)}\big(x_0-(x_0-\log (L))\big)-\log(D_6)\\
&= f(x_0)-\frac{\log(D_6)}{\log(3/2)}\log (L)-\log(D_6) -\alpha \big(x_0-\log (L)-(x_0-\log (L))\big)
\end{align*}
so $x_\star \leq x_0-\log L$.  If for some $z>0$,
\[
x=x_0 - \frac{\frac{\log(D_6)}{\log(3/2)}\log (L)+\log(D_6)-\alpha\log L+\log(C)}{\alpha'-\alpha} - z
\]
then by equation~\eqref{eq:recordA}
\begin{align*}
f(x) &\leq f(x_0) - \alpha'(x_0-x) +\log(C)\\
&\leq f(x_0) - (\alpha' - \alpha)(x_0-x) -\alpha(x_0-x)+\log(C)\\
&= f(x_0) - \frac{\log(D_6)}{\log(3/2)}\log (L)-\log(D_6)+\alpha\log L - (\alpha'-\alpha)z -\alpha(x_0-x)\\
&= f(x_0)-\frac{\log(D_6)}{\log(3/2)}\log (L)-\log(D_6) -\alpha (x_0-\log (L)-x) - (\alpha'-\alpha)z\\
&< f(x_0)-\frac{\log(D_6)}{\log(3/2)}\log (L)-\log(D_6) -\alpha (x_0-\log (L)-x)
\end{align*}
and so $x_\star \geq x_0 - \frac{\frac{\log(D_6)}{\log(3/2)}\log (L)+\log(D_6)-\alpha\log L+\log(C)}{\alpha'-\alpha}$.  If we set $m=e^{x_\star}$ then
\[
m \in [n (CD_6)^{-\frac1{\alpha'-\alpha}}L^{-\frac1{\alpha'-\alpha}(\frac{\log(D_6)}{\log(3/2)}-\alpha)}, nL^{-1}]
\]
For $C_{11}$ and $L\geq C_{10}$ large enough,
\[
(CD_6)^{-\frac1{\alpha'-\alpha}}L^{-\frac1{\alpha'-\alpha}(\frac{\log(D_6)}{\log(3/2)}-\alpha)} \leq L^{-C_{11}}
\]
and so $m\in [nL^{-C_2},nL^{-1}]$ as required.  It remains to check that $m$ is an $\alpha$-record point.  If it was not a record point then there must be some $m'<m$ such that $Q_{m'}\left(\frac{m}{m'}\right)^\alpha > Q_m$ and so if $x'=\log m'$ then
\[
f(x') > f(x_\star) - \alpha(x_\star - x').
\]
But then we would have
\begin{align*}
f(x') &> f(x_\star) - \alpha(x_\star - x')\\
&\geq f(x_0)-\frac{\log(D_6)}{\log(3/2)}\log (L)-\log(D_6) -\alpha (x_0-\log (L)-x_\star) - \alpha(x_\star - x')\\
&= f(x_0)-\frac{\log(D_6)}{\log(3/2)}\log (L)-\log(D_6) -\alpha (x_0-\log (L)-x')
\end{align*}
which contradicts the defining property of $x_\star$.  Hence $m$ is a record point satisfying the required properties.
\end{proof}

\subsection{Growth of $Q_n$}

We next establish improved control over the growth rate of $Q_n$, in particular that it is locally sub-linear.

\begin{lemma}
\label{l:growth34}
There exists $D_5>0$ and $n_\star$ such that for all $n_\star\leq n\leq n'$ we have that $Q_{n'}\leq D_{5}\left(\frac{n'}{n}\right)^{3/4}Q_n$.
\end{lemma}
\begin{proof}
Given Lemma~\ref{l:Qgrowth}, it suffices to show that for some $M$ large enough that
\begin{equation}\label{eq:growth34recond}
Q_{n M} \leq M^{3/4} Q_n,
\end{equation}
whenever $n'=nM$ is an $\alpha$-record point.  By equation~\eqref{eq:recordXstar} this implies that
\begin{equation}\label{eq:growth34deviationBound}
\exp(1-x_\star^\theta)=\P[|X_{n'}-\E X_{n'}|> x_\star Q_{n'}]
\end{equation}
for some $x_\star\in[1,(4D_1 R)^4]$.  Set $\delta>0$ small enough such that $\delta < \frac12\exp(1-(4D_1 R)^{4\theta})$ and
\[
\int_{0}^\infty \delta\wedge\exp(1-t^\theta)dt < \frac13\exp(1-(4D_1 R)^{4\theta}).
\]
We claim that for all record points $n'$,
\begin{equation}\label{eq:notTooConcentrated}
\max_{y}\P[|X_{n'}-y|\leq \frac13 Q_{n'}] \leq 1-\delta.
\end{equation}
By~\eqref{eq:growth34deviationBound} we know that either,
\[
\P[X_{n'}-\E X_{n'}> Q_{n'}] \geq \frac12\exp(1-(4D_1 R)^{4\theta})\hbox{ or } \P[X_{n'}-\E X_{n'} < - Q_{n'}] \geq \frac12\exp(1-(4D_1 R)^{4\theta}).
\]
Assume that the latter holds.  Then
\[
\E[(X_{n'}-\E X_{n'})I(X_{n'}-\E X_{n'}< 0)] \leq -\frac12\exp(1-(4D_1 R)^{4\theta}) Q_{n'}.
\]
If $\P[X_{n'}-\E X_{n'}>0]<\delta$ then
\begin{align*}
\E[(X_{n'}-\E X_{n'})I(X_{n'}-\E X_{n'}> 0)] &= \int_0^\infty \P[X_{n'}-\E X_{n'}>t] dt\\
& \leq \int_0^\infty \delta\wedge\exp(1-t^\theta)dt < \frac13\exp(1-(4D_1 R)^{4\theta})
\end{align*}
which contradicts the fact that
\[
\E[(X_{n'}-\E X_{n'})I(X_{n'}-\E X_{n'}> 0)] + \E[(X_{n'}-\E X_{n'})I(X_{n'}-\E X_{n'}< 0)]=0.
\]
Hence we have that
\[
\P[X_{n'}-\E X_{n'} < - Q_{n'}]>\delta, \quad \P[X_{n'}-\E X_{n'} >0]>\delta,
\]
and so there can be no $y$ with $\P[|X_{n'}-y|\leq \frac13 Q_{n'}] > 1-\delta$.  The case when $\P[X_{n'}-\E X_{n'}> Q_{n'}] \geq \frac12\exp(1-(4D_1 R)^{4\theta})$ follows similarly which establishes~\eqref{eq:notTooConcentrated}.  We will complete the lemma by showing that there is a $y$ such that $\P[|X_{n'} - y|>\frac13 M^{3/4}Q_n]\leq \delta$.

Following the percolation method from Section~\ref{s:percgen}, we define $\ell_i$ to be the line $x=in$ and let $\gamma$ be the optimal path for $X_{nM}$ and let $t_i:=\inf\{t\in[0,1]: \gamma_1(t)=in\}$ and $k^\gamma_i = \lfloor \gamma_2(t_i)/W_n\rfloor$
Let $\cJ=\{\max_{1\leq i \leq M} |k_i^\gamma| \leq \frac14 n/W_n\}$ and $\cA=\{\max_{1\leq i \leq M} |k_i^\gamma| \leq M^4\}$.
Similarly to equation~\eqref{eq:cJBound} have that $\P[\cJ^c]\leq \exp(-Cn^{\kappa/5})$.

Next define
\[
\bX^{\Lambda_i}_{a_{i-1}a_i}:= \min\Big\{\max \Big \{X^{\Lambda_i}_{a_{i-1}a_i},\mu|a_i-a_{i-1}|+ M^\epsilon Q_n \Big\}, \E \mu|a_i-a_{i-1}| -  M^\epsilon Q_n \Big\}
\]
which thresholds values of $X^{\Lambda_i}_{a_{i-1}a_i}$ to be at most $(M^{\epsilon}+O(1))Q_n$ from its mean in either the positive or the negative direction.
Our first approximation to the passage time from $\origin$ to $(Mn,0)$ will be
\[
\Phi_{n}=\inf_{\ua\in \cA^*} \sum_{i=1}^M \bX^{\Lambda_i}_{a_{i-1}a_i}
\]
where $\cA_{*}$ is the set of all $\ua=(a_{i})$ satisfying the condition in $\cA$. 
Note that by construction
\begin{equation}\label{eq:PhiBound}
Mn\mu - M^{1+\epsilon}Q_n \leq \Phi_{n}\leq Mn\mu + M^{1+\epsilon}Q_n.
\end{equation}
To compare $X^{\Lambda_i}_{a_{i-1}a_i}$ and $\bX^{\Lambda_i}_{a_{i-1}a_i}$, we estimate the maximum deviation from its expected value of point to point distances on our canonical parallelograms.  Let
\[
\cV =\left\{\max_{1\leq i \leq M}\max_{-M^4\leq k,k'\leq M^4}\max_{u_1\in L_{\cP_{i,k,k'}}}\max_{u_2\in R_{\cP_{i,k,k'}}} \big| X^{\Lambda_i}_{u_1 u_2}- |u_2-u_1|\mu \big| \leq M^\epsilon Q_n\right\}.
\]
By Lemma~\ref{l:paraChange} and a union bound over $i,k,k'$ we have that for large enough $M$,
\begin{equation}\label{eq:Vbound1}
\P[\cV^c]\leq M^{10}\left(\exp(1-C Q_n^{\kappa/10}) + \exp(1-(\tfrac12 M^\epsilon)^\theta)\right) \leq \exp(1-(\tfrac13 M^\epsilon)^\theta).
\end{equation}
Let $\cB$ be the event
\[
\cB:=\{|X_{nM} -  \sum_{i=1}^M X_{\gamma(t_{i-1}),t(\gamma_i)}| +\max_{1\leq i \leq M}\max_{|y|,|y'|\leq M^4 W_n} |X_{((i-1)n,y),(in,y')} - X^{\Lambda_i}_{((i-1)n,y),(in,y')}|\leq Q_n\}.
\]
Then by Lemma~\ref{l:Gamma} and Assumption~\ref{as:resamp1} we have that
\[
\P[\cJ\cap \cB^c]\leq \exp(1-n^{\alpha\kappa/2}).
\]
As before we set
\[
\cZ_{i,k,k'}:=\left(-(Z^{-,\Lambda_i}_{i,k,k'} -n\mu)/Q_n + \frac{k^2}{32} \right)I(|k|\vee |k'|\leq n/W_n).
\]
Now if $\cJ,\cV,\cB$ all holds and $\cA$ fails then $\gamma$ had a large transversal fluctuation and $\tau_1(\uk^\gamma)\geq M^4$.  Moreover, the length of the path must be at most $\Phi^{(1)}_{n}+Q_n$ so by equation~\eqref{eq:PhiBound} and Corollary~\ref{c:percolation} we have that
\begin{align}\label{eq:growth34Perc}
\P[\cA^c,\cJ,\cV,\cB]&\leq \P[\cA^c,\cJ,\sum_{i=1}^M X^{\Lambda_i}_{((i-1)n,y),(in,y')} -Mn\mu \leq (1+M^{1+\epsilon})Q_n]\\
&\quad\leq \P[\max_{\uk\in\mathfrak{K}_M,\tau(\uk)\geq M^4}\sum_{i=1}^M \cZ_{i,k,k'} - \frac{(k_{i-1}-k_i)^2}{32} \geq - (1+M^{1+\epsilon})]\nonumber\\
&\quad\leq \P[\max_{\uk\in\mathfrak{K}_M}\sum_{i=1}^M \cZ_{i,k,k'} \geq \frac{M^4}{32} - (1+M^{1+\epsilon})]\nonumber\\
&\quad \leq  \exp(1-CM^{4\theta})
\end{align}
for some $C>0$.  Combining the above estimates we have that
\begin{align*}
\P[\cA,\cJ,\cV,\cB] &\geq 1-\exp(1-(\tfrac13 M^\epsilon)^\theta) - \exp(1-CM^{4\theta}) - \exp(1-n^{\alpha\kappa/2})-\exp(-Cn^{\kappa/5})\\
&\geq 1-\exp(1-(\tfrac14 M^\epsilon)^\theta)
\end{align*}
provided $M$ and $n$ are large enough.  Now if $\cA,\cJ,\cV,\cB$ all hold it means that $\max |k_i^\gamma | \leq M^4$ and that $|X_{Mn} - \Phi|\leq Q_n$.  Now note that the $X^{\Lambda_i}_{a_{i-1}a_i}$ are independent and vary by at most $2M^{\epsilon}Q_n$ so by the Azuma-Hoeffding Inequality,
\[
\P[|\Phi-\E\Phi| > M^{2/3}Q_n] \leq 2\exp(-\frac1{8}M^{\frac13-2\epsilon}).
\]
Hence we have that
\[
\P[|X_{Mn}-\Phi|> (1+M^{2/3})Q_n]\leq   2\exp(-\frac1{8}M^{\frac13-2\epsilon}) + \exp(1-(\tfrac14 M^\epsilon)^\theta)
\]
which is less that $\delta$ if $M$ is large enough.  Hence by equation~\eqref{eq:notTooConcentrated} we must have
\[
\frac13 Q_{Mn} \leq (1+M^{2/3})Q_n
\]
for $M$ large enough which establishes~\eqref{eq:growth34recond} and completes the lemma.
\end{proof}

The exponent $\frac34$ is not tight and in fact any exponent strictly greater than $\frac12$ would suffice.  In a followup paper \cite{BSS3} {we will show in fact that one can take the exponent to be exactly $\frac12$}.

\section{Local transversal fluctuations}
\label{s:ltf}
The objective of this section is to control transversal fluctuation of a geodesic of length $n'\gg n$ at a distance $n$ from the starting point. In Section \ref{s:trans} we showed that the maximum transversal fluctuation of such a geodesic is typically $O(W_{n'})$ but one expects the transversal fluctuations at distance $n$ to be much smaller. Indeed, comparing to the results known in exactly solvable models one would expect the local transversal fluctuations at scale $n$ to be $O(W_{n})$ typically and this is what we shall prove in this section. In fact, we show something stronger: as in Section \ref{s:trans}, we show that with large probability, any path having large local transversal fluctuation will have significantly longer length than the first passage time between its endpoints. The arguments in this section might appear notation heavy, but the ideas are similar to the exactly solvable LPP cases; in fact we essentially do a quantitative version of the arguments in \cite[Theorem 3]{BSS19}, \cite[Proposition 2.1]{BBB23}; with the additional difficulty coming from the fact that we do not have precise control on growth of $W_{n}$ (for the exactly solvable cases we can simply take $W_{n}=n^{2/3}$, we use Lemma \ref{l:growth34} as a proxy for this) and also from the fact that we need to deal with very large deviations as well as backtracks due to the undirected nature of the FPP models. As we use Lemma \ref{l:growth34} (which in turn depends on results about the record points), the results in this section require the polynomial growth of $\SD(X_{n})$ (and hence planarity), and unlike the global transversal fluctuation results in Section \ref{s:trans} shall not apply to higher dimensional model without the polynomial growth assumption on $\SD(X_{n})$. 

to make things formal, we define the following set of paths. For integers $k,x,y$ with $k\geq 0$ and $x+2^k < x'$ define
\begin{align*}
\Xi^{(n),R}_{x,x',y,k,z}:=
\Big\{\gamma':&\gamma'(0)\in\ell_{xn,yW_n,(y+1)W_n},\gamma'(1)\in\ell_{x'n,(y-\lceil (x'-x)^{{89/100}}\rceil)W_n,(y+\lceil (x'-x)^{{89/100}}\rceil)W_n},\\ &\sup\{|w|:((x+2^k)n,yW_n+w) \in \gamma'\}\geq z2^{9k/10}W_n\Big\}.
\end{align*}
\begin{center}
\begin{figure}
\includegraphics[width=5in]{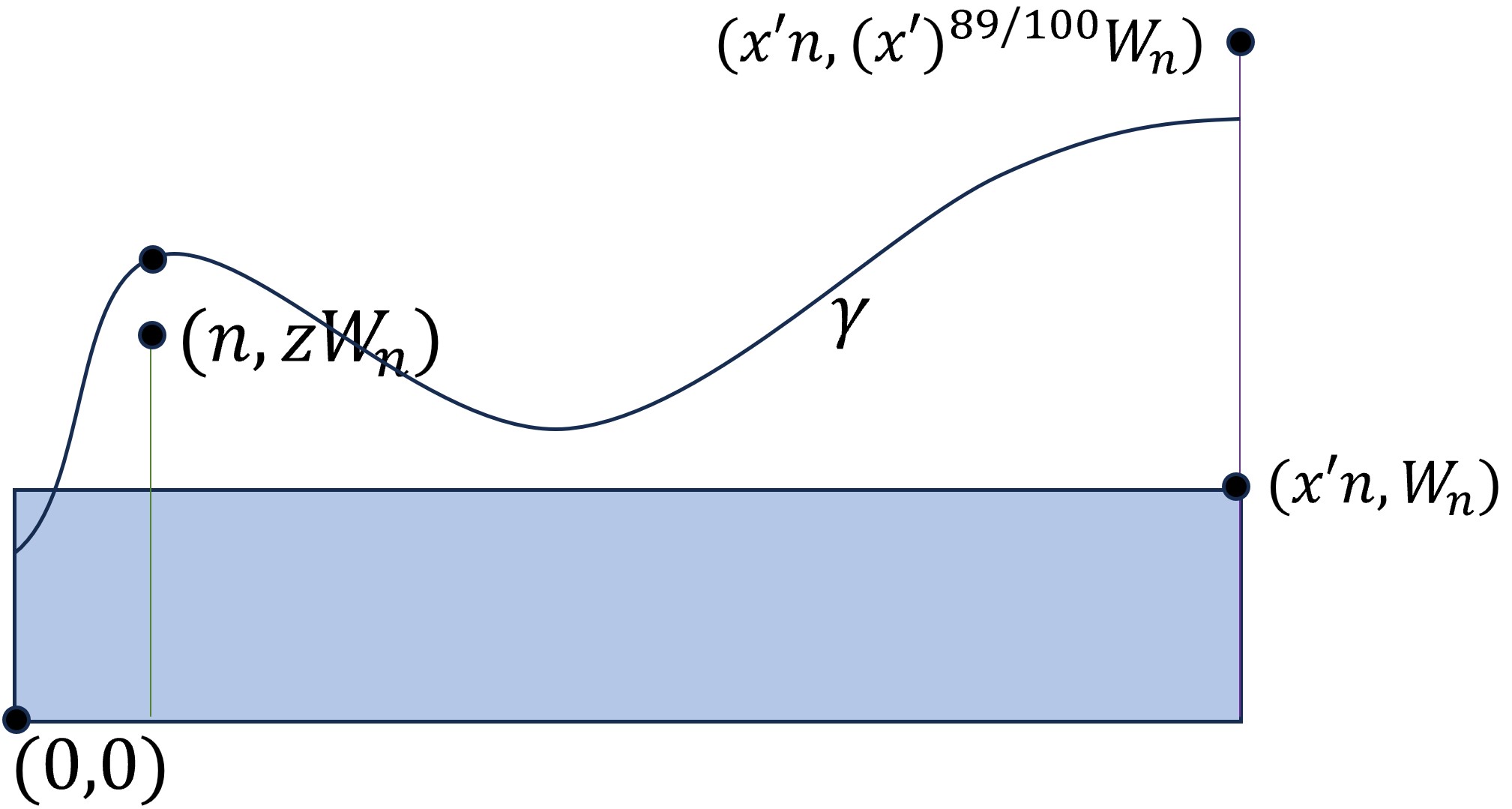}\label{f:xi.local.trans}
\caption{The set of paths $\Xi^{(n),R}_{0,x',0,0,z}$ are those that intersect the line $y=n$ above $zW_n$ or below $-zW_n$. Lemma \ref{l:local.trans.proof} shows that for large $z$, on the large probability event $\cL^{(n),R}_{0,x',0,z}$ all such paths have length significantly larger (at scale $Q_{n}$) than the first passage time between its endpoints. }
\end{figure}
\end{center}
This is illustrated in Figure~\ref{f:xi.local.trans}.  Similarly if $x'+2^k < x$ define
\begin{align*}
\Xi^{(n),L}_{x,x',y,k,z}:=
\Big\{\gamma':&\gamma'(0)\in\ell_{xn,yW_n,(y+1)W_n},\gamma'(1)\in\ell_{x'n,(y-\lceil (x-x')^{89/100}\rceil)W_n,(y+\lceil (x-x')^{89/100}\rceil)W_n},\\ &\sup\{|w|:((x-2^k)n,yW_n+w) \in \gamma'\}\geq z2^{9k/10}W_n\Big\}.
\end{align*}

For $x'>x$, $y\in \Z$ and $z\ge 0$, define the event
\begin{align*}
\cL^{(n),R}_{x,x',y,z}&=\bigcap_{j=0}^{j_{\max}} \bigcap_{w=-z2^jW_n/W_{n2^{j}}}^{z2^jW_n/W_{n2^{j}}} \cH^{(2^j n)}_{x2^{-j}-1,yW_n/W_{n2^{j}}+w,\frac{z2^{j/20}}{10000} } \\
&\qquad  \bigcap \bigcap_{w=-z2^{j_{\max}}W_n/W_{n2^{j_{\max}}}}^{z2^{j_{\max}}W_n/W_{n2^{j_{\max}}}} \cH^{((x'-x-2^{j_{\max}}) n)}_{\frac{x+2^{j_{\max}}}{x'-x-2^{j_{\max}}},yW_n/W_{(x'-x-2^{j_{\max}})n}+w,\frac{z2^{j_{\max}/20}}{10000} } \\
&\qquad \bigcap_{j=0}^{j_{\max}} \cS^{(n2^j)}_{z,(xn,yW_n)}
\end{align*}
where $j_{\max} = \lfloor \log_2 (x'-x)-1\rfloor$.
Similarly, for $x'<x$, define 
\begin{align*}
\cL^{(n),L}_{x,x',y,z}&=\bigcap_{j=0}^{j_{\max}} \bigcap_{w=-z2^jW_n/W_{n2^{j}}}^{z2^jW_n/W_{n2^{j}}} \cH^{(2^j n)}_{x2^{-j},yW_n/W_{n2^{j}}+w,\frac{z2^{j/20}}{10000} } \\
&\qquad  \bigcap \bigcap_{w=-z2^{j_{\max}}W_n/W_{n2^{j_{\max}}}}^{z2^{j_{\max}}W_n/W_{n2^{j_{\max}}}} \cH^{((x-2^{j_{\max}}-x') n)}_{\frac{x'}{x-2^{j_{\max}}-x'},yW_n/W_{(x-2^{j_{\max}}-x')n}+w,\frac{z2^{j_{\max}/20}}{10000} } \\
&\qquad \bigcap_{j=0}^{j_{\max}} \cS^{(n2^j)}_{z,(xn,yW_n)}
\end{align*}
where $j_{\max} = \lfloor \log_2 (x-x')-1\rfloor$.

\begin{lemma}\label{l:cL.bound}
There exists $D>0$ such that for all $x,y,y'$ and $z>0$, 
\[
\P[\cL^{(n),R}_{x,x',y,z}] \geq 1-\exp(1-Dz^{\frac12 \epsilon\theta}).
\]
\end{lemma}
\begin{proof}
By Lemmas~\ref{l:paraBounds} and~\ref{l:trans.SOGam},
\begin{align*}
\P[(\cL^{(n),R}_{x,x',y,z})^c]&=\sum_{j=0}^{j_{\max}} \sum_{w=-z2^jW_n/W_{n2^{j}}}^{z2^jW_n/W_{n2^{j}}} \exp(1-C(\frac{z2^{j/20}}{10000})^\theta) \\
&\qquad  + \sum_{w=-z2^{j_{\max}}W_n/W_{n2^{j_{\max}}}}^{z2^{j_{\max}}W_n/W_{n2^{j_{\max}}}} \exp(1-C(\frac{z2^{j_{\max}/20}}{10000} )^\theta) \\
&\qquad +\sum_{j=0}^{j_{\max}} \exp(1-C(zn2^j)^{\frac12\epsilon\theta})\\
&\leq \sum_{j=0}^{j_{\max}} z4^j \exp(1-C(\frac{z2^{j/20}}{10000})^\theta) + z4^j \exp(1-C(\frac{z2^{j/20}}{10000})^\theta) +\sum_{j=0}^{j_{\max}} \exp(1-C(zn2^j)^{\frac12\epsilon\theta})\\
&\leq \exp(1-Dz^{\frac12 \epsilon\theta}).
\end{align*}
\end{proof}

Observe that, by reflection symmetry the same bound holds for $\P[\cL^{(n),L}_{x,x',y,z}]$; we shall not state this separately. 

\begin{center}
\begin{figure}
\includegraphics[width=5in]{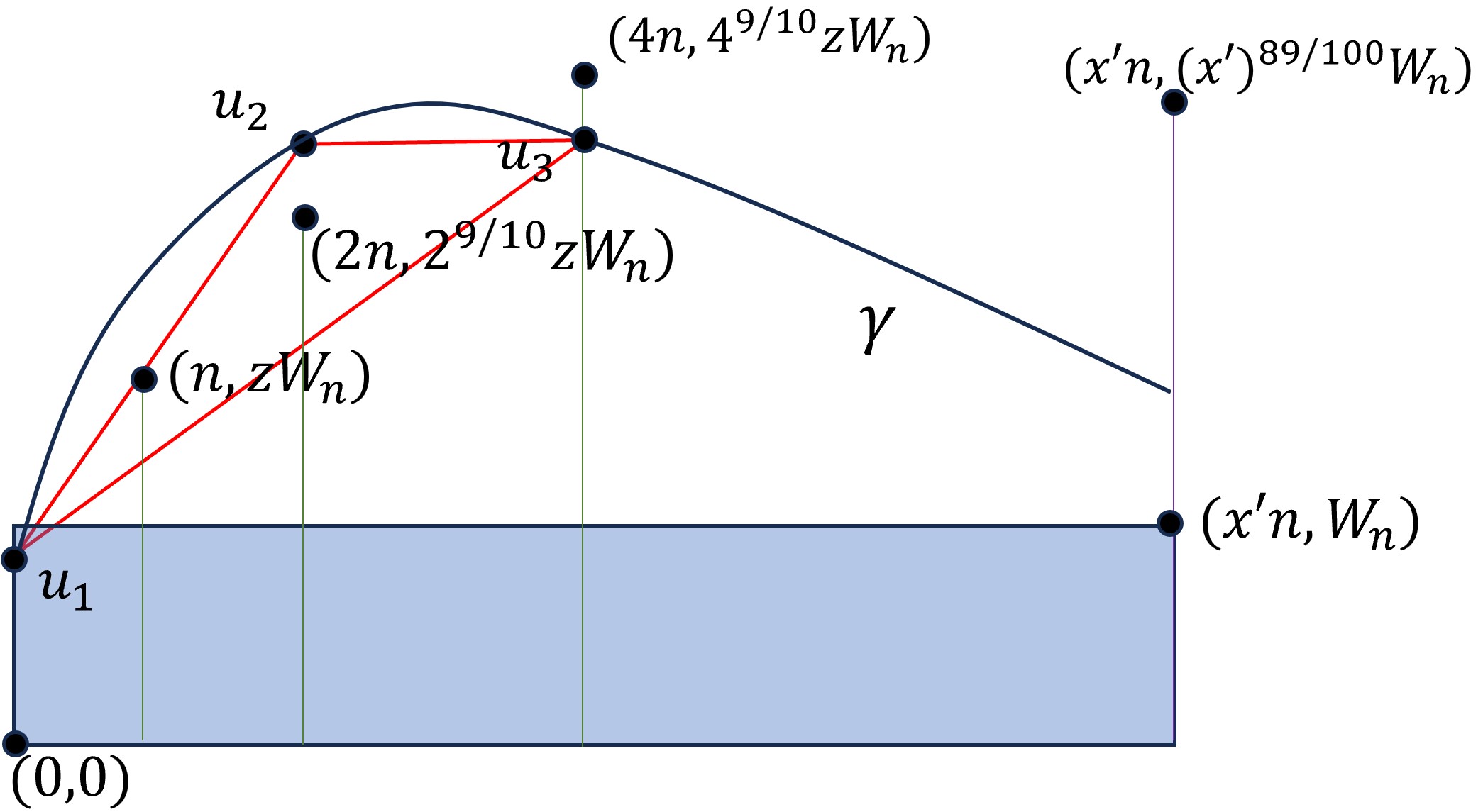}
\caption{The chaining argument in the proof of Lemma \ref{l:local.trans.proof}: the vertical lines represent intervals $\gamma$ must pass through to satisfy $A_j$ for different dyadic scales.  Here $A_j$ fails for $j=3$ at $i=3$.  The extra length required to traverse from $u_1$ to $u_3$ via $u_2$ rather than directly is used to rule out such paths under the event $\cL^{(n),R}_{x,x',y,z}$.}
\label{f:chaining.local.trans}
\end{figure}
\end{center}

\begin{lemma}\label{l:local.trans.proof}
There exists an absolute constant $z_0$ such that for all $z\geq z_0$, on the event $\cL^{(n),R}_{x,x',y,z}$ we have that for all $k\in [1,\lfloor \log_2 (x'-x)-1\rfloor]$,
\begin{equation}\label{eq:local.trans.proof}
\inf_{\gamma' \in \Xi^{(n),R}_{x,x',y,k,z}} X_{\gamma'}-X_{\gamma'(0),\gamma'(1)} \geq \frac{{z^{1/3}}\mu}{4000} Q_{n}.
\end{equation}
\end{lemma}

\begin{proof}First assume without loss of generality that $x=y=0$.

{\bf Case 1:} $zW_n \leq n^{1-10\epsilon}$.

Suppose that $\gamma' \in \Xi^{(n),R}_{0,x',0,k}$ and $X_{\gamma'}-X_{\gamma'(0),\gamma'(1)} \leq \frac{z\mu}{4000} Q_{n}$.  Let $j$ be the maximal integer in $[k,j_{\max}]$ such that
\[
\sup\{|w|:(2^jn,w) \in \gamma'\}\geq z2^{9j/10}W_n.
\]
By definition of $\Xi^{(n),R}_{x,x',y,k,z}$ we have $j\geq k$.  Let us consider the case $j<j_{\max}$.  Then set $u_1=\gamma'(0)$ and let $u_2=(2^j n,w_2)\in \gamma'$ be intersection of $\gamma'$ with the line $x=2^j n$ that maximize $|w_2|$.  And set $u_3=(2^{j+1}n,w_3)$ as the last point of intersection of $\gamma'$ with the line $x=2^{j+1} n$.  See Figure~\ref{f:chaining.local.trans}.

Let us denote $f =w_2/W_n$ and $g=w_3/W_n$.  
Note by construction of $j$ we have that $|f|\geq z2^{9j/10}$ and $|g|\leq z2^{9(j+1)/10}$ and so $|f|-|g|/2 \geq 2^{-4}|f|$.  First let us consider the case that $|f|\leq (2^j n)^{1-10\epsilon}/W_n$.
Then by ~\eqref{eq:Pythag2},
\begin{align*}
&|u_1-u_2|+|u_2-u_3| -|u_1-u_3| \\
&\quad\geq\sqrt{(n2^j)^2 + ((|f|-1)W_n)^2} + \sqrt{(n2^j)^2 + ((f-g)W_n)^2}\\
&\quad\qquad - \sqrt{(n2^{j+1})^2 + (g W_n)^2}\\
&\quad\geq\frac{((|f|-1)W_n)^2+((f-g)W_n)^2 - \frac12(g W_n)^2}{n2^{j+1}} - o(2^{-j}|f|^2 Q_n)\\
&\quad\geq 2^{-(j+1)}((|f|-|g|/2)^2-2|f|)Q_n - o(2^{-j}|f|^2 Q_n)\\
&\quad\geq 2^{-(j+10)}|f|^2Q_n .
\end{align*}

Now 
\begin{align*}
X_{\gamma'}-X_{v_1 v_2} &\geq X_{u_1 u_2 } + X_{u_2 u_3} - X_{u1 u_3} - 3\Gamma_{n^2}\\
&\geq \mu(|u_1-u_2|+|u_2-u_3| -|u_1-u_3|) - 3\Gamma_{n^2}\\
&\qquad - \Big| X_{u_1 u_2 } -\mu|u_1-u_2|\Big|- \Big| X_{u_2 u_3 } -\mu|u_2-u_3|\Big|- \Big| X_{u_1 u_3 } -\mu|u_1-u_3|\Big|\\
&\geq \mu 2^{-(j+10)}|f|^2Q_n - 3z 2^{(j+1)/20}Q_{2^{j+1}n}- 3\Gamma_{n^2}\\
&\qquad- \frac{\theta^2}{1000}\Big(\frac{(2|f|+|g|)W_n}{W_{2^j n}} Q_{2^j n} + \frac{|g|W_n}{W_{2^{j+1} n}} Q_{2^{j+1} n}\Big)\\
&\geq \mu 2^{-(j+10)}|f|^2Q_n - \bigg(3\sqrt{D_5}z 2^{(j+1)4/5}+ \frac{5\sqrt{D_5}\theta^2|f|2^{-j/8}}{1000}\bigg)Q_{n}- 3\Gamma_{n2^j}.
\end{align*}
where the first inequality is by the definition of $\Gamma$, the second by the triangle inequality, the third by $\cL$ and the  fourth inequality follows by Lemma~\ref{l:growth34}.  For large enough $z\geq z_0$ we have that
\begin{align*}
\mu 2^{-(j+12)}|f|^2Q_n &\geq \mu z^2 2^{4j/5-12} Q_n \geq 3\sqrt{D_5}z 2^{\tfrac45(j+1)} Q_{n}- 3\Gamma_{n2^j},\\
\mu 2^{-(j+12)}|f|^2Q_n &\geq \mu z 2^{-(j/10+12)}|f|Q_n \geq \frac{5\sqrt{D_5}\theta^2|f|2^{-j/8}}{1000}Q_{n}.
\end{align*}
Hence combining the last two equations we have that
\begin{align*}
X_{\gamma'}-X_{v_1 v_2} &\geq  \mu 2^{-(j+11)}|f|^2Q_n\geq 2^{4j/5-11}z^2 Q_n \geq \frac{z\mu}{4000}Q_n,
\end{align*}
for all large enough $z$ and $n$ which contradicts the assumption.  

Next suppose that $|f|\geq (2^j n)^{1-10\epsilon}/W_n$.  Then
\begin{align*}
|u_1-u_2|+|u_2-u_3| -|u_1-u_3| &\geq 2^{-(j+10)}((2^j n)^{1-10\epsilon}/W_n)^2 Q_n\\
&\geq C z n^{1-10\epsilon-\frac34} 2^{j-10} Q_n\\
&\geq z 2^{j/5} \mu (2^jn)^\epsilon Q_{3n2^j}.
\end{align*}
Let $u_2'$ be the first point on the path $\gamma'$ from $u_1$ to $u_3$ such that 
\[
|u_1-u_2'|+|u_2'-u_3| -|u_1-u_3| \geq z 2^{j/5} \mu (2^jn)^\epsilon Q_{3n2^j}.
\]
We have that $u_1,u_3\in B_{n2^{j+1}}(\origin)$ so $u_2'\in B_{3n2^{j}}(\origin)$. 
Then by $\cS^{(n2^j)}_{z,(0,0)}$,
\begin{align*}
X_{\gamma'}-X_{v_1 v_2} &\geq X_{u_1 u_2' } + X_{u_2' u_3} - X_{u1 u_3} - 3\Gamma_{n 2^j}\\
&\geq \mu(|u_1-u_2|+|u_2-u_3| -|u_1-u_3|) - \frac3{10}\mu (n2^j)^\epsilon Q_{3n2^j} - 3\Gamma_{n 2^j}\\
&\geq \frac{z\mu}{4000} Q_{n}.
\end{align*}
The case for $j=j_\star$ follows similarly but with $u_3=v_2$.

{\bf Case 2:} $zW_n > n^{1-10\epsilon}$.

Define $j_*=j_{\max}\wedge \max\{j: zW_n \leq 2^{j-5} n\}$ and set
\[
a_j=\begin{cases}
zW_n - j(zW_n)^{9/10} &0\leq j < j_*,\\
n 2^{j_* -5 + \frac9{10}(j-j^*)} &j_* \leq j\leq j_{\max}.
\end{cases}
\]
Note that $\frac12 zW_n \leq a_j\leq z2^{9j/10}W_n$.  Let $j$ be the maximal integer in $[k,j_{\max}]$ such that
\[
\sup\{|w|:(2^j n,w) \in \gamma'\}\geq a_j.
\]
If $j<j_*$ then set $u_1=\gamma'(0)$ and $u_2=(2^j n,w_2),u_3=(2^{j+1} n,w_3)$ similarly to the first case.  We will assume that $w_2>0$, the case $w_2<0$ follows similarly.  Note that by construction in this case $a_j-a_{j+1} \geq (zW_n)^{9/10}$.  Then
\begin{align*}
|u_1-u_2|+|u_2-u_3| -|u_1-u_3| 
&\quad\geq |u_2|+|u_2-u_3| -|u_3|-2W_n\\
&\quad = |(2^j n , w_2)|+|(2^j n , w_2-w_3)| -|(2^{j+1} n , w_3)|-2W_n\\
&\quad\geq |(2^j n , a_j)|+|(2^j n , a_j-a_{j+1})| -|(2^{j+1} n , a_{j+1})|-2W_n\\
&\quad\geq |(2^{j+1} n , 2 a_j-a_{j+1})| -|(2^{j+1} n , a_{j+1})|-2W_n\\
&\quad\geq \frac1{10} |a_j-a_{j+1}| -2W_n \geq \frac1{40} (zW_n)^{9/10}.
\end{align*}
where the first inequality used the fact that $|u_1|\leq W_n$, the second that the expression is increasing in $w_2$ and decreasing in $w_3$, the third by the triangle inequality.  The third inequality follows by the fact that $a_{j+1} \geq 2^{j-4} n$ and
\[
\inf_{w\geq 2^{j-4} n} \frac{d}{dw} \sqrt{(2^{j+1} n )^2 + w^2} \geq \frac1{40}
\]
we have that
\[
|(2^{j+1} n , 2 a_j-a_{j+1})| -|(2^{j+1} n , a_{j+1})|-2W_n \geq \frac{|a_j-a_{j+1}|}{20}  -2W_n \geq \frac1{40} (zW_n)^{9/10}
\]
and the last from the fact that  $\frac1{40} (zW_n)^{9/10}\geq \frac1{40}  n^{\frac9{10}(1-10\epsilon)} \geq 2W_n$ for large $n$.  It may be that $u_2$ is outside $B_{3zn}(\origin)$ so let $u_2'$ be the first point on $\gamma'$ such that
\[
|u_1-u'_2|+|u'_2-u_3| -|u_1-u_3| \geq  \frac1{40} (zW_n)^{9/10}.
\]
Any point $u'$ on the boundary of $B_{3zn}(\mathbf{0})$ would satisfy the above equation so we have that $u_2'\in B_{3zn}(\mathbf{0})$ and by construction comes before $u_3$.  Then, since $Q_{n2^j} \leq \sqrt{n2^j} \leq 2^{-5}(zW_n)^{1/2}$
\begin{align*}
X_{\gamma'}-X_{v_1 v_2} &\geq X_{u_1 u_2' } + X_{u_2' u_3} - X_{u1 u_3} - 3\Gamma_{n 2^j}\\
&\geq \mu(|u_1-u'_2|+|u'_2-u_3| -|u_1-u_3|) -\frac6{10} \mu z^{1/10}Q_{n2^j}\\
&\geq \frac{\mu}{50} (zW_n)^{9/10}\geq \frac{z^{1/3}\mu}{4000} Q_{n}.
\end{align*}
Suppose instead that $j\in[j_*,j_{\max})$.  Then with $u_i$ as before, using~\eqref{eq:Pythag}
\begin{align*}
&|u_1-u_2|+|u_2-u_3| -|u_1-u_3| \\
&\quad\geq |(2^j n , a_j)|+|(2^j n , a_j-a_{j+1})| -|(2^{j+1} n , a_{j+1})|-2W_n\\
&\quad\geq \frac{a_j^2}{2^{j+1} n} -\frac18 \frac{a_j^4}{(2^{j+1} n)^3} + \frac{(a_{j+1} -a_j)^2}{2^{j+1} n} -\frac18 \frac{(a_{j+1} -a_j)^4}{(2^{j+1} n)^3} - \frac{a_{j+1}^2}{2^{j+2} n}-2W_n\\
&\quad \geq \frac{(2a_{j} -a_{j+1})^2}{2^{j+2} n}-\frac14 \frac{a_j^4}{(2^{j+1} n)^3}-2W_n\\
&\quad \geq \frac{(2-2^{9/10})^2a_j^2}{2^{j+2} n}-\frac14 \frac{a_j^2 2^{2j-10}n^{2}}{(2^{j+1} n)^3}-2W_n\\
&\quad \geq \frac{a_j^2}{2^{j+5} n}-\frac{a_j^2 }{2^{3j+13} n}-2W_n\\
&\quad \geq \frac{a_j^2}{2^{j+6} n}-2W_n \geq \frac{n^2 2^{2j_*-10+\frac{9}{5}(j-j_*)}}{2^{j+6} n}-2W_n \geq 2^{-20} (n2^j)^{4/5}.
\end{align*}
We choose 
$u_2'$ be the first point on $\gamma'$ such that
\[
|u_1-u_2|+|u'_2-u_3| -|u_1-u_3| \geq 2^{-20} (n2^j)^{4/5}.
\]
and have that 
\begin{align*}
X_{\gamma'}-X_{v_1 v_2} &\geq X_{u_1 u_2' } + X_{u_2' u_3} - X_{u1 u_3} - 3\Gamma_{n 2^j}\\
&\geq \mu(|u_1-u_2|+|u_2-u_3| -|u_1-u_3|)  - \frac6{10} \mu z^{1/10}Q_{n2^j}\\
&\geq  \mu 2^{-21} (n2^j)^{4/5}\\
&\geq  \mu 2^{-21} z^{5/15}W_n^{5/15}(n2^j)^{7/15}\\
&\geq  \mu 2^{-21} z^{1/3}W_n^{4/5}\geq  \frac{z^{1/3}\mu}{4000} Q_{n}.
\end{align*}
The case for $j=j_{\max}$ follows similarly but with $u_3=v_2$.  This establishes~\eqref{eq:local.trans.proof} in all the cases.
\end{proof}

Combining Lemmas~\ref{l:cL.bound} and~\ref{l:local.trans.proof} we have the following corollary.
\begin{corollary}\label{c:local.trans}
There exist absolute constants $D,z_0$ such that for all $z\geq z_0$ and all $k\in [1,\lfloor \log_2 (x-x')-1\rfloor]$,
\begin{equation}\label{eq:local.trans.proof.corol}
\P\Big[\inf_{\gamma' \in \Xi^{(n),R}_{x,x',y,k,z}} X_{\gamma'}-X_{\gamma'(0),\gamma'(1)} \geq \frac{{z^{1/3}}\mu}{4000} Q_{n}\Big] \leq \exp(1-Dz^{\frac12\epsilon\theta}).
\end{equation}
\end{corollary}
By reflection symmetry an identical bound holds for paths in $\Xi^{(n),L}_{x,x',y,k,z}$.

\section{Left tail lower bound}
\label{s:left}

The purpose of this section is to prove that for quasi-record points $n$, $X_{n}-n\mu$ has positive mass arbitrarily far in the left tail at scale $Q_{n}$. Recall \eqref{e:alpharange}; throughout this section we shall work with $\alpha'\in (\alpha, \alpha_{*})$ such that all our estimates are valid for this $\alpha'$ (i.e., when $Q_{n}=\sup_{1\le m\le n} (n/m)^{\alpha'}\hat{Q}_{m}$). We have the following proposition.

\begin{proposition}
\label{p:left1}
Let $\alpha'\in (\alpha, \alpha_*)$ be fixed. For each $L,C>0$, there exists $\delta>0$ such that for all $(C,\alpha')$-quasi record point $n$ sufficiently large we have 
$$\P(X_{n}<n\mu-LQ_{n})\ge \delta.$$
\end{proposition}

We shall actually prove a slightly stronger result. 
{\begin{proposition}
\label{p:left}
Let $M\geq 1$ and $\alpha'\in (\alpha, \alpha_*)$ be fixed. For each $L,C>0$, there exists $\delta>0$ such that for all $(C,\alpha')$-record point $n$ sufficiently large and $n'\in[{\frac12}n,Mn]$ we have 
$$\P(X_{n'}<n'\mu-LQ_{n'})\ge \delta.$$
\end{proposition}}

First we shall need the following lemma which shows that with positive probability there is some mass on the left tail.  

\begin{lemma}
\label{l:leftweak}
There exists $\varepsilon>0$ such that for all $\alpha$-record points $m$ sufficiently large we have 
$$\P(X_{m}<m\mu-\varepsilon Q_{m})\ge \varepsilon.$$
\end{lemma}

We shall first prove Proposition \ref{p:left} assuming Lemma \ref{l:leftweak}. 

\begin{proof}[Proof of Proposition \ref{p:left}]
We only need to prove this proposition for $L$ sufficiently large. Recall from Lemma \ref{l:growth34} that for $m\le n$ sufficiently large, 
$Q_{m}\ge D_5^{-1}(\frac{m}{n})^{3/4}Q_{n}$. Given $L$ from the statement of the proposition and $\varepsilon$ from Lemma \ref{l:leftweak}, define 
$L_1:=(\frac{5D_5L}{\varepsilon})^{4}$. It is easy to see that for any $m\le nL^{-1}$ and $n'\ge \frac{1}{2}n$, we have  $\lfloor \frac{n'}{m}\rfloor \varepsilon Q_{m}\ge 2LQ_{n'}$. Now, increasing the value of $L_1$ if necessary, fix an $\alpha$-record point $m\in [nL_1^{-C_{11}},nL_1^{-1}]$ where $C_{11}$ is as in Lemma \ref{l:grmain}; such a record point exists by Lemma \ref{l:grmain}.

For $n'\in [\frac{1}{2}n,Mn]$, for convenience of notation, let us locally set $K=\lfloor \frac{n'}{m}\rfloor$. Then we have 

$$X_{n'}\le \sum_{j=1}^{K} X_{(j-1)\mathbf{m},j\mathbf{m}}+ X_{K\mathbf{m}, \mathbf{n}'}\le \sum_{j=1}^{K} (X^{\Lambda_{j}}_{(j-1)\mathbf{m},j\mathbf{m}}+Y_{j})+X^{\Lambda_{K+1}}_{K\mathbf{m}, \mathbf{n}'}+Y_{K+1}$$
where $\Lambda_{j}$ denotes the strip $\{(x,y): x\in ((j-1)m,jm]\}$ for $1\le j\le K$, $\Lambda_{K+1}$ denotes the strip $\{(x,y): x\in (Km,n']\}$
$$Y_{j}:=X_{(j-1)\mathbf{m},j\mathbf{m}}- X^{\Lambda_{j}}_{(j-1)\mathbf{m},j\mathbf{m}}$$
for $1\le j \le K$ and 
$$Y_{K+1}:=X_{K\mathbf{m},n'}- X^{\Lambda_{K+1}}_{K\mathbf{m},n'}.$$
By the independence (and identical distribution) of  $X^{\Lambda_{j}}_{(j-1)\mathbf{m},j\mathbf{m}}$ and Lemma \ref{l:leftweak} (together with our choice of $m$) implies 
$$\P\left(\sum_{j=1}^{K\rfloor} X^{\Lambda_{j}}_{(j-1)\mathbf{m},j\mathbf{m}}\le \lfloor K\rfloor m\mu-2LQ_{n'}
\right)\ge \left(\P(X_{m}\le m\mu-\varepsilon Q_{m}\right)^{\lfloor n/m \rfloor} \ge (\varepsilon)^{ML_1^{C_2}}:=4\delta.$$
Further, 
$$\P(X^{\Lambda_{K+1}}_{K\mathbf{m}, \mathbf{n}'}\le (n'-Km)\mu+\frac{L}{Q_{n'}}/2)\ge \frac{1}{2}$$
by definition of $Q_{\cdot}$, the fact that $Q_{n'}\ge Q_{n'-Km}$ and since $L$ is sufficiently large. 
Therefore, 
$$\P(\sum_{j=1}^{K} (X^{\Lambda_{j}}_{(j-1)\mathbf{m},j\mathbf{m}}+X^{\Lambda_{K+1}}_{K\mathbf{m}, \mathbf{n}'}\le n'\mu-\frac{3L}{2}Q_{n'})\ge 2\delta.$$
Finally, by Assumption \ref{as:resamp1} it follows that 
$$\P\left(\sum_{j=1}^{K+1}Y_{j}\le -LQ_{n'}/2\right)\le (K+1)\exp(1-D_0(LQ_{n'}/2K\log^{1/\kappa}n')^{\kappa})\le \delta$$
for $n$ sufficiently large since $Q_{n'}\ge C(n')^{1/12}$ and $K\le ML_1^{C_{11}}$. 
Combining the above two estimates we get the desired result. 
\end{proof}

\subsection{Thin rectangle estimates}
\label{s:thin}
For the proof of Lemma \ref{l:leftweak}, we need the following lemma about passage times across thin rectangles. Let $Y^{\delta,\pm}_{n}:=Y^{\pm}_{\mathcal{R}_{n,\delta W_{n}}}$ be the passage time across a rectangle or width $n$ and height $\delta W_{n}$. For this proof we shall use the shorthand $\mathcal{R}_{\delta}$ for the rectangle $\mathcal{R}_{n,\delta W_{n}}$.
 We have the following lemma which says that the best and worst passage times across a thin rectangle are close.

\begin{lemma}
\label{l:thin}
Let $\delta>0$ be fixed and sufficiently small. Then for all $n$ sufficiently large we have and all $x\ge \delta^{3\alpha/7}$
$$\P(Y^{\delta,+}_{n}-Y^{\delta,-}_{n}\ge xQ_{n})\le C\left(\exp(-\delta^{-\epsilon\theta/4})+\delta^{-1/2}\exp(-cx^{\theta}\delta^{-3\alpha \theta/7})+\exp(-n^{\kappa/100})\right).$$
In particular, the same estimate holds for $X_{n}-Y^{\delta,-}_{n}$.  
\end{lemma}

\begin{center}
\begin{figure}
\includegraphics[width=5in]{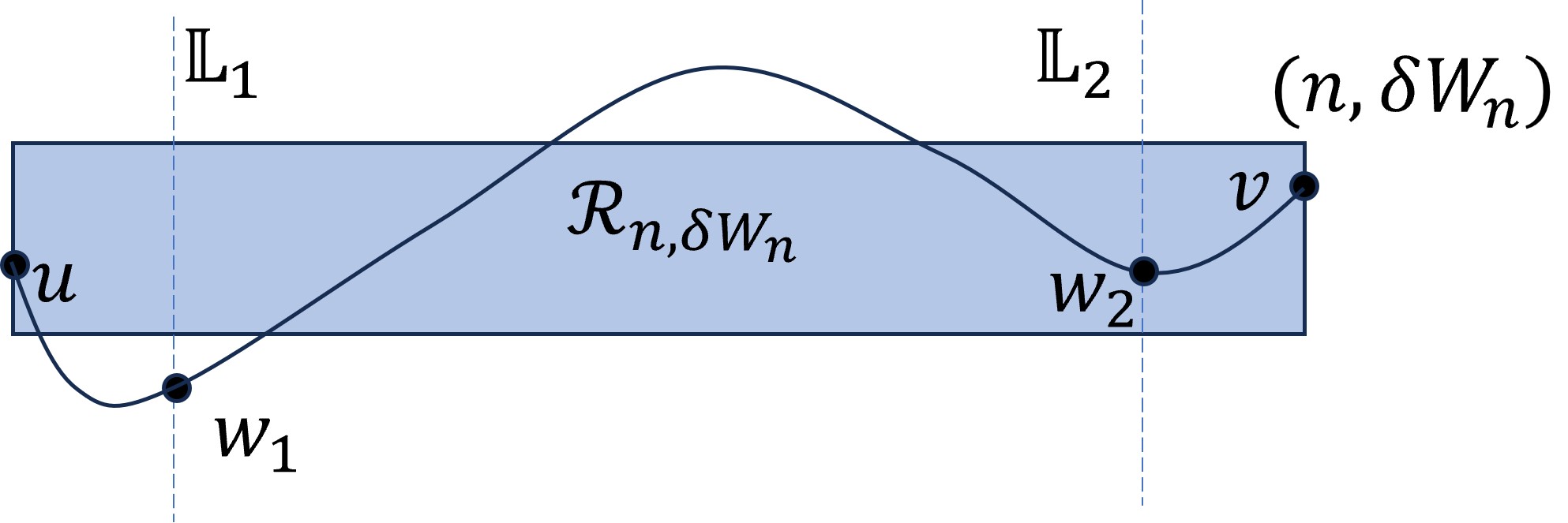}
\caption{Illustration of proof of Lemma~\ref{l:thin}. We control $Y^{\delta, +}_{n}-Y_{n}^{\delta,-}$ by looking at points $w_1$ and $w_{2}$ where the geodesic attaining $Y^{\delta,-}_{n}$ intersects the lines $\mathbb{L}_1$ and $\mathbb{L}_2$ and considering paths from the left side to the right side of the rectangle which are concatenation of geodesic segments up to $w_1$, between $w_1$ and $w_2$ and from $w_2$. }
\label{f:thin.rect}
\end{figure}
\end{center}

\begin{proof}
We shall fix parameters $L=\delta^{-1/2}$ and $\varepsilon=\delta^{3/7}$. Consider now the line segments $\mathbb{L}_1:=\{(\varepsilon n, y): |y|\le LW_{\epsilon n}\}$ and $\mathbb{L}_2:=\{(n-\epsilon n, y): |y|\le LW_{\varepsilon n}\}$. For $u\in L_{\mathcal{R}_{\delta}}$ and $v\in R_{\mathcal{R}_{\delta}}$, Let $w^1_{u,v}$ (resp.\ $w^2_{u,v}$) denote the first (resp.\ last) intersection of $\gamma_{uv}$ with the line $x=\varepsilon n$ (resp.\ $x=n-\varepsilon n$), see Figure~\ref{f:thin.rect}. Consider the event 
\[
\mathcal{A}=\left\{w^1_{u,v}\in \mathbb{L}_1,w^2_{u,v}\in \mathbb{L}_2 \forall u,v\right\}.
\]

Let $\mathcal{B}_1$ denote the event that 
$$\mathcal{B}_1= \left\{\sup_{u,u'\in L_{\mathcal{R}_{\delta}},v\in \mathbb{L}_1}|X_{u,v}-X_{u',v}| \le xQ_{\varepsilon n}/3\right\}$$
and let $\mathcal{B}_2$ denote the event that
$$\mathcal{B}_2= \left\{\sup_{v,v'\in R_{\mathcal{R}_{\delta}},u\in \mathbb{L}_2}|X_{u,v}-X_{u,v'}| \le xQ_{\varepsilon n}/3\right\}.$$
Observe that for $w_1\in \mathbb{L}_1$ and $w_2\in \mathbb{L}_2$, 
$$Y^{\delta,+}_{n}\le \sup_{u\in L_{\mathcal{R}_{\delta}}}X_{u,w_1}+X_{w_1,w_2}+\sup_{v\in R_{\mathcal{R}_{\delta}}}X_{w_2,v}+3\Gamma_{n}$$ and hence on 
$\mathcal{A}\cap \mathcal{B}_1\cap \mathcal{B}_2\cap\{\Gamma_{n}\le xQ_{n\varepsilon}/9\}$ we have 
$$Y^{\delta,+}_{n}\le Y^{\delta,-}_{n}+xQ_{n\varepsilon}.$$
By Lemma \ref{l:growth34} we have that $W_{\varepsilon n} \ge \delta W_{n}$.
Now, by  Corollary~\ref{c:local.trans} we get that 
$$\P(\mathcal{A}^c)\le C\exp(-L^{\frac12\epsilon\theta}).$$
To upper bound $\P(\mathcal{B}^c_1)$ we partition $\mathbb{L}_1$ into $2L$ segments, each of length $W_{\varepsilon n}$. Further, choose $\varepsilon$ such that $\delta W_{n}\ll L^{-1}W_{\varepsilon n}$. By Pythagoras' theorem we know that this implies that for $u, u'\in L_{\mathcal{R}_{\delta}}, v\in \mathbb{L}_1$, $||u-v|-|u'-v||\le Q_{\varepsilon n}$.
Next, choose $n$ sufficiently large such that $L^2\ll Q_{\varepsilon n}^{1/9}$. Now for a fixed sub-segment of $\mathbb{L}_1$ as above, (call it $\mathbb{L}_{1,i}$) using Lemma \ref{l:paraChange}, Lemma \ref{l:YMinusBound} and Lemma \ref{l:YPlusBound}, it follows that 
$$\P\left(\sup_{u,u'\in L_{\mathcal{R}_{\delta}},v\in \mathbb{L}_{1,i}}|X_{u,v}-X_{u',v}| \ge xQ_{\varepsilon n}/3\right)\le C\exp(-cx^{\theta})+\exp(1-c(Q_{\varepsilon n})^{\kappa/10}).$$
By taking a union bound over all sub-segments we get 
$$\P(\mathcal{B}_1^{c})\le CL\exp(-cx^{\theta})+2L\exp(1-c(Q_{\varepsilon n})^{\kappa/10}). $$
As identical argument leads to 
$$\P(\mathcal{B}_2^{c})\le CL\exp(-cx^{\theta})+2L\exp(1-c(Q_{\varepsilon n})^{\kappa/10}). $$
Finally by Lemma \ref{l:Gamma}, we also know that 
$$\P(\Gamma_{n}\ge xQ_{n\varepsilon}/9)\le  \exp(1-(xQ_{\varepsilon n}/\log^{C}n)^{\kappa}).$$
It therefore follows that for $x\ge 1$ and all $n$ sufficiently large
$$\P\left(Y^{\delta,+}_{n}\ge Y_n^{\delta,-}+xQ_{\varepsilon n}\right)\le C\left(\exp(-\delta^{-\epsilon\theta/4})+\delta^{-1/2}e^{-cx^{\theta}}+e^{-n^{\kappa/100}}\right).$$
Since $Q_{\varepsilon n}\le \varepsilon^{-\alpha}Q_{n}$ it follows that for all $x\ge \delta^{3\alpha/7}$ and for all $n$ sufficiently large we have 
$$\P\left(Y^{\delta,+}_{n}\ge Y_n^{\delta,-}+xQ_{n}\right)\le C\left(\exp(-\delta^{-2\theta})+\delta^{-1}\exp(-x^{\theta}\delta^{-3\alpha \theta/7})+\exp(-n^{\kappa/100})\right)$$
as desired. 
\end{proof}

\subsection{Proof of Lemma \ref{l:leftweak}}
We shall prove Lemma \ref{l:leftweak} by contradiction. Let us therefore assume that for each $\varepsilon>0$, there exists an $\alpha$-record point $m$ sufficiently large such that $\P(X_{m}-m\mu \le -\varepsilon Q_{n})\le \varepsilon$ for all $\varepsilon>0$. We first need the following basic bound.

\begin{lemma}
\label{l:rtail}
There exists $\varepsilon_3>0$ such that for all record points $m$ sufficiently large we have $\P(X_{m}-m\mu\ge \varepsilon_3 Q_{m})\ge \varepsilon_3$. 
\end{lemma}

\begin{proof}
Recall from Lemma \ref{l:varbd} that ${\rm Var}(X_{m})\ge CQ^2_{m}$ for some $C>0$ independent of $m$. Now, define $Z_{m}=Q_{m}^{-1}(X_{m}-\E X_{m})$. Since $|Z_{m}|$ has stretched exponential tails uniformly in $m$ (or more specifically $\E |Z_{m}|^{4}$ is bounded in $m$), it follows from the above that there exists $\epsilon>0$ (sufficiently small depending on $C$) independent of $m$ such that
$$\P(|Z_{m}|\ge \epsilon)\ge \epsilon$$
and hence $\E|Z_{m}|\ge \epsilon^2.$ Let $Z_m^+=Z_m\vee 0$ and so $\E Z^{+}_{m}=\frac12\E |Z_{m}|\ge 2^{-1}\epsilon^2$. Again, using that $Z^{+}_{m}$ has stretched exponential tails uniformly in $m$ there exists $\varepsilon_3$ sufficiently small depending on $\epsilon$ such that $\P(Z^{+}_{m}\ge \varepsilon_3)\ge \varepsilon_3$. The proof is completed by noticing that $\E X_{m}\ge m\mu$ by sub-additivity. 
\end{proof}

We now come to the main result leading to a contradiction. 

\begin{lemma}
\label{l:contracorr}
If the conclusion of Lemma \ref{l:leftweak} is false then for any $L_0>0$ there exist $L>L_0$ and $n$ sufficiently large ($n\ge n_0(L)$) with 
$$\P(Y^{-}_{n}\le n\mu-L^{-20}Q_{n})\le L^{-20}.$$
\end{lemma}

Lemma \ref{l:contracorr} will follow immediately from the following lemma. 

\begin{lemma}
\label{l:contra}
Let $L>0$ be a fixed sufficiently large integer. There exists $\varepsilon>0$ and $m_0$ (depending on $L$) such that for all  $\alpha$-record point $m\geq m_0$ that satisfy $\P(X_{m}\le m\mu-\varepsilon Q_m)\le \varepsilon$, we have that 
\[
\P(Y^{-}_{mL}\ge mL\mu+L^{-1000}Q_{mL})\ge 1-L^{-20000}.
\]
\end{lemma}

\begin{proof}
For the proof, we shall write $n=mL$ to reduce the notational overhead. We define parameters $\delta=L^{-100}$, $K=2000/\alpha$ and $\varepsilon=\exp(-L)$. The proof is long and technical so we divide it into several steps for the convenience of the reader. 

\noindent
\textbf{Step 1: Constructions of favourable events.} We shall first construct a number of events on whose intersection, the large probability events in the statement of the lemma will be shown to hold. We first set-up some basic notations. 

For $i=0,1,2,\ldots, L$ and $j\in \Z$, let $L_{i,j,\delta}$ denote the line segments 
$im\times [j\delta W_{m}, (j+2)\delta W_{m}]$ (notice that in our previous notation this would be denoted as $\ell_{im,j\delta W_{m}, (j+2)\delta W_{m}}$, but to simplify notation we are locally going to use the $L_{i,j,\delta}$ notation). We shall also use the following shorthands (these are used locally and are slightly different from the notations defined in Section \ref{s:rectpara}):
$$Z^{-}_{i,j,\delta}:=\inf_{u\in L_{i-1,j,\delta}, v\in L_{i,j,\delta}} X_{uv};$$
$$Z^{-}_{i,j,j',\delta}:=\inf_{u\in L_{i-1,j,\delta}, v\in L_{i,j',\delta}} X_{uv}.$$

\noindent
\textbf{Event $\cA$ for paths across thin rectangles:}
For $j\in \Z$ with $|j|\le \frac{W_{n}}{\delta W_{m}}$, let $\mathcal{A}_{j}$ denote the event that 
$$\min_{u\in L_{0,j,\delta}, v\in L_{L,j,\delta}} X_{u,v}\ge n\mu+ 10^{-5}\delta^{4}Q_{m}.$$
Let $\mathcal{A}:=\cap_{j}\mathcal{A}_{j}$. 

\noindent
\textbf{Parallelogram events $\cB'$ and $\cB''$:}
Let $\cB''$ denote the event 
$$\mathcal{B}'':=\left\{Z^{-}_{i,j,j',\delta^{K}}\ge m\mu- 2\delta^{K\alpha/4-10}Q_{m}~~ \forall i \forall j,j'~\text{with }|j|, |j'|\le L\frac{W_{n}}{\delta^{K} W_{m}}\right\}.$$
%
Consider next the event $\cB'$ given by
$$\mathcal{B}':=\left\{Z^{-}_{i,j,j',\delta^{K}}\ge m\mu+ 10^{-4}Q'~~ \forall i \forall j,j'~\text{with }|j|, |j'|\le L\frac{W_{n}}{\delta^{K} W_{m}}, |j-j'|\ge \frac{1}{10}\delta^{2-K}\right\}$$
where $Q':=\frac{(\delta^2 W_{m})^2}{m}=\delta^4 Q_{m}$.

\noindent
Let $\mathcal{T}$ denote the event that for all $u\in L_{\mathcal{R}_{n,W_{n}}}, v\in R_{\mathcal{R}_{n,W_{n}}}$, $\gamma_{uv}\in 
\Upsilon_{n,u,v,L}$ (recall that this is the class of paths having global transversal fluctuation bounded by $LW_{n}$). Let $\mathcal{T}'$ denote the event  
$\{\Gamma_{n}\le L^{-1}\delta^{K\alpha/4} Q_{m}\}$. 
We shall next prove some consequences of these events. 

\noindent
\textbf{Claim:}  On $\cB''\cap \cB' \cap \mathcal{T}\cap \mathcal{T}'\cap \mathcal{A}$, we have 
$$Y^{-}_{n}\ge n\mu+ 10^{-5}\delta^{4}Q_{m}\ge n\mu+ D_5^{-1}10^{-5}\delta^{4}L^{-1/4}Q_{n}$$
where $D_5$ is as in Lemma \ref{l:growth34} and the final inequality follows from the same lemma. 

By definition of $\cA$, we only need to consider $X_{uv}$ such that there does not exist any $j$ such that $u\in L_{0,j,\delta}$ and $v\in L_{L,j,\delta}$. In that case, it follows by taking $\delta$ small enough such that $L\delta^2 \ll \delta$, there must exist $i,j,j'$ such that for some points $u_1$ and $v_1$ on $\Gamma_{uv}$ with $u_{1}\in L_{i-1,j,\delta^{K}}$, $v_1\in L_{i,j',\delta^{K}}$ with $|j-j'|\ge 
\frac{\delta^{2-K}}{10}$. Therefore the minimum of $X_{uv}$ over all such pairs is lower bounded by 
$$\min_{\{j_{i}\}}\left( \sum_{i=1}^{L} Z^{-}_{i,j_{i-1},j_{i},\delta^K}\right) -L\Gamma_{n}$$
where the minimum is taken over all sequences $\{j_{i}\}_{0\le i\le  L}$ with $|j_{i}|\le \frac{L W_{n}}{\delta^{K}W_{m}}$ such that $|j_{i-1}-j_{i}|\ge \delta^{2-K}/10$ for some $i$. For any such fixed $\{j_{i}\}$ at least one summand above is lower bounded by $m\mu+10^{-4}\delta^{4}Q_{m}$ (by $\cB'$) whereas all other terms are lower bounded by $m\mu-2\delta^{K\alpha/4-10}Q_{m}$ (by $\cB''$). By using the definition of $\mathcal{T}'$, we get that the above minimum is lower bounded by 
$$n\mu+10^{-4}\delta^{4}Q_{m}-2L\delta^{\alpha K/4-10}Q_{m}-L10^{-5}\delta^{5} Q_{m}.$$
Since $\alpha K>100$ and $\delta=L^{-100}$ it follows that $10^{-4}\delta^{4}-2L\delta^{\alpha K/4-10}-L10^{-5}\delta^{5}\ge 10^{-5}\delta^4$ and this completes the proof of the claim.  
 
It remains to bound the probabilities of the events constructed in the above argument. 

\noindent
\textbf{Step 2: Estimating the probabilities of favourable events.}
We shall upper bound the probabilities of the complements of the favourable events defined above. 

\noindent
\textbf{Estimating the probabilities of $\cB'$ and $\cB''$:}
Let us again fix $i,j,j'$ as in the definition of $\cB''$. We shall apply Corollary \ref{c:paraRect} with $n=m$, $W=\delta^{K}W_{m}$, $Q=m^{-1}W^2=\delta^{2K}Q_{m}$ and $k=j'-j$. Clearly these choices satisfy the hypotheses in Corollary \ref{c:paraRect}. By \eqref{eq:paraRectB},
$$\P(Z^{-}_{i,j,j',\delta^{K}}\le m\mu-2\delta^{K\alpha/4-10}Q_{m})\le \P\left(Y^{-}_{\mathcal{R}_{m},W}\le m\mu-\frac{(k^2-8k)Q}{16}-2\delta^{K\alpha/4-10}Q_{m}\right)+\exp(1-CQ^{\kappa/10}).$$
Clearly, $\min_{k\in \Z} \frac{k^2-8k}{16}\ge -4$ and since $Q\ll \delta^{K\alpha/4} Q_{m}$ for small $\delta$, and $Q\gg \delta^{2K}m^{\alpha/2}$, it follows that for $m$ sufficiently large we have 
$$\P(Z^{-}_{i,j,j',\delta^{K}}\le m\mu-2\delta^{K\alpha/4-10}Q_{m})\le \P\left(Y^{-}_{\mathcal{R}_{m},W}\le m\mu-\frac{3}{2}\delta^{K\alpha/4-10}Q_{m}\right)+\delta^{1000+2K}.$$
To bound the first term above we use Lemma \ref{l:thin} with $m$ in place of $m$, $\delta^{K}$ in place of $\delta$ and $x=\delta^{K\alpha/4-10}$. By choosing $\delta$ sufficiently small and $m$ sufficiently large we get 

$$\P\left(Y^{-}_{\mathcal{R}_{m},W}\le m\mu-\frac{3}{2}\delta^{K\alpha/4-10}Q_{m}\right)\le \P(X_{m}\le m\mu-\frac{1}{2}\delta^{K\alpha/4-10} Q_{m})+\delta^{1000+2K}.$$
Since $\varepsilon=e^{-L}\le \min (\delta^{1000+2K},\frac{1}{2}\delta^{K\alpha/4-10})$  and using the hypothesis $\P(X_{m}\le m\mu-\varepsilon Q_{m})\le \varepsilon$ we get that 
$$\P(Z^{-}_{i,j,j',\delta^{K}}\le m\mu-2\delta^{K\alpha/4-10}Q_{m})\le 3\delta^{1000+2K}. $$

It is easy to see that the total number of triplets of $i,j,j'$ in this case is bounded by $L(\frac{LW_{n}}{\delta^{K}W_{m}})^2=O(\delta^{-2K}L^{15/4})$. Taking a union bound with $\delta$ to be sufficiently small we get 
$\P((\cB'')^{c})\le \delta^{900}$.

The bound for $\P((\cB')^{c})$ is very similar. If $|j-j'|\ge \frac{1}{10}\delta^{2-K}$ then for $k=j'-j$ we get $\frac{(k^2-8k)Q}{16}\ge 2\times 10^{-4}\delta^{4}Q_{m}$ and hence in this case, using \eqref{eq:paraRectB} again,
$$\P(Z^{-}_{i,j,j',\delta^{K}}\le m\mu+10^{-4}\delta^{4}Q_{m})\le \P\left(Y^{-}_{\mathcal{R}_{m},W}\le m\mu-10^{-4}\delta^{4}Q_{m}\right)+\exp(1-CQ^{\kappa/10}).$$
Using Lemma \ref{l:thin} again using $K\alpha\ge 100$, for $m$ sufficiently large 
$$\P(Z^{-}_{i,j,j',\delta^{K}}\le m\mu+10^{-4}\delta^{4}Q_{m})\le 2\delta^{1000+2K}.$$
By taking a union bound over all triplets $i,j,j'$ with $|j-j'|\ge \frac{1}{10}\delta^{2-K}$ we get, for $\delta$ sufficiently small that $\P((\cB')^c)\le \delta^{900}$. 

\begin{center}
\begin{figure}
\includegraphics[width=5in]{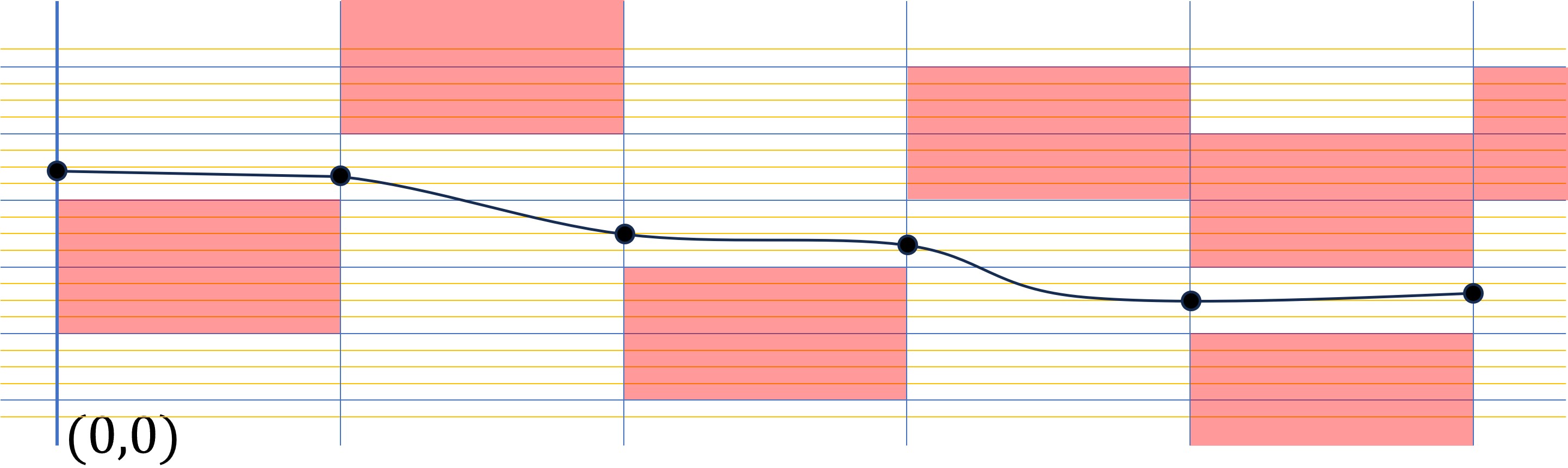}
\caption{The figure displays the grid with blue horizontal lines spaced distance $\delta W_n$ and orange horizontal lines space by $\delta^K W_n$. Red boxes indicate where $Y^{\Lambda_{i+1},-}_{i,j,\delta}\ge m\mu+\frac{\varepsilon_3}{2} Q_{m}$.  With very high probability every row will have at least $\frac{L\epsilon_3}{10}$ red boxes. Any $\gamma$ that avoids side to side crossings of red boxes must have transversal fluctuations across many orange lines and so on $\cB'\cap\cB''$ have an increased passage time.}
\label{f:bad.row}
\end{figure}
\end{center}
\noindent
\textbf{Bounding the probability of $\cA^{c}\cap \cB' \cap \cB''\cap \mathcal{T}\cap  \mathcal{T}'$:} See Figure~\ref{f:bad.row} for illustration of this argument. Let $u\in L_{0,j',\delta}$ and $v\in L_{L,j',\delta}$ be such that $X_{uv}$ is a witness of the event $\cA_{j'}^{c}$. We claim that on $\cB'\cap \cB''\cap \mathcal{T}\cap \mathcal{T}'$ there exists $j\in\{j'-1,j',j'+1\}$  such that the event
$$\widetilde{\cA}^c_{j}:=\left\{\sum_{i=1}^{L} Z^{-}_{i,j,\delta}\le n\mu+10^{-6}\delta^4Q_{m}\right\}$$
holds (here we use the facts that the intervals $L_{i,j,\delta}$ are overlapping in $j$). Indeed, otherwise there will be $i$ such that $\gamma_{uv}$ has a large (in the sense of $\cB'$) vertical deviation from the left to the right of some column $\Lambda_{i}=\{(x,y):(i-1)m \le x\le im\}$ and arguing as in the proof of the claim, we can lower bound $X_{uv}$ on $\cB'\cap \cB''\cap \mathcal{T}\cap \mathcal{T}'$ to show that it cannot be a witness to $\cA^{c}_{j'}$.

It therefore remains to bound $\P(\widetilde{\cA}_j^c)$. Using translation invariance and a union bound we get that
$$\P(\widetilde{\cA}_{j}^{c})\le \P\left(\sum_{i=1}^{L} Z^{\Lambda_{i},-}_{i,j,\delta}\le n\mu+10^{-7}\delta^4 Q_{m}\right)+ L\P\left(|Y^{-}_{0,j,\delta}-Y^{\Lambda_{1},-}_{0,j,\delta}|\ge L^{-1}10^{-7}\delta^4 Q_{m}\right).$$
By Assumption \ref{as:resamp1}, (since $Q_{m}$ grows polynomially in $m$), we have that the second term above is upper bounded by $\delta^{1000}$ for $m$ sufficiently large. As for the first term, notice that the sequence of random variables $Z^{\Lambda_{i},-}_{i,j,\delta}$ are i.i.d.\ and let $N$ denote the number of indices $i$ for which $Z^{\Lambda_{i},-}_{i,j,\delta}\ge m\mu+\frac{\varepsilon_3}{2} Q_{m}$ where $\varepsilon_3$ is as in Lemma \ref{l:rtail}. Using Lemma \ref{l:rtail} and Lemma \ref{l:thin} and choosing $\delta$ sufficiently small it follows that $N$ stochastically dominates a $\mbox{Bin}(L,\frac{\varepsilon_3}{2})$ random variable. 

Let $\cB_{i,j}$ denote the event that $Z^{\Lambda_{i},-}_{i,j,\delta}\ge m\mu- 2\delta^{\alpha/4}Q_{m}$ and set $\cB_{j}:=\cap_{i}\cB_{i,j}$. It follows that, on $\cB_{j}$, 
$$\sum_{i=1}^{L} Z^{\Lambda_{i},-}_{i,j,\delta}\ge n\mu +  
(N\varepsilon_3/2-2L\delta^{\alpha/4})Q_{m}.$$ Now, choose $\delta$ sufficiently small so that $\delta^{\alpha/4}\le \varepsilon^{3/2}/100$ and $10^{-7}\delta^{4}\le L\varepsilon^{2}/100$. By the above discussion, 

\begin{equation}
\label{eq:atbound}
\P(\widetilde{\cA}_{j}^{c})\le \delta^{100}+\P(\cB^c_j)+\P\left(\mbox{Bin}(L,\varepsilon_3/2)\le \frac{L\varepsilon_3}{10}\right).
\end{equation}

By Lemma \ref{l:thin} for $n=m$ and $x=\delta^{\alpha/4}$ it follows that for $\delta$ sufficiently small, 
$$\P(\cB^c_{i,j})\le \delta^{1000}+\P(X_{m}\le m\mu-\delta^{\alpha/4}Q_{m})\le 2\delta^{1000}.$$
By taking a union bound over all $i$,
$\P(\cB^c_{j})\le \delta^{900}$. Also, the last term in \eqref{eq:atbound} is upper bounded (by  a Chernoff bound) by $e^{-cL}$ for some $c>0$ (depending only on $\varepsilon_3$). Since $\delta$ is only a polynomial of $L^{-1}$, for $L$ sufficiently large this is upper bounded by $\delta^{100}$ as well. Combining all these, we get from \eqref{eq:atbound} that $\P(\widetilde{\cA}^{c}_{j})\le \delta^{850}$. Taking a union bound over all $j$, and choosing $\delta$ sufficiently small, we get 
$$\P(\cA^{c}\cap \cB' \cap \cB''\cap \mathcal{T}\cap  \mathcal{T}')\le \delta^{810}.$$

\noindent
\textbf{Probabilities of $\mathcal{T}$ and $\mathcal{T}'$:}
Using Lemma \ref{l:trans.SOGam} and Lemma \ref{l:trans.events} we get $\P(\mathcal{T}^c)\le \delta^{1000}$ for $L$ sufficiently large. By Lemma \ref{l:Gamma}, $\P((\mathcal{T}')^c)\le \delta^{1000}$ for $m$ sufficiently large. 

It follows from the claim and above bounds that
$$\P(Y_{n}^{-}\ge n\mu+L^{-1000}Q_{n}) \ge 1-L^{-20000},$$
completing the proof. 

\end{proof}

Our next step for deriving a contradiction is to extend the estimates from Lemma \ref{l:contracorr} to the case of parallelograms. We define some notations, these will only be locally used so we do not care about overlap with previously defined ones. For $\delta>0$ and $k\in \Z$, let
$A_{\delta,n}$ (resp.\ $B_{\delta,n,k}$) denote the line segment $\{0\}\times [-\delta W_{n},\delta W_{n}]$ (resp.\ $\{n\}\times [(k-1)\delta W_{n}, (k+1)\delta W_{n}]$). Let $Y_{n,k,\delta}^{-}:=\inf_{u\in A_{\delta,n},v\in B_{\delta,n,k}} X_{u,v}$. We have the following lemma. 

\begin{lemma}
\label{l:contrapara}
Suppose $n$ and $L$ are such that the conclusion of Lemma \ref{l:contracorr} holds. Then for $|k|<\frac{nL^{3/4}}{W_{n}}$ we have 
$$\P(Y^{-}_{n,k,L^{-3/4}}\ge n\mu+\frac{1}{2L}Q_{n}) \ge 1-2L^{-20}.$$
\end{lemma}

\begin{proof}
We shall again use Corollary \ref{c:paraRect} with $W=L^{-3/4}W_{n}$; clearly this satisfies the hypotheses for $n$ sufficiently large. By our choice for $W$, and for $Q:=\frac{W^2}{n}=L^{-3/2}Q_{n}$ it follows that $\min_{k\in \Z} \frac{(k^2-8k)Q}{16}\ge -50L^{-3/2}Q_{n}$. From \eqref{eq:paraRectB} it follows that for $L$ sufficiently large 
$$\P(Y^{-}_{n,k,L^{-3/4}}\ge n\mu+\frac{1}{2L}Q_{n}) \ge \P(Y^{-}_{n}\ge n\mu+L^{-1}Q_{n})+\exp(1-CQ^{\kappa/10}).$$
For $n$ sufficiently large, $\exp(1-CQ^{\kappa/10})\le L^{-20}$ and we get the desired result by the hypothesis.  
\end{proof}

We are finally ready to prove Lemma \ref{l:leftweak}. 

\begin{proof}[Proof of Lemma \ref{l:leftweak}]
We argue by contradiction. Assuming the negation of the statement of the lemma, Lemma \ref{l:contracorr} and Lemma \ref{l:contrapara} implies that for arbitrarily large $L$ there exists $n$ sufficiently large for which the conclusion of Lemma \ref{l:contrapara} holds. That is, for $i\in \Z_{\ge 0}$, and $j\in \Z$
denoting (locally) by $L_{i,j}$ the line segments $\{in\}\times [(j-1)L^{-3/4}W_{n},(j+1)L^{-3/4}W_{n}]$ and setting $Z_{i,j,j'}=\inf_{u\in L_{i-1,j},v\in L_{i,j'}}X_{uv}$, we have 
\begin{equation}\label{eq:para.right.tail.contra}
\P(Z_{i,j,j'}\ge n\mu+\frac{1}{2L}Q_{n})\ge 1-2L^{-20}.
\end{equation}
for $|j'-j|\le nL^{3/4}/{W_{n}}$. 
We shall show that under this hypothesis, for $n'=nL^{5}$ we have that 
\begin{equation}
\label{e:meancontra}
\P(X_{n'}\ge n'\mu+L^{1/8}Q_{n'})\ge \frac{1}{2}. 
\end{equation}
Before proving \eqref{e:meancontra}, let us show how this leads to a contradiction. Indeed, since $Q^{-1}_{n'}|X_{n'}-\E X_{n'}|$ has stretched exponential tails it follows that for $L$ sufficiently large $\E X_{n'}\ge n\mu+\frac{1}{2}L^{1/8}Q_{n'}$; this contradicts Lemma~\ref{l:AnBound} for $L$ sufficiently large.

We now come to the proof of \eqref{e:meancontra}. Let $n$ now be sufficiently large so that $L^{5}W_{n}\ll n$ (this can be done because $Q_{n}$ and hence $W_{n}$ grows sub-linearly). Let $\mathcal{T}$ denote the event that $\gamma_{\mathbf{0},\mathbf{n'}}\in \Upsilon_{n',\mathbf{0},\mathbf{n'},L^{5}W_{n}/W_{n'}}$. Denote $\Lambda_{i}:=\{(x,y):x\in [in,(i+1)n]\}$, and let $\cB$ denote the event 
\[
\bigcap_{0\le i< L^{5}}\bigcap_{|j|, |j'| \le L^{23/4}} \bigg\{Z_{i,j,j'}\ge n\mu+\frac{1}{2L}Q_{n}, |Z_{i,j,j'}-Z^{\Lambda_{i}}_{i,j,j'}|\le L^{-100}Q_{n'}\bigg\}
\]
Finally, let $\mathcal{T}'$ denote the event $\{\Gamma_{n'}\le L^{-100}Q_{n'}\}$.
Observing that on $\mathcal{T}\cap\mathcal{T}'\cap\mathcal{B}$,
$$X_{n'}\ge \min_{\{j_i\}} \sum_{i=1}^{L^5}Z_{i,j_{i-1},j_{i}}-L^{5}\Gamma_{n'}\geq L^5(n\mu+\frac{1}{2L}Q_{n} - L^{-100}Q_{n'}) \geq n'\mu+L^{1/8}Q_{n'}$$
where the minimum is taken over all sequences $\{j_{i}\}_{i}$ with $|j_i|\le L^{23/4}$.  It follows that for $L$ sufficiently large
\begin{equation}
\label{eq:sumboud}
\P(X_{n'}< n'\mu+L^{1/8}Q_{n'})\le \P(\cB^c)+\P(\cT^{c})+\P((\cT')^{c}).
\end{equation}

By equation~\eqref{eq:para.right.tail.contra} and a union bound it follows that $\P(\cB^c)\le 2L^{5+46/4-20}\leq L^{-2}$. Since, by Lemma \ref{l:growth34}, 
$$L^{5}W_{n}=L^{5}\sqrt{nQ_{n}}\ge D_5^{-1/2}L^{25/8}\sqrt{nQ_{n'}}\ge D_5^{-1/2}L^{5/8}W_{n'},$$ it follows from Lemma \ref{l:trans.SOGam} and Lemma \ref{l:trans.events} that $\P(\cT^c)\le L^{-100}$ for $L$ sufficiently large. Finally, it follows from Lemma \ref{l:Gamma} that $\P((\mathcal{T}')^{c})\le L^{-100}$ for $n$ sufficiently large.  Combining these estimates with \eqref{eq:sumboud} we establish equation~\eqref{e:meancontra} which completes the result.
\end{proof}

\section{Lower bounding the variance}
\label{s:varlb}
Recall that our strategy is to show that for $M$ sufficiently large we have
\begin{equation}
\label{e:varlb1}
    \hbox{Var}(X_{Mn})\geq C M^{1/10} \hbox{Var}(X_n)
\end{equation}
for all sufficiently large record points $n$, and use it to show \eqref{eq:QequivSD}. This section and the next one is devoted to this. As mentioned at the beginning, our approach will be a block version of the proof of variance lower bound in~\cite{NP95} where we reveal blocks of size $n\times W_n$ one at a time and understand the variance contributed by each one. This is rather more complicated than in~\cite{NP95} where a single edge was revealed at a time.  When resampling an edge if it decreases then all passage times that use that edge will decease. Resampling a block on the other hand may cause some passage times to increase and others to decrease.  We will construct events that guarantee that the passage time from $0$ to $Mn$ have a positive probability of decreasing by a certain amount if $\gamma$ passes through that block. Our objective in this section is to construct such events and show that these events occur at a large fraction of locations along the geodesic $\gamma$ from $(0,0)$ to $(Mn,0)$ with large probability.

We first introduce the basic geometric set-up. Let $R$ be a large integer to be chosen later. We shall fix $M$ large enough depending on $R$ afterwards.  We will consider a block $$\Lambda_{i,j}=[in,(i+1)n]\times [jW_n,(j+1)W_n]$$ and define its enlargements by
\begin{align*}
\Lambda^+_{i,j}&=[(i-2\Theta)n,(i+1+2\Theta)n]\times [(j-2\Theta)W_n,(j+1+2\Theta)W_n],\\
\Lambda^{++}_{i,j}&=[(i-3\Theta)n,(i+1+3\Theta)n]\times [(j-3\Theta)W_n,(j+1+3\Theta)W_n]
\end{align*}
where $\Theta$ is the smallest power of 2 greater than $\lceil\log^{1000/\epsilon^2} M\rceil$. Define the following points in $\Lambda^+_{i,j}$ (see Figure~\ref{f:VarianceDiagram})
\begin{align*}
a_1^\pm &=((i-\Theta)n, (j+1\pm\Theta^{88/100})W_n), \quad a_2^\pm =((i+\Theta)n, (j+1\pm\Theta^{88/100})W_n),\\
h_1^\pm &=((i-1)n, (j\pm 2R) W_n), \quad h_2^\pm =((i+2)n, (j\pm 2R) W_n),\\
s_1^\pm &=(i n, (j\pm R) W_n), \quad s_2^\pm =((i+1)n, (j\pm R) W_n).
\end{align*}
\begin{center}
\begin{figure}
\includegraphics[width=5in]{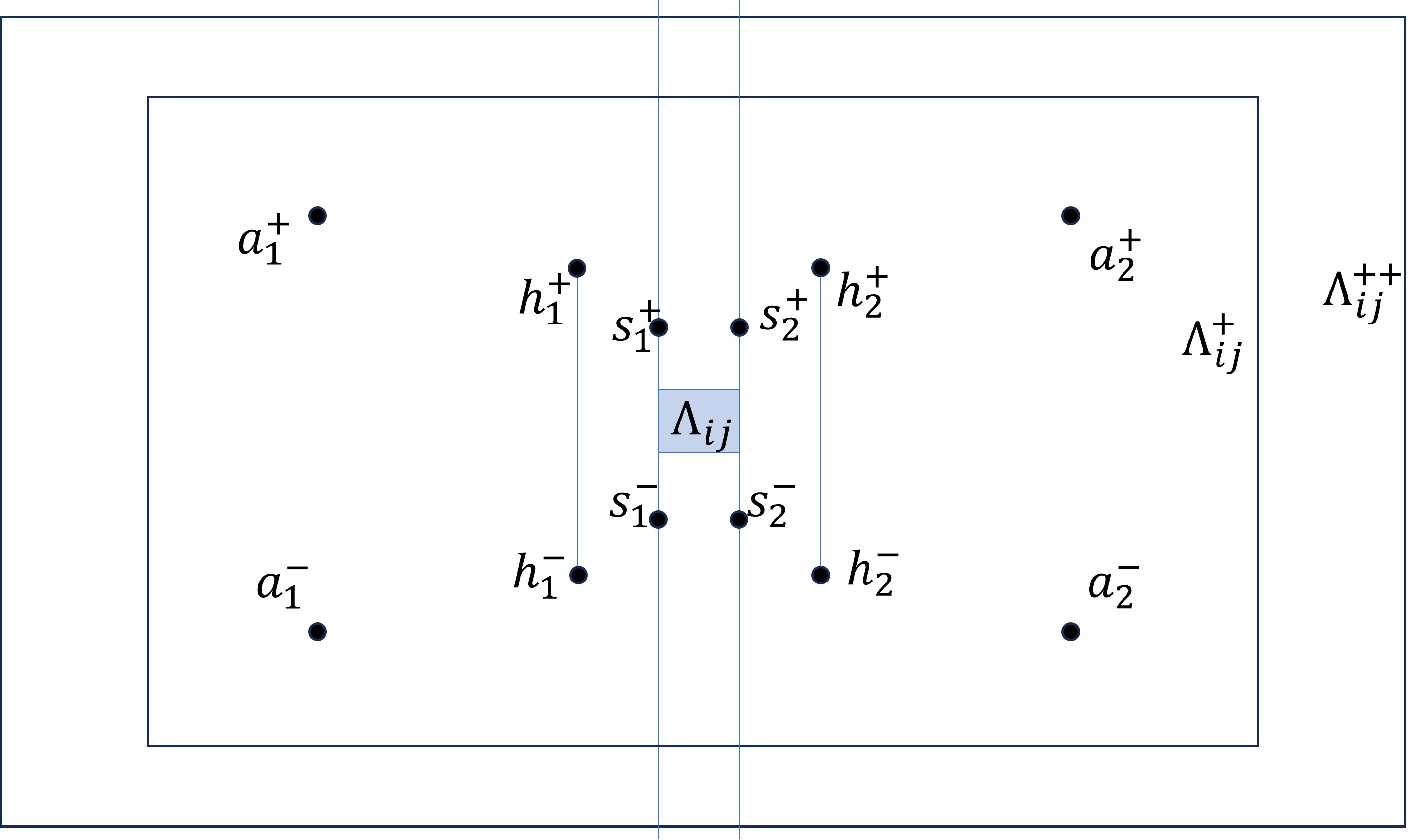}
\caption{Basic construction of points and rectangles around the block $\Lambda_{i,j}$. The points $a^{\pm}_i$ are $\Theta$ many columns (of width $n$) way from $\Lambda_{i,j}$ whereas the points $h^{\pm}_{i}$ are only one column away.}
\label{f:VarianceDiagram}
\end{figure}
\end{center}

We are now ready to define the favourable events. 

\subsection{Favourable events}
Define the first event $\cC^{(1)}_{i,j}$ as 
\begin{align*}
\cC^{(1)}_{i,j}&=\bigg\{\inf_{\substack{v_1\in\ell_{in,(j- R) W_n,(j- R+1) W_n}\\v_2\in\ell_{(i+1)n,(j- R) W_n,(j- R+1) W_n}}} \inf \{X_{\gamma'} - X_{v_1,v_2}:\gamma'\in \Upsilon_{n,v_1,v_2,\frac12 R,(in,(j- R) W_n)}\} \geq Q_n\bigg\}\\
&\quad \cap \bigg\{\inf_{\substack{v_1\in\ell_{in,(j+ R-1) W_n,(j+ R) W_n}\\v_2\in\ell_{(i+1)n,(j+ R-1) W_n,(j+ R) W_n}}} \inf \{X_{\gamma'} - X_{v_1,v_2}:\gamma'\in \Upsilon_{n,v_1,v_2,\frac12 R,(in,(j+ R-1) W_n)}\} \geq Q_n\bigg\}.
\end{align*}

This event essentially stipulates that the transversal fluctuation for geodesics from points around $s_1^{\pm}$ to points around $s_2^{\pm}$ are typical; see Figure \ref{f:VarianceC1}.

\begin{center}
\begin{figure}[t]
\includegraphics[width=2.5in]{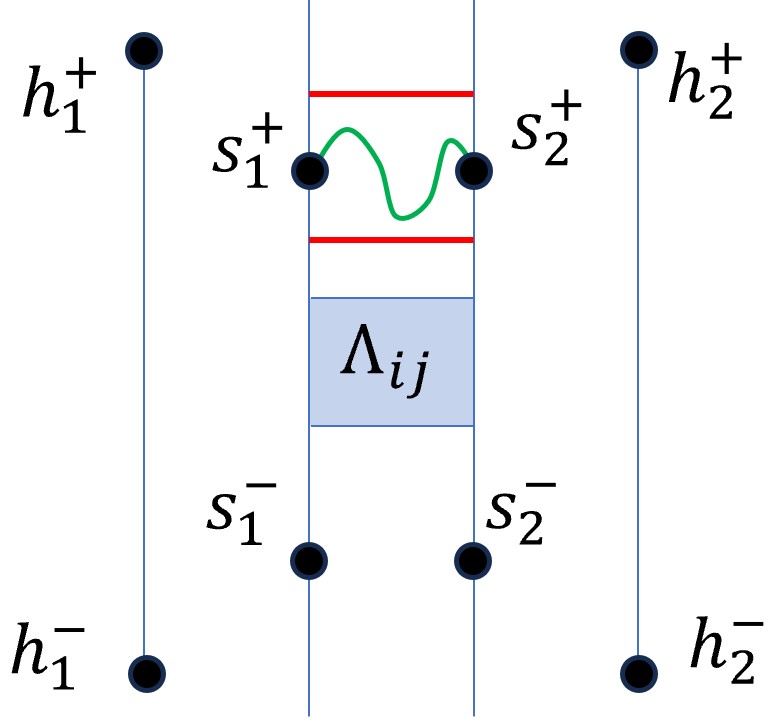}
\caption{Illustration of the event $\cC^{(1)}_{i,j}$: it essentially asks that the geodesics from points near $s^{\pm}_{1}$ to points near $s^{\pm}_2$ have typical transversal fluctuations, in particular these geodesics stay away from $\Lambda_{i,j}$.}
\label{f:VarianceC1}
\end{figure}
\end{center}

Define the event $\cC^{(2)}_{i,j}$ as
\begin{align*}
\cC^{(2)}_{i,j}&=\bigcap_{i'=i}^{i+1}\bigcap_{j'=j-R{-1}}^{j+R{+1}}\bigcap_{k=0}^1\bigg\{\inf_{\gamma' \in \Xi^{(n),R}_{i',i+\Theta,j',k,R/10}} X_{\gamma'}-X_{\gamma'(0),\gamma'(1)} \geq Q_n\bigg\}\\
&\quad \cap \bigcap_{i'=i}^{i+1}\bigcap_{j'=j-R{-1}}^{j+R{+1}}\bigcap_{k=0}^1\bigg\{\inf_{\gamma' \in \Xi^{(n),L}_{i',i-\Theta,j',k,R/10}} X_{\gamma'}-X_{\gamma'(0),\gamma'(1)} \geq Q_n\bigg\}.
\end{align*}

This event asks for control of local transversal fluctuations for geodesics between $s^{\pm}_{1},s_2^{\pm}$ to $a_1^{\pm}, a_2^{\pm}$; see Figure \ref{f:VarianceC2}.

\begin{center}
\begin{figure}
\includegraphics[width=5in]{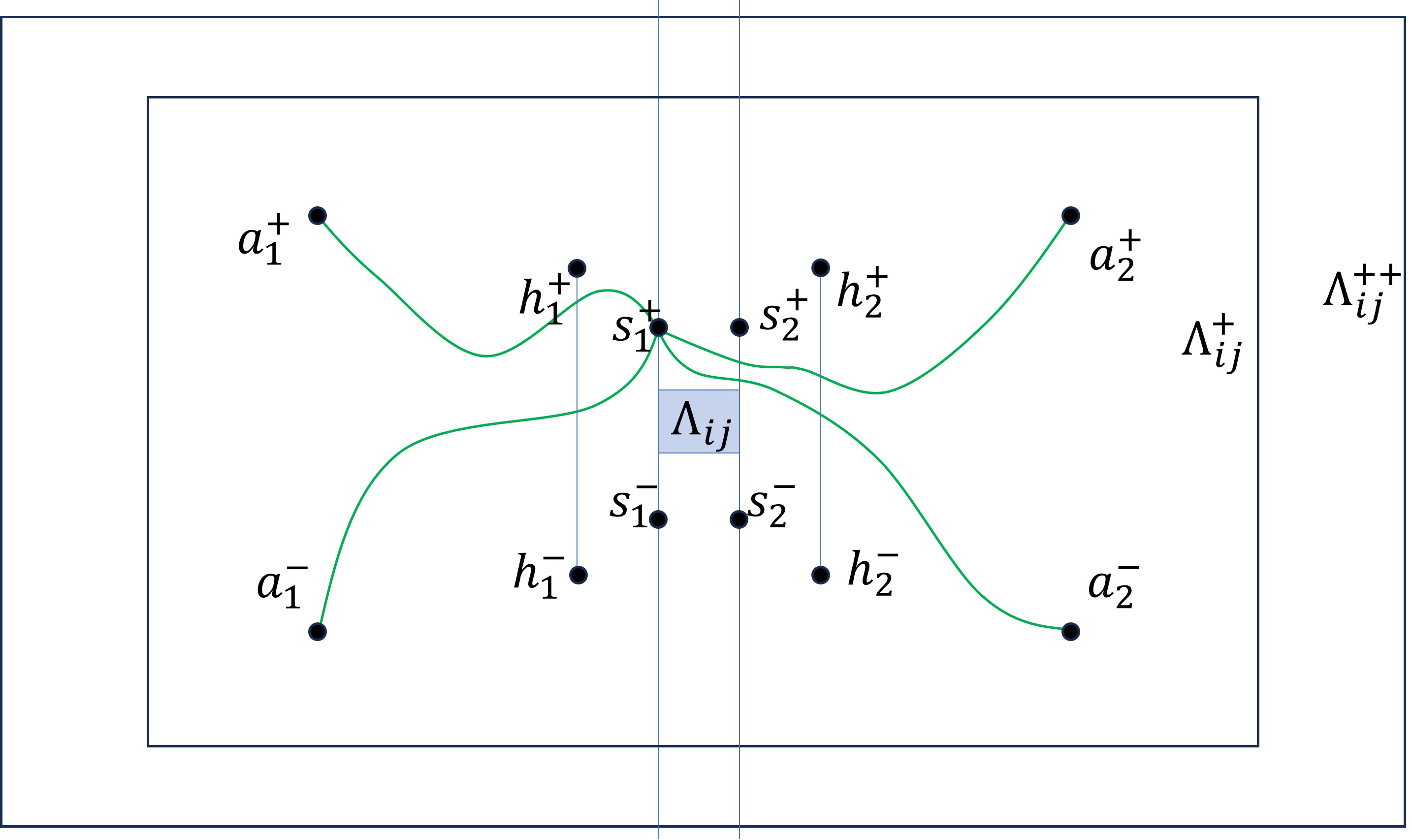}
\caption{Illustration of the event $\cC^{(2)}_{i,j}$: it essentially asks that the geodesics from points $s^{\pm}_1, s_{2}^{\pm}$ to points $a^{\pm}_1, a_2^{\pm}$ have typical local transversal fluctuation one column away from $\lambda_{i,j}$, these geodesics pass through the line segment joining $h^{+}_{1}$ and $h^{-}_{1}$ as well as the one joining $h^{+}_{2}$ and $h^{-}_{2}$.}
\label{f:VarianceC2}
\end{figure}
\end{center}

Finally, let
\begin{align*}
\cC^{(3)}_{i,j}&=\bigcap_{i'=i-1}^{i+1}\bigg\{\sup_{\substack{v_1\in\ell_{i'n,(j- 2R) W_n,(j+2R) W_n}\\v_2\in\ell_{(i'+1)n,(j- 2R) W_n,(j+2R) W_n}}} \Big|X_{v_1 v_2} - \mu|v_1-v_2| \Big| \leq R Q_n\bigg\}\\
&\qquad\bigcap\bigg\{\sup_{\substack{v_1\in\ell_{(i-1)n,(j- 2R) W_n,(j+2R) W_n}\\v_2\in\ell_{(i+2)n,(j- 2R) W_n,(j+2R) W_n}}} \Big|X_{v_1 v_2} - \mu|v_1-v_2| \Big| \leq R Q_n\bigg\}.
\end{align*}

This event asks that passage times across the three columns of width $n$, with the column containing $\Lambda_{i,j}$ in the middle are typical (fluctuations at scale $Q_{n}$) in the region around $\Lambda_{i,j}$.

Let $\cC_{i,j}=\cC^{(1)}_{i,j}\cap\cC^{(2)}_{i,j}\cap \cC^{(3)}_{i,j}$. The main result in this section is to show that the event $\cC_{i,j}$ occurs at a large fraction of the locations along the geodesic from $(0,0)$ to $(Mn,0)$ with large probability. Let $\gamma$ be the optimal path from $\mathbf{0}$ to $(Mn,0)$ and let $(in,y_i )$  be the location of its first intersection with the line $x=in$. Define $J_i=\lfloor y_i/W_n\rfloor$. We have the following proposition. 

\begin{proposition}
\label{p:percevent}
There exists $R$ such that for all large enough $M$ and large enough $n$
\[
\P\bigg[ \sum_{i=1}^{M}  I(\cC_{i,J_i}) \geq \frac{9}{10} M \bigg ] \geq 1-M^{-10}.
\]
\end{proposition}

Postponing the proof of this proposition for now, we define a local proxy of the event $\cC_{i,j}$ which will be useful in our variance decomposition later. 

\subsubsection{Localized events}
The events $\cC^{(k)}_{i,j}$  depend on all of  $\omega$ but we will show that they can be well approximated by $\omega_{\Lambda^{++}_{ij}}$.  Let $\omega^{loc}$ be an independent environment and define $\omega^{i,j,loc}$ as
\[
\omega^{i,j,loc}_{\Lambda^{++}_{ij}} = \omega_{\Lambda^{++}_{ij}}, \qquad \omega^{i,j,loc}_{(\Lambda^{++}_{ij})^c} = \omega^{loc}_{(\Lambda^{++}_{ij})^c}.
\]
We define the events $\cC^{(k),loc}_{i,j}$ to be the same as $\cC^{(k)}_{i,j}$ except with $X$ replaced by $X^{\omega^{i,j,loc}}$ (i.e., the passage times are computed in the environment $\omega^{i,j,loc}$). 

\begin{lemma}\label{l:CC.local.approx}
We have that
\[
\P[\cC_{a,b} \cap (\cC^{loc}_{a,b} )^c] \leq M^{-100}.
\]
\end{lemma}

\begin{proof}
The probability is invariant in $a$ and $b$ so consider $\cC_{0,0}$.  Define the set of paths
\[
\Psi_{i,i'}=\bigg\{\gamma': \gamma'(0)\in\ell_{in,- \Theta W_n,\Theta W_n}, \gamma'(1)\in\ell_{i'n,- \Theta W_n,\Theta W_n}, \gamma'\not\subset \Lambda_{0,0}^{+}\bigg\}
\]
and define the events
\begin{align*}
\cQ^{(1)}&= \bigcap_{i= -\Theta}^{\Theta} \bigcap_{i'= i+1}^{\Theta}
\bigg\{\inf_{\gamma' \in \Psi_{i,i'}} X_{\gamma'}-X_{\gamma'(0),\gamma'(1)} \geq 2 Q_n\bigg\},\\
\cQ^{(2)}&= \bigcap_{i= -\Theta}^{\Theta} \bigcap_{i'= i+1}^{\Theta}
\bigg\{\inf_{\gamma' \in \Psi_{i,i'}} X^{\omega^{0,0,loc}}_{\gamma'}-X^{\omega^{0,0,loc}}_{\gamma'(0),\gamma'(1)} \geq 2 Q_n\bigg\},\\
\cQ^{(3)}&= \Big\{\sup_{\gamma' \subset \Lambda^{+}_{ij}} |X_{\gamma'} - X^{\omega^{0,0,loc}}_{\gamma'}|=0 \Big\}.
\end{align*}
We will show that $\P[\cQ^{(1)}\cap \cQ^{(2)}\cap \cQ^{(3)}] \geq 1-M^{-100}$ and $\cC_{0,0} \cap (\cC^{loc}_{0,0} )^c \subset (\cQ^{(1)}\cap \cQ^{(2)}\cap \cQ^{(3)})^c$.  By Assumption~\ref{as:resamp2} we have that
\[
\P[\cQ^{(3)}]\geq 1-n^{-10} \geq 1- M^{-200}.
\]
Let $\gamma^+$ be the geodesic joining $b_1^+=(-\frac43\Theta n,\frac43 \Theta W_n)$ and $b_2^+=(\frac43\Theta n,\frac43 \Theta W_n)$ and let $\gamma^-$ be the geodesic joining $b_1^-=(-\frac43\Theta n,\frac43 \Theta W_n)$ and $b_2^-=(\frac43\Theta n,\frac43 \Theta W_n)$.
Suppose that the event
\begin{align*}
\cA = &\Big\{\inf_{\gamma^* \in \Upsilon_{\frac83\Theta n,b_1^+,b_2^+,\Theta^{1/10},b_1^+}} X_{\gamma^*} - X_{b_1^+,b_2^+} \geq Q_{\frac83 \Theta n}\Big\} \\
&\qquad\bigcap \Big\{\inf_{\gamma^* \in \Upsilon_{\frac83\Theta n,b_1^-,b_2^-,\Theta^{1/10},b_1^-}} X_{\gamma^*} - X_{b_1^-,b_2^-} \geq Q_{\frac83 \Theta n}\Big\}
\bigcap  \cS_{n^\epsilon,\origin}^{Mn}
\end{align*}
holds.  This implies that 
\[
\gamma^+ \not \in \Upsilon_{\frac83\Theta n,b_1^+,b_2^+,\Theta^{1/10},b_1^+}, \quad \gamma^- \not \in \Upsilon_{\frac83\Theta n,b_1^-,b_2^-,\Theta^{1/10},b_1^-}.
\]
and so the transversal fluctuations of $\gamma^\pm$ are not too large.  In particular, since by Lemma~\ref{l:growth34} we have that $Q_{\frac83 \Theta n} \Theta^{1/10} \leq \frac1{10} \Theta Q_n$ it follows that
\[
\gamma^+,\gamma^- \subset [-\frac32\Theta n,\frac32\Theta n]\times [-\frac32\Theta W_n,\frac32\Theta W_n] \subset \Lambda^+_{0,0}.
\]
The lines $\ell_{-\frac32 \Theta n,-\frac32 \Theta W_n,\frac32 \Theta W_n},\ell_{-\frac32 \Theta n,-\frac32 \Theta W_n,\frac32 \Theta W_n}$ and the curves $\gamma^+,\gamma^-$ enclose a region $\mathfrak{O}$ such that
\begin{equation}\label{eq:frak.o}
[-\Theta n,\Theta n]\times [-\Theta W_n,\Theta W_n] \subset \mathfrak{O} \subset \Lambda^+_{0,0}.
\end{equation}
\begin{center}
\begin{figure}
\includegraphics[width=5in]{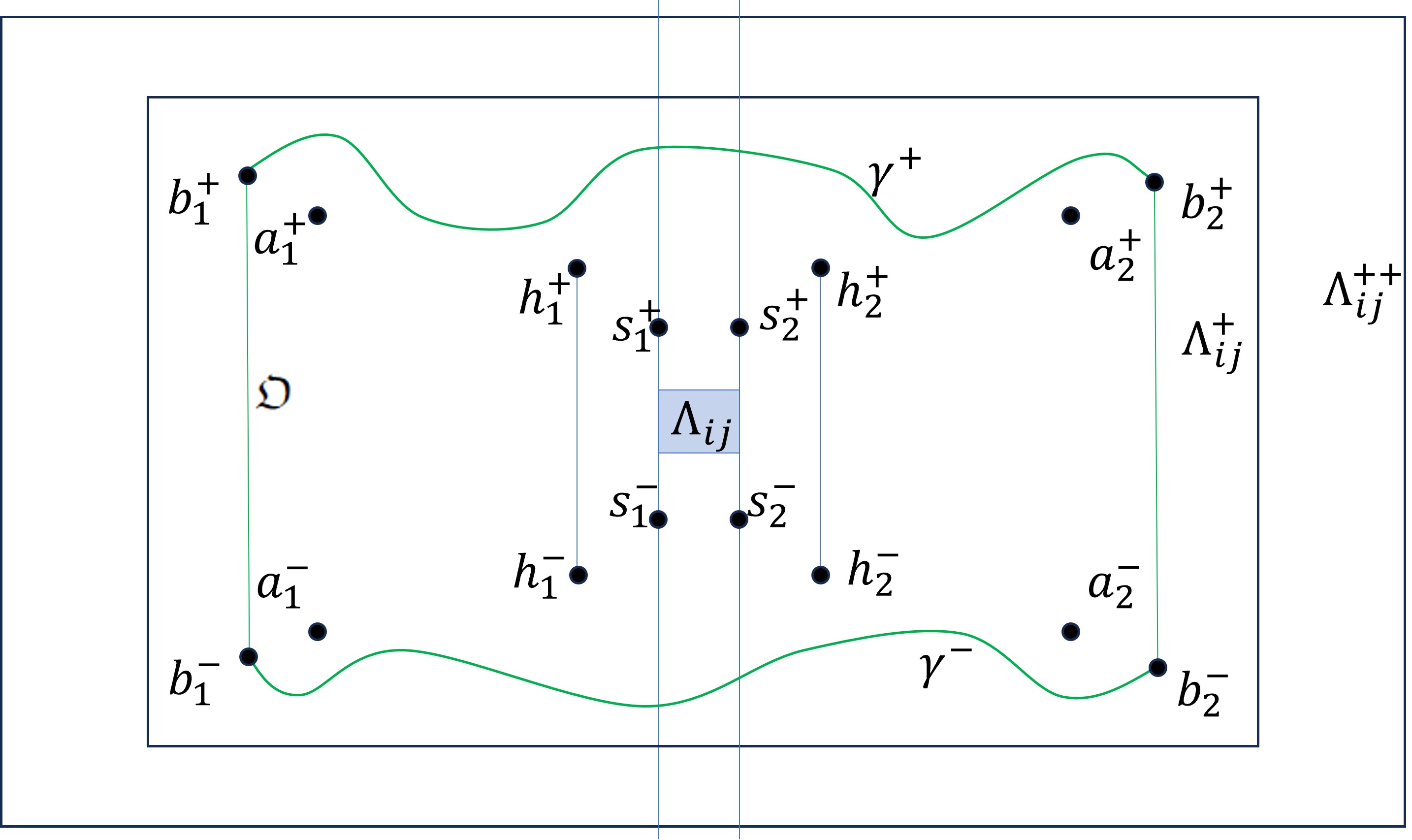}
\caption{Construction in the proof of Lemma \ref{l:CC.local.approx}: localizing the events $\cC_{i,j}$. It is shown that by resampling the environment outside $\lambda^{++}_{i,j}$, the event $\cC_{i,j}$ barely changes.}
\label{f:VarianceLocal}
\end{figure}
\end{center}
See Figure~\ref{f:VarianceLocal}.  Let $\gamma'\in\Psi_{i,i'}$ for some $-\Theta\leq i < i'\leq \Theta$.  Then $\gamma'(0),\gamma'(1)\in \mathfrak{O}$ but the path exists $\Lambda^+_{0,0}$ so it leaves $\mathfrak{O}$.  If it hits the line $y=-\frac32\Theta n$ or $y=\frac32 \Theta n$ then we can find {$u\in \gamma' \cap B_{Mn}(\origin)$} such that
\[
|\gamma'(0) - u| + |u - \gamma'(1)| \geq n + |\gamma'(0)  - \gamma'(1)|.
\]
Then by $\cS_{n^\epsilon,\origin}^{Mn}$ we have that
\begin{align*}
X_{\gamma'} &\geq X_{\gamma'(0), u} + X_{u,\gamma'(1)} - 2\Gamma_{3Mn}\\
&\geq \mu(|\gamma'(0) - u| + |u - \gamma'(1)| ) -\frac{2}{10} \mu n^\epsilon Q_{3n} -2 n^{2\epsilon\theta/\kappa + \epsilon}\\
&\geq \mu n + \mu |\gamma'(0)  - \gamma'(1)|-\frac{2}{10} \mu n^\epsilon Q_{3n} -2 n^{2\epsilon\theta/\kappa + \epsilon}\\
&\geq X_{\gamma'(0) , \gamma'(1)} + \mu n-\frac{3}{10} \mu n^\epsilon Q_{3n} -2 n^{2\epsilon\theta/\kappa + \epsilon}\\
&> X_{\gamma'(0) , \gamma'(1)} +2 Q_n.
\end{align*}
So suppose $\gamma'$ avoids the lines $y=-\frac32\Theta n$ and $y=\frac32 \Theta n$.  If $\gamma'$ exits $\mathfrak{O}$ through $\gamma^+$ and then exits $\Lambda^+_{0,0}$ it must return through $\gamma^+$ (see Figure~\ref{f:VarianceLocal}).  Let $w_1,w_2$ be the entry and exit points.  The path $\gamma^*$ that follows $\gamma^+$ from $b_1^+$ to $w_1$ then $\gamma'$ from $w_1$ to $w_2$ and then $\gamma^+$ from $w_2$ to $b_2^+$ is in $\Upsilon_{\frac83\Theta n,b_1^+,b_2^+,\Theta^{1/10},b_1^+}$ and so
\begin{align*}
X_{b_1^+,w_1}^{\gamma^+} +X_{w_1,w_2}^{\gamma'} + X_{w_2,b_2^+}^{\gamma^+} \geq   X_{\gamma^*}  \geq Q_{\frac83 \Theta n} + X_{b_1^+,b_2^+} \geq Q_{\frac83 \Theta n} + X_{b_1^+,w_1}^{\gamma^+} +X_{w_1,w_2}^{\gamma^+} + X_{w_2,b_2^+}^{\gamma^+} - 3\Gamma_{Mn}
\end{align*}
and so
\[
X_{w_1,w_2}^{\gamma'} - X_{w_1,w_2}^{\gamma^+}  \geq   Q_{\frac83 \Theta n}  - 3\Gamma_{Mn}.
\]
Now let $\hat\gamma$  be the path that follows $\gamma'$ from $\gamma'(0)$ to $w_1$ then $\gamma^+$ from $w_1$ to $w_2$ and then $\gamma'$ from $w_2$ to $\gamma'(1)$.  We have that 
\begin{align*}
X_{\gamma'(0) , \gamma'(1)}& \leq X_{\hat\gamma}\\
&\leq X_{\gamma'(0),w_1}^{\gamma'} +X_{w_1,w_2}^{\gamma^+} + X_{w_2,\gamma'(1)}^{\gamma'}\\
&\leq X_{\gamma'(0),w_1}^{\gamma'} +X_{w_1,w_2}^{\gamma'} + X_{w_2,\gamma'(1)}^{\gamma'} + Q_{\frac83 \Theta n}  - 3\Gamma_{Mn}\\
&\leq X_{\gamma'}  + Q_{\frac83 \Theta n}  - 6\Gamma_{Mn}\\
&\leq X_{\gamma'}  + 3 Q_{\frac83 \Theta n}
\end{align*}
where the last inequality follows by Lemma~\ref{l:growth34} and $\cS_{n^\epsilon,\origin}^{Mn}$.  We get same inequality if $\gamma'$ exits through $\gamma^-$ instead.  Hence $\cQ^{(1)}$ holds.  So by Lemmas~\ref{l:trans.SOGam} and~\ref{l:trans.events}
\[
\P[\cQ^{(1)}] = \P[\cQ^{(2)}] \geq \P[\cA] \geq 1 - M^{-200}.
\]
Hence we have that
\[
\P[\cQ^{(1)}\cap \cQ^{(2)}\cap \cQ^{(3)}] \geq 1-M^{-100}.
\]
Now observe that we define the event that all passage times defined by starting and ending points in $\Psi_{i,i'}$ across are equal under environments
\[
\cQ^{(4)}= \bigcap_{i= -\Theta}^{\Theta} \bigcap_{i'= i+1}^{\Theta} \Big\{\forall u\in\ell_{in,- \Theta W_n,\Theta W_n}, v\in\ell_{i'n,- \Theta W_n,\Theta W_n} \quad  X_{u,v}=X^{\omega^{0,0,loc}}_{u,v} \Big\}
\]
then we have that
\begin{equation}\label{eq:equal.passage.Cs}
\cQ^{(1)}\cap \cQ^{(2)}\cap \cQ^{(3)}  \subset \cQ^{(4)}
\end{equation}
since under $\cQ^{(1)}\cap \cQ^{(2)}\cap \cQ^{(3)}$ both optimal geodesics between $u$ and $v$ must remain in $\Lambda^+_{0,0}$ by $\cQ^{(1)}\cap \cQ^{(2)}$ and so the equality follows by $\cQ^{(3)}$.
It remains to check 
\[
\cC_{0,0} \cap (\cC^{loc}_{0,0} )^c \subset (\cQ^{(1)}\cap \cQ^{(2)}\cap \cQ^{(3)})^c = (\cQ^{(1)}\cap \cQ^{(2)}\cap \cQ^{(3)}\cap \cQ^{(4)})^c.
\]
Suppose that $\cC_{0,0}^{(2)} \cap (\cC^{(2),loc}_{0,0})^c$ holds.  
We must have for some $\gamma' \in \Xi^{(n),R}_{i',\Theta,j',k,R/10}$
\[
X^{\omega^{0,0,loc}}_{\gamma'} - X^{\omega^{0,0,loc}}_{\gamma'(0),\gamma'(1)} < Q_n, \quad X_{\gamma'}-X_{\gamma'(0),\gamma'(1)}  \geq Q_n,
\]
where $i',j',k$ are as in $\cC^{(2)}_{0,0}$.  If $X^{\omega^{0,0,loc}}_{\gamma'(0),\gamma'(1)}\neq X_{\gamma'(0),\gamma'(1)}$ then $\cQ^{(4)}$ fails.  If $X^{\omega^{0,0,loc}}_{\gamma'(0),\gamma'(1)}= X_{\gamma'(0),\gamma'(1)}$ and  $\gamma'\subset \Lambda^+_{0,0}$ then $\cQ^{(3)}$ must fail since $X^{\omega^{0,0,loc}}_{\gamma'} \neq X_{\gamma'}$.  Finally if $X^{\omega^{0,0,loc}}_{\gamma'(0),\gamma'(1)}= X_{\gamma'(0),\gamma'(1)}$ and $\gamma' \not\subset \Lambda^+_{0,0}$ then $\gamma\in \Psi_{i',\Theta}$ and so $\cQ^{(2)}$ fails since $X^{\omega^{0,0,loc}}_{\gamma'} - X^{\omega^{0,0,loc}}_{\gamma'(0),\gamma'(1)} < Q_n$.  Hence $\cC^{(2)}_{0,0} \cap (\cC^{(2),loc}_{0,0} )^c \subset (\cQ^{(1)}\cap \cQ^{(2)}\cap \cQ^{(3)})^c$.  The argument for $\cC^{(1)}_{0,0} \cap (\cC^{(1),loc}_{0,0} )^c$ and $\cC^{(3)}_{0,0} \cap (\cC^{(3),loc}_{0,0} )^c$ follows similarly so
\[
\P[\cC_{0,0} \cap (\cC^{loc}_{0,0} )^c] \leq \P[(\cQ^{(1)}\cap \cQ^{(2)}\cap \cQ^{(3)})^c] \leq M^{-100}.
\]
\end{proof}

\begin{corollary}\label{c:path.c.loc.bound}
There exists $R$ such that for all large enough $n$,
\[
\P\bigg[ \sum_{i=1}^M  I(\cC^{loc}_{i,J_i}) \geq \frac{9}{10} M \bigg ] \geq 1-M^{-9}.
\]
\end{corollary}

\begin{proof}
    From Theorem \ref{t:trans.main}, it follows that $\P(\max_{i}|J_{i}|\ge M)\le M^{-100}$. The result now follows from Proposition \ref{p:percevent}, Lemma \ref{l:CC.local.approx} and taking a union bound over all $a,b$ with $1\le a \le M$ and $|b|\le M$.  
\end{proof}

Using the favourable events constructed above, Proposition \ref{p:percevent} and Corollary \ref{c:path.c.loc.bound}, we shall next establish \eqref{e:varlb1} in this section and use it to complete the proof of Theorem~\ref{t:tightness}. As mentioned above to lower bound the variance we shall decompose it into sums contributions coming from resampling $n\times W_n$ blocks one by one. We first study the contribution of resampling a single block. 

\subsection{Fluctuations from resampling a single block}

We begin by defining a collection of unlikely events that are rare enough that we may take a union bound over all the blocks to rule these out.  Let
\[
\cB^{(1)} =\left\{\max_{0\leq i \leq  M} |J_i| \geq \frac12 M^{3/5}\right\}.
\]
Let
\[
\cB^{(2)} =\left\{\max_{\frac14 M\leq i \leq \frac34 M} |J_i-J_{i+\Theta}| \geq \frac16 \Theta^{88/100}\right\}.
\]
With $\gamma_{i,y,i',y'}$ the optimal path from $(in,y)$ to $(i'n,y')$
\[
\cB^{(3)} =\left\{\max_{\substack{0 \leq i,i',i'' \leq M\\ y,y'\in[-MW_n, MW_n]}} \max_{\substack{(x_1,y_1), (x_2,y_2)\in \gamma_{i,y,i',y'}\\ |x_1-i''n| \leq \log^3 n\\ |x_2-i''n| \leq \log^3 n}} |y_1-y_2|  \geq \frac14 W_n \right\}.
\]
Let
\[
\cB^{(4)} =\bigcup_{i=0}^M \bigcup_{j=-M}^M \{\cC_{i,j} \cap (\cC^{loc}_{i,j} )^c\}.
\]
Let
\[
\cB^{(5)} =\{\Gamma_{10Mn}\geq n^\epsilon\}.
\]
Finally set $\cB=\bigcup_{j=1}^5 \cB^{(j)}$.
\begin{lemma}\label{l:cB.bound}
For all large $n$ we have that, if $n$ is such that $Q_{Mn} \leq M^{1/10} Q_n$ then
\[
\P[\cB] \leq M^{-90}.
\]
\end{lemma}
\begin{proof}
By Lemma~\ref{l:Gamma} we have that
\[
\P[\cB^{(5)}] \leq M^{-100}.
\]
By Theorem~\ref{t:trans.main}, if $Q_{Mn} \leq M^{1/10} Q_n$ then
\[
\P[\cB^{(1)} ]\leq \P[\trans_{Mn} \geq \frac12 M^{3/5} W_n]\leq \P[\trans_{Mn} \geq \frac12 M^{1/20} W_{Mn}]\leq \exp(1-DM^{\epsilon\theta/40})\leq M^{-100}.
\]
Now if $(\cB^{(1)})^c\cap (\cB^{(5)})^c$ holds and $|J_i-J_{i-\Theta}|\geq \frac16 \Theta^{88/100}$ and $J_i=j$ then the geodesics $\gamma$ from $\origin$ to $(Mn,0)$ restricted between $\mathbf{0}$ and $(in,y_i W_n)$ is a path in the set
\[
\Xi^{(\Theta n),L}_{\frac{i}{\Theta},0,j\frac{W_n}{W_{\Theta n}},0,\frac16\Theta^{88/100}\frac{W_n}{W_{\Theta n}}}.
\]
Note that by Lemma~\ref{l:growth34},
\[
{\frac16\Theta^{88/100}\frac{W_n}{W_{\Theta n}} \geq \frac1{6D_5} \Theta^{1/200}}
\]
Furthermore we have
\begin{align*}
X^\gamma_{\origin, (in,y_i W_n)} + X_{(in,y_i W_n),(Mn,0)} - 2\Gamma_{Mn}&\leq X^\gamma_{\origin, (in,y_i W_n)} + X^\gamma_{(in,y_i W_n),(Mn,0)} - 2\Gamma_{Mn}\\
&\leq X_{\origin,(Mn,0)} \leq X_{\origin, (in,y_i W_n)} + X_{(in,y_i W_n),(Mn,0)} 
\end{align*}
and so
\[
X^\gamma_{\origin, (in,y_i W_n)} -X_{\origin, (in,y_i W_n)}  \leq 2\Gamma_{Mn}\leq n^\epsilon < \frac{(\frac1{6D_5} \Theta^{1/100})\mu}{4000} Q_{\Theta n}.
\]
By Lemma~\ref{l:local.trans.proof} it follows that the event 
\[
\left(\cL^{(n),R}_{\frac{i}{\Theta},0,j\frac{W_n}{W_{\Theta n}},\frac16\Theta^{88/100}\frac{W_n}{W_{\Theta n}}}\right)^c
\]
holds and so
\begin{align*}
\P[(\cB^{(1)})^c\cap (\cB^{(5)})^c \cap |J_i-J_{i+\Theta}|,J_i=j] &\leq \P[(\cL^{(n),R}_{\frac{i}{\Theta},0,j\frac{W_n}{W_{\Theta n}},\frac16\Theta^{88/100}\frac{W_n}{W_{\Theta n}}})^c]\\
&\leq \exp\Big(1-D\big(\frac1{6D_5} \Theta^{1/200}\big)^\theta\Big)\leq M^{-100}.
\end{align*}
By a union bound over $i$ and $j$
\[
\P[(\cB^{(1)})^c\cap (\cB^{(5)})^c \cap \cB^{(2)}] \leq M^{-97}.
\]
We will show that the event $\cS_{n^\epsilon,\origin}^{(Mn)}\cap\cB^{(3)}$ is empty.  Suppose that it holds for some $\gamma'=\gamma_{i,y,i',y'}$.  If $\gamma'$ exits $B_{3Mn}(\origin)$ then we can find $0=t_0\leq t_1<t_2\leq t_3 =1$ such that $\gamma'(t_i) \in B_{3Mn}(\origin)$ and 
\[
\sum_{i=1}^3 |\gamma'(t_i)-\gamma'(t_{i-1})| \geq |\gamma'(1)-\gamma'(0)| + \frac15 M n  \geq |\gamma'(1)-\gamma'(0)| + \frac15 M n.
\]
The condition of $\cB^{(3)}$ means $(x_1,y_1),(x_2,y_2)\in \gamma'$ 
such that 
\[
|(x_1,y_1)-(x_2,y_2)| \geq |x_1 - x_2|+ \frac15 W_n.
\]
Hence we can find $0=t_0\leq t_1<t_2\leq t_3 =1$ such that
\[
\sum_{i=1}^3 |\gamma'(t_i)-\gamma'(t_{i-1})| \geq |\gamma'(1)-\gamma'(0)| + \frac15 W_n
\]
even if $\gamma'$ stays within $B_{3Mn}(\origin)$.  In both cases $\gamma'(t_i) \in B_{3Mn}(\origin)$.  Then by $\cS_{n^\epsilon,\origin}^{(Mn)}$ we have that
\[
\sum_{i=1}^3 X_{\gamma'(t_i),\gamma'(t_{i-1})} \geq \mu |\gamma'(1)-\gamma'(0)| + \mu \frac15 W_n - \frac3{10}\mu (Mn)^\epsilon Q_{3Mn}.
\]
On the other hand we also have
\[
\sum_{i=1}^3 X_{\gamma'(t_i),\gamma'(t_{i-1})} \leq X_{\gamma'(1),\gamma'(0)}+3\Gamma_{3Mn} \leq \mu |\gamma'(1)-\gamma'(0)| +{3\Gamma_{3Mn}} + \frac3{10}\mu (Mn)^\epsilon Q_{3Mn}
\]
and so combining the last two equations we have that
\begin{equation}\label{eq:b4.contradiction}
3\Gamma_{3Mn} + \frac4{10}\mu (3Mn)^\epsilon Q_{3Mn} \geq \mu \frac15 W_n.
\end{equation}
Then for large $n$, by Lemma~\ref{l:growth34} and the definition of $\cS_{n^\epsilon,\origin}^{(Mn)}$ the left hand side of~\eqref{eq:b4.contradiction} is at most
\[
3(Mn)^{\epsilon (2\theta/\kappa+1)} + \frac{4D\mu}{10} (Mn)^\epsilon (3M)^{3/4} Q_n \leq  n^{\frac12+2\epsilon}
\]
while the right hand side of~\eqref{eq:b4.contradiction} is at least
\[
\frac{\mu}5 \sqrt{n Q_n} \geq n^{\frac12 +\frac12\alpha}
\]
which is a contradiction so $\cS_{n^\epsilon,\origin}^{(Mn)}\cap\cB^{(3)}$ is empty.  Hence
\[
\P[\cB^{(3)}]\leq \P [(\cS_{n^\epsilon,\origin}^{(Mn)})^c] \leq M^{-100}.
\]
Finally by Lemma~\ref{l:CC.local.approx} and a union bound
\[
\P[\cB^{(4)}]\leq M^{-98}.
\]
The result follows from the above estimates by a union bound.
\end{proof}

Define the event that $\gamma$ comes close to  $\Lambda_{i,j}$
\[
\cI_{i,j}=\Big\{\inf \{\|x-y\|_\infty :x\in\gamma,y\in\Lambda_{i,j})\}\leq \log^2 n\Big\}
\]
The event $\cB^c$ together with $\cC^{loc}_{i,j}$ will serve to constrain the geodesic $\gamma$ from $0$ to $Mn$.  Note that by the definition of $\cB^{(4)}$ we have $\cB^c \cap \cC^{loc}_{i,j}\subset \cC_{i,j}$.  The following proposition says that if $\cI_{i,j}$ holds then $\gamma$ must also pass between $h_1^-$ and $h_1^+$ and between $h_2^-$ and $h_2^+$.

\begin{proposition}\label{p:path.hit.H}
For $\frac13M\leq i \leq \frac23 M$ and $|j| \leq M^{3/5}$, on the event $\cB^c \cap \cC^{loc}_{i,j} \cap \cI_{i,j}$ we have that,
\[
j-2R < J_{i-1} < j+ 2R, \quad j-2R < J_{i+2} < j+ 2R.
\]
\end{proposition}

We will first need a couple of lemmas. 
\begin{lemma}\label{l:transversal.blocking}
On the event $\cB^c \cap \cC^{loc}_{i,j}$ if $J_i, J_{i+1}\geq j+R$ or if $J_i, J_{i+1}< j-R$ then $\cI_{i,j}^c$.
\end{lemma}

\begin{proof}
Assume that $J_i, J_{i+1}\geq j+R$ in which case $\gamma$ passes above $s_1^+$ and $s_2^+$.  By property $(\cB^{(3)})^c$ the path $\gamma$ does not pass within distance $\log^2 n$ of the left and right and bottom side of $\Lambda_{ij}$.  If it passes within distance $\log^2 n$ it therefore enters it from its top side.  Suppose that this occurs.

Let $\gamma'$ be the path from $w_1=(i n , (j + R-1)W_n)$ to $w_2=((i+1) n , (j + R-1)W_n)$.  By $\cC^{(1)}_{i,j}$ we have that $\gamma'$ lies above the line $y=(j-1+\frac{R}{2})W_n \geq (j+2)W_n$ so if $\gamma$ is within distance $\log^2 n$ it must cross below $\gamma'$.  Events $(\cB^{(3)})^c$ means that the paths must cross at points we will call $d_1,d_2$ (illustrated in Figure~\ref{f:VarianceAbove}). The path $\hat \gamma$ that follows $\gamma'$ from $w_1$ to $d_1$, then $d_1$ to $d_2$ along $\gamma$ then from $d_2$ to $w_2$ along $\gamma'$ is in $\Upsilon_{n,w_1,w_2,\frac12 R,w_1}$ and so by the event $\cC^{(1)}_{i,j}$ we have

\begin{align*}
X^{\gamma'}_{w_1 d_1} + X^{\gamma'}_{d_1 d_2} + X^{\gamma'}_{d_2 w_2} -3 \Gamma_{10Mn} \leq X_{\gamma'} \leq X_{\hat\gamma} -Q_n\leq X^{\gamma'}_{w_1 d_1} + X^\gamma_{d_1,d_2} + X^{\gamma'}_{d_2 w_2} -Q_n
\end{align*}
and so $X^{\gamma'}_{d_1 d_2} \leq X^\gamma_{w_1 d_1} +3\Gamma_{10Mn}-Q_n$.  It follows that if $\check\gamma$ is the path following $\gamma$ from $(0,0)$ to $d_1$ then from $d_1$ to $d_2$ along $\gamma'$ then from $d_2$ to $(Mn,0)$ along $\gamma$ we have that
\[
X_\gamma\geq X^\gamma_{\origin,d_1} + X^\gamma_{d_1,d_2}+X^\gamma_{d_2,(Mn,0))} - 3\Gamma_{10Mn} 
\geq X^\gamma_{\origin,d_1} + X^{\gamma'}_{d_1,d_2}+X^\gamma_{d_2,(Mn,0))} + Q_n - 6\Gamma_{10Mn}  > X_{\check\gamma}
\]
which is a contradiction to $\gamma$ being the optimal path from $0$ to $(Mn,0)$.  Hence $d(\gamma, \Lambda_{i,j})> \log^2 n$.  The case of $J_i, J_{i+1}< i-R$ follows similarly.
\end{proof}

\begin{center}
\begin{figure}
\includegraphics[width=2in]{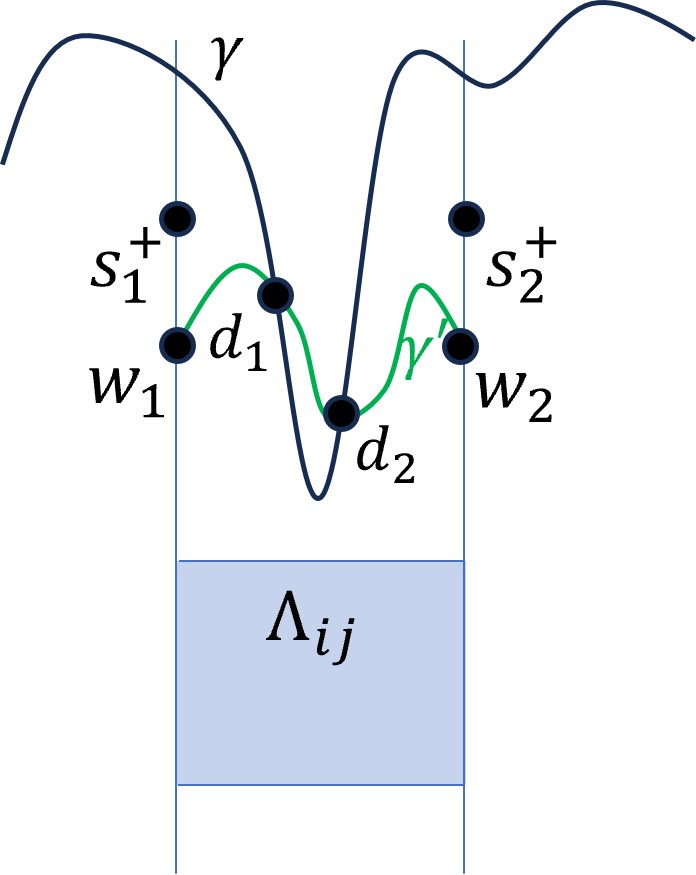}
\caption{Proof of Lemma \ref{l:transversal.blocking}: we show that on the favourable events it is extremely unlikely for the geodesic $\gamma$ from $\mathbf{0}$ to $(Mn,0)$ to come close to $\Lambda_{i,j}$ without being close to it on the left or the right boundary of the corresponding column. 
}
\label{f:VarianceAbove}
\end{figure}
\end{center}

\begin{lemma}\label{l:local.transversal.blocking}
On the event $\cB^c \cap \cC^{loc}_{i,j}$ if $J_i\geq j-R$ or $J_{i+1}\geq j-R$, then 
\[
J_{i-1}\geq j-\frac32 R, J_{i+2}\geq j-\frac32 R.
\]
Similarly if $J_i< j+R$ or $J_{i+1}< j+R$ 
\[
J_{i-1}\leq j+\frac32 R, J_{i+2}\leq j+\frac32 R.
\]
\end{lemma}
\begin{proof}
Suppose that $J_i\geq j-R$ but assume for the sake of contradiction that $J_{i-1}< j-\frac32 R$.  Let $\gamma'$ be the geodesic from $a_1^-$ to $w_1^{-}=s_1^- -(0,-W_n)$.  By $\cC^{(2)}_{i,j}$ we have that $\gamma'$ intersects the line $x=(i-1)n$ above  $(j-\frac43 R)W_n$. By $(\cB^{(2)})^c$ we have that $\gamma$ passes above $a_1^-$.  So if $J_{i-1}< j-\frac32 R$ then $\gamma$ must intersect $\gamma'$ at $d_1,d_2$ in between the lines $x=(i-\Theta)n$ and $x=in$ (illustrated in Figure~\ref{f:VarianceLT}.  Let $\hat\gamma$ be the path that follows $\gamma'$ from $a_1^-$ to $d_1$, then $d_1$ to $d_2$ along $\gamma$ then $d_2$ to $w_1^{-}$ along $\gamma'$.  By $\cC^{(2)}_{i,j}$, since $\hat\gamma \in \Xi^{(n),L}_{i,i-\Theta,j-R-1,0,R/10}$
\[
X_{\hat\gamma} \geq X_{\gamma'} +Q_n
\]
and so
\[
X^{\gamma'}_{a_1^-,d_1} + X^{\gamma}_{d_1,d_2} + X^{\gamma'}_{d_2, w_1^-} \geq X^{\gamma'}_{a_1^-,d_1} + X^{\gamma'}_{d_1,d_2} + X^{\gamma'}_{d_2, w_1^-} -3 \Gamma_{10Mn} + Q_n,
\]
and so $X^{\gamma}_{d_1,d_2} \geq X^{\gamma'}_{d_1,d_2} -3 \Gamma_{10Mn} + Q_n$.  Let $\check \gamma$ be the path from $\origin$ to $d_1$ along $\gamma$ then $d_1$ to $d_2$ along $\gamma'$ and then $d_2$ to $(Mn,0)$ along $\gamma$.  Then
\begin{align*}
X_{\check\gamma}&\leq X_{\origin,d_1}^\gamma + X^{\gamma'}_{d_1,d_2} + X^{\gamma}_{d_2,(Mn,0)}\\
&\leq  X_{\origin,d_1}^\gamma + X^{\gamma}_{d_1,d_2} + X^{\gamma}_{d_2,(Mn,0)} -3 \Gamma_{10Mn}+Q_n \\
&\leq X_\gamma -6 \Gamma_{10Mn}+Q_n < X_\gamma 
\end{align*}
 which contradicts $\gamma$ being the optimal path.  The other conclusions all follow similarly.
\end{proof}
\begin{center}
\begin{figure}
\includegraphics[width=5in]{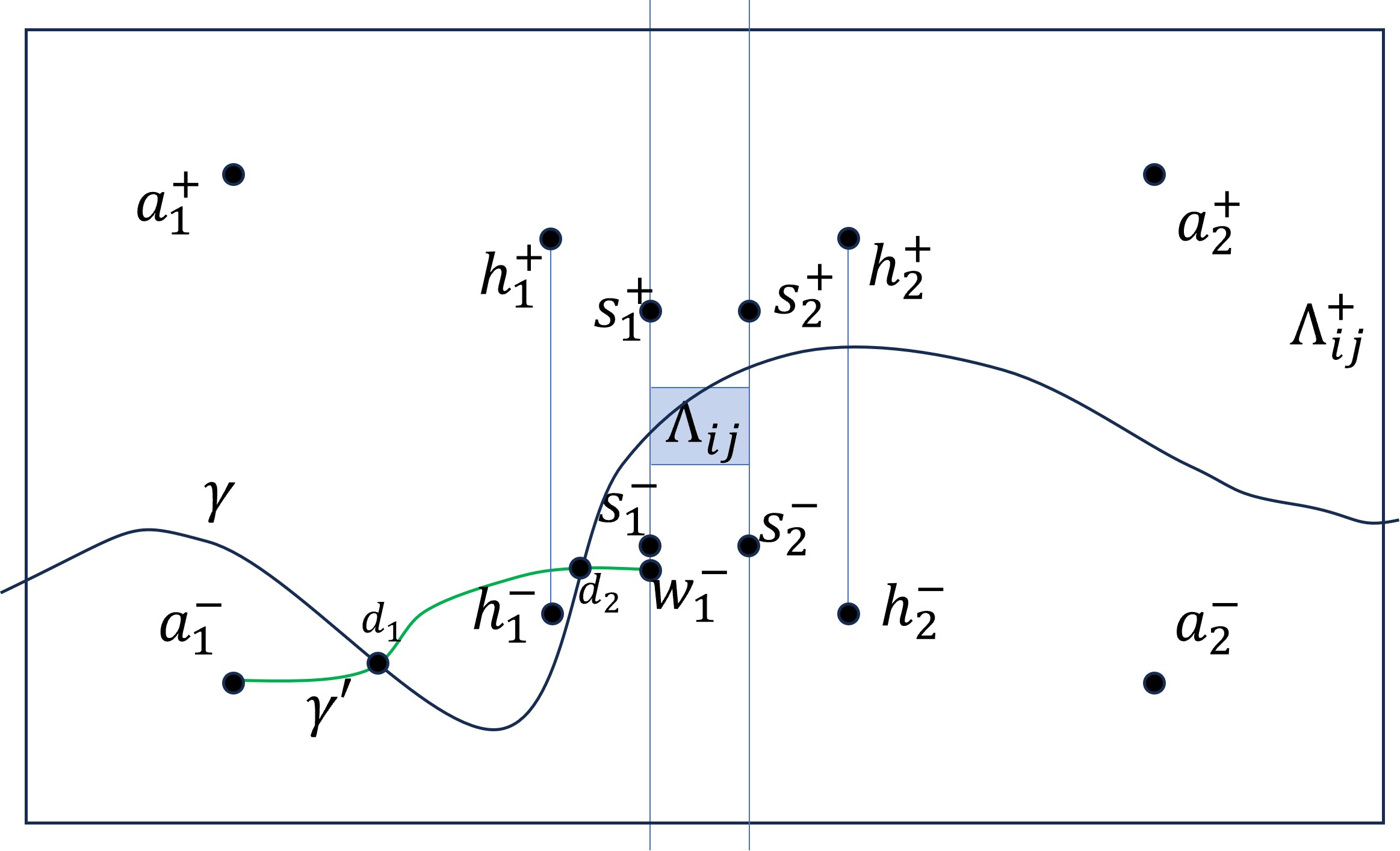}
\caption{Proof of Lemma \ref{l:local.transversal.blocking}: we show that if the geodesic $\gamma$ is close to $\Lambda_{i,j}$ either at its left or right boundary, then $\gamma$ cannot be too far away from the same height at the boundaries one column away to the left or right.}
\label{f:VarianceLT}
\end{figure}
\end{center}

\begin{proof}[Proof of Proposition~\ref{p:path.hit.H}]
Assume the event $\cB^c \cap \cC^{loc}_{i,j} \cap \cI_{i,j}$.
The path $\gamma$ must pass either between $s_1^-$ and $s_1^+$ or below $s_1^-$ or above $s_1^+$.  Similarly, it must pass either between $s_2^-$ and $s_2^+$ or below $s_2^-$ or above $s_2^+$.  By Lemma~\ref{l:transversal.blocking} it cannot pass above both $s_1^+$ and $s_2^+$ as then it would not pass within $\log^2 n$ of $\Lambda_{i,j}$.  Similarly it cannot pass below both $s_1^-$ and $s_2^-$.  

Hence we must have both the events $\{J_i\geq j-R\}\cup \{ J_{i+1}\geq j-R\}$ and $\{J_i< j+R \} \cup \{J_{i+1}< j+R\}$ and so the conclusion follows by Lemma~\ref{l:local.transversal.blocking}.
\end{proof}

\subsubsection{Resampling events:}
Let $\omega'$ be an independent environment and let $\omega^{ij}$ be the environment with 
\[
\omega^{ij}_{\Lambda_{ij}} = \omega'_{\Lambda_{ij}}, \quad \omega^{ij}_{\Lambda_{ij}^c} = \omega_{\Lambda_{ij}^c}
\]
so only $\Lambda_{ij}$ is resampled. We now define another collection of following events.  The first two events ask that passage times in columns away from $i$ do not change much
\begin{align*}
\cG^{(1)}_{i,j}&=\bigg\{\sup_{\substack{v_1\in\ell_{(i-1)n,(j- R) W_n,(j+R) W_n}\\v_2\in\ell_{in,(j- 2R) W_n,(j+2R) W_n}}} \Big|X_{v_1 v_2} - X_{v_1 v_2}^{\omega^{ij}} \Big| \leq n^\epsilon\bigg\}\\
&\qquad\bigcap\bigg\{\sup_{\substack{v_1\in\ell_{(i+1)n,(j- R) W_n,(j+R) W_n}\\v_2\in\ell_{(i+2)n,(j- 2R) W_n,(j+2R) W_n}}} \Big|X_{v_1 v_2} - X_{v_1 v_2}^{\omega^{ij}} \Big| \leq n^\epsilon\bigg\},
\end{align*}
\begin{align*}
\cG^{(2)}_{i,j}&=\bigg\{\sup_{v\in\ell_{(i-1)n,-M W_n,M W_n}} \Big|X_{\origin v} - X_{\origin v}^{\omega^{ij}} \Big| \leq n^\epsilon\bigg\}\\
&\qquad\bigcap\bigg\{\sup_{v\in\ell_{(i+2)n,-M W_n,M W_n}} \Big|X_{v,(Mn,0)} - X_{v,(Mn,0)}^{\omega^{ij}} \Big| \leq n^\epsilon\bigg\}.
\end{align*}
Define a slight reduction and enlargement of $\Lambda_{i,j}$ as follows
\begin{align*}
\Lambda^-_{i,j}&=[in+\log^2 n,(i+1)n-\log^2 n]\times [jW_n+\log^2 n,(j+1)W_n - \log^2 n], \\
\Lambda^\star_{i,j}&=[in-\log^2 n,(i+1)n+\log^2 n]\times [jW_n-\log^2 n,(j+1)W_n + \log^2 n]
\end{align*}
and define the resampling events 
\begin{align*}
\cG^{(3)}_{i,j}=\bigg\{ \sup_{\gamma' \subset \Lambda^-_{i,j}} |X_{\gamma'}^{\omega^{ij}} - X_{\gamma'}|=0 \bigg\}.
\end{align*}
and
\begin{align*}
\cG^{(4)}_{i,j}=\bigg\{ \sup_{\gamma' \subset (\Lambda^\star_{i,j})^c} |X_{\gamma'}^{\omega^{ij}} - X_{\gamma'}|=0 \bigg\}.
\end{align*}
Let
\begin{align*}
\cG^{(5)}_{i,j}&=\bigg\{ X_{(in, (j+\frac12)W_n),(in+n^{8/10}, (j+\frac12)W_n)}^{\omega^{ij}}\leq \mu n^{8/10} + Q_n\bigg\}\\
&\qquad \cap \bigg\{X_{((i+1)n-n^{8/10}, (j+\frac12)W_n),((i+1)n, (j+\frac12)W_n)}^{\omega^{ij}}\leq \mu n^{8/10} + Q_n \bigg\}
\end{align*}
and
\begin{align*}
\cG^{(6)}_{i,j}=\big\{ \Gamma_{Mn}^{\omega^{ij}} \leq n^\epsilon \big\}.
\end{align*}
Finally the intersection of all events is defined as
\[
\cG_{i,j} = \bigcap_{k=1}^6 \cF^{(k)}_{i,j}.
\]

\begin{lemma}\label{l:cG.bound}
We have that
\[
\P[\cG_{i,j}] \geq 1-M^{-90}.
\]
\end{lemma}

\begin{proof}
Recall that $\Lambda_i=[(i-1)n,in]\times\R$.
Let $\omega^*$ be a third independent environment and define $\omega^{*i}$ as 
\[
\omega^{*i}_{\Lambda_{i}} = \omega_{\Lambda_{i}}, \quad \omega^{*i}_{\Lambda_{i}^c} = \omega^*_{\Lambda_{i}^c},
\]
which corresponds to resampling $\Lambda_{i-1}^c$ with $\omega^*$.  It also corresponds to starting with $\omega^{ij}$ and  resampling $\Lambda_{i}^c$ with $\omega^*$.  Hence by Assumption~\ref{as:resamp1},
\begin{align*}
&\P\bigg[\sup_{\substack{v_1\in\ell_{(i-1)n,(j-2R) W_n,(j+2R) W_n}\\v_2\in\ell_{in,(j- 2R) W_n,(j+2R) W_n}}} \Big|X_{v_1 v_2} - X_{v_1 v_2}^{\omega^{ij}} \Big| > n^\epsilon\bigg] \\
&\quad\leq \P\bigg[\sup_{\substack{v_1\in\ell_{(i-1)n,(j- 2R) W_n,(j+2R) W_n}\\v_2\in\ell_{in,(j- 2R) W_n,(j+2R) W_n}}} \Big|X_{v_1 v_2} - X_{v_1 v_2}^{\omega^{*i}} \Big| > \frac12 n^\epsilon\bigg] \\
&\qquad + \P\bigg[\sup_{\substack{v_1\in\ell_{(i-1)n,(j- 2R) W_n,(j+2R) W_n}\\v_2\in\ell_{in,(j- 2R) W_n,(j+2R) W_n}}} \Big|X_{v_1 v_2}^{\omega^{*i}} - X_{v_1 v_2}^{\omega^{ij}} \Big| > \frac12 n^\epsilon\bigg] \\
&\quad \leq 2 \exp\Big(1-D (n^\epsilon/\log^{1/\kappa}n)^\kappa\Big) \leq M^{-100}.
\end{align*}
We can bound the second part of $\cG^{(1)}_{i,j}$ similarly so
\[
\P[(\cG^{(1)}_{i,j})^c] \leq 2M^{-100}.
\]
By essentially the same proof,
\[
\P[(\cG^{(2)}_{i,j})^c] \leq 2M^{-100}.
\]
By Assumption~\ref{as:resamp2},
\[
\P[(\cG^{(3)}_{i,j})^c] \leq n^{-10},\qquad \P[(\cG^{(4)}_{i,j})^c] \leq n^{-10}.
\]
Using \eqref{eq:Q} for concentration at the scale $Q_{n^{8/10}}$,
\[
\P[(\cG^{(5)}_{i,j})^c] \leq 2\exp\Big( 1 - (\frac{Q_n}{Q_{n^{8/10}}})^\theta\Big) \leq M^{-100}.
\]
By Lemma~\ref{l:Gamma}
\[
\P[(\cG^{(6)}_{i,j})^c] \leq \exp\Big( 1 - (\frac{n^\epsilon }{Q_{\log^C n}})^\kappa\Big) \leq M^{-100}.
\]
The lemma then holds by a union bound for large enough $n$.
\end{proof} 

We define one last event which is,
\begin{align*}
\cH_{i,j}&=\left\{\inf_{\substack{\gamma'\subset \Lambda^-_{i,j},\\ \gamma'(0)=(in+n^{8/10}, (j+\frac12)W_n),\\ \gamma'(1)=((i+1)n-n^{8/10}, (j+\frac12)W_n)}} X_{\gamma'}^{\omega'} \leq (n-2n^{8/10})\mu - R^{10} Q_n\right\}.
\end{align*}
which we note is independent of $\omega$.

\begin{lemma}\label{l:cH.bound}
There exists $\delta>0$ such that for all large enough $\alpha'$-record points $n$,
\[
\P[\cH_{i,j}] \geq \delta.
\]
\end{lemma}
\begin{proof}
Without loss of generality, we take $i=j=0$.  For all large enough $n_\star$, by Proposition~\ref{p:left} if there is an $\alpha''\in(\alpha,\alpha')$ record point in $[\frac12 n_\star, 2n_\star]$ then
\[
\P\Big[X_{(n_\star^{9/10},0),(n_\star - n_\star^{9/10},0)} \leq \mu (n_\star -2n_\star^{9/10}) - Q_{n_\star}\Big] \geq \delta_\star>0.
\]
Let $\gamma_\star$ be the optimal path joining $(n_\star^{9/10},0)$ and $(n_\star - n_\star^{9/10},0)$. By Lemmas~\ref{l:trans.SOGam} and~\ref{l:trans.events} and the fact that $W_{n_\star}=o(n_\star^{9/10})$, we can choose an absolute constant $T$ large enough such that
\begin{equation}\label{eq:short.contained.path}
\P\Big[X_{(n_\star^{9/10},0),(n_\star - n_\star^{9/10},0)} \leq \mu (n_\star -2n_\star^{9/10}) - Q_{n_\star},\gamma_\star\subset [n_\star^{8/10},n_\star-n_\star^{8/10}] \times [-Tn_\star,Tn_\star] \Big] \geq \delta_\star/2.
\end{equation}
for all sufficiently large $n_\star$.

Choose $m$ to be the largest value such that $n/m$ is an integer,
\[
\frac{W_n}{W_m} \geq \sqrt{n/m} \geq 4T,\qquad  \frac{Q_m}{Q_n} \geq C (m/n)^{3/4} \geq 2R^{10}\frac{m}{n} 
\]
and there is an $\alpha''$ record point in $[\frac12 m,2m]$.  By Lemma~\ref{l:grmain}, we can find $m$ such that
\[
n/m \leq (T\vee R)^C,
\]
and so in particular $n/m$ is bounded above independently of $n$.

Let 
\[
\Delta_i=[(i-1)m+\log^2n,im-\log^2 n]\times[\log^2 n,W_m-\log^2 n]\subset \Lambda^-_{0,0}
\]
and let $\gamma_i$ be the geodesic joining $((i-1)m+m^{9/10},\frac{W_n}{2})$ to $(im-m^{9/10},\frac{W_n}{2})$ in the environment $(\omega')^{\Delta_i}$ recalling that this is defined as $\omega'$ resampled outside of $\Delta_i$.  The resampling is done independently for each $i$ so the environments $\omega^{\Delta_i}$ are independent since the $\Delta_i$ are disjoint.
Define the events
\[
\cW_i = \Big\{ X^{(\omega')^{\Delta_i}}_{\gamma_i} \leq \mu (m -2m^{9/10}) - Q_{m},\gamma_i\subset \Delta_i \Big\}
\]
Then by equation~\eqref{eq:short.contained.path} and independence
\[
\P\left[\bigcap_{i=1}^{n/m} \cW_i\right] \geq (\delta_\star/2)^{n/m}.
\]
Let $\cU_i$ be the event that $X^{\omega'}_{\gamma_i} \leq \mu (m -2m^{9/10}) - Q_{m}$.
By Assumption~\ref{as:resamp2},
\[
\P\left[\bigcap_{i=1}^{n/m} \cW_i, \ \bigcap_{i=1}^{n/m} \cU_i\right] \geq (\delta_\star/2)^{n/m}-o(1).
\]
For $1\leq i \leq n/m-1$ let $\hat{\gamma}_i$ be the geodesic in environment $\omega'$ from $(im-m^{9/10},\frac{W_n}{2})$ to $(im+m^{9/10},\frac{W_n}{2})$.  Let $\hat{\gamma}_0$ be the geodesic in environment $\omega'$ from $(n^{8/10},\frac{W_n}{2})$ to $(m^{9/10},\frac{W_n}{2})$ and let $\hat{\gamma}_{n/m}$ be the geodesic in environment $\omega'$ from $(n-m^{9/10},\frac{W_n}{2})$ to $(n-n^{8/10},\frac{W_n}{2})$.  Then for $0\leq i \leq n/m$ define the events
\[
\cV_i=\{\hat{\gamma}_i\subset \Lambda_{0,0}^-, X^{\omega'}_{\hat{\gamma}_i(0),\hat{\gamma}_i(1)} \leq 2\mu |\hat{\gamma}_i(0) - \hat{\gamma}_i(1)| + \frac12(n/m)^{-1}Q_n\}.
\]
By Lemmas~\ref{l:trans.SOGam} and~\ref{l:trans.events} and the fact that $Q_{2m^{9/10}} = o(Q_n)$ we have that
\[
\P[\cV_i]=1-o(1).
\]
\begin{center}
\begin{figure}
\includegraphics[width=5in]{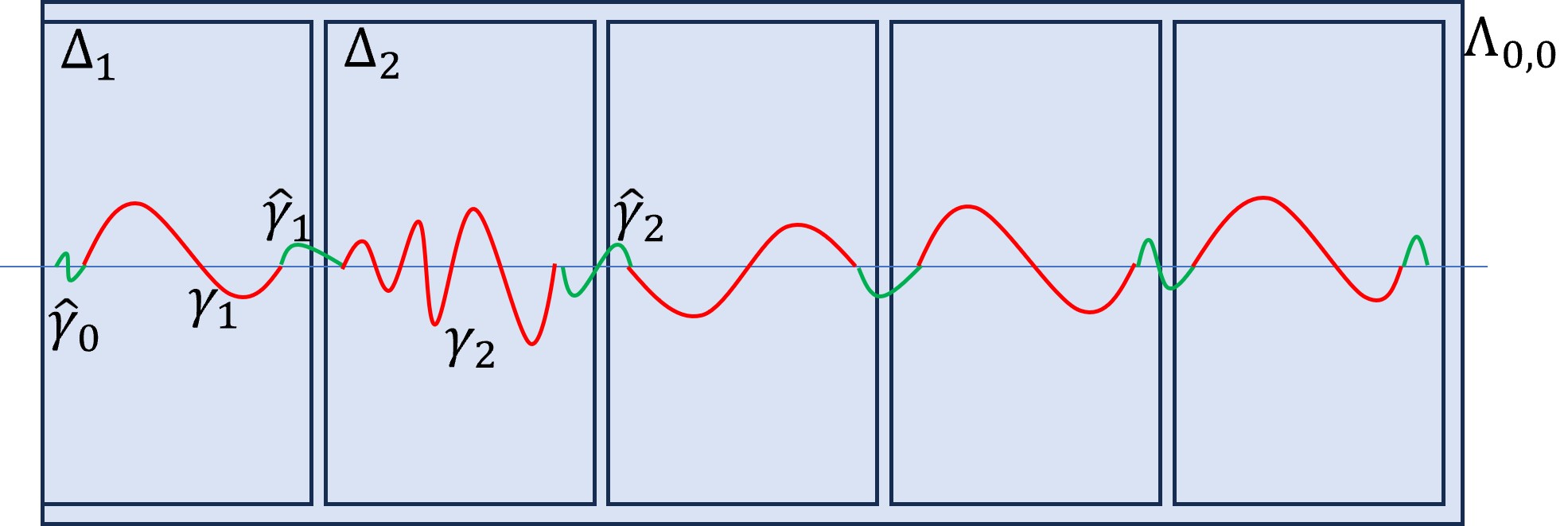}
\caption{Proof of Lemma \ref{l:cH.bound}: we use Proposition \ref{p:left} to show that with positive probability, paths restricted within an $n\times W_{n}$ rectangles can also have arbitrarily large deviation to the left at scale $Q_{n}$. }
\label{f:VarianceGoodold}
\end{figure}
\end{center}
Hence, using Lemma~\ref{l:Gamma},
\[
\P\left[\bigcap_{i=1}^{n/m} \cW_i, \ \bigcap_{i=1}^{n/m} \cU_i, \ \bigcap_{i=1}^{n/m} \cV_i, \ \Gamma_n^{\omega'} \leq n^\epsilon\right] \geq (\delta_\star/2)^{n/m}-o(1),
\]
where $\Gamma_n^{\omega'}$ is the event $\Gamma_n$ applied in the environment $\omega'$.  If we let $\gamma'$ be the concatenation of the paths $\hat{\gamma}_0, \gamma_1,\hat{\gamma}_1,\gamma_1,\ldots,\gamma_{n/m},\hat{\gamma}_{n/m}$ then $\gamma'$ is a path joining $(n^{8/10},\frac{W_n}{2})$ to $(n-n^{8/10},\frac{W_n}{2})$.  See Figure~\ref{f:VarianceGoodold} for illustration.  On the event $\{\bigcap_{i=1}^{n/m} \cW_i, \bigcap_{i=1}^{n/m} \cU_i, \bigcap_{i=1}^{n/m} \cV_i,\Gamma_n^{\omega'} \leq n^\epsilon\}$ we have that $\gamma'\subset \Lambda^-_{0,0}$ and
\begin{align*}
X_{\gamma'}^{\omega'} &\leq \sum_{i=1}^{n/m} X_{\gamma_i}^{\omega'} + \sum_{i=0}^{n/m} X_{\hat{\gamma}_i}^{\omega'} + (2\frac{n}{m}+1)\Gamma_n^{\omega'}\\
&\leq \mu (n-2n^{8/10}) - \frac{n}{m} Q_m + (\frac{n}{m}+1)\frac12(\frac{n}{m})^{-1}Q_n + (2\frac{n}{m}+1)n^\epsilon\\
&\leq \mu (n-2n^{8/10}) - 2R^{10} + Q_n + o(Q_n) \geq \mu (n-2n^{8/10}) - R^{10}
\end{align*}
and so  
\[
\left\{ \bigcap_{i=1}^{n/m} \cW_i, \bigcap_{i=1}^{n/m} \cU_i, \bigcap_{i=1}^{n/m} \cV_i,\Gamma_n^{\omega'} \leq n^\epsilon \right\} \subset \cH_{0,0},
\]
which completes the proof.
\end{proof}

We now come to the conclusion of this long series of events. 
\begin{lemma}\label{l:resample.decrease}
On the event $\cG_{i,j}\cap  \cI_{i,j}^c$ we have
\[
X_{\origin,(Mn,0)} - X_{\origin,(Mn,0)}^{\omega^{ij}} \geq 0.
\]
On the event $\cC^{loc}_{i,j}\cap \cH_{i,j}\cap \cG_{i,j}\cap \cB^c \cap \cI_{i,j}$ then
\[
X_{\origin,(Mn,0)} - X_{\origin,(Mn,0)}^{\omega^{ij}} \geq \frac12 R^{10} Q_n.
\]
\end{lemma}
\begin{proof}
For the first part of the lemma, by $\cG_{i,j}^{(4)}$, since $\gamma$ avoids $\Lambda_{i,j}^\star$ we have that
\[
X_{\origin,(Mn,0)} = X_{\gamma}^{\omega^{ij}} \geq X_{\origin,(Mn,0)}^{\omega^{ij}}.
\]
For the second part of the lemma, we assume $\cC^{loc}_{i,j}\cap \cH_{i,j}\cap \cG_{i,j}\cap \cB^c \cap \cI_{i,j}$.  Let   
\[
u_0=\origin, \quad u_1=((i-1)n,y_{i-1}), \quad u_2=((i+2)n,y_{i+2}), \quad u_3=(Mn,0).
\]
By Proposition~\ref{p:path.hit.H},
\[
y_{i-1}, y_{i+2} \in [(j-2R)W_n, (j+2R)W_n]
\]
and so by $\cC^{(3)}_{i,j}$
\[
X_{u_1,u_2} \geq \mu|u_1 - u_2| - R Q_n \geq 3n\mu - R Q_n.
\]
Now define the following points (see Figure~\ref{f:VarianceGood}),
\begin{align*}
v_a &= (in, (j+\frac12)W_n), \quad v_b = (in+n^{8/10}, (j+\frac12)W_n), \\ v_c &= ((i+1)n-n^{8/10}, (j+\frac12)W_n), \quad v_d = ((i+1)n, (j+\frac12)W_n).
\end{align*}
\begin{center}
\begin{figure}
\includegraphics[width=5in]{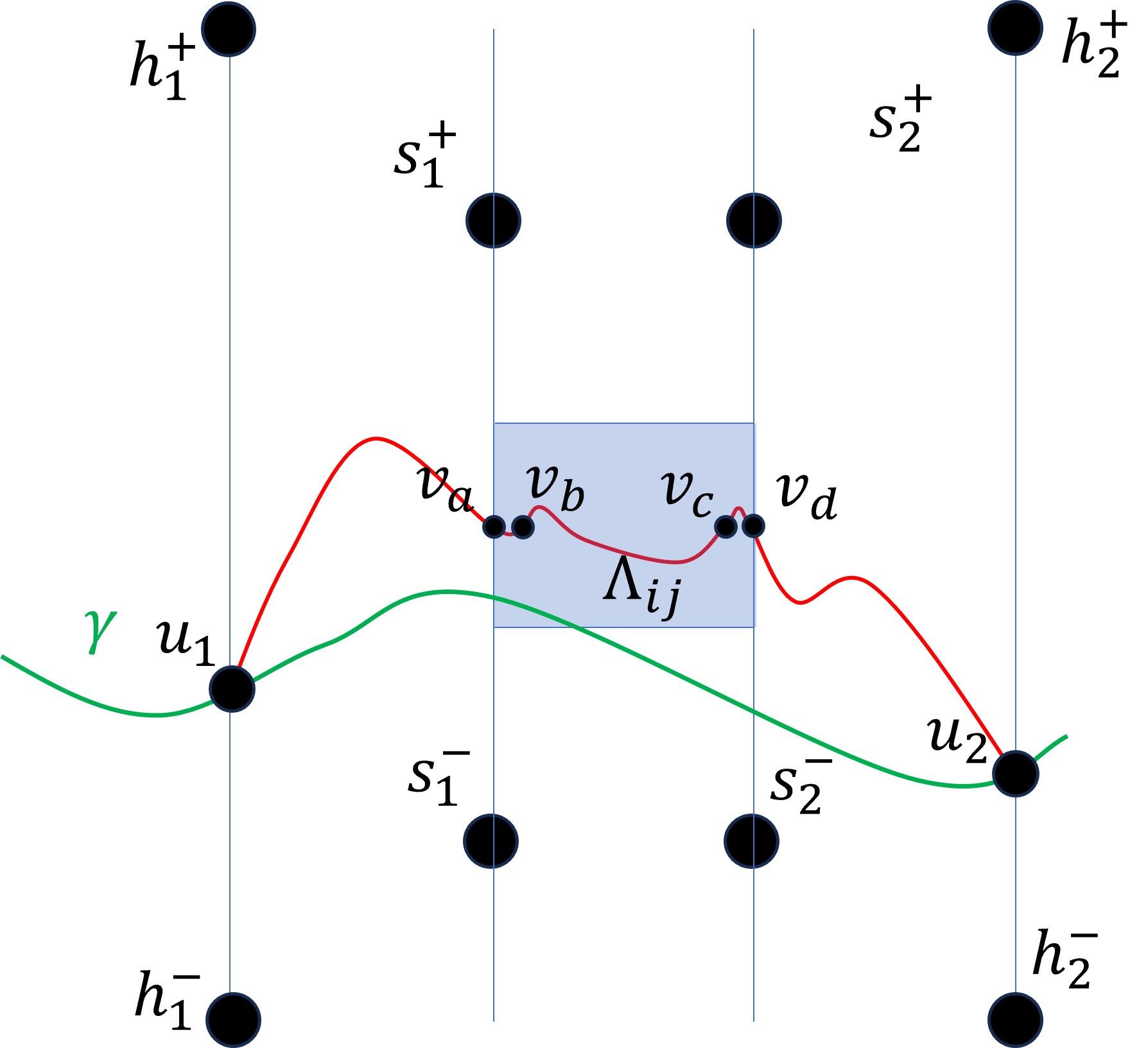}
\caption{Proof of Lemma \ref{l:cH.bound}: we use Proposition \ref{p:left} to show that if $n$ is a record point, with positive probability, paths restricted within an $n\times W_{n}$ rectangles can also have arbitrarily large deviation to the left at scale $Q_{n}$. Thus using a very good path from $v_b$ to $v_c$ in the resampled environment we can show that with positive probability the resampling of $\Lambda_{ij}$ leads to the passage time decreasing significantly at scale $Q_{n}$.}
\label{f:VarianceGood}
\end{figure}
\end{center}
By $\cG_{i,j}^{(5)}$,
\[
X_{v_a,v_b}^{\omega^{ij}}, X_{v_c,v_d}^{\omega^{ij}} \leq n^{8/10}\mu + Q_n
\]
while by $\cH_{i,j}$ and $\cC^{(3)}_{i,j}$,
\[
X_{v_b,v_c}^{\omega^{ij}} \leq (n-n^{8/10})\mu - R^{10} Q_n.
\]
By  $\cC^{(3)}_{i,j}$ and $\cG_{i,j}^{(1)}$,
\[
X_{u_1, v_a}^{\omega^{ij}} \leq \mu |u_1 - v_a| + R Q_n +n^\epsilon \leq \mu (n + 4R^2 Q_n) + R Q_n +n^\epsilon, 
\quad X_{v_d u_3}^{\omega^{ij}} \leq \mu (n + 4R^2 Q_n) + R Q_n +n^\epsilon,
\]
and by $\cG_{i,j}^{(2)}$
\[
X_{u_0, u_1}^{\omega^{ij}} \leq X_{u_0, u_1} - n^\epsilon, \quad X_{u_2, u_3}^{\omega^{ij}} \leq X_{u_2, u_3} - n^\epsilon.
\]
Finally by $\cB^{(5)}$,
\[
X_{\origin,(Mn,0)} \geq X_{u_0,u_1} + X_{u_1,u_2}  + X_{u_2,u_3} - 3n^\epsilon.
\]
Combining the estimates we have that
\begin{align*}
X_{\origin,(Mn,0)} - X_{\origin,(Mn,0)}^{\omega^{ij}} &\geq X_{u_0,u_1} + X_{u_1,u_2}  + X_{u_2,u_3} - 3n^\epsilon\\
&\quad - X_{u_0,u_1}^{\omega^{ij}} - X_{u_1,v_a}^{\omega^{ij}} -  X_{v_a v_b}^{\omega^{ij}} -  X_{v_b v_c}^{\omega^{ij}} -  X_{v_c v_d}^{\omega^{ij}} -  X_{v_d u_2}^{\omega^{ij}} - X_{u_2,u_3}\\
&\geq 3n\mu - R Q_n -2(\mu (n + 4R^2 Q_n) + R Q_n +n^\epsilon) - (n\mu - (R^{10}-2)Q_n) - 7n^\epsilon\\
&\geq \frac12 R^{10} Q_n
\end{align*}
provided $R$ is large enough.
\end{proof}

\subsection{Variance decomposition}
We will estimate the variance by revealing the blocks $\omega_{\Lambda_{i,j}}$ one by one.
\begin{proposition}\label{p:var.decomp}
For $\frac{M}{3} \leq i \leq \frac{2M}3$ and $-M^{3/5}\leq j \leq M^{3/5}$ and for $M,n$ large enough with $n$ an $\alpha'$-record point, if $Q_{Mn} \leq M^{1/10} Q_n$ then
\[
\E\Big[\big(\E[X_{\origin,(Mn,0)} \mid \omega_{\Lambda_{i,j}^{++}}] - \E[X_{\origin,(Mn,0)} \mid \omega_{\Lambda_{i,j}^{++}\setminus \Lambda_{i,j}}]\big)^2\Big] \geq \frac18 R^{20} Q_n^2 \bigg( \big(\delta \P[  I(\cC^{loc}_{i,j}, \cI_{i,j}] - 2M^{-40}\big)^+ \bigg)^2.
\]
\end{proposition}

\begin{proof}
Comparing the variance from revealing versus resampling $\omega_{\Lambda_{i,j}}$ we have that
\begin{align}\label{eq:variance.decompA}
& \E\Big[\big(\E[X_{\origin,(Mn,0)} \mid \omega_{\Lambda_{i,j}^{++}}] - \E[X_{\origin,(Mn,0)} \mid \omega_{\Lambda_{i,j}^{++}\setminus \Lambda_{i,j}}]\big)^2\Big] \nonumber\\
 &\qquad = \frac12\E\Big[\big(\E[X_{\origin,(Mn,0)} - X_{\origin,(Mn,0)}^{\omega^{ij}} \mid \omega_{\Lambda_{i,j}^{++}},\omega'_{\Lambda_{i,j}}]\big)^2\Big]   \nonumber\\
 &\qquad \geq \frac12\E\Big[\big(\E[(X_{\origin,(Mn,0)} - X_{\origin,(Mn,0)}^{\omega^{ij}})I(\cC^{loc}_{i,j}, \cH_{i,j}) \mid \omega_{\Lambda_{i,j}^{++}},\omega'_{\Lambda_{i,j}}]\big)^2\Big]  
\end{align}
where the last inequality follows since $\cC^{loc}_{i,j}$ and $\cH_{i,j}$ are measurable with respect to $\omega_{\Lambda_{i,j}^{++}}$ and $\omega'_{\Lambda_{i,j}}$.
By Jensen's Inequality and Lemma~\ref{l:resample.decrease}
\begin{align}\label{eq:variance.decompC}
&\E\Big[\Big(\E[ (X_{\origin,(Mn,0)} - X_{\origin,(Mn,0)}^{\omega^{ij}}) I(\cC^{loc}_{i,j}, \cH_{i,j},\cG_{i,j}, \cB^c) \mid \omega_{\Lambda_{i,j}^{++}},\omega'_{\Lambda_{i,j}}] \Big)^2\Big] \nonumber\\
&\qquad\geq \Big(\E[ \frac12 R^{10} Q_n I(\cC^{loc}_{i,j}, \cH_{i,j},\cG_{i,j}, \cB^c , \cI_{i,j}) ] \Big)^2 \nonumber\\
&\qquad\geq \frac14 R^{20} Q_n^2\Big(\big(\P[  \cC^{loc}_{i,j}, \cI_{i,j},\cH_{i,j}] - \P[\cG^c_{i,j}\cup \cB]\big)^+ \Big)^2 \nonumber\\
&\qquad\geq \frac14 R^{20} Q_n^2 \Big(\big(\delta \P[  \cC^{loc}_{i,j}, \cI_{i,j}] - 2M^{-90}\big)^+ \Big)^2
\end{align}
where the last inequality follows by Lemma~\ref{l:cB.bound}, Lemma~\ref{l:cG.bound}, Lemma~\ref{l:cH.bound} and the fact that $\cH_{i,j}$ depends only on $\omega'$ and so is independent of $\cC^{loc}_{i,j}\cap \cI_{i,j}$.
By conditional Jensen's Inequality and Cauchy-Schwartz,
\begin{align*}
&\E\Big[\Big(\E[(X_{\origin,(Mn,0)} - X_{\origin,(Mn,0)}^{\omega^{ij}})I(\cC^{loc}_{i,j}, \cH_{i,j})I(\cG^c_{i,j}\cup \cB) \mid \omega_{\Lambda_{i,j}^{++}},\omega'_{\Lambda_{i,j}}]\Big)^2\Big]\\
&\qquad \leq \E\Big[ (X_{\origin,(Mn,0)} - X_{\origin,(Mn,0)}^{\omega^{ij}})^2 I(\cG^c_{i,j}\cup \cB)  \Big]\\
&\qquad \leq \sqrt{\E\big[ (X_{\origin,(Mn,0)} - X_{\origin,(Mn,0)}^{\omega^{ij}})^4\big] \P\big[\cG^c_{i,j}\cup \cB \big]}
\end{align*}
Since $X_{\origin,(Mn,0)}$ and $X_{\origin,(Mn,0)}^{\omega^{ij}}$ are equal in distribution, by Jensen's Inequality,
\begin{align*}
\E\big[ (X_{\origin,(Mn,0)} - X_{\origin,(Mn,0)}^{\omega^{ij}})^4 \big]  
&\leq 8 \E\big[ (X_{\origin,(Mn,0)} - \E[X_{\origin,(Mn,0)}] )^4 \big] + 8 \E\big[ (X_{\origin,(Mn,0)}^{\omega^{ij}} - \E[X_{\origin,(Mn,0)}^{\omega^{ij}}] )^4 \big]\\
&=16 \int_0^\infty 4x^3 \P\big[ |X_{\origin,(Mn,0)} - \E[X_{\origin,(Mn,0)}] |>x \big]dx\\
&\leq 16 \int_0^\infty 4x^3 \exp(1-(x/Q_{Mn})^\theta) dx\\
&\leq C Q_{Mn}^4 \leq C' M^{3} Q_{n}^4
\end{align*}
where the second inequality follows by the definition of $Q_n$ and the last inequality by Lemma~\ref{l:growth34}.  By Lemma~\ref{l:cB.bound}, Lemma~\ref{l:cG.bound} we have that
\[
\P\big[\cG^c_{i,j}\cup \cB \big]\leq 2M^{-90}
\]
so for large enough $M$,
\begin{equation}\label{eq:variance.decompD}
  \E\Big[\Big(\E[(X_{\origin,(Mn,0)} - X_{\origin,(Mn,0)}^{\omega^{ij}})I(\cC^{loc}_{i,j}, \cH_{i,j})I(\cG^c_{i,j}\cup \cB) \mid \omega_{\Lambda_{i,j}^{++}},\omega'_{\Lambda_{i,j}}]\Big)^2\Big] \leq M^{-80} Q_n^2.  
\end{equation}
For two random variables $Z_1,Z_2$ by Cauchy-Schwartz,
\[
\E[(Z_1+Z_2)^2] \geq \E[Z_1^2] + \E[Z_2^2] - 2\sqrt{\E[Z_1^2]} \sqrt{\E[ Z_2^2]} = \bigg(\sqrt{\E[Z_1^2]} - \sqrt{\E[ Z_2^2]}\bigg)^2
\]
so by equations~\eqref{eq:variance.decompC} and~\eqref{eq:variance.decompD} we have that
\begin{align*}
&\E\Big[\big(\E[(X_{\origin,(Mn,0)} - X_{\origin,(Mn,0)}^{\omega^{ij}})I(\cC^{loc}_{i,j}, \cH_{i,j}) \mid \omega_{\Lambda_{i,j}^{++}},\omega'_{\Lambda_{i,j}}]\big)^2\Big]\\
&\geq \bigg( \bigg(\frac12 R^{10} Q_n \big(\delta \P[  \cC^{loc}_{i,j}, \cI_{i,j}] - 2M^{-90}\big)^+ - M^{-40} Q_n\bigg)^+\bigg)^2\\
&\geq \frac14 R^{20} Q_n^2 \bigg( \big(\delta \P[  \cC^{loc}_{i,j}, \cI_{i,j}] - 2M^{-40}\big)^+ \bigg)^2.
\end{align*}
Plugging this into equation~\eqref{eq:variance.decompA} completes the proof.
\end{proof}
We can now complete the main result of this section.
\begin{theorem}\label{t:var.increase}
If $n$ is an $\alpha'$-record point then,
\[
Q_{Mn}\geq M^{1/10} Q_n.
\]
\end{theorem}
\begin{proof}
Suppose that $Q_{Mn} \leq M^{1/10} Q_n$.  By Corollary~\ref{c:path.c.loc.bound} and Lemma~\ref{l:cB.bound}, since $\{J_i=j\}\subset \cI_{i,j}$
\[
\sum_{i=\frac13 M}^{\frac23 M} \sum_{j=-M^{3/5}}^{M^{3/5}} \P[\cC^{loc}_{i,j}, \cI_{i,j}] \geq \sum_{i=\frac13 n}^{\frac23 n} \P[\cC^{loc}_{i,J_i}]-\P[\cB] \geq \frac15 M.
\]
By the Pigeonhole Principle, for some integers $a,b\in\{0,1,\ldots,4\Theta-1\}$ we have that
\begin{equation}\label{eq:Pigeonhole}
\sum_{i=\frac13 M}^{\frac23 M} \sum_{j=-M^{3/5}}^{M^{3/5}} \P[\cC^{loc}_{i,j}, \cI_{i,j}] (i\equiv a, j \equiv b \hbox{ mod }  4\Theta) \geq \frac{1}{80\Theta^2} M.
\end{equation}
Let $(i_1,j_1), (i_2,j_2),\ldots, (i_K,j_K)$ be an ordering of the pairs $(i,j)$ with $i\equiv a, j \equiv b \hbox{ mod }  4\Theta$ and let $\cF_k$ be the filtration
\begin{align*}
\cF_{2k} = \sigma\Big\{ \big(\omega_{\Lambda^{++}_{i_\ell,j_\ell}}\big)_{\ell \leq k} \Big \}, \qquad \cF_{2k+1} = \sigma\Big\{ \big(\omega_{\Lambda^{++}_{i_\ell,j_\ell}}\big)_{\ell \leq k}, \omega_{\Lambda^{++}_{i_{k+1},j_{k+1}}\setminus \Lambda_{i_{k+1},j_{k+1}}} \Big \}.
\end{align*}
Note that by the $4\Theta$ spacing of the $(i_\ell,i_\ell)$ that we have taken each $\Lambda_{i_\ell,i_\ell}$ is disjoint from the other $\Lambda^{++}_{i_{\ell'},j_{\ell'}}$.  Hence, by Proposition~\ref{p:var.decomp} we have that
\begin{align*}
\hbox{Var}(X_{Mn}) &\geq \sum_k \E\Big[\big(\E[X_{Mn}\mid \cF_{2k}] - \E[X_{Mn}\mid \cF_{2k-1}]\big)^2 \Big]\\
&\geq \sum_{k=1}^K \E\Big[\big(\E[X_{Mn}\mid \omega_{\Lambda^{++}_{i_{k},j_{k}}}] - \E[X_{Mn}\mid \omega_{\Lambda^{++}_{i_{k},j_{k}}\setminus \Lambda_{i_{k},j_{k}}}]\big)^2 \Big]\\
&\geq \sum_{k=1}^K  \frac18 R^{20} Q_n^2 \bigg( \big(\delta \P[  I(\cC^{loc}_{i_k,j_k}, \cI_{i_k,j_k}] - 2M^{-40}\big)^+ \bigg)^2
\end{align*}
By the power mean inequality and equation~\eqref{eq:Pigeonhole},
\begin{align*}
&\frac1{K} \sum_{k=1}^K  \frac18 R^{20} Q_n^2 \bigg( \big(\delta \P[  I(\cC^{loc}_{i_k,j_k}, \cI_{i_k,j_k}] - 2M^{-40}\big)^+ \bigg)^2\\
&\qquad\geq \frac18 R^{20} Q_n^2 \Bigg(\frac1{K}\sum_{k=1}^K \big(\delta \P[  I(\cC^{loc}_{i_k,j_k}, \cI_{i_k,j_k}] - 2M^{-40}\big)^+ \Bigg)^2\\
&\qquad\geq \frac18 R^{20} Q_n^2 \frac1{K^2}\Bigg(\frac{\delta}{80\Theta^2} M - 2KM^{-40} \Bigg)^2
\end{align*}
Since $K\leq M^{8/5}$, combining the last two equations we have that
\begin{align*}
\hbox{Var}(X_{Mn}) &\geq \frac18 R^{20} Q_n^2 \frac1{K}\Bigg(\frac{\delta}{80\Theta^2} M - 2KM^{-40} \Bigg)^2\\
&\geq C_1 Q_n^2 M^{2/5}.
\end{align*}
Hence by equation~\eqref{eq:varUpperQn} for large enough $M$,
\[
Q_{Mn}^2 \geq \frac1{C} \hbox{Var}(X_{Mn}) \geq C_2 M^{2/5} Q_n^2 > M^{2/10} Q_n^2.
\]
which completes the result.
\end{proof}

We are now ready show the equivalence of $Q_n$ and the standard deviation.
\begin{proposition}\label{p:qn.SD.equivalence}
There exists $C>1$ such that for all $n\geq 1$,
\[
C^{-1} Q_n \leq \mathrm{SD}(X_n) \leq C Q_n.
\]
\end{proposition}

\begin{proof}
The upper bound was given in equation~\eqref{eq:varUpperQn} so the lower bound remain.  We claim that we can find $n_0,n_1,\ldots,$ such that each $n_i$ is an $\alpha'$-record point and $\frac{n_{i+1}}{n_i} \in [2,M]$ for some large constant $M$.  We may assume that $n_0$ is large enough that Theorem~\ref{t:var.increase} applies.  Suppose that for some $n_i$ in our sequence there is no $\alpha'$-record point in $[2n_i,M n_i]$.  Then
\begin{equation}\label{eq:var.increase.contradiction.assump}
Q_{Mn_i} = \Big(\frac{M}{2}\Big)^{\alpha'} Q_{2n_i} \leq D_{5} 2^{3/4} \Big(\frac{M}{2}\Big)^{\alpha'} Q_{n_i}.
\end{equation}
where the first equation follows by the absence of a record point and the second by Lemma~\ref{l:growth34}.  By Theorem~\ref{t:var.increase} we have that
\[
Q_{Mn_i}\geq M^{1/10} Q_{n_i}.
\]
which contradicts~\eqref{eq:var.increase.contradiction.assump} provided $M$ is large enough.  It follows that for every $n'\geq n_0$ we can find an $\alpha'$ record point $n\in [\frac{n'}{M},n']$.  Then, by Proposition~\ref{p:left}, there exist a constant $\delta>0$ 
\[
\P[X_{n'} < n'\mu - Q_n]\geq \delta.
\]
But since $\E[X_{n'}] \geq n'\mu$,
\[
\hbox{Var}(X_{n'}) = \E[(X_{n'}-\E[X_{n'}])^2] \geq \delta Q_n^2  \geq \delta D_{5}^{-2} M^{-3/2} Q_{n'}^2 
\]
where the last inequality follows by Lemma~\ref{l:growth34}.  Hence
\[
\hbox{SD}(X_n) \geq \delta D_{5}^{-1} M^{-3/4} Q_{n}
\]
for all $n\geq n_0$ and we can pick $C$ large enough so that the proposition holds for $n\in[1,n_0]$ which completes the proof.
\end{proof}

\subsection{Proofs of main theorems}
We can now complete the proof of Theorem \ref{t:tightness} and the upper bounds in Theorems \ref{t:tf} and \ref{t:nr}. 

\begin{proof}[Proof of Theorem~\ref{t:tightness}]
The result follows from Proposition~\ref{p:qn.SD.equivalence} and the definition of $\hat{Q}_n$ since
\[
\P\left[\frac{|X_n - \E X_n|}{\SD(X_n)} > x\right] \leq \P\left[\frac{|X_n - \E X_n|}{\frac1{C}Q_n} > x\right] \leq \P\left[\frac{|X_n - \E X_n|}{\hat{Q}_n} > (x/C)\right] \leq \exp(1- (x/C)^\theta).
\]
\end{proof}

\begin{proof}[Proof of Theorem \ref{t:tf}, upper bound]
Observe that by definition of $W_n$, together with Proposition \ref{p:qn.SD.equivalence} we get that $W_{n}\le C\sqrt{n\mbox{SD}(X_{n})}$ for all $n\ge 1$. The upper bound in Theorem \ref{t:tf} is now immediate from Theorem \ref{t:trans.main}.
\end{proof}

\begin{proof}[Proof of Theorem \ref{t:nr}, upper bound]
This is immediate from Lemma \ref{l:AnBound} and Proposition \ref{p:qn.SD.equivalence}.      
\end{proof}

We have also established Propositions \ref{p:conc} and \ref{p:lowertail}.

\begin{proof}[Proof of Proposition \ref{p:conc}]
This follows from Lemmas \ref{l:YMinusBound} and \ref{l:YPlusBound} together with Proposition \ref{p:qn.SD.equivalence}.
\end{proof}
\begin{proof}[Proof of Proposition \ref{p:lowertail}]
   This follows from Proposition \ref{p:left1} and Proposition \ref{p:qn.SD.equivalence}.  
\end{proof}

\section{Lower bound on transversal fluctuations}
\label{s:lower}

We shall prove the lower bounds in Theorems \ref{t:tf} and \ref{t:nr} in this section.
First we show that large probability, the transversal fluctuation $\mathfrak{W}{n}$ of the geodesic $\gamma_{n}$ from $\mathbf{0}$ to $\mathbf{n}$ is at least of the order $W_{n}=\sqrt{nQ_{n}}$.  We shall prove the following result, which together with Proposition \ref{p:qn.SD.equivalence} clearly implies the lower bound in Theorem \ref{t:tf}. 

\begin{proposition}
    \label{p:tflower}
    Let $\varepsilon\in (0,1)$ be fixed. There exists $\delta>0$ such that for all $n$ sufficiently large we have 
    $$\P(\mathfrak{W}_{n}\le \delta W_{n})\le \varepsilon.$$
\end{proposition}

En route to the proof of Proposition~\ref{p:tflower}, we shall prove the lower bound in Theorem~\ref{t:nr} (see Lemma \ref{l:mean}). We shall first prove a weaker version of Proposition \ref{p:tflower}. 
 Let $\Psi_{n,\delta}$ denote the set of paths
    \[
\Psi_{n,\delta}=\{\gamma':\gamma'(0)=0,\gamma'(1)=n,\sup_t\inf_{x\in[0,n]} |\gamma'(t) - (x,0)|\leq \delta W_n\}
\]

\begin{lemma}
    \label{l:tflowerw}
    There exists $\delta>0$ such that for all large enough $n$,
\begin{equation}\label{eq:narrow.path.lower.bound}
\P[\inf_{\gamma'\in \Psi_{n,\delta}} X_{\gamma'} \geq X_n + \delta Q_n]> \delta
\end{equation}
and in particular,
\[
\P(\mathfrak{W}_n \ge \delta W_{n})\ge \delta.
\]
\end{lemma}

\begin{center}
\begin{figure}
\includegraphics[width=5in]{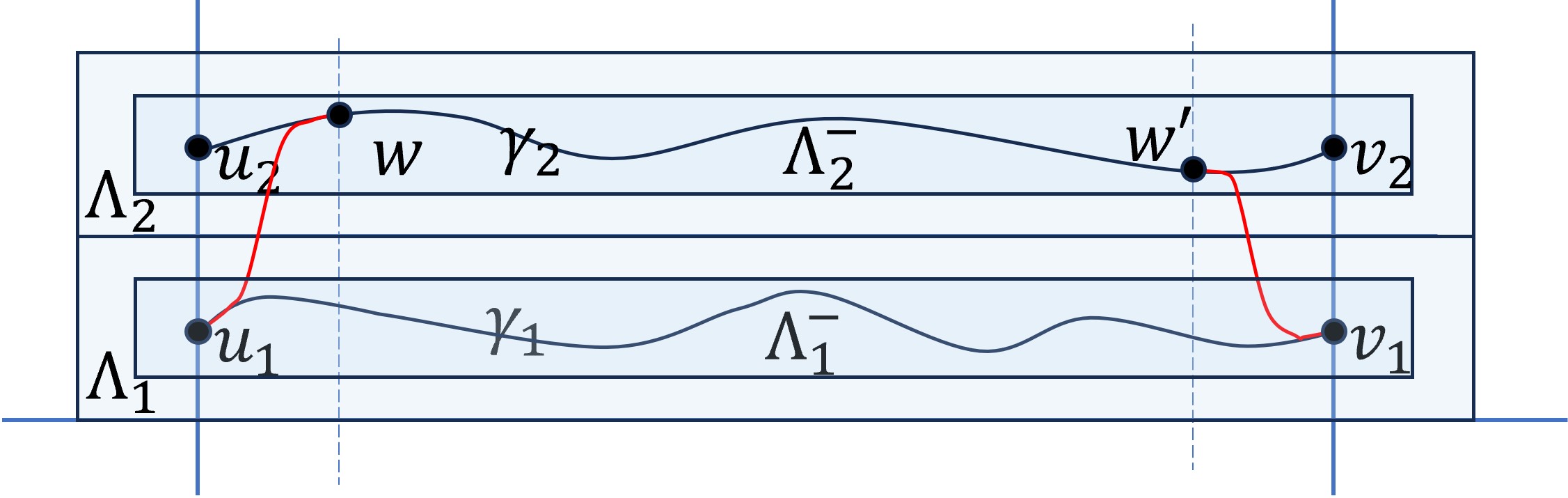}
\caption{Illustration of the construction in Lemma~\ref{l:tflowerw}.}
\label{f:trans.lower}
\end{figure}
\end{center}

\begin{proof}
We can find $\epsilon'$ independent of $n$ such that if $X_n,X_n'$ are two independent copies of $X_n$ then
\[
\P[X_n>X_n'+\epsilon' Q_n]>\epsilon'.
\]
Indeed, one can see this for example from Proposition \ref{p:left} (and Proposition \ref{p:qn.SD.equivalence}) and Lemma~\ref{l:AnBound}.

For $\zeta\in (0,\frac13)$ let $G_\zeta$ be the event
\[
G_{\zeta}=\{|Y_{\cR_{\zeta n}}^- -n\mu|\leq \frac1{10}\epsilon'Q_n\} \bigcap \{|Y_{\cR_{\zeta n}}^+ -n\mu|\leq \frac1{10}\epsilon'Q_n\}
\]
By Lemmas~\ref{l:YMinusBound} and~\ref{l:YPlusBound} we have that
\[
\P[G_{\zeta}^{c}] \leq 2\exp(1-C(\frac1{10}\epsilon'Q_n/Q_{\zeta n})^\theta) \leq 2\exp(1-C(\frac1{10}\epsilon' \zeta^{-\alpha})^\theta)
\]
We can pick $\zeta$, not depending on $n$, small enough such that such that $2\exp(1-C(\frac1{10}\epsilon' \zeta^{-\alpha})^\theta)<\frac1{10}\epsilon'$ so $\P[G_{\zeta}]<\frac1{10}\epsilon'$.  Let $\cR_{\zeta n}^*$ be $\cR_{\zeta n}$ flipped in the line $x=n/2$, that is the rectangle with corners, $(n,0),(n,W_{\zeta n}),((1-\zeta)n,0),((1-\zeta)n,W_{\zeta n})$ and let
\[
G^*_{\zeta}=\{|Y_{\cR^*_{\zeta n}}^- -n\mu|\leq \frac1{10}\epsilon'Q_n\} \bigcap \{|Y_{\cR^*_{\zeta n}}^+ -n\mu|\leq \frac1{10}\epsilon'Q_n\}
\]
so by symmetry $\P[(G^*_{\zeta})^{c}]<\frac1{10}\epsilon'$.  Pick $\delta>0$, independent of $n$, such that $\delta<\frac1{10}\epsilon'$ and $W_{\zeta n} > 10 \delta W_n$.  Suppose that equation~\eqref{eq:narrow.path.lower.bound} does not hold for this choice of $\delta$.

Let $u_1=(0,2\delta W_n),u_2=(0,8\delta W_n)$ and $v_1=(n,2\delta W_n),v_2=(n,8\delta W_n)$.
Let $\Lambda^-_i$ be the rectangle with corners $u_i+(-n,-\delta W_n), u_i+(-n,\delta W_n), v_i+(n,-\delta W_n)$ and $v_i+(n,\delta W_n)$ and let $\Lambda_i$ be the rectangle with corners $u_i+(-2n,-2\delta W_n), u_i+(-2n,2\delta W_n), v_i+(2n,-2\delta W_n)$ and $v_i+(2n,2\delta W_n)$.  Let $\gamma_i$ be a path from $u_i$ to $v_i$ in $\Lambda_i^-$ that minimizes the distance $X_{\gamma_i}$ among paths $\gamma_i\subset \Lambda_i^-$, see Figure~\ref{f:trans.lower} for an illustration.  Let $\gamma_i'$ be a path from $u_i$ to $v_i$ in $\Lambda_i^-$ that minimizes the resampled distance $X_{\gamma_i'}^{\Lambda_i}$ among paths $\gamma_i'\subset \Lambda_i^-$.

Let $A_i$ be the event that
\[
A_i=\{X_{\gamma_i} = X_{\gamma_i}^{\Lambda_i}, X_{\gamma_i'} = X_{\gamma_i'}^{\Lambda_i}\}
\]
which by Assumption~\ref{as:resamp2} satisfies
\[
\P[A_i] \geq 1 -n^{-10}. 
\]
On the event $A_i$ both $\gamma_i$ and $\gamma_i'$ must be optimal for $X$ and $X^{\Lambda_i}$ so we have that $X_{\gamma_i}=X_{\gamma_i'}^{\Lambda_i}$. 
Define that events
\[
B_i=\{X_{u_i v_i}+\delta Q_n \geq X_{\gamma_i}\},\quad E_i=\{X^{\Lambda_i}_{u_i v_i}+\delta Q_n \geq X^{\Lambda_i}_{\gamma'_i}\}.
\]

By our assumption we have that
\[
\P[B_i] \geq 1-\delta,\quad \P[C_i] \geq 1-\delta.
\]
Finally, let $F$ be the event
\[
F=\{X^{\Lambda_2}_{u_2 v_2} \leq X^{\Lambda_1}_{u_1 v_1} - \epsilon' Q_n\}.
\]
Since the $\Lambda_i$ are disjoint we have that $X^{\Lambda_1}$ and $X^{\Lambda_2}$ are independent so
\[
\P[F] \geq \epsilon'.
\]
Let $$H=A_1\cap A_2 \cap B_1\cap B_2 \cap E_1\cap E_2\cap F\cap G_{\zeta} \cap G_{\zeta}^*\cap \{\Gamma_{2n} \leq \frac13 \delta Q_n\}$$ and altogether we have that
\[
\P[H]\geq \P[F] - 2n^{-10}-4\delta - \frac2{10}\epsilon' -o(1) \geq \frac1{10}\epsilon' -o(1) > 0.
\]
So in particular $E$ has a positive probability which we will show is a contradiction.  Let $w$ and $w'$ be the first intersection of $\gamma_2$ with the lines $x=\zeta n$ and $x=(1-\zeta)n$ respectively which by construction are on the sides of $\cR_{\zeta n}$ and $\cR^*_{\zeta n}$.  Consider the length of the path $\hat{\gamma}$ we follows the optimal path from $u_1$ to $w$, follows $\gamma_2$ from $w$ to $w'$ and the optimal path from $w$ to $v_1$.  Then
\begin{align*}
X_{u_1 v_1} &\leq X_{u_1 w} + X_{w w'} + X_{w' v_1}\\
&\leq X_{u_1 w} + X_{w w'} + X_{w' v_1} +(X_{\gamma_2} - X_{u_2 w} - X_{ww'}-X_{w' v_2} +3\Gamma_n) \\
&\leq X_{\gamma_2} + \frac{5}{10} \epsilon' Q_n \\
&= X_{\gamma'_2}^{\Lambda_2} + \frac{5}{10} \epsilon' Q_n \\
&\leq  X_{u_2 v_2}^{\Lambda_2} + \frac{6}{10} \epsilon' Q_n \\
&\leq  X_{u_1 v_1}^{\Lambda_1} - \frac{4}{10} \epsilon' Q_n \\
&\leq  X_{\gamma_1'}^{\Lambda_1} - \frac{4}{10} \epsilon' Q_n \\
&=  X_{\gamma_1} - \frac{4}{10} \epsilon' Q_n \\
&\leq  X_{u_1 v_1} - \frac{3}{10} \epsilon' Q_n 
\end{align*}
where the first line follows by the triangle inequality assumption (Assumption \ref{as:tri}), the second by the definition of $\Gamma_n$, the third by the events $G_{\zeta}, G_{\zeta}^*, \{\Gamma_{2n}\leq \frac13 \delta Q_n\}$, the fourth by $A_2$, the fifth by $E_2$, the sixth by $F$, the seventh by the fact that $\gamma_1'$ joins $u_1$ to $v_1$, the eighth by $A_1$ and the ninth by $B_1$.  This gives a contradiction.
\end{proof}

Before proving Proposition \ref{p:tflower}, we use Lemma \ref{l:tflowerw} to prove the following result which (together with Proposition \ref{p:qn.SD.equivalence}) shows the lower bound in Theorem \ref{t:nr}.

\begin{lemma}
    \label{l:mean}
    There exists $D_4>0$ such that $\E X_{n}\ge n\mu+D_4Q_{n}$ for all $n\ge 1$. 
\end{lemma}

\begin{proof}
Clearly it suffices to prove this for $n$ sufficiently large. Let $\delta>0$ be as in Lemma \ref{l:tflowerw} and let $n$ be sufficiently large so that the conclusion of that Lemma holds. Fix a large integer $M$ and for $i=1,2,\ldots, M$, let $\gamma_{i}$ denote a geodesic from $(i-1)\mathbf{n}$ to $i\mathbf{n}$, and let $\gamma$ denote the path from $0$ to $M\mathbf{n}$ obtained by the concatenation of $\gamma_{i}$s. Let $\gamma'$ denote the optimal path from $0$ to $M\mathbf{n}$ satisfying $\sup_{t}\inf_{x\in [0,Mn]}|\gamma'(t)-(x,0)|\le \delta W_{Mn}$. Since $\delta W_{Mn}\ge \delta M^{1/2+\alpha}W_{n}$ it follows from Theorem~\ref{t:trans.main} that for $M$ sufficiently large depending on $\delta$
\[
\P(\sup_{t}\inf_{x\in [0,Mn]}|\gamma'(t)-(x,0)|\le \delta W_{Mn}) \ge 1-M\exp(1-DM^{\theta(1/2+\alpha)})\ge 1-\delta/2
\]
and so
    $$\P(X_{\gamma}\ge X_{\gamma'})\ge 1-\delta/2.$$ This together with Lemma~\ref{l:tflowerw} (which states $\P(X_{\gamma'}\ge X_{Mn}+\delta Q_{Mn})\ge \delta$)
    implies that 
    $$\P(X_{\gamma}\ge X_{Mn}+\delta Q_{Mn})\ge \delta/2.$$ 
    Since $X_{\gamma}\ge X_{Mn}$ by definition it follows that 
    $\E X_{\gamma}\ge \E X_{Mn}+\delta^2Q_{Mn}/2$. Since $\E X_{\gamma}=M\E X_{n}$ we have that
    $$\E X_{n}\ge \frac1{M}(\E X_{Mn}+ \frac{\delta^2}{2}Q_{Mn})\ge n\mu+\frac{\delta^2}{2M^{1-\alpha}}Q_{n}$$
    completing the proof of the lemma. 
\end{proof}

We shall strengthen Lemma \ref{l:mean} to the following result which will be used in the proof of Proposition \ref{p:tflower}.

\begin{lemma}
    \label{l:meanpara}
    There exists $\delta_0>0$ and $c>0$ such that for all $\delta<\delta_0$ and all $n$ sufficiently large 
    $\E Y^{-}_{R_{n,\delta W_{n}}}\ge n\mu+cQ_{n}$. 
\end{lemma}

\begin{proof}
By Lemma~\ref{l:mean} it suffices to show that for small enough $\delta$,
\begin{equation}
\E|Y^{-}_{R_{n,\delta W_{n}}} - X_n| \leq \frac12 D_4 Q_n.
\end{equation}
Now
\begin{align*}
\E|Y^{-}_{R_{n,\delta W_{n}}} - X_n| &\leq \delta^{3\alpha/7}Q_n + \E\Big[|Y^{-}_{R_{n,\delta W_{n}}} - X_n|I(|Y^{-}_{R_{n,\delta W_{n}}} - X_n|\geq \delta^{3\alpha/7}Q_n)\Big]\\
&\leq \delta^{3\alpha/7}Q_n +  \P[|Y^{-}_{R_{n,\delta W_{n}}} - X_n|\geq \delta^{3\alpha/7}Q_n]^{1/2}\E\Big[|Y^{-}_{R_{n,\delta W_{n}}} - X_n|^2\Big]^{1/2}\\
&\leq \delta^{3\alpha/7}Q_n +C\left(\exp(-\delta^{-\epsilon\theta/4})+\delta^{-1/2}\exp(-cx^{\theta}\delta^{-3\alpha \theta/7})+\exp(-n^{\kappa/100})\right)^{1/2}Q_n\\
&\leq \frac12 D_4 Q_n
\end{align*}
where the second inequality is by Cauchy-Schwartz, the third is by Lemma~\ref{l:thin} and our concentration estimates for passage times across rectangles and the final inequality holds for $\delta>0$ small enough.
\end{proof}

We are now ready to prove Proposition \ref{p:tflower}. 

\begin{proof}[Proof of Proposition \ref{p:tflower}]
Let $\varepsilon>0$ be fixed. Let $C=C(\varepsilon)>0$ be such that 
$\P(X_{n}\ge n\mu+CQ_{n})\le \varepsilon/2$. It therefore suffices to prove that for $\delta$ sufficiently small, 
$$\P(X_{n}\le n\mu+CQ_{n}, \mathfrak{W}_{n}\le \delta W_{n})\le \varepsilon/2.$$ 
For notational convenience, let $B_{\delta}$ denote the event $\{\mathfrak{W}_{n}\le \delta W_{n}\}$. 
Let $\rho>0$ be sufficiently small to be chosen later appropriately such that $\rho^{-1}$ is an integer. For $i=1,2,\ldots, \rho^{-1}$ let $\Lambda_{i,\rho}$ denote the strip $\{(i-1)\rho n \le x \le i\rho n\}$. Define  $$Y^{-}_{i,\delta}:=\inf_{u\in L_{i-1},v\in L_{i}} X_{uv}$$
where $L_{i}$ is the line segment $\{i\rho n\}\times [-\delta W_{n}, \delta W_{n}]$. It follows from Lemma \ref{l:Gamma} that for $n$ sufficiently large,
$$\P\left(X_{n}\le n\mu+CQ_{n}, \sum_{i} Y^{-}_{i,\delta} \ge n\mu+2CQ_{n}, B_{\delta}\right)\le \varepsilon/4.$$
Notice, further that from Assumption \ref{as:resamp1}, it follows that for $n$ sufficiently large
$$\P\left(\sum_{i} Y^{-}_{i,\delta} \le n\mu+2CQ_{n}, \sum_{i} Y^{-,\Lambda_{i,\rho}}_{i,\delta} \ge n\mu+3CQ_{n}\right)\le \varepsilon/8$$
and hence it suffices to show that 
\begin{equation}
    \label{e:lowerred}
    \P\left(\sum_{i} Y^{-,\Lambda_{i,\rho}}_{i,\delta} \le n\mu+3CQ_{n}\right)\le \varepsilon/8.
\end{equation}
To prove \eqref{e:lowerred}, choose $\delta$ sufficiently small compared to $\rho$ such that $\E Y^{-,\Lambda_{i,\rho}}_{i,\delta}\ge \rho n\mu +c Q_{\rho n}$, by Lemma \ref{l:meanpara}. Observe also that $Y^{-,\Lambda_{i,\rho}}_{i,\delta}$ are i.i.d.\ and from Lemma~\ref{l:YMinusBound} $Q_{\rho n}^{-1}|Y^{-,\Lambda_{i,\rho}}_{i,\delta}-\rho n \mu|$ has stretched exponential tails uniformly in $n$. Also, by Lemma \ref{l:growth34}, 
$\rho^{-1}Q_{\rho n}\gg Q_{n}$ for $\rho$ sufficiently small, and hence 
$$ \E \sum_{i} Y^{-,\Lambda_{i,\rho}}_{i,\delta} \ge n\mu+ c\rho^{-1}Q_{\rho n}\geq n\mu+3CQ_{n}+ \frac12 c\rho^{-1}Q_{\rho n}.$$
Notice also, that by standard concentration results for sums of i.i.d. random variables with stretched exponential tails (see, e.g.\ \cite[Proposition 2.1]{GH23}), 
$$\P\bigg[\Big|\sum_{i} Y^{-,\Lambda_{i,\rho}}_{i,\delta}-\E \sum_{i} Y^{-,\Lambda_{i,\rho}}_{i,\delta}\Big| \geq \frac12 c\rho^{-1}Q_{\rho n} \bigg]\leq \varepsilon/8$$
for small enough $\rho$ which establishes~\eqref{e:lowerred} and completes the proof of the proposition.
\end{proof}

\section{Proof of Proposition \ref{p:percevent}}
\label{s:percevent}
We now prove Proposition \ref{p:percevent}. That local events which happen with large probability in typical environments also happen at most locations along the geodesic with large probability has been previously shown in exactly solvable setting~\cite{BSS14, BB23}. Our argument follows the same approach. The basic idea is to use a \emph{percolation argument}, to take a union bound over all possible blocks of paths (the control on the number of such blocks of paths is provided by Corollary \ref{c:percolation}), and control the probability of there being a significant fraction of bad locations along any given block path. The main difference between the argument here and the existing ones in the literature is that the events constituting $\cC_{i,j}$ are not all local, so a multi-scale argument is required specifically to deal with the events $\cC^{(2)}$.

\subsection{Percolation Events}
For convenience of implementing the percolation argument, we shall use proxies for the event $\cC_{i,j}$ which we define as follows.  Define
\begin{align*}
\cD^{(1)}_{i,j}&=\cS^{(n)}_{\frac12 R,(in,(j- R) W_n))}\cap \cO^{(n)}_{\frac12 R,(in,(j- R) W_n))}\cap \cS^{(n)}_{\frac12 R,(in,(j+ R-1) W_n))}\cap \cO^{(n)}_{\frac12 R,(in,(j+ R-1) W_n))},
\end{align*}
and
\begin{align*}
\cD^{(2)}_{i,j}&=\bigcap_{i'=i}^{i+1}\bigcap_{j'=j-R}^{j+R} \cL^{(n),R}_{i',(i+\Theta),j',R/10}\\
&\quad \cap \bigcap_{i'=i}^{i+1}\bigcap_{j'=j-R}^{j+R} \cL^{(n),L}_{i',(i-\Theta),j',R/10}
\end{align*}
We let $\cD^{(3)}_{i,j}=\cC^{(3)}_{i,j}$ and set $\cD_{i,j}=\cD^{(1)}_{i,j}\cap\cD^{(2)}_{i,j}\cap \cD^{(3)}_{i,j}$.  By the transversal fluctuation estimates Lemma \ref{l:trans.events} and Lemma \ref{l:local.trans.proof} $\cD_{i,j}\subset \cC_{i,j}$ for large enough $R$.

We shall prove the following proposition which shall imply Proposition \ref{p:percevent}. 

\begin{proposition}
    \label{p:perceventloc}
    There exists $R$ such that for $M$ sufficiently large and  $n$ sufficiently large (depending on $M$)
\[
\P\bigg[ \sum_{i=1}^{M}  I(\cD_{i,J_i}) \geq \frac{9}{10} M \bigg ] \geq 1-M^{-10}.
\]
\end{proposition}

Recall that $\cD_{i,j}$ is an intersection of three events. Proposition \ref{p:perceventloc} will follow from the following three lemmas which deal with these three events.

\begin{lemma}
    \label{l:d1}
     There exists $R$ such that for $M$ sufficiently large and  $n$ sufficiently large (depending on $M$)
\[
\P\bigg[ \sum_{i=1}^{M}  I(\cD^{(1)}_{i,J_i}) \geq \frac{99}{100} M \bigg ] \geq 1-\frac{1}{3}M^{-10}.
\]
\end{lemma}

\begin{lemma}
    \label{l:d2}
      There exists $R$ such that for $M$ sufficiently large and  $n$ sufficiently large (depending on $M$)
\[
\P\bigg[ \sum_{i=1}^{M}  I(\cD^{(2)}_{i,J_i}) \geq \frac{99}{100} M \bigg ] \geq 1-\frac{1}{3}M^{-10}.
\]
\end{lemma}

\begin{lemma}
    \label{l:d3}
      There exists $R$ such that for $M$ sufficiently large and  $n$ sufficiently large (depending on $M$)
\[
\P\bigg[ \sum_{i=1}^{M}  I(\cD^{(3)}_{i,J_i}) \geq \frac{99}{100} M \bigg ] \geq 1-\frac{1}{3}M^{-10}.
\]
\end{lemma}

Assuming these three lemmas, it is easy to complete the proof of Proposition \ref{p:perceventloc}. 

\begin{proof}[Proof of Proposition \ref{p:perceventloc}]
The proposition follows from Lemmas \ref{l:d1}, \ref{l:d2}, \ref{l:d3} together with a union bound.     
\end{proof}

Before starting with the proofs of Lemmas \ref{l:d1}, \ref{l:d2} and \ref{l:d3} we prove the following result controlling fluctuations of the geodesics which will be repeatedly used throughout this section. 

\begin{lemma}
    \label{l:tauboundgeo}
    There exists $H_0>0$ such that 
\begin{equation}
    \label{e:d3taubound}
     \P\left(\sum_{i}|J_{i}-J_{i-1}|\ge H_0M)\right)\le e^{-cM^{\theta/4}}
\end{equation}
for $M,n$ sufficiently large and some $c>0$. 
\end{lemma}

\begin{proof}
    The proof is an application of Corollary \ref{c:percolation}. We shall use the same argument leading to~\eqref{eq:lowerTailContBound}. As in the argument following \eqref{eq:cJBound}, we define for some large constant $H$ 
    \[
\cZ_{i,k,k'}:=\left(-(Z^{-,\Lambda_i}_{i,k,k'} -n\mu)/Q_n + \frac{(k-k')^2}{32} +H \right)I(|k|\vee |k'|\leq n/W_n).
\]
As argued there, it follows that $\cZ_{i,k,k'}$ satisfy the hypothesis of Corollary \ref{c:percolation}. Observe now that by Lemma \ref{l:YPlusFlexible} and the fact that $Q_{Mn}\le M^{3/4}Q_{n}$ (by Lemma \ref{l:growth34}) we have that for $M$ large enough 
$$\P(X_{Mn}-Mn\mu\ge MQ_{n})\le \exp(-cM^{\theta/4}).$$
We also know, by Theorem \ref{t:trans.main}, that 
$$\P(\max_{i} |J_{i}|\ge n/W_{n})\le e^{-cM^{\theta}}$$
for $n$ sufficiently large. Using Assumption~\ref{as:resamp1} to control $Z^{-}_{i,k,k'}-Z^{-,\Lambda_i}_{i,k,k'}$ as in \eqref{eq:lowerTailContBound} it suffices to show that 
\[
\P\left[\max_{\uk\in \mathfrak{K}_{M} |k_i|\le \frac{n}{W_{n}},\tau_1(\uk)\ge H_0M} \sum_{i=1}^M Z^{-, \Lambda_i}_{i,k_{i-1}, k_{i}} - Mn\mu \le 2MQ_{n} \right] \leq e^{-cM^{\theta/4}}.
\]
Using the definition of $\cZ_{i,k,k'}$, and $M$ sufficiently large this reduces to showing
\[
\P\left[\max_{\uk\in \mathfrak{K}_{M},\tau_1(\uk)\ge H_0M} \sum_{i=1}^M \cZ^{-}_{i,k_{i-1}, k_{i}} -\frac{1}{32}\sum_{i}(k_{i}-k_{i-1})^{2} \ge -(2+H)M \right] \leq e^{-cM^{\theta/4}}.
\]
Observe now that for $\uk$ such that $2^{\ell}H_0 M\le \tau_1(\uk) \le 2^{\ell+1}M$ we have by the Cauchy-Schwarz inequality
$$\sum_{i}(k_{i}-k_{i-1})^{2} \ge  2^{2\ell}H_0^2 M.$$
Therefore, if $H_0$ is sufficiently large (depending on $H$) it follows using Corollary \ref{c:percolation} (for $\lambda=1/64$) that 
\[
\P\left[\max_{\uk\in \mathfrak{K}_{M},2^{\ell+1}H_0M\ge \tau_1(\uk)\ge 2^{\ell}H_0M} \sum_{i=1}^M \cZ^{-}_{i,k_{i-1}, k_{i}} -\frac{1}{32}\sum_{i}(k_{i}-k_{i-1})^{2} \ge -(2+H)M \right] \leq e^{-c(2^{2\ell}M)^{\theta/4}}.
\]
Summing over $\ell$ gives the desired result. 
\end{proof}

\subsection{Proof of Lemma \ref{l:d3}}
For this proof we shall work with the non-backtracking model. Recall  ${\Lambda}_{i}:=[(i-1)n,in]\times \R$ and set $\hat{\Lambda}_{i}:=[(i-1)n,(i+2)n]\times \R$. 
Let us define the event 
\begin{align*}
\widetilde{\cD}^{(3)}_{i,j}&=\bigcap_{i'=i-1}^{i+1}\bigg\{\sup_{\substack{v_1\in\ell_{i'n,(j-2R) W_n,(j+2R) W_n}\\v_2\in\ell_{(i'+1)n,(j- 2R) W_n,(j+2R) W_n}}} \Big|X^{{\Lambda}_{i'+1}}_{v_1 v_2} - \mu|v_1-v_2| \Big| \leq R Q_n/2\bigg\}\\
&\qquad\bigcap\bigg\{\sup_{\substack{v_1\in\ell_{(i-1)n,(j- 2R) W_n,(j+2R) W_n}\\v_2\in\ell_{(i+2)n,(j- 2R) W_n,(j+2R) W_n}}} \Big|X^{\hat{\Lambda}_{i}}_{v_1 v_2} - \mu|v_1-v_2| \Big| \leq R Q_n/2\bigg\}.
\end{align*}

We first need the following easy lemma. 

\begin{lemma}
    \label{l:d3bound}
    Let $\varepsilon>0$ be fixed. Then there exists $R=R(\varepsilon)>0$ such that 
    for all $i,j$, $\P(\widetilde{\cD}^{(3)}_{i,j})\ge (1-\varepsilon)$
    for all $n$ sufficiently large and $T$ sufficiently large. 
\end{lemma}

\begin{proof}
    We shall show that each of the four events whose intersection defines $\widetilde{\cD}^{(3)}_{i,j}$ has probability at least $1-\varepsilon/4$ provided $R$ is sufficiently large; clearly this implies the lemma by a union bound. 

    The proofs for each of the four bounds are almost identical; therefore we shall only provide details in one of the cases; let us show 
\begin{equation}
    \label{e:d3toshow}
    \P\left(\bigg\{\sup_{\substack{v_1\in\ell_{(i-1)n,(j- 2R) W_n,(j+2R) W_n}\\v_2\in\ell_{(i+2)n,(j-2R) W_n,(j+2R) W_n}}} \Big|X^{\hat{\Lambda}_{i}}_{v_1 v_2} - \mu|v_1-v_2| \Big| \leq R Q_n/2\bigg\}\right)\ge 1-\varepsilon/4
\end{equation}
if $R$ is sufficiently large. 

Divide the line segments $\ell_{(i-1)n,(j- 2R) W_n,(j+2R) W_n}$ and $\ell_{(i+2)n,(j- R) W_n,(j+R) W_n}$ into $4R$ line segments of length $W_{n}$ each, denote these line segments locally by $\ell_{1,a}$ and $\ell_{2,b}$ respectively where $a$ and $b$ vary from $1$ to $4R$. Let $A_{a,b}$ locally denote the event that $\Big|X^{\hat{\Lambda}_{i}}_{v_1 v_2} - \mu|v_1-v_2| \Big| \leq R Q_n/2$ for all $v_1\in \ell_{1,a}$ and $v_2\in \ell_{2,b}$. By Lemma \ref{l:paraplusminus} we have for large enough $R$ that $\P(A_{a,b})\ge 1-Ce^{-cR^{\theta}}$ for some $C,c>0$ and all $n$ sufficiently large. By taking a union bound over $a,b$ and choosing $R$ sufficiently large depending on $\varepsilon$, \eqref{e:d3toshow} follows. Lower bounding the probabilities of the other three events follows similarly and together with a union bound completes the proof of the lemma. 
\end{proof}

To complete the proof of Lemma \ref{l:d3} we shall need the following general result. Let $K>0$ be fixed. For $i=1,2,\ldots, T$ and $j\in \Z$ let us consider a family of events $G_{i,j}$ such that the events $G_{i_1,j_1}, G_{i_2,j_2},\ldots, G_{i_k,j_{k}}$ are independent whenever $\min_{a,a'} |i_{a}-i_{a'}|\ge K$. Then we have the following proposition.

\begin{proposition}
    \label{p:percgen1}
    Let $K\ge 1$ be fixed and let $G_{i,j}$ be as above. There exists $\varepsilon=\varepsilon(K)>0$ such that if $\P(G_{i,j})\ge 1-\varepsilon$ for all $i,j$, then there exists $c>0$ such that for all $M$ sufficiently large 
    \[
\P\bigg[ \sum_{i=1}^{M}  I(G_{i,J_i}) \geq \frac{99}{100} M \bigg ] \geq 1-e^{-cM^{\theta/4}}.
\]
\end{proposition}

\begin{proof}
   Let $H_0$ be as in \eqref{e:d3taubound}. Let us now consider the set $\mathcal{J}$ of all sequences 
    $\{j_0,j_1,\ldots, j_{M}\}$ satisfying $j_0=0$ and 
    $\sum_{i} |j_{i}-j_{i-1}|\le H_0M$. An easy counting argument gives that $|\mathcal{J}|\le C^{M}$ for some $C>0$ depending on $H_0$. Fix now a sequence $\{j_0,j_1,\ldots, j_{M}\}\in \mathcal{J}$. 
    By a union bound 
     \[
\P\bigg[ \sum_{i=1}^{M}  I(G_{i,j_i}) \geq \frac{99}{100} M \bigg ] \geq 1-\sum_{k=0}^{K-1} \P(A_{k})
\]
where $A_{k}$ is event 
$$\bigg\{ \sum_{i=1}^{\lfloor M/K \rfloor}  I(G^{c}_{iK+k,j_{iK+k}}) \geq \frac{1}{100K}M \bigg \}.$$
By our assumption on $G_{i,j}$s, the indicators in the event $A_{k}$ are independent for each fixed $K$ and the events all have probability bounded above by $\varepsilon$. Using a Chernoff bound it follows that 
$$\P(A_{k})\leq \exp\Big(-\frac{M}{k}(K\log C + K)\Big)\le C^M e^{-M}$$
for $\varepsilon$ sufficiently small. By a union bound it follows that 
\[
\P\bigg[ \sum_{i=1}^{M}  I(G_{i,j_i}) \geq \frac{99}{100} M \bigg] \geq 1-Ke^{-M}.
\]
It follows that one can choose $\varepsilon$ sufficiently small so that by taking a union bound over all elements of $\mathcal{J}$ we get 
\[
\P\bigg[ \inf_{\mathcal{J}}\sum_{i=1}^{M}  I(G_{i,j_i}) \geq \frac{99}{100} M\bigg ] \geq 1-e^{-M/2}.
\]
This together with \eqref{e:d3taubound} completes the proof of the lemma. 
\end{proof}

We can now complete the proof of Lemma \ref{l:d3}. 

\begin{proof}[Proof of Lemma \ref{l:d3}]
Observe first, by the definition of $\widetilde{\cD}^{(3)}_{i,j}$, $\widetilde{D}^{(3)}_{i,j}$ and $\widetilde{D}^{(3)}_{i',j'}$ are independent if $|i-i'|\ge 3$ and hence Proposition \ref{p:percgen1} applies with $K=3$. Using Proposition \ref{p:percgen1} and Lemma \ref{l:d3bound} it follows that 
\[
\P\bigg[ \sum_{i=1}^{M}  I(\widetilde{\cD}^{(3)}_{i,J_i}) \geq \frac{99}{100} M \bigg ] \geq 1-e^{-cM^{\theta/4}}.
\]
Let $\mathcal{A}$ denote the event that there exists $1\le i \le M$ and $j$ with $|j|\le M$ such that $\widetilde{\cD}^{(3)}_{i,j}\cap (\cD^{(3)}_{i,j})^{c}$ holds. Let $\mathcal{B}$ denote the event that there exists $1\le i \le M$ such that $|J_{i}|\ge M$. It is clear that on $\mathcal{A}^{c}\cap \mathcal{B}^{c}$ we have $\cD^{(3)}_{i,J_{i}}\supset \widetilde{\cD}^{(3)}_{i,J_i}$ for all $i$ and hence 
  \[
\P\bigg[ \sum_{i=1}^{M}  I({\cD}^{(3)}_{i,J_i}) \geq \frac{99}{100} M \bigg ] \geq 1-e^{-cM^{\theta/4}}-\P(\mathcal{A})-\P(\mathcal{B}).
\]  
Now by the first resampling hypothesis Assumption \ref{as:resamp1} and a union bound  we have $\P(\mathcal{A})\le M^{2}\exp(-cn^{\epsilon})\le M^{-100}$. By Theorem \ref{t:trans.main} (and the fact that $W_{Mn}\le M^{7/8}W_{n}$) it follows that $\P(\cB)\le M^{-100}$. The lemma follows. 
\end{proof}

\subsection{Proof of Lemma \ref{l:d1}}
This proof will be divided into two parts, one dealing with the $\mathcal{S}$ events in the definition of $\cD^{(1)}_{i,j}$, the others dealing with the $\mathcal{O}$ events.
For notational convenience, let us write 

$$\cD^{(1,1)}_{i,j}=\cS^{(n)}_{R/2,(in,(j- R) W_n))}\cap \cS^{(n)}_{R/2,(in,(j+ R-1) W_n))};$$

$$\cD^{(1,2)}_{i,j}=\cO^{(n)}_{R/2,(in,(j- R) W_n))}\cap \cO^{(n)}_{R/2,(in,(j+ R-1) W_n))}.$$

We have the following two lemmas 

\begin{lemma}
   \label{l:d1s}
   For each fixed $M$ sufficiently large, and $n$ sufficiently large depending on $M$,
   $\P(\exists 1\le i \le M : I(\cD^{(1,1)}_{i,J_{i}})=0)\le M^{-100}$.
\end{lemma}

\begin{lemma}
    \label{l:d1o}
    For each fixed $M$ sufficiently large, and $n$ sufficiently large depending on $M$,
    \[
\P\bigg[ \sum_{i=1}^{M}  I(\cD^{(1,2)}_{i,J_i}) \geq \frac{99}{100}M \bigg ] \geq 1-M^{-100}.
\]
\end{lemma}

Postponing the proof of the above two lemmas momentarily, we complete the now straightforward proof of Lemma \ref{l:d1}.

\begin{proof}[Proof of Lemma \ref{l:d1}]
Since $\cD^{(1)}_{i,j}=\cD^{(1,1)}_{i,j}\cap \cD^{(1,2)}_{i,j}$, Lemma \ref{l:d1} follows from Lemmas \ref{l:d1s} and \ref{l:d1o} by a union bound. 
\end{proof}

We now prove Lemma \ref{l:d1s}. 

\begin{proof}[Proof of Lemma \ref{l:d1s}]
    Let us define the events 
    \[
     \cA =\bigcup_{i=1}^M\{I(\cD^{(1,1)}_{i,J_{i}})=0\},\quad \cB = \bigcap_{i=1}^M \bigcap_{|j|\leq M}\{I(\cD^{(1,1)}_{i,j})=1\},\quad \mathcal{C}=\bigcap_{i=1}^M\{|J_{i}|\le M \}.
    \]
Notice now that by Lemma \ref{l:trans.SOGam} we know that if $n$ is sufficiently large, then for all $i,j$ 
$$\P(\cD^{(1,1)}_{i,j})\ge 1-n^{-100}.$$
Therefore by taking a union bound over $i\in [1,M]$ $j\in [-M,M]$ it follows by choosing $n$ sufficiently large compared to $M$ that 
$$\P(\cB)\ge 1-M^{1000}.$$

Next, since $W_{Mn}\le M^{7/8}W_{n}$ it follows from Theorem \ref{t:trans.main} that for $M$ sufficiently large 
$$\P(\mathcal{C})\ge 1-M^{-1000}.$$
Since $\cA\subset \cB^{c}\cup \mathcal{C}^{c}$, the lemma follows by a union bound. 
\end{proof}

Finally, we give the proof of Lemma \ref{l:d1o}. This proof is similar to the proof of Lemma \ref{l:d3} using  Proposition \ref{p:percgen1}. We shall define events $G_{i,j}$ which satisfy the hypothesis of Proposition \ref{p:percgen1}. 

Recall the event $\cH^{(n)}_{x,y,z}$ from Section \ref{s:trans}. Define the event $\widetilde{\cH}^{(n)}_{x,y,z}$ where all passage times $X_{u_1,u_2}$ in the definition of $\cH^{(n)}_{x,y,z}$ are replaced by $X_{u_1,u_2}^{\Lambda_{x+1}}$ where $\Lambda_{x+1}$ denotes the strip $[xn,(x+1)n]\times \R$. 

Define 
$$\widetilde{\cO}^{(n)}_{\tfrac{R}{2},(in,tW_{n})}=\bigcap_{j=0}^{3\epsilon \log_2 n}\bigcap_{x=0}^{2^{j}-1}\bigcap_{y=-\tfrac{R}{2}W_{n}/W_{n2^{-j}}}^{\tfrac{R}{2}W_{n}/W_{n2^{-j}}} \widetilde{\cH}^{(n2^{-j})}_{i+x2^{-j},tW_{n}/W_{n2^{-j}}+y,\tfrac{R}{2}2^{\theta j}\theta^{2}/2000}.$$
and set 
$$\widetilde{\cD}^{(1,2)}_{i,j}= \widetilde{\cO}^{(n)}_{\tfrac{R}{2},(in,(j- R) W_n))}\cap \widetilde{\cO}^{(n)}_{\tfrac{R}{2},(in,(j+ R-1) W_n))}.$$

We have the following lemma. 

\begin{lemma}
    \label{l:d1oind}
    Given $\varepsilon>0$, there exists $R$ sufficiently large such that 
    $\P(\widetilde{\cD}^{(1,2)}_{i,j})\ge 1-\varepsilon$ for all $i,j$.  
\end{lemma}

\begin{proof}
    By the same argument as in the proof of Lemma \ref{l:trans.SOGam} we get that if $R$ is sufficiently large that  
    $$\P(\widetilde{\cO}^{(n)}_{\tfrac{R}{2},(in,(j- R) W_n))})\ge 1-\varepsilon/2;~~\P(\widetilde{\cO}^{(n)}_{z,(in,(j- R) W_n))})\ge 1-\varepsilon/2$$
    and the lemma follows by a union bound. 
\end{proof}

Notice that by definition the events $\widetilde{D}^{(1,2)}_{i,j}$ are independent across $i$, therefore setting $G_{i,j}=\widetilde{D}^{(1,2)}_{i,j}$ satisfies the hypothesis of Proposition \ref{p:percgen1} with $K=1$. Using Proposition \ref{p:percgen1} we immediately have the following lemma. 

\begin{lemma}
    \label{l:d1oind2}
    If $R$ is sufficiently large, then 
    $$\P\left(\sum_{i=1}^{M} I(\widetilde{\cD}^{(1,2)}_{i,J_{i}})\ge \frac{99}{100}M\right)\ge 1- e^{-cM^{\theta/4}}$$
    for some $c>0$. 
\end{lemma}

We also need the following lemma. 

\begin{lemma}
    \label{l:d1oerr}
    For each $i,j$, $\P(\widetilde{\cD}^{(1,2)}_{i,j}\cap (\cD^{(1,2)}_{i,j})^{c})\le n^{-10}$.
\end{lemma}

\begin{proof}
    Let us fix $j\in [0,3\epsilon\log_2 n], x\in [0,2^{j}-1]$ and $y\in [-\tfrac{R}{2}W_{n}/W_{n2^{-j}},\tfrac{R}{2}W_{n}/W_{n2^{-j}}]$. It follows from Assumption~\ref{as:resamp1} that for $n$ sufficiently large we have 
    $$\P\left(({\cH}^{(n2^{-j})}_{i+x2^{-j},tW_{n}/W_{n2^{-j}}+y,\tfrac{R}{2}2^{\theta j}\theta^2/1000})^{c}\cap \widetilde{\cH}^{(n2^{-j})}_{i+x2^{-j},tW_{n}/W_{n2^{-j}}+y,\tfrac{R}{2}2^{\theta j}\theta^2/2000}\right)\le n^{-100}.$$
    Now taking a union bound over all $y$ (at most $2^{j}$ terms), all $x$ ($2^{j}$ terms) and finally over all $j$ between $0$ and $3\epsilon\log_2 n$ the result follows from the definitions of $\widetilde{\cD}^{(1,2)}_{i,j}$ and ${\cD}^{(1,2)}_{i,j}$.
\end{proof}

We can now complete the proof of Lemma \ref{l:d1o}. 

\begin{proof}[Proof of Lemma \ref{l:d1o}]
   Let $\mathcal{A}$ denote the event that for all $1\le i \le M$ and for all $1\le j \le M$ we have $I(\widetilde{\cD}^{(1,2)}_{i,j}\cap (\cD^{(1,2)}_{i,j})^{c})=0$. Further, let $\cB$ denote the event that $|J_{i}|\le M$ for each $i$. From Lemma \ref{l:d1oerr}, it follows that 
   $\P(\cA)\ge 1-M^{-1000}$ and by transversal fluctuations estimates as in the proof of Lemma \ref{l:d3}, $\P(\cB)\ge 1-M^{-1000}$ for $M$ sufficiently large. Since on $\cA\cap \cB$
   $$\sum_{i=1}^{M} I({\cD}^{(1,2)}_{i,J_{i}})\ge \sum_{i=1}^{M} I(\widetilde{\cD}^{(1,2)}_{i,J_{i}}),$$ 
   the result follows by a union bound from Lemma \ref{l:d1oind2} for $M$ sufficiently large. 
\end{proof}

\subsection{Proof of Lemma \ref{l:d2}}
This proof is similar to the previous two, but more complicated. We define the following events.

$$\cD^{(2,1)}_{i,j}=\bigcap_{j'=j-R}^{j+R} \cL^{(n),R}_{in,(i+\Theta)n,j'W_n,R/10};$$

$$\cD^{(2,2)}_{i,j}=\bigcap_{j'=j-R}^{j+R} \cL^{(n),R}_{(i+1)n,(i+\Theta)n,j'W_n,R/10};$$

$$\cD^{(2,3)}_{i,j}=\bigcap_{j'=j-R}^{j+R} \cL^{(n),L}_{in,(i-\Theta)n,j'W_n,R/10}$$

$$\cD^{(2,4)}_{i,j}=\bigcap_{j'=j-R}^{j+R} \cL^{(n),L}_{(i+1)n,(i-\Theta)n,j'W_n,R/10}.$$

Proof of Lemma \ref{l:d2} will follow from the following four lemmas together with a union bound.

\begin{lemma}
    \label{l:d21}
    There exists $R$ sufficiently large such that for all large $M$ we have 
    $$\P\left(\sum_{i=1}^{M} I({\cD}^{(2,1)}_{i,J_{i}})\ge \frac{999}{1000}M\right)\ge 1- M^{-100}$$
    for all $n$ sufficiently large.
\end{lemma}

\begin{lemma}
    \label{l:d22}
    There exists $R$ sufficiently large such that for all large $M$ we have 
    $$\P\left(\sum_{i=1}^{M} I({\cD}^{(2,2)}_{i,J_{i}})\ge \frac{999}{1000}M\right)\ge 1- M^{-100}$$
    for all $n$ sufficiently large. 
\end{lemma}

\begin{lemma}
    \label{l:d23}
    There exists $R$ sufficiently large such that for all large $M$ we have 
    $$\P\left(\sum_{i=1}^{M} I({\cD}^{(2,3)}_{i,J_{i}})\ge \frac{999}{1000}M\right)\ge 1- M^{-100}$$
    for all $n$ sufficiently large. 
\end{lemma}

\begin{lemma}
    \label{l:d24}
    There exists $R$ sufficiently large such that for all large $M$ we have 
    $$\P\left(\sum_{i=1}^{M} I({\cD}^{(2,4)}_{i,J_{i}})\ge \frac{999}{1000}M\right)\ge 1- M^{-100}$$
    for all $n$ sufficiently large. 
\end{lemma}

By symmetry, Lemmas \ref{l:d21} and \ref{l:d22} will imply Lemmas \ref{l:d23} and \ref{l:d24}, so we shall only prove Lemmas \ref{l:d21} and \ref{l:d22}. {The proof of these two lemmas will be almost identical} so we shall need to prove only Lemma \ref{l:d21}.

The main difference between Lemma \ref{l:d21} and the previously proved Lemma \ref{l:d3} and Lemma \ref{l:d1o} is that the events ${\cD}^{(2,1)}_{i,j}$ does not have finite range of dependence across $i$, and hence Proposition \ref{p:percgen1} cannot directly be applied. 
Recall the definition of $\cL^{(n),R}_{i,(i+\Theta),j,R/10}$. Since $\Theta$ is an integer power of 2 it is easy to check that the definition simplifies to the following:

\begin{align*}
\cL^{(n),R}_{i,i+\Theta,j,R/10}&=\bigcap_{\ell=0}^{\ell_{\max}+1} \bigcap_{w=-(R/10)2^{\ell}W_n/W_{n2^{\ell}}}^{(R/10)2^{\ell}W_n/W_{n2^{\ell}}} \cH^{(2^{\ell} n)}_{i2^{-\ell},jW_n/W_{n2^{\ell}}+w,\frac{R2^{(\ell \wedge \ell_{\max})/20}}{100000} } \\
&\qquad \bigcap_{\ell=0}^{\ell_{\max}} \cS^{(n2^\ell)}_{R/10,(in,jW_n)}
\end{align*}
where $\ell_{\max} = \lfloor \log_2 \Theta -1\rfloor$.

As before, we shall divide the event $\cD_{i,j}^{(2,1)}$ into two parts. For $\ell\in [0,\ell_{\max}+1]$ we set 

$$\cD^{(2,1,1)}_{i,j,\ell}=\bigcap_{j'=j-R}^{j+R} \bigcap_{w=-(R/10)2^{\ell}W_n/W_{n2^{\ell}}}^{(R/10)2^{\ell}W_n/W_{n2^{\ell}}} \cH^{(2^{\ell} n)}_{i2^{-\ell},j'W_n/W_{n2^{\ell}}+w,\frac{R2^{(\ell\wedge \ell_{\max})/20}}{100000}}.$$

We also set 

$$\cD^{(2,1,2)}_{i,j}=\bigcap_{\ell=0}^{\ell_{\max}} \bigcap_{j'=j-R}^{j+R} \cS^{(n2^\ell)}_{R/10,(in,j'W_n)}.$$

Lemma \ref{l:d21} shall follow from the following two lemmas.

\begin{lemma}
    \label{l:d21l}
    There exists $R$ such that the following holds for all $M$ sufficiently large. For each $\ell\in [0,\ell_{\max}+1]$ 
 $$\P\left(\sum_{i=1}^{M} I(({\cD}^{(2,1,1)}_{i,J_{i},\ell})^{c})\ge \frac{(\ell \vee 1)^{-100}}{10000} M\right)\le M^{-1000}$$
 for all $n$ sufficiently large. 
\end{lemma}

\begin{lemma}
\label{l:d213}
    There exists $R$ such that the following holds for all $M$ sufficiently large: for all $n$ large enough we have 
 $$\P\left(\sum_{i=1}^{M} I(({\cD}^{(2,1,2)}_{i,J_{i}})^{c})\ge \frac{1}{10000} M\right)\le M^{-1000}.$$ 
\end{lemma}

Postponing the proof of these lemmas let us first complete the proof of Lemma \ref{l:d21}, which  at this point is quite straightforward.

\begin{proof}[Proof of Lemma \ref{l:d21}]
    For $\ell\in [0,\ell_{\max}+1]$, let $A_{\ell}$ be the set of indices $i\in [1,M]$ such that $I(D^{(2,1,1)}_{i,J_i,\ell})=0$, and let $A=\cup_{\ell}A_\ell$. Since $1+\sum_{\ell=1}^{\infty} \ell^{-100}<5$ it follows that if $|A|\ge \frac{1}{2000}M$ then there must exist $\ell\in [0,\ell_{\max}+1]$ such that  
    $$\sum_{i=1}^{M} I(({\cD}^{(2,1,1)}_{i,J_{i},\ell})^{c})\ge \frac{(\ell \vee 1)^{-100}}{10000} M.$$
    Notice further that if 
    $$\sum_{i=1}^{M}I(({\cD}^{(2,1)}_{i,J_{i}})^{c})\ge \frac{1}{1000}M$$
    then by definition, either $|A|\ge \frac{1}{2000}M$ or 
$$\sum_{i=1}^{M} I(({\cD}^{(2,1,2)}_{i,J_{i}})^{c})\ge \frac{1}{10000}M.$$    By Lemmas \ref{l:d21l} and \ref{l:d213} together with a union bound it follows that 
 $$\P\left(\sum_{i=1}^{M}I(({\cD}^{(2,1)}_{i,J_{i}})^{c})\ge \frac{1}{1000}M\right)\le (\ell_{\max}+3)M^{-1000}.$$
    Since $\ell_{\max}\le  C\log \log M$ for some $C$ (depending on the model parameter $\epsilon$), this completes the proof by choosing $M$ sufficiently large. 
    \end{proof}

It remains to prove Lemmas \ref{l:d21l} and \ref{l:d213}. 

\begin{proof}[Proof of Lemma \ref{l:d213}]
    Notice first that by Lemma \ref{l:trans.SOGam} if $n$ is sufficiently large we know that for each $i,j'$ and for each $\ell\in [0,\ell_{\max}]$ we have
    $$\P(\cS^{(n2^\ell)}_{R/10,(in,j'W_n)})\ge 1-n^{-100}.$$
    Taking a union bound over $j'=j-R, \ldots, j+R$ and over $0\le \ell \le \ell_{\max}$ (and using $\ell_{\max}=O(\log \log M)$) it follows that for all $i,j$ we have 
    $$\P(({\cD}^{(2,1,2)}_{i,j})^{c})\le n^{-99}$$
    for all $n$ sufficiently large. The proof can now be completed by arguing exactly as in the proof of Lemma \ref{l:d1s}, taking a union bound over all $1\le i \le M$ and $j\in [-M,M]$ and using our control of the transversal fluctuations. We omit the details. 
\end{proof}

The final remaining piece of the argument is the proof of Lemma \ref{l:d21l}. To reduce notational overhead, we shall, without loss of generality, assume $1\le \ell \le \ell_{\max}$. The reader can easily check that the proofs for the cases $\ell=0$ and $\ell=\ell_{\max}+1$ will follow by the same argument with minor adjustments. Recall the definition of $\cD^{(2,1,1)}_{i,j,\ell}$ for $1\le \ell \le \ell_{\max}$: 

$$\cD^{(2,1,1)}_{i,j,\ell}=\bigcap_{j'=j-R}^{j+R}\bigcap_{w=-(R/10)2^{\ell}W_n/W_{n2^{\ell}}}^{(R/10)2^{\ell}W_n/W_{n2^{\ell}}} \cH^{(2^{\ell} n)}_{i2^{-\ell},j'W_n/W_{n2^{\ell}}+w,\frac{R2^{\ell/20}}{100000}}.$$

The first step of the proof is to replace the passage times by the resampled passage times so that the events will be independent for $i$s  which are well separated. Recall from the definition of $\cH^{(n)}_{(x,y,z)}$ that the event $\cD^{(2,1,1)}_{i,j,\ell}$ depends on passage times between pairs of points one of which is on the line $x=in$ and the other on $x=(i+2^{\ell})n$. Let us consider the event 
$\widetilde{\cH}^{(2^{\ell} n)}_{i2^{-\ell},j'W_n/W_{n2^{\ell}}+w,\frac{R2^{(\ell\wedge \ell_{\max})/20}}{200000}}$ which is defined identically to $\mathcal{H}$, except with the resampled passage times $X^{\Lambda_{i,\ell}}_{u_1,u_2}$ where $\Lambda_{i,\ell}$ denotes the strip $x\in [in,(i+2^{\ell})n]$. Let us define 
$$\widetilde{\cD}^{(2,1,1)}_{i,j,\ell}=\bigcap_{j'=j-R}^{j+R}\bigcap_{w=-(R/10)2^{\ell}W_n/W_{n2^{\ell}}}^{(R/10)2^{\ell}W_n/W_{n2^{\ell}}} \widetilde{\cH}^{(2^{\ell} n)}_{i2^{-\ell},j'W_n/W_{n2^{\ell}}+w,\frac{R2^{\ell/20}}{200000}}.$$

We have the following basic bound which is an easy consequence of Lemma \ref{l:paraBounds} and a union bound over $j'$ and $w$. 

\begin{lemma}
    \label{l:htildebound}
    There exists $C,C'>0$ such that for each $\ell$ and each $1\le i\le M,j\in \Z$, 
    $$\P((\widetilde{\cD}^{(2,1,1)}_{i,j,\ell})^{c})\le CR^22^{\ell}\exp(-C'(R2^{\ell/20})^{\theta})$$
    for some $C,C'>0$. 
\end{lemma}

We can now complete the proof of Lemma \ref{l:d21l}.

\begin{proof}[Proof of Lemma \ref{l:d21l}]
    The proof will consist of the following two steps. 

    \textbf{Step 1:} First we shall show that there exists $R$ large enough such that for $M$ sufficiently large and all $n$ sufficiently large we have 
$$\P\left(\sum_{i=1}^{M}I(\widetilde{\cD}^{(2,1,1)}_{i,J_{i},\ell})^{c}\ge \frac{\ell^{-100}}{10000}M \right)\le M^{-1001}.$$

\textbf{Step 2:}
Then we shall show that the probability that there exists $1\le i \le M$, $|j|\le M$ such that $\widetilde{\cD}^{(2,1,1)}_{i,j,\ell})^{c}$ holds but $\cD^{(2,1,2)}_{i,j,\ell}$ does not is at most $n^{-10}$. Step 2 is actually an immediate consequence of the first resampling hypothesis as has been argued a number of times already (e.g.\ proof of Lemma \ref{l:d1oerr}). 

Given Step 1 and Step 2, the proof can be completed as argued in the proof of Lemma \ref{l:d3}, we shall omit the details to avoid repetition. It remains to show Step 1. 

The proof of Step 1 will be similar to the proof of Proposition \ref{p:percgen1} except that we need to keep track of the range of the dependence. As in the proof of Proposition \ref{p:percgen1}, we now consider the set $\mathcal{J}$ of all sequences 
    $\{j_0,j_1,\ldots, j_{M}\}$ satisfying $j_0=0$ and 
    $\sum_{i} |j_{i}-j_{i-1}|\le H_0 M$ where $H_0$ is such that (as in \eqref{e:d3taubound}) 
    $$\P\left(\sum_{i=1}^{M} |J_{i}-J_{i-1}|\ge H_0M\right)\le \exp(-cM^{\theta}).$$

Let $\cA_{\mathcal{J}}$ denote the event that there exists $\mathsf{J}=\{j_0,\ldots, j_{M}\}\in \mathcal{J}$ such that 
$$\sum_{i=1}^{M}I(\widetilde{\cD}^{(2,1,1)}_{i,J_{i},\ell})^{c}\ge \frac{\ell^{-100}}{10000}M.$$
Let $\cB$ denote the event that $J=\{J_{0},\ldots, J_{M}\}\notin \cJ$. Clearly the required probability is upper bounded by $\P(\cA_{\cJ})+\P(\cB)$ and by our choice of $H_0$ it suffices to show (for $M$ sufficiently large) that 

$$\P\left(\max_{\mathsf{J}\in \cJ}\sum_{i=1}^{M}I(\widetilde{\cD}^{(2,1,1)}_{i,j_{i},\ell})^{c}\ge \frac{\ell^{-100}}{10000}M \right)\le M^{-1002}.$$

Recall that $2^{\ell}\le 2^{\ell_{\max}}$ and hence is bounded above by a polynomial of $\log M$. Let us set $a=\lceil 2^{-\ell} M \rceil -1$. 
For $s=1,2,, \ldots 2^{\ell}$, let $\cJ_{s}$ denote the set of all sequences 
$\{j_{s}, j_{2^{\ell}+s}, \ldots, j_{a2^{\ell}+s}\}$ (assuming, without loss of generality, that $a2^{\ell}+s\le M$, otherwise the final term of the sequence will be $j_{(a-1)2^{\ell}+s})$ such that 
$$|j_{s}|+\sum_{i=1}^{a} |j_{i2^{\ell}+s}-j_{(i-1)2^{\ell}+s}|\le H_0 M.$$
Clearly, by the triangle inequality, if $\mathsf{J}=\{j_0,\ldots, j_{M}\}\in \cJ$ then its restriction 
$\mathsf{J}_{s}:=\{j_{s}, j_{2^{\ell}+s}, \ldots, j_{a2^{\ell}+s}\}\in \cJ_{s}$. 
Let $\cA_{s}$ denote the event that 
$$\max_{\mathsf{J}_{s}\in \cJ_{s}}\sum_{i=0}^{a}I(\widetilde{\cD}^{(2,1,1)}_{i2^{\ell}_s,j_{i2^{\ell+s}},\ell})^{c}\ge \frac{\ell^{-100}}{10000}2^{-\ell}M.$$
Clearly, $\cA\subseteq \cup_{s}\cA_{s}$ and hence it suffices to show that for each $s$, $\P(\cA_{s})\le 2^{-\ell}M^{-1002}$. 

Fix $s$ and $\mathsf{J}_{s}\in \cJ_{s}$. Notice that, by definition, the events $\{\widetilde{\cD}^{(2,1,1)}_{i2^{\ell}_s,j_{i2^{\ell+s}},\ell}\}_{i=0}^{a}$ are mutually independent. Recall Chernoff's inequality: for independent $0-1$ valued random variables $X_{i}$ with $\E X_{i}\le \mu$, we have 
$$\P(\sum X_{i}\ge (1+\delta)\mu)\le \exp(-c\mu(1+\delta)\log (1+\delta))$$
for some $c>0$. Using these and Lemma \ref{l:htildebound}, for a fixed $\mathsf{J}_{s}\in \cJ_{s}$ we have
$$\P\left(\sum_{i=0}^{a}I(\widetilde{\cD}^{(2,1,1)}_{i2^{\ell}_s,j_{i2^{\ell+s}},\ell})^{c}\ge \frac{\ell^{-100}}{10000}2^{-\ell}M \right) \le \exp\left(-c\ell^{-100}2^{-\ell}M \log \biggl(c'\ell^{-100}R^{-2}2^{-\ell}\exp(c''R^{\theta}2^{\ell\theta/20})\biggr)\right)$$
for some constants $c,c',c''>0$. Now if $R$ (depending on $c',c''$) is chosen sufficiently large we get for all $1\le \ell \le \ell_{\max}$,
$$\log (c'\ell^{-100}R^{-2}2^{-\ell}\exp(c''R^{\theta}2^{\ell\theta/20})\ge R^{\theta/100}2^{\ell\theta/100}.$$
It follows that 
\begin{equation}
 \label{e:rlarge}
\P\left(\sum_{i=0}^{a}I(\widetilde{\cD}^{(2,1,1)}_{i2^{\ell}_s,j_{i2^{\ell+s}},\ell})^{c}\ge \frac{\ell^{-100}}{10000}2^{-\ell}M \right) \le \exp\left(-c2^{-\ell}M R^{\theta/200}2^{\ell\theta/200}\right)
\end{equation}
Now we need to take a union bound over all $\mathsf{J}_{s}\in \cJ_{s}$, and for this we need to bound $|\cJ_{s}|$. An easy counting argument as in the proof of Proposition \ref{p:percgen1} shows that there exists $C,C',C''>0$ such that for all $1\le\ell \le \ell_{\max}$ and all $1\le s\le 2^{\ell}$ $|\cJ_{s}|\le 2^{a+1} (C2^{\ell})^{a+1}\le C'\exp(C''\ell 2^{-\ell} M)$. 
Choosing $R$ sufficiently large (notice that $R$ only needs to be large to deal with the case when $\ell$ is small, for larger values of $\ell$ the term $2^{\ell\theta/200}$ on the RHS of \eqref{e:rlarge} is sufficient, and hence $R$ can be chosen large independently of $\ell$) and using \eqref{e:rlarge} then gives us for some $c>0$, all $\ell$ and all $s$, 
$$\P\left(\sum_{i=0}^{a}I(\widetilde{\cD}^{(2,1,1)}_{i2^{\ell}_s,j_{i2^{\ell+s}},\ell})^{c}\ge \frac{\ell^{-100}}{10000}2^{-\ell}M \right)\le \exp(-c2^{-\ell}M).$$
The proof is completed by noting that $\ell_{\max}=O(\log \log M)$ and choosing $M$ sufficiently large.  
\end{proof}

\section{Proof of Proposition \ref{p:perc1}}
\label{s:appa}

This section is devoted to the proof of the percolation estimate Proposition \ref{p:perc1}. Clearly, by replacing $\mathcal{Z}_{i,k_{i},k_{i}}$ by its positive part it suffices to only consider non-negative random variables. 
Also, by adjusting the constant $C_3$, it suffices to prove the result for $R$ sufficiently large. For the rest of the proof let $R$ sufficiently large be fixed. Let $C_4>0$ also be sufficiently large which will be fixed later independent of $R$. Let us also fix $z$ sufficiently large. 

\begin{center}
\begin{figure}
\includegraphics[width=5in]{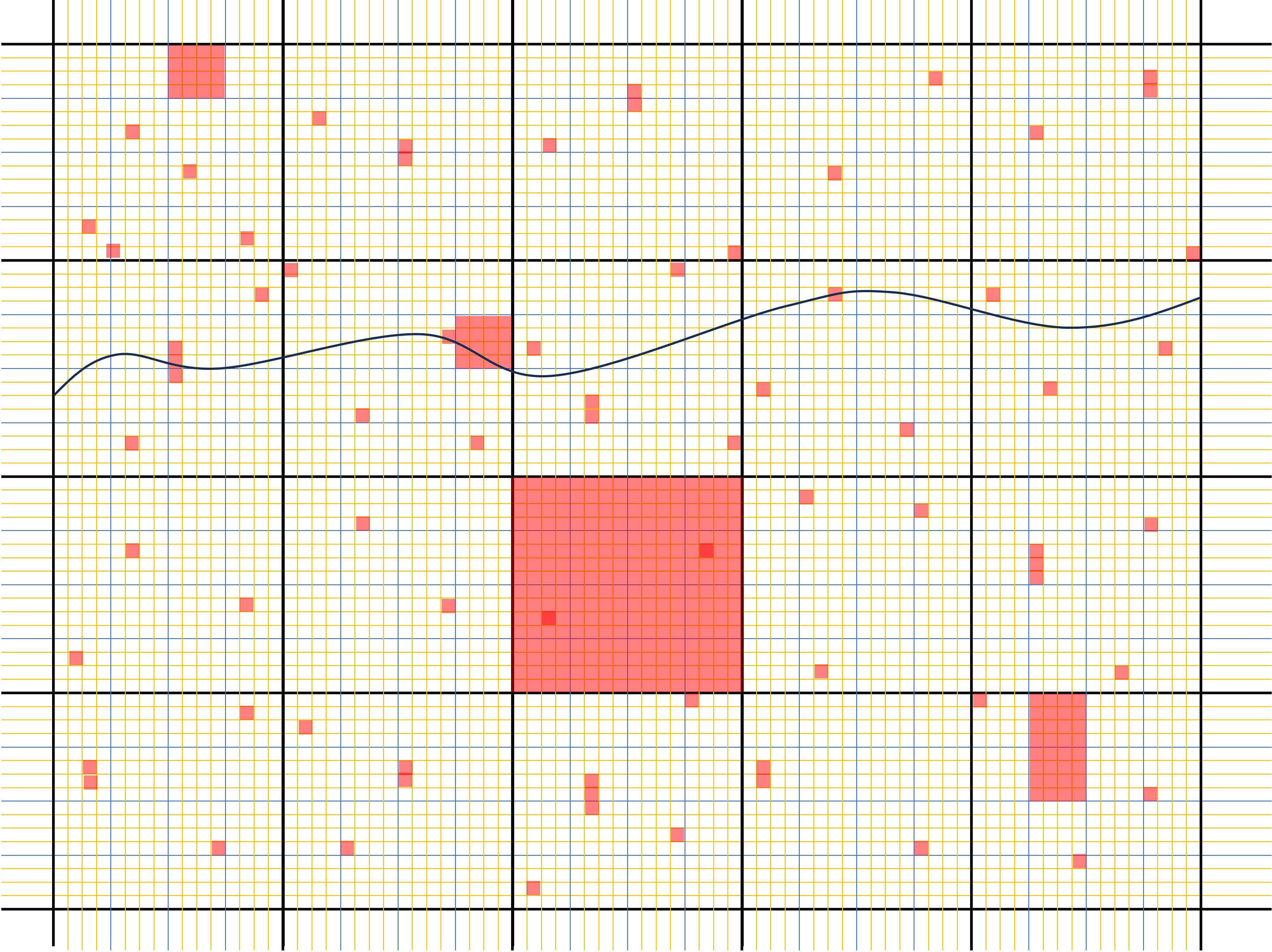}
\caption{In the proof of Proposition \ref{p:perc1} we divide the plane into boxes on a series of dyadic scales.  In this figure, the black, blue and orange lines denote progressively finer mesoscopic scales and bad boxes are marked in red.  Bad boxes are independent in different columns and their density decreases rapidly with the scale.  Our percolation argument shows that no path $\gamma$ with $\tau_1(\uk)\leq R M$ will spend a large fraction of its time in the red region.}
\label{f:meso}
\end{figure}
\end{center}

The first step is to  truncate $\mathcal{Z}_{i,k_{i-1},k_{i}}$ into dyadic intervals. For $j\ge 1$, set $\mathcal{Z}^{j}_{i,k,k'}=2^{j+1}I(\mathcal{Z}_{i,k,k'}\in [2^{j},2^{j+1}])$. Therefore we have,
\begin{equation}
    \label{e:zsum}
    \max_{\underline{k}}\sum_{i=1}^{M} \mathcal{Z}_{i,k_{i-1},k_i}\le M+\sum_{j\ge 1} \max_{\underline{k}}\sum_{i} \mathcal{Z}^{j}_{i,k_{i-1},k_i}.
\end{equation}
We shall bound the terms 
$$\max_{\underline{k}}\sum_{i} \mathcal{Z}^{j}_{i,k_{i-1},k_i}$$
separately for three different regimes of $j$. Let us set $j_{\max}=\lceil \log_2 ((\log^{C_4}(RM)+z)\wedge z_{\max})\rceil$ and $j_{\min}=\lfloor \log_2 (\log^{C_4/2}(R)) \rfloor$. 
Notice that, we have, deterministically, 
\begin{equation}
    \label{e:jmin}
    \sum_{j<j_{\min}} \max_{\underline{k}}\sum_{i} \mathcal{Z}^{j}_{i,k_{i-1},k_i}\le \log^{3C_4/4}(R)M
\end{equation}
for $R$ sufficiently large. 

For $j>j_{\max}$ we have the following lemma. 

\begin{lemma}
    \label{l:jmax}
    There exist $C_5,C_6>0$ depending only on $C_1$, $C_2$ and $\beta$ such that 
    $$\P\left(\sum_{j> j_{\max}} \max_{\underline{k}}\sum_{i} \mathcal{Z}^{j}_{i,k_{i-1},k_i}>0\right)\le C_5\exp(-C_6z^{\beta}-C_6(z/z_{\max})z_{\max}^{\beta}).$$
\end{lemma}

\begin{proof}
    Since $k_0=0$, it is easy to see that there exists a deterministic set $\mathsf{A}=\mathsf{A}_{R,M}$ of triples $(i,k_{i-1},k_{i-1})$ of size at most $4RM^3$ such that for any $\underline{k}\in \mathfrak{k}_{M}$ with $\tau_1(\underline{k})\le RM$ we have $(i,k_{i-1},k_{i})\in \mathsf{A}$. Now clearly, for $j>j_{\max}$,
    $$\P\left(\max_{\underline{k}}\sum_{i} \mathcal{Z}^{j}_{i,k_{i-1},k_i}>0\right)\le \P(\max_{(i,k_{i-1},k_{i})\in \mathsf{A}} \mathcal{Z}_{i,k_{i-1},k_i}>2^{j})\le 4RM^{3}C_1\exp(-C_22^{j\beta}).$$
    Since $j>j_{\max}$, either $2^{j}>z_{\max}$ or $z_{\max}\ge 2^{j}\ge \log^{C_4}(RM)+z$. In the former case, the above upper bound can be improved to $0$ (since $\mathcal{Z}_{i,k_{i-1},k-i}\le z_{\max}$) and in the latter case we get $2^{j\beta}\ge z^{\beta} \ge (z/z_{\max})z_{\max}^{\beta}$ where the final inequality is true since $z_{\max}\ge z$ and $\beta\le 1$. Therefore in this case,
    $$4RM^{3}C_1\exp(-C_22^{j\beta})\le C\exp(-c2^{j\beta}-c(z/z_{\max})z^{\beta}_{\max})$$
    for some $C,c>0$ depending only on $C_1,C_2$ and $\beta$. Notice that the $RM^3$ term can be absorbed in the exponent since $2^{j}\ge \log^{C_4}(RM)$ and $C_4$ is sufficiently large. 
    Summing this probability bound for $j> j_{\max}$ completes the proof of the lemma. 
\end{proof}

The next lemma deals with the case $j_{\min} \le j \le j_{\max}$.

\begin{lemma}
    \label{l:minmax}
    There exist constant $C_3,C_4,C_5,C_6>0$ depending only on $C_1,C_2$ and $\beta$ such that 
    $$\P\left( \sum_{j=j_{\min}}^{j_{\max}} \max_{\underline{k}} \sum_i \mathcal{Z}^{j}_{i,k_{i-1},k_{i}} \ge (C_3+\frac{1}{2}\log^{C_4}(R))M+z \right)\le C_5\exp(-C_6z^{\beta}-C_6(z/z_{\max})z_{\max}^{\beta}).$$
\end{lemma}

Postponing the proof of Lemma \ref{l:minmax} momentarily, let us complete the proof of Proposition \ref{p:perc1}. 

\begin{proof}[Proof of Proposition \ref{p:perc1}]
The proof follows from \eqref{e:zsum}, together with \eqref{e:jmin}, Lemma \ref{l:jmax} and Lemma \ref{l:minmax} by taking $R$ sufficiently large. 
\end{proof}

\subsection{Proof of Lemma \ref{l:minmax}}
The proof of this lemma is somewhat long so we shall divide it into a number of steps. An illustration for this argument is given in Figure \ref{f:meso}. 

\noindent
\textbf{Step 1: Mesoscopic coarsening of $\underline{k}$:}
To reduce the entropy of the the number of sequences $\underline{k}$ we shall consider the following discretization, which we call mesoscopic coarsening (or $j$-mesoscopic coarsening).  
Fix $j_{\min}\leq j\leq j_{\max}$. Let $w=w_j$ be a fixed and large integer, to be chosen appropriately later. 

For $\underline{k}\in \mathfrak{K}_{M}$ define 
$\underline{k}^{w}$ by $k^{w}_{i}= \lfloor \frac{k_{iw}}{w^2}\rfloor$ for $i=1,2,\ldots \lfloor\frac{M}{w}\rfloor$. Define $\underline{k}_{\rm large}$ by 
$$\{i: 1\le i \le M: |k_{i}-k_{i-1}|\ge w\},$$
i.e., the set of locations where $\underline{k}$ has a large jump. Finally, let 
$$\underline{k}^1_{\rm large}:=\{k_{i}:i\in \underline{k}_{\rm large}~\text{or}~(i+1)\in \underline{k}_{\rm large}\}.$$ 
The mesoscopic coarsening of $\underline{k}$ is given by the triple 
$$\underline{k}_{\rm M}^{w}=(\underline{k}^w,\underline{k}_{\rm large},\underline{k}^1_{\rm large}).$$ 

We need the following estimate to count the number of distinct $\underline{k}_{\rm M}^{w}$ as $\underline{k}$ varies over all $\underline{k}\in \mathfrak{K}_{M}$ with $\tau_1(\underline{k})\le RM$. 

\begin{lemma}
\label{l:mesocount}
Let $\mathsf{M}^{w}_{R}$ denote the set of all $\underline{k}_{\rm M}^{w}$
for $\underline{k}\in \mathfrak{K}_{M}$ with $\tau_1(\underline{k})\le RM$. Then there exists a constant $c>0$ such that 
$$|\mathsf{M}^w_{R}|\leq  \exp \left(cRM\log w/w\right).$$
\end{lemma}

\begin{proof}
    First, let us fix the size of $\underline{k}_{\rm large}$ to be $s$; denote the corresponding subset of $\mathsf{M}^{w}_{R}$ by $\mathsf{M}^{w}_{R}(s)$. Clearly, $s\le RM/w$ since $\tau_1(\underline{k})\le RM$.

    To bound $|\mathsf{M}^w_{R}(s)|$ for $s\le RM/w$, first fix the elements of $\underline{k}_{\rm large}$; clearly there are $\binom{M}{s}$ many choices. Then we fix the sizes of big jumps, i.e., ${k}_i-k_{i-1}$ for all $i\in \underline{k}_{\rm large}$. Since the sum total of the absolute values of these jumps can be at most $RM$, by a standard counting argument, the number of choices here is at most $2^s \binom {RM +s}{s}$ (the factor $2^{s}$ comes from the fact that each jump can be either positive or negative). We still need to determine $k_{i-1}$s for $i\in \underline{k}_{\rm large}$ and $\underline{k}^w$; we determine the number of choices for them together. Observe that if $k^{w}_{i_0}$ is fixed for some $i$ then traversing the block $[i_0w,(i_0+1)w]$ from left to right, we can determine $k_{i-1}$ for for every $i\in i\in \underline{k}_{\rm large}$, such that $i-1\in [i_0w,(i_0+1)w]$ up to an error of $\pm 2w^2$. Further, $k^{w}_{i_0+1}$ is also determined (using the information from the value of the large jumps fixed earlier) up to an error of $\pm 1$. Therefore, the total number of choices we have for this is at most $3^{M/w}(4w^2)^s$.   

    Putting all these together, we get 
\begin{eqnarray*}
|\mathsf{M}^{w}_{R}(s)|  &\leq & 8^s3^{M/w}\binom{M}{s} \binom{RM+s}{s} w^{2s}\\
& \leq & \exp \biggl( c(s+M/w +s\log (M/s)+s\log (RM/s +1)+2s\log w)\biggr)\\
& \leq & \exp \left(cRM \log w/w\right)
\end{eqnarray*}
for some constant $c>0$ where in the last inequality we have used the fact that $s\le RM/w$. Summing over $s$ from $0$ to $RM/w$ we get the required result. 
\end{proof}

\noindent
\textbf{Step 2: Bounding the tails for a fixed $j\in [j_{\min},j_{\max}]$:}
The usefulness of defining the mesoscopic coarsening is shown in the following lemma. 

\begin{lemma}
\label{l:fixedj}
Let $j\in [j_{\min}, j_{\max}]$ be fixed. Fix $\underline{k}^*\in \mathsf{M}^w_{R}$. Then, 
$$\max_{\underline{k}:\underline{k}_{\rm M}^{w}=\underline{k}^*} \sum_{i} \mathcal{Z}^j_{i,k_{i-1},k_{i}}$$ 
is stochastically dominated by $2^{j+1}\mathrm{Bin}(M,8w^3C_1e^{-C_2(2^{j\beta})})$. 
\end{lemma}

\begin{proof}
Clearly, it suffices to show that for $i=1,2,\ldots, M$, 
$\max_{\underline{k}:\underline{k}_{\rm M}^{w}=\underline{k}^*}I(\mathcal{Z}^j_{i,k_{i-1},k_{i}}\ne 0)$ are stochastically dominated by independent $\mbox{Ber}(8w^3C_1e^{-C_2(2^j\beta)})$ variables.

Once $\underline{k}_{\rm M}^{w}=\underline{k}^*$, is fixed, for each $i$, the choice of $(k_{i-1},k_{i})$ is fixed in a deterministic set of size at most $8w^3$. To see this, notice that if $i\in \underline{k}_{\rm large}$ then both $k_{i-1}$ and $k_{i}$ are determined and hence the claim follows. If $i\notin \underline{k}_{\rm large}$, then given $\underline{k}^{w}$, $\underline{k}_{\rm large}$ and 
    $\underline{k}^1_{\rm large}$, $k_{i-1}$ is determined up to an error of $4w^2$, and fixing this $k_{i}$ can take one of the $2w$ possible values since $|k_{i}-k_{i-1}|\le w$. 

Therefore, using the tail hypothesis on $\mathcal{Z}_{i,k_{i-1},i}$ for each $i$, 
    $$\P(\max_{\underline{k}:\underline{k}_{\rm M}^{w}=\underline{k}^*}\mathcal{Z}^j_{i,k_{i-1},k_{i}}>0)\le 8w^3C_1e^{-C_22^{j\beta}}.$$
    Since $\mathcal{Z}_{i,k_{i-1},i}$ are independent across $i$, the claim and hence the lemma follows. 
 \end{proof}

 The next lemma provides a tail bound for the above sum for a fixed $j$. At this point let us fix $w=w_{j}=\exp(c_02^{j\beta})$ for some $c_0>0$ sufficiently small. In particular, we choose $w$ so that $8w^{3}C_1e^{-C_22^{j\beta}}\le w^{-10}$. Also, let $\varepsilon>0$ be fixed  such that $\beta+\varepsilon<1$. For $z>0$, let us set $z_j=2^{-j}M+(z+\log^{C_4}(RM))2^{\varepsilon(j-j_{\max})}$. For notational convenience, we shall denote the term $(z+\log^{C_4}(RM))$ by $\tilde{z}$.

\begin{lemma}
    \label{l:minmaxj}
    There exists $c>0$, such that for $j\in [j_{\min},j_{\max}]$,
    $$\P\left(\max_{\underline{k}:\tau_1(\underline{k})\le RM} \sum_{i} \mathcal{Z}^j_{i,k_{i-1},k_{i}} \ge z_j\right)\le \exp\left(-\frac{c\tilde{z}2^{(\beta+\varepsilon-1)(j-j_{\max})}}{2^{(1-\beta)j_{\max}}}\right).$$ 
\end{lemma}

Postponing the proof of Lemma \ref{l:minmaxj} for now, we first complete the proof of Lemma \ref{l:minmax}. 

\noindent
\textbf{Step 3: Completing the proof of Lemma \ref{l:minmax}:}
First note the it suffices to prove the lemma for $z,R$ and $M$ sufficiently large. Notice that for $R$ and $M$ sufficiently large
$$\sum_{j=j_{\min}}^{j_{\max}} z_{j} \le M+C_7(z+\log^{C_4}(RM))\le M+\frac{1}{2}\log^{C_4}(R)M+C_7z$$
for some $C_7\in (0,\infty)$ since $\varepsilon>0$. 
Therefore, it suffices to upper bound 
$$\sum_{j=j_{\min}}^{j_{\max}} \P\left(\max_{\underline{k}:\tau_1(\underline{k})\le RM} \sum_{i} \mathcal{Z}^j_{i,k_{i-1},k_{i}} \ge z_j\right)$$
at which point we use Lemma \ref{l:minmaxj}. Notice that, since $\beta+\varepsilon<1$, it follows that the while summing the upper bound in Lemma \ref{l:minmaxj} from $j=j_{\min}$ to $j_{\max}$ the dominating term is the term corresponding to $j=j_{\max}$. Therefore we get, 
 $$\P\left( \sum_{j=j_{\min}}^{j_{\max}} \max_{\underline{k}} \sum_i \mathcal{Z}^{j}_{i,k_{i-1},k_{i}} \ge (1+\frac{1}{2}\log^{C_4}(R))M+C_7z \right)\le C\exp \left(-\frac{c\tilde{z}}{2^{j_{\max}(1-\beta)}}\right).$$
 Now recall that $2^{j_{\max}}=z_{\max}\wedge \tilde{z}$. If $2^{j_{\max}}=z_{\max}\le \tilde{z}$, we have 
  $$ \frac{\tilde{z}}{2^{j_{\max}(1-\beta)}} \ge (\tilde{z}/z_{\max})z_{\max}^{\beta}\ge \tilde{z}^{\beta}.$$
  On the other hand if $2^{j_{\max}}=\tilde{z}\le z_{\max}$, we have 
  
   $$ \frac{\tilde{z}}{2^{j_{\max}(1-\beta)}}\ge \tilde{z}^{\beta}\ge (\tilde{z}/z_{\max})z_{\max}^{\beta}.$$
   The proof of the lemma is completed by observing $\tilde{z}\ge z$. 
   \qed

We finish by proving Lemma \ref{l:minmaxj}.

\noindent 
\textbf{Step 4: Proof of Lemma \ref{l:minmaxj}:}
Using Lemma \ref{l:fixedj} and the hypothesis on $w$, taking a union bound over all choices of $\underline{k}^{w}_{\rm M}$, and using Lemma \ref{l:mesocount}, it follows that 
$$\P\left(\max_{\underline{k}:\tau_1(\underline{k})\le RM} \sum_{i} \mathcal{Z}^j_{i,k_{i-1},k_{i}} \ge z_j\right)\le \exp(cRM\log w/w)\P\left(\mbox{Bin}(M,w^{-10})\ge 2^{-(j+1)}z_j\right).$$

By our choice of $w$, again, since $j\ge j_{\min}$ is sufficiently large, we also have $2^{-(j+1)}z_j\ge 2Mw^{-10}$. Therefore, a Chernoff bound gives

\begin{align*}
 &~ \exp(cRM\log w/w)\P\left(\mbox{Bin}(M,w^{-10})\ge 2^{-(j+1)}z_j\right)\\ & \le  \exp(cRM\log w/w)\exp\left(-c'2^{-(j+1)}z_j \log (w^{10}M^{-1}2^{-(j+1)}z_{j})\right)\\
  &\le  \exp\left(\frac{cRM\log w}{w}- \frac{c'M\log(w^{10}2^{-(2j+1)})}{2^{2j+1}}-\frac{c'\tilde{z}2^{\varepsilon(j-j_{\max})}\log(w^{10}2^{-(2j+1)})}{2^{j+1}} \right)
\end{align*}

Since $j\ge j_{\min}$, $2^{j\beta}\ge \log^{C_4\beta/2}(R)$. By choosing $C_4$ sufficiently large depending on $\beta$ and $c_0$ this guarantees that 
$w=e^{c_02^{j\beta}}$, $\sqrt{w}\ge R$. Using this together with the fact that $w$ grows doubly exponentially in $j$ and $j$ is sufficiently large (since $j\ge j_{\min}$) and $R$ is sufficiently large it follows that 
$$\frac{c'M\log(w^{10}2^{-(2j+1)})}{2^{2j+1}} \ge \frac{cRM\log w}{w}$$
and also $w^{10}2^{-(2j+1)}\ge w$.
Hence the bound above can be further upper bounded by 

$$\exp\left(-\frac{c'c_0\tilde{z}2^{\varepsilon(j-j_{\max})}2^{j\beta}}{2^{j+1}}\right)\le \exp\left(-\frac{c\Tilde{z}2^{(j-j_{\max})(\beta+\varepsilon-1)}}{2^{j_{\max}(1-\beta)}}\right).$$
This completes the proof of the lemma. 
\qed

\bibliography{fpp}

\begin{thebibliography}{10}

\bibitem{AH16}
D.~Ahlberg and C.~Hoffman.
\newblock Random coalescing geodesics in first-passage percolation.
\newblock Preprint, arXiv:1609.02447.

\bibitem{Ale11}
Kenneth~S. Alexander.
\newblock Sub{G}aussian rates of convergence of means in directed first passage
  percolation.
\newblock Preprint arXiv 1101.1549.

\bibitem{Ale93}
Kenneth~S. Alexander.
\newblock {A Note on Some Rates of Convergence in First-Passage Percolation}.
\newblock {\em The Annals of Applied Probability}, 3(1):81 -- 90, 1993.

\bibitem{Ale97}
Kenneth~S. Alexander.
\newblock {Approximation of subadditive functions and convergence rates in
  limiting-shape results}.
\newblock {\em The Annals of Probability}, 25(1):30 -- 55, 1997.

\bibitem{Ale20}
Kenneth.~S. Alexander.
\newblock Geodesics, bigeodesics, and coalescence in first passage percolation
  in general dimension.
\newblock {\em arXiv preprint arXiv:2001.08736}, 2020.

\bibitem{Ale21}
Kenneth.~S. Alexander.
\newblock Uniform fluctuation and wandering bounds in first passage
  percolation.
\newblock {\em arXiv preprint arXiv:2011.07223}, 2020.

\bibitem{ADH13}
Antonio Auffinger, Michael Damron, and Jack Hanson.
\newblock Rate of convergence of the mean for sub-additive ergodic sequences.
\newblock {\em Advances in Mathematics}, 285:138--181, 2015.

\bibitem{ADH15}
Antonio Auffinger, Michael Damron, and Jack Hanson.
\newblock {\em 50 years of first-passage percolation}, volume~68.
\newblock American Mathematical Soc., 2017.

\bibitem{BBB23}
M\'arton Bal\'azs, Riddhipratim Basu, and Sudeshna Bhattacharjee.
\newblock Geodesic trees in last passage percolation and some related problems.
\newblock {\em arXiv preprint arXiv:2308.07312}, 2023.

\bibitem{BB21}
Riddhipratim Basu and Manan Bhatia.
\newblock Small deviation estimates and small ball probabilities for geodesics
  in last passage percolation.
\newblock {\em arXiv preprint arXiv:2101.01717}, 2021.

\bibitem{BB23}
Riddhipratim Basu and Manan Bhatia.
\newblock A peano curve from mated geodesic trees in the directed landscape.
\newblock {\em arXiv preprint arXiv:2304.03269}, 2023.

\bibitem{BBF22}
Riddhipratim Basu, Ofer Busani, and Patrik~L. Ferrari.
\newblock On the exponent governing the correlation decay of the
  $\text{Airy}_1$ process.
\newblock {\em Communications in Mathematical Physics}, pages 1171--1211, 2023.

\bibitem{BG18}
Riddhipratim Basu and Shirshendu Ganguly.
\newblock Time correlation exponents in last passage percolation.
\newblock In M.E. Vares, R.~Fern\'andez, L.R. Fontes, and C.M. Newman, editors,
  {\em In and Out of Equilibrium 3: Celebrating Vladas Sidoravicius}, volume~77
  of {\em Progress in Probability}. Birkh{\"a}user, 2021.

\bibitem{BGHH22}
Riddhipratim Basu, Shirshendu Ganguly, Alan Hammond, and Milind Hegde.
\newblock Interlacing and scaling exponents for the geodesic watermelon in last
  passage percolation.
\newblock {\em Comm. Math. Phys.}, 393(3):1241--1309, 2022.

\bibitem{BGZ21}
Riddhipratim Basu, Shirshendu Ganguly, and Lingfu Zhang.
\newblock Temporal correlation in last passage percolation with flat initial
  condition via brownian comparison.
\newblock {\em Communications in Mathematical Physics}, 383, 2021.

\bibitem{BHS18}
Riddhipratim Basu, Christopher Hoffman, and Allan Sly.
\newblock Nonexistence of bigeodesics in planar exponential last passage
  percolation.
\newblock {\em Communications in Mathematical Physics}, 389, 2022.

\bibitem{BSS19}
Riddhipratim Basu, Sourav Sarkar, and Allan Sly.
\newblock Coalescence of geodesics in exactly solvable models of last passage
  percolation.
\newblock {\em Journal of Mathematical Physics}, 60(9):093301, 2019.

\bibitem{BSS14}
Riddhipratim Basu, Vladas Sidoravicius, and Allan Sly.
\newblock Last passage percolation with a defect line and the solution of the
  {S}low {B}ond {P}roblem.
\newblock Preprint arXiv 1408.3464.

\bibitem{BSS3}
Riddhipratim Basu, Vladas Sidoravicius, and Allan Sly.
\newblock Rotationally invariant first passage percolation: Breaking the
  $n/\log n$ variance barrier.
\newblock In preparation.

\bibitem{BSS2}
Riddhipratim Basu, Vladas Sidoravicius, and Allan Sly.
\newblock Rotationally invariant first passage percolation: Models and
  examples.
\newblock In preparation.

\bibitem{BR08}
Michel Bena{\"\i}m and Rapha{\"e}l Rossignol.
\newblock Exponential concentration for first passage percolation through
  modified poincar{\'e}inequalities.
\newblock {\em Annales de l'Institut Henri Poincar{\'e}, Probabilit{\'e}s et
  Statistiques}, 44(3):544--573, 6 2008.

\bibitem{BKS04}
Itai Benjamini, Gil Kalai, and Oded Schramm.
\newblock First passage percolation has sublinear distance variance.
\newblock {\em Ann. Probab.}, 31(4):1970--1978, 10 2003.

\bibitem{Cha08}
Sourav Chatterjee.
\newblock Chaos, concentration, and multiple valleys.
\newblock Arxiv 0810.4221, 2008.

\bibitem{Cha11}
Sourav Chatterjee.
\newblock The universal relation between scaling exponents in first-passage
  percolation.
\newblock {\em Ann. Math. (2)}, 177(2):663--697, 2013.

\bibitem{CD81}
J.~Theodore Cox and Richard Durrett.
\newblock {Some Limit Theorems for Percolation Processes with Necessary and
  Sufficient Conditions}.
\newblock {\em The Annals of Probability}, 9(4):583 -- 603, 1981.

\bibitem{DHS14}
Michael Damron, Jack Hanson, and Philippe Sosoe.
\newblock {Subdiffusive concentration in first passage percolation}.
\newblock {\em Electronic Journal of Probability}, 19:1 -- 27, 2014.

\bibitem{DHS13}
Michael Damron, Jack Hanson, and Philippe Sosoe.
\newblock Sublinear variance in first-passage percolation for general
  distributions.
\newblock {\em Probability Theory and Related Fields}, 163(1):223--258, 2015.

\bibitem{DW16}
Michael Damron and Xuan Wang.
\newblock {Entropy reduction in Euclidean first-passage percolation}.
\newblock {\em Electronic Journal of Probability}, 21:1 -- 23, 2016.

\bibitem{DEP21}
Barbara Dembin, Dor Elboim, and Ron Peled.
\newblock Coalescence of geodesics and the {BKS} midpoint problem in planar
  first-passage percolation.
\newblock Preprint, arXiv:2204.02332.

\bibitem{GH23}
Shirshendu Ganguly and Milind Hegde.
\newblock Optimal tail exponents in general last passage percolation via
  bootstrapping \& geodesic geometry.
\newblock {\em Probability Theory and Related Fields}, 186(1):221--284, 2023.

\bibitem{GZ22}
Shirshendu Ganguly and Lingfu Zhang.
\newblock Cdiscrete geodesic local time converges under {KPZ} scaling.
\newblock {\em arXiv preprint arXiv:2212.09707}, 2021.

\bibitem{HW65}
J.~M. Hammersley and D.~J.~A. Welsh.
\newblock First-passage percolation, subadditive processes, stochastic
  networks, and generalized renewal theory.
\newblock In Jerzy Neyman and Lucien~M. Le~Cam, editors, {\em Bernoulli 1713
  Bayes 1763 Laplace 1813: Anniversary Volume}, pages 61--110. 1965.

\bibitem{HS18}
Alan Hammond and Sourav Sarkar.
\newblock Modulus of continuity for polymer fluctuations and weight profiles in
  poissonian last passage percolation.
\newblock {\em Electron. J. Probab.}, 25:38 pp., 2020.

\bibitem{HNGS15}
Christian Hirsch, D~Neuh{\"a}user, C~Gloaguen, and V~Schmidt.
\newblock First passage percolation on random geometric graphs and an
  application to shortest-path trees.
\newblock {\em Advances in Applied Probability}, 47(2):328--354, 2015.

\bibitem{H00}
C~Douglas Howard.
\newblock Lower bounds for point-to-point wandering exponents in euclidean
  first-passage percolation.
\newblock {\em Journal of applied probability}, 37(4):1061--1073, 2000.

\bibitem{H01}
C~Douglas Howard.
\newblock Differentiability and monotonicity of expected passage time in
  euclidean first-passage percolation.
\newblock {\em Journal of applied probability}, 38(4):815--827, 2001.

\bibitem{HN97}
C.~Douglas Howard and Charles~M. Newman.
\newblock Euclidean models of first-passage percolation.
\newblock {\em Probability Theory and Related Fields}, 108(2):153--170, 1997.

\bibitem{HN98}
C.~Douglas Howard and Charles~M. Newman.
\newblock From greedy lattice animals to euclidean first-passage percolation.
\newblock In Maury Bramson and Rick Durrett, editors, {\em Perplexing Problems
  in Probability: Festschrift in Honor of Harry Kesten}, pages 107--119.
  Birkh{\"a}user Boston, Boston, MA, 1999.

\bibitem{HN01}
C.~Douglas Howard and Charles~M. Newman.
\newblock {Special Invited Paper: Geodesics And Spanning Tees For Euclidean
  First-Passage Percolaton}.
\newblock {\em The Annals of Probability}, 29(2):577 -- 623, 2001.

\bibitem{J00}
Kurt Johansson.
\newblock Transversal fluctuations for increasing subsequences on the plane.
\newblock {\em Probability theory and related fields}, 116(4):445--456, 2000.

\bibitem{KPZ86}
Mehran Kardar, Giorgio Parisi, and Yi-Cheng Zhang.
\newblock Dynamic scaling of growing interfaces.
\newblock {\em Phys. Rev. Lett.}, 56:889--892, 1986.

\bibitem{Kes93}
Harry Kesten.
\newblock On the speed of convergence in first-passage percolation.
\newblock {\em Ann. Appl. Probab.}, 3(2):296--338, 1993.

\bibitem{LW10}
T.~LaGatta and J.~Wehr.
\newblock A shape theorem for riemannian first-passage percolation.
\newblock {\em Journal of Mathematical Physics}, 51(5), 2010.

\bibitem{LR10}
Michel Ledoux and Brian Rider.
\newblock Small deviations for beta ensembles.
\newblock {\em Electron. J. Probab.}, 15:1319--1343, 2010.

\bibitem{LNP96}
C.~Licea, C.~M. Newman, and M.~S.~T. Piza.
\newblock Superdiffusivity in first-passage percolation.
\newblock {\em Probability Theory and Related Fields}, 106(4):559--591, 1996.

\bibitem{MSZ21}
James Martin, Allan Sly, and Lingfu Zhang.
\newblock Convergence of the environment seen from geodesics in exponential
  last-passage percolation.
\newblock {\em arXiv preprint arXiv:2106.05242}, 2021.

\bibitem{MR96}
Ronald Meester and Rahul Roy.
\newblock {\em Continuum Percolation}.
\newblock Cambridge Tracts in Mathematics. Cambridge University Press, 1996.

\bibitem{New95}
Charles~M. Newman.
\newblock A surface view of first-passage percolation.
\newblock In S.~D. Chatterji, editor, {\em Proceedings of the International
  Congress of Mathematicians}, pages 1017--1023, Basel, 1995. Birkh{\"a}user
  Basel.

\bibitem{NP95}
Charles~M. Newman and Marcelo S.~T. Piza.
\newblock {Divergence of Shape Fluctuations in Two Dimensions}.
\newblock {\em The Annals of Probability}, 23(3):977 -- 1005, 1995.

\bibitem{Pen03}
Mathew Penrose.
\newblock {\em Random Geometric Graphs}.
\newblock Oxford University Press, 2003.

\bibitem{Pim11}
Leandro~PR Pimentel.
\newblock Asymptotics for first-passage times on delaunay triangulations.
\newblock {\em Combinatorics, Probability and Computing}, 20(3):435--453, 2011.

\bibitem{SSZ21}
Sourav Sarkar, Allan Sly, and Lingfu Zhang.
\newblock Infinite order phase transition in the slow bond {TASEP}.
\newblock {\em arXiv preprint arXiv:2109.04563}, 2021.

\bibitem{Ser97}
H.C. Serfani.
\newblock First passage percolation on the {D}elaunay graph of a
  $d$-dimensional poisson process, 1997.
\newblock Ph.D. Thesis, New York University.

\bibitem{Tal95}
Michel Talagrand.
\newblock Concentration of measure and isoperimetric inequalities in product
  spaces.
\newblock {\em Publications Math{\'e}matiques de l'Institut des Hautes
  {\'E}tudes Scientifiques}, 81(1):73--205, 1995.

\bibitem{VW90}
Mohammad~Q Vahidi-Asl and John~C Wierman.
\newblock First-passage percolation on the voronoi tessellation and delaunay
  triangulation.
\newblock In {\em Random graphs}, volume~87, pages 341--359, 1990.

\bibitem{VW92}
Mohammad~Q Vahidi-Asl and John~C Wierman.
\newblock A shape result for first-passage percolation on the voronoi
  tessellation and delaunay triangulation.
\newblock In {\em Random graphs}, volume~2, pages 247--262, 1992.

\bibitem{VW93}
MQ~Vahidi-Asl and JC~Wierman.
\newblock Upper and lower bounds for the route length of first-passage
  percolation in voronoi tessellations.
\newblock {\em Bull. Iranian Math. Soc}, 19(1):15--28, 1993.

\bibitem{Zha19}
Lingfu Zhang.
\newblock Optimal exponent for coalescence of finite geodesics in exponential
  last passage percolation.
\newblock {\em Electronic Communications in Probability}, 25(none):1 -- 14,
  2020.

\end{thebibliography}
\bibliographystyle{plain}

\end{document}